%% file: main.tex
\documentclass{article}

\input{packages}
\input{math_commands}

\icmltitlerunning{Fundamental Benefit of Alternating Updates in Minimax Optimization}

\begin{document}


\twocolumn[
\icmltitle{Fundamental Benefit of Alternating Updates in Minimax Optimization}



\icmlsetsymbol{equal}{*}

\begin{icmlauthorlist}
\icmlauthor{Jaewook Lee}{equal,kaist}
\icmlauthor{Hanseul Cho}{equal,kaist}
\icmlauthor{Chulhee Yun}{kaist}
\end{icmlauthorlist}

\icmlaffiliation{kaist}{KAIST AI, South Korea}

\icmlcorrespondingauthor{Chulhee Yun}{chulhee.yun@kaist.ac.kr}

\icmlkeywords{Minimax Optimization, Gradient Descent-Ascent, Alex-GDA}

\vskip 0.3in
]



\printAffiliationsAndNotice{\icmlEqualContribution} 

\begin{abstract}
The Gradient Descent-Ascent (GDA) algorithm, designed to solve minimax optimization problems, takes the descent and ascent steps either simultaneously (Sim-GDA) or alternately (Alt-GDA).
While Alt-GDA is commonly observed to converge faster, the performance gap between the two is not yet well understood theoretically, especially in terms of global convergence rates.
To address this theory-practice gap, we present fine-grained convergence analyses of both algorithms for strongly-convex-strongly-concave and Lipschitz-gradient objectives. 
Our new iteration complexity upper bound of Alt-GDA is strictly smaller than the lower bound of Sim-GDA; \emph{i.e.}, Alt-GDA is provably faster.
Moreover, we propose Alternating-Extrapolation GDA (Alex-GDA), a general algorithmic framework that subsumes Sim-GDA and Alt-GDA, for which the main idea is to alternately take gradients from extrapolations of the iterates.
We show that Alex-GDA satisfies a smaller iteration complexity bound, identical to that of the Extra-gradient method, while requiring less gradient computations.
We also prove that Alex-GDA enjoys linear convergence for bilinear problems, for which both Sim-GDA and Alt-GDA fail to converge at all.
\end{abstract}

\input{sec1}
\input{sec2}
\input{sec3}
\input{sec4}
\input{sec5}
\input{sec6}
\input{sec7}
\input{sec8}

\section*{Acknowledgements}
This work was supported by Institute of Information \& communications Technology Planning \& evaluation (IITP) grant (No.\ RS-2019-II190075, Artificial Intelligence Graduate School Program (KAIST)) funded by the Korea government (MSIT). The work was also supported by the National Research Foundation of Korea (NRF) grant (No.\ RS-2023-00211352) funded by the Korea government (MSIT). CY acknowledges support from a grant funded by Samsung Electronics Co., Ltd.

\section*{Impact Statement}
This paper presents work whose goal is to advance the field of Machine Learning.
There are many potential societal consequences of our work, none of which we feel must be specifically highlighted here. 

\bibliography{refs}
\bibliographystyle{icml2024}


\newpage
\onecolumn
\appendix

\input{seca}
\newpage
\input{secb1}
\input{secb2}
\input{secb3}
\input{secb4}
\input{secc1}
\input{secc2}
\input{secc3}
\input{secd1}
\input{secd2}
\input{secd3}
\input{secd4}
\input{secd5}
\input{sece1}
\input{sece2}
\input{sece3}
\newpage
\input{secf}
\newpage
\input{secg1}

\input{secg2}
\input{secg4}
\newpage
\input{sech}

\end{document}

%% file: packages.tex
\usepackage{microtype}
\usepackage{graphicx}
\usepackage{subfigure}
\usepackage{booktabs} 

\usepackage[x11names,table,xcdraw]{xcolor} 

\usepackage[pagebackref, backref=section]{hyperref}

\usepackage[accepted]{icml2024}

\usepackage{amsmath}
\usepackage{amssymb}
\usepackage{mathtools}
\usepackage{amsthm}

\usepackage{multirow}
\usepackage{bbm}
\usepackage{enumitem}
\usepackage{float}
\usepackage{thmtools, thm-restate}

\usepackage[capitalize,noabbrev]{cleveref}

\theoremstyle{plain}
\newtheorem{theorem}{Theorem}[section]
\newtheorem{proposition}[theorem]{Proposition}

\newtheorem{corollary}[theorem]{Corollary}
\newtheorem{conjecture}[theorem]{Conjecture} 
\theoremstyle{definition}
\newtheorem{definition}[theorem]{Definition}

\theoremstyle{remark}
\newtheorem{remark}[theorem]{Remark}
\newtheorem*{remark*}{Remark}

\usepackage[textsize=tiny]{todonotes}

\usepackage{listings}
\definecolor{codegreen}{rgb}{0,0.6,0}
\definecolor{codegray}{rgb}{0.5,0.5,0.5}
\definecolor{codepurple}{rgb}{0.58,0,0.82}
\definecolor{backcolour}{rgb}{0.95,0.95,0.92}
\lstdefinestyle{mystyle}{
    backgroundcolor=\color{backcolour},   
    commentstyle=\color{codegreen},
    keywordstyle=\color{magenta},
    numberstyle=\tiny\color{codegray},
    stringstyle=\color{codepurple},
    basicstyle=\ttfamily\footnotesize,
    breakatwhitespace=false,         
    breaklines=true,                 
    captionpos=b,                    
    keepspaces=true,                 
    showspaces=false,                
    showstringspaces=false,
    showtabs=false,                  
    tabsize=2
}
\crefname{lstlisting}{Listing}{Listings}
\Crefname{lstlisting}{Listing}{Listings}

\def\tightparagraph#1{\noindent\textbf{#1}~}

%% file: math_commands.tex
\usepackage{amsmath,amsfonts,amssymb,bm,bbm,pifont}

\def\norm#1{\lVert#1\rVert}
\def\inner#1{\left\langle#1\right\rangle}

\def\bignorm#1{\left\lVert#1\right\rVert}
\def\bigabs#1{\left|#1\right|}
\def\bigopen#1{\left(#1\right)}
\def\bigset#1{\left\{#1\right\}}
\def\bigclosed#1{\left[#1\right]}

\newcommand{\red}[1]{\textcolor{Firebrick3}{#1}}%
\newcommand{\green}[1]{\textcolor{Chartreuse4}{#1}}%
\newcommand{\blue}[1]{\textcolor{DodgerBlue3}{#1}}%
\newcommand\note[1]{#1}

\newcommand{\Psisim}{\Psi^{\text{Sim}}}
\newcommand{\Psialt}{\Psi^{\text{Alt}}}

\newcommand{\simgda}{\textbf{\textcolor{Chartreuse4}{Sim-GDA}}}
\newcommand{\altgda}{\textbf{\textcolor{Firebrick3}{Alt-GDA}}}
\newcommand{\alexgda}{\blue{\textbf{Alex-GDA}}}

\def\1{\bm{1}}

\def\eps{{\epsilon}}

\def\vzero{{\bm{0}}}

\def\va{{\bm{a}}}
\def\vb{{\bm{b}}}
\def\vc{{\bm{c}}}

\def\vg{{\bm{g}}}

\def\vu{{\bm{u}}}
\def\vv{{\bm{v}}}
\def\vw{{\bm{w}}}
\def\vx{{\bm{x}}}
\def\vy{{\bm{y}}}
\def\vz{{\bm{z}}}

\def\mA{{\bm{A}}}
\def\mB{{\bm{B}}}
\def\mC{{\bm{C}}}

\def\mI{{\bm{I}}}

\def\mM{{\bm{M}}}

\def\mP{{\bm{P}}}

\def\mS{{\bm{S}}}

\def\mU{{\bm{U}}}
\def\mV{{\bm{V}}}
\def\mW{{\bm{W}}}
\def\mX{{\bm{X}}}
\def\mY{{\bm{Y}}}

\DeclareMathAlphabet{\mathsfit}{\encodingdefault}{\sfdefault}{m}{sl}
\SetMathAlphabet{\mathsfit}{bold}{\encodingdefault}{\sfdefault}{bx}{n}

\def\gF{{\mathcal{F}}}

\def\gN{{\mathcal{N}}}
\def\gO{{\mathcal{O}}}

\def\sS{{\mathbb{S}}}

\newcommand{\R}{\mathbb{R}}

\DeclareMathOperator{\nullspace}{null}

\DeclareMathOperator{\row}{row}
\DeclareMathOperator{\rank}{rank}
\DeclareMathOperator{\diag}{diag}
\DeclareMathOperator{\Uniform}{Uniform}

%% file: sec1.tex
\vspace*{-10pt}
\section{Introduction}
\vspace*{-5pt}
\label{sec:1}

\begin{figure}[ht]
    \centering
    \includegraphics[width=0.85\linewidth]{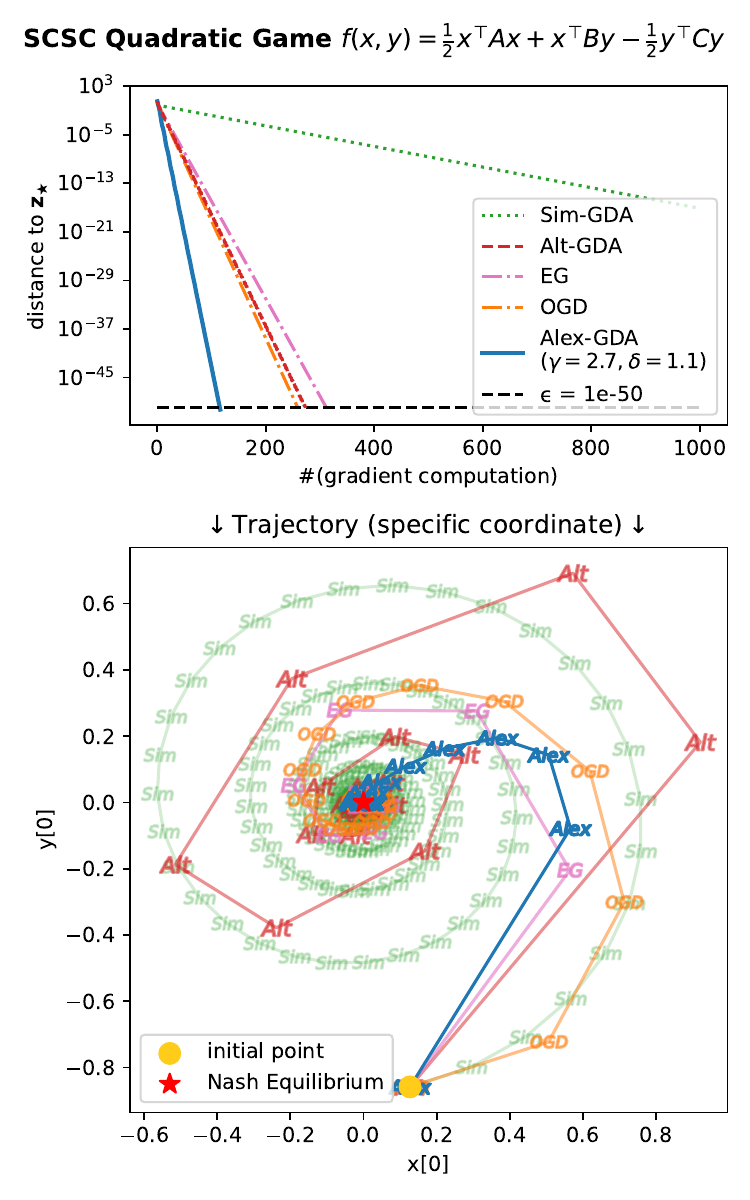}
    \vspace*{-16pt}
    \caption{(\textbf{Top}) Comparing the convergence speeds of algorithms: \simgda{}, \altgda{}, {\color{magenta}\bfseries EG}, {\color{orange}\bfseries OGD} and \alexgda{}. 
    (\textbf{Bottom}) Trajectory of the algorithms. (Partial visualization. Originally, the trajectory is $6$-dimensional since $d_x=d_y=3$).}
    \label{fig:scsc}
    \vspace*{-15pt}
\end{figure}

The \textit{minimax problem} aims to solve:
\begin{align}
    \min_{\vx \in \R^{d_x}} \max_{\vy \in \R^{d_y}} f(\vx, \vy). \label{eq:minimax}
\end{align}
This has been popularized since the work by \citet{neumann1928zur} and is widely studied in mathematics, economics, computer science, and machine learning.
Particularly, in modern machine learning, there are many important settings which fall within problem~\eqref{eq:minimax}, including but not limited to generative adversarial networks (GANs)  \cite{arjovsky2017wasserstein, goodfellow2020generative, heusel2017gans}, adversarial training and robust optimization \cite{latorre2023finding, madry2018towards, sinha2018certifiable, yu2022fast}, reinforcement learning \cite{li2019robust}, and area-under-curve (AUC) maximization \cite{liu2020stochastic, ying2016stochastic, yuan2021large}.

The simplest baseline algorithm for solving minimax problems is \textit{gradient descent-ascent} (GDA) \cite{demyanov1972numerical}, which naturally generalizes the idea of gradient descent for minimization problems. 
The GDA algorithm updates $\vx$ in the direction of decreasing the objective function $f$ while updating $\vy$ in the direction of increasing $f$, either simultaneously (\simgda) or alternately (\altgda). 
Unfortunately, it is not easy for both algorithms to converge to an optimal point even in a convex-concave minimax problem: 
in an unconstrained bilinear problem $\min_x \max_y xy$, for example, \simgda{} diverges all the way out while \altgda{} generates bounded but non-convergent iterates \cite{bailey2020finite,gidel2018a,gidel2019negative,zhang22near}.

To tackle the issues of vanilla GDA(s), numerous algorithms have been introduced and analyzed for smooth minimax problems, including Extra-gradient (\textbf{EG}) \cite{korpelevich1976extragradient}, Optimistic Gradient Descent (\textbf{OGD}) \cite{popov1980modification}, negative momentum \cite{gidel2019negative}, and many more \cite{lee2021fast,park2022exact,yoon2021accelerated,yoon2022accelerated}. 
Although these algorithms enjoy accelerated convergence rates compared to vanilla GDA, the majority of these works focus on simultaneous updates of $\vx$ and $\vy$, mainly because of the simplicity of analysis. 
However,
in minimax problems applied in practical machine learning, it is more natural for the training procedure to work in an alternating sense.
In training GANs, for instance, the discriminator should update its weight based on the outcome of the generator, and vice versa.
Moreover, there exist substantial amounts of empirical evidence of \altgda{} exhibiting faster convergence \cite{goodfellow2020generative,mescheder2017numerics}, as we demonstrate in \cref{fig:scsc}. In contrast, we still lack a theoretical understanding of why and how much \altgda{} is faster, especially compared to \simgda{}.
To fill this gap between theory and practice, it is a timely and important subject to study which one is a winner between simultaneous and alternating updates. 

An existing work by \citet{zhang22near} proposes a theoretical explanation involving \emph{local} convergence guarantees for $\mu$-strongly-convex-strongly-concave (SCSC), $L$-Lipschitz gradient functions.
Their results constructively explain that \altgda{} (of iteration complexity $\tilde\gO(\kappa)$) has a faster convergence rate than \simgda{} ($\tilde\gO(\kappa^2)$), where $\kappa={L}/{\mu}$ is the condition number of the problem. However,
their results are confined to guaranteeing \emph{local} convergence rates,
\note{which is only valid after enough iterations.}

Overall, this raises the following question:
\begin{align}
\begin{aligned}
&\,\,\text{%
\textit{For minimax problems~\eqref{eq:minimax}, are \textbf{alternating} updates}
}\\
&\,\,\text{%
\textit{strictly better than \textbf{simultaneous} updates,}%
}\\
&\,\,\text{%
\textit{even in terms of \textbf{global convergence}}?%
}
\end{aligned}
\label{eq:quote}
\end{align}

\vspace{-10pt}
\subsection{Summary of Contributions}
\vspace{-5pt}

Our contributions are largely twofold. First, we eliminate the limitations of prior work by providing \emph{global} convergence guarantees that elucidate the fundamental strength of \altgda{} over \simgda{}. Second, we propose a novel algorithm called \textbf{Al}ternating-\textbf{Ex}trapolation \textbf{GDA} (\alexgda{})
that achieves an identical rate to the Extra-gradient (\textbf{EG}) method with the same number of gradient computations per iteration as \simgda{} and \altgda{}.

For the following results, we assume $(\mu_x, \mu_y)$-strongly-convex-strongly-concave (SCSC), $(L_x, L_y, L_{xy})$-Lipschitz gradient objectives with condition numbers $\kappa_x = {L_x}/{\mu_x}$, $\kappa_y = {L_y}/{\mu_y}$, and $\kappa_{xy} = {L_{xy}}/{\sqrt{\mu_x \mu_y}}$.%
\footnote{For the definitions of SCSC and Lipschitz-gradient functions, please refer to \cref{def:scsc,def:lipg}. For the definition of condition numbers $\kappa_x$, $\kappa_y$, and $\kappa_{xy}$, please refer to \cref{def:kappa}.}
In particular, we study the upper and lower bounds on the rates of the iteration complexity $K$ required to achieve $\| \vz_K - \vz_{\star} \|^2 \le \epsilon$.

\begin{itemize}
    \item In \cref{sec:3}, we prove that \simgda{} satisfies an iteration complexity rate of
    \begin{align*}
        \Theta \bigopen{(\kappa_x + \kappa_y + \kappa_{xy}^2) \cdot \log(1/\epsilon)}
    \end{align*}
    by showing tightly matching upper and lower bounds.
    Our fine-grained convergence rate highlights the fact that the term $\kappa_{xy}^2$ is the main cause of slow convergence, which previously known results do not capture.
    
    \item In \cref{sec:4}, we prove that \altgda{} satisfies an iteration complexity rate upper bound of
    \begin{align*}
        {\gO} \bigopen{\bigopen{\kappa_x + \kappa_y + \kappa_{xy} (\sqrt{\kappa_x} + \sqrt{\kappa_y})} \cdot \log(1/\epsilon)},
    \end{align*}
    which, compared to the results in \cref{sec:3}, concludes that \emph{\altgda{} is provably faster than \simgda{}.}

    \item In \cref{sec:5}, we propose a new algorithm, \textbf{Al}ternating-\textbf{Ex}trapolation \textbf{GDA} (\alexgda{}), and prove a smaller iteration complexity rate of
    \begin{align*}
        \Theta \bigopen{\bigopen{\kappa_x + \kappa_y + \kappa_{xy}} \cdot \log(1/\epsilon)}
    \end{align*}
    by showing tightly matching upper and lower bounds.
    We also show that \textbf{EG}---which requires twice the number of gradient computations per iteration---yields the same rate by showing an identical lower bound.
\end{itemize}

Next, we turn to bilinear objectives $f(\vx, \vy) = \vx^\top \mB \vy$, for which both \simgda{} and \altgda{} fail to converge.
\begin{itemize}
    \item In \cref{sec:6}, we show that \alexgda{} enjoys linear convergence with an iteration complexity rate upper bound of
    \begin{align*}
        {\gO} \bigopen{\bigopen{L_{xy} / \mu_{xy}}^2 \cdot \log(1/\epsilon)},
    \end{align*}
    where $\mu_{xy}$, $L_{xy}$ are the smallest, largest nonzero singular values of the coupling matrix $\mB$, respectively.
\end{itemize}

Long story short, our results altogether answer the ground-setting question~\eqref{eq:quote} in the positive. For the optimization community|we believe that our fundamental comparison between simultaneous and alternating updates could provide fruitful insights for future investigations to unveil new rate-optimal algorithms by exploiting alternating updates.

%% file: sec2.tex
\section{Preliminaries}
\label{sec:2}

\tightparagraph{Notation.} 
We study unconstrained minimax problems with objective function $f : \R^{d_x} \times \R^{d_y} \rightarrow \R$, where $\vx \in \R^{d_x}$ and $\vy \in \R^{d_y}$ are the variables.
In some cases we use $\vz = (\vx, \vy) \in \R^{d_x} \times \R^{d_y}$ for notational simplicity. 
We denote by $\|\cdot\|$ the Euclidean $\ell_2$-norm for vectors and the spectral norm ({\it i.e.},  maximum singular value) for matrices.
We denote by $\langle\cdot,\cdot\rangle$ the usual inner product between vectors in Euclidean space of the same dimension.
The spectral radius ({\it i.e.}, maximum absolute eigenvalue) of a matrix $\mM$ is denoted by $\rho(\mM)$.
The letters $\gO$, $\Omega$, $\omega$, and $\Theta$ are for the conventional asymptotic notations, while the tilde notation (\textit{e.g.}, $\tilde{\gO}$ and $\tilde{\Omega}$) hides polylogarithmic factors.

\subsection{Function Class}

We first introduce the definitions we need in order to characterize the function class we will mainly focus on.
\begin{definition}[Strong convexity/concavity]
    For a given constant $\mu > 0$, we say that a differentiable function $f : \R^{d} \rightarrow \R$ is \emph{$\mu$-strongly convex} if
    \begin{align*}
        f(\vz') &\ge f(\vz) + \inner{\nabla f(\vz), \vz' - \vz} + \frac{\mu}{2} \| \vz' - \vz \|^2
    \end{align*}
    for all $\vz, \vz' \in \R^{d}$, and \emph{$\mu$-strongly concave} if $-f(\vz)$ is $\mu$-strongly convex. If the above inequality holds for $f$ (or $-f$) and $\mu = 0$, then we say that $f$ is \emph{convex} (or \textit{concave}).
\end{definition}
\begin{definition}[Strong-convex-strong-concavity]
    \label{def:scsc}
    For given constants $\mu_x, \mu_y > 0$, we say that a differentiable function $f : \R^{d_x} \times \R^{d_y} \rightarrow \R$ is \textit{$(\mu_x, \mu_y)$-strong-convex-strong-concave} (or \textit{$(\mu_x, \mu_y)$-SCSC}) if
    \begin{itemize}
        \item $f(\cdot, \vy)$ is $\mu_x$-strongly convex for all $\vy \in \R^{d_y}$,
        \item $f(\vx, \cdot)$ is $\mu_y$-strongly concave for all $\vx \in \R^{d_x}$.
    \end{itemize}
    If $\mu_x=\mu_y=0$, we say that $f$ is \emph{convex-concave}.
\end{definition}

\begin{definition}[Lipschitz gradients]
    \label{def:lipg}
    For given constants $L_x, L_y \ge 0$ and $L_{xy} \ge 0$, we say that a differentiable function $f : \R^{d_x} \times \R^{d_y} \rightarrow \R$ has \emph{$(L_x, L_y, L_{xy})$-Lipschitz gradients} if 
    \begin{align*}
        \| \nabla_{\vx} f(\vx', \vy) - \nabla_{\vx} f(\vx, \vy) \| &\le L_x \| \vx' - \vx \|,
    \end{align*}
    for all $\vx, \vx' \in \R^{d_x}$ and $\vy \in \R^{d_y}$,
    \begin{align*}
        \| \nabla_{\vy} f(\vx, \vy') - \nabla_{\vy} f(\vx, \vy) \| &\le L_y \| \vy' - \vy \|,
    \end{align*}
    for all $\vy, \vy' \in \R^{d_y}$ and $\vx \in \R^{d_x}$, and
    \begin{align*}
        \| \nabla_{\vx} f(\vx, \vy') - \nabla_{\vx} f(\vx, \vy) \| &\le L_{xy} \| \vy' - \vy \|, \\
        \| \nabla_{\vy} f(\vx', \vy) - \nabla_{\vy} f(\vx, \vy) \| &\le L_{xy} \| \vx' - \vx \|
    \end{align*}
    for all $\vx, \vx' \in \R^{d_x}$ and $\vy, \vy' \in \R^{d_y}$.
\end{definition}
For SCSC and Lipschitz-gradient objective functions, the convergence rates of algorithms usually depend on the ratio between the parameters $\mu_x, \mu_y$ and $L_x, L_y, L_{xy}$, which we often refer to as the \textit{condition number}.
\begin{definition}[Condition numbers]
    \label{def:kappa}
    For given constants $0 < \mu_{x} \le L_{x}$, $0 < \mu_{y} \le L_{y}$, and $L_{xy} \ge 0$, we define the \textit{condition numbers} as $\kappa_x := L_x / \mu_x$, $\kappa_y := L_y / \mu_y$, and $\kappa_{xy} := L_{xy} / \sqrt{\mu_x \mu_y}$.
\end{definition}

The definitions of $\kappa_x$ and $\kappa_y$ are completely analogous to the definition widely used in convex optimization literature, and we have $\kappa_x, \kappa_y \ge 1$ since $\mu_{x} \le L_{x}$, $\mu_{y} \le L_{y}$.
The number $\kappa_{xy} \ge 0$ additionally takes into account how the coupling between the two variables can affect the convergence speed. 

\begin{definition}[Function class]
    For $0 < \mu_{x} \le L_{x}$, $0 < \mu_{y} \le L_{y}$, and $L_{xy} \ge 0$, we define $\gF (\mu_x, \mu_y, L_x, L_y, L_{xy})$ as the function class containing all $f : \R^{d_x} \times \R^{d_y} \rightarrow \R$ that are (i) \textit{twice-differentiable}, (ii) \textit{$(\mu_x, \mu_y)$-SCSC}, and (iii) \textit{has $(L_x, L_y, L_{xy})$-Lipschitz gradients}.
\end{definition}

Considering the minimax problem as in (\ref{eq:minimax}), the optimal solution is often characterized as in \cref{def:nash}.

\begin{definition}
    \label{def:nash}
    A \textit{Nash equilibrium} of a function $f : \R^{d_x} \times \R^{d_y} \rightarrow \R$ is defined as a point $(\vx_{\star}, \vy_{\star}) \in \R^{d_x} \times \R^{d_y}$ which satisfies for all $\vx \in \R^{d_x}$ and $\vy \in \R^{d_y}$:
    \begin{align*}
        f(\vx_{\star}, \vy) \le f(\vx_{\star}, \vy_{\star}) \le f(\vx, \vy_{\star}).
    \end{align*}
\end{definition}

It is well known that if $f \in \gF(\mu_x, \mu_y, L_x, L_y, L_{xy})$, then the Nash equilibrium $(\vx_{\star}, \vy_{\star})$ of $f$ uniquely exists (see \citet{zhang22near}).

\subsection{Algorithms}

We focus on GDA algorithms with constant step sizes $\alpha, \beta > 0$.
In \cref{sec:3,sec:4}, we provide convergence analyses for \simgda{} and \altgda{}, shown in \cref{alg:simaltgda}.
In \cref{sec:5,sec:6}, we construct a new algorithm called \textbf{Al}ternating-\textbf{Ex}trapolation \textbf{GDA} (\alexgda{}), shown in \cref{alg:alexgda}, which we formally define later.

\begin{algorithm}[ht]
    \caption{\simgda{} and \altgda{}}
    \label{alg:simaltgda}
    \begin{algorithmic}
        \STATE {\bfseries Input:} Number of epochs $K$, step sizes $\alpha, \beta > 0$
        \STATE {\bfseries Initialize:} $(\vx_0, \vy_0) \in \R^{d_x} \times \R^{d_y}$
        \FOR{$k = 0, \dots, K-1$}
            \STATE $\vx_{k+1} = \vx_k - \alpha \nabla_{\vx} f(\vx_k, \vy_k)$
            \IF{\simgda{}}
                \STATE $\vy_{k+1} = \vy_k + \beta \nabla_{\vy} f(\green{\vx_k}, \vy_k)$
            \ELSIF{\altgda{}}
                \STATE $\vy_{k+1} = \vy_k + \beta \nabla_{\vy} f(\red{\vx_{k+1}}, \vy_k)$
            \ENDIF
        \ENDFOR
        \STATE {\bfseries Output:} $(\vx_K, \vy_K) \in \R^{d_x} \times \R^{d_y}$
    \end{algorithmic}
\end{algorithm}

\subsection{Lyapunov Function}

Originally designed for stability analysis of dynamical systems \citep{kalman1960control}, the \textit{Lyapunov function} defined as in \cref{def:lyap} \note{(sometimes referred to as the \textit{potential function})} is widely used as a strategy to obtain convergence guarantees
in optimization studies \citep{bansalgupta19,taylor18lyapunov}.

\begin{definition}[Lyapunov function]
    \label{def:lyap}
    Suppose that we have a function $f$ with optimal point $\vz_{\star}$, an initialization point $\vz_0$, and an algorithm that outputs $\vz_k$ at the $k$-th iterate.
    A \textit{Lyapunov function} is defined as a continuous function $\Psi: \R^{d} \rightarrow \R$ such that:
    \begin{itemize}
        \item $\Psi(\vz) \ge 0$ and $\Psi(\vz) = 0$ if and only if $\vz = \vz_{\star}$,
        \item $\Psi(\vz) \rightarrow \infty$ as $\norm{\vz} \rightarrow \infty$,
        \item $\Psi(\vz_{k+1}) \le \Psi(\vz_{k})$ for all $k \ge 0$.
    \end{itemize}
\end{definition}
For an algorithm that outputs $\{ \vz_k \}_{k \ge 0}$ and a Lyapunov function $\Psi$, we define the sequence $\{ \Psi_k \}_{k \ge 0}$ as $\Psi_k :=  \Psi(\vz_k)$,
which we will refer to as, with a bit of an abuse of notation, just the \textit{Lyapunov function} throughout the paper.
\begin{definition}
    \label{def:lyapvalid}
    We say that a Lyapunov function $\{\Psi_k\}_{k \ge 0}$ is \textit{valid} if it satisfies $ \Psi_k \ge A \| \vz_k - \vz_{\star} \|^2$ for all $k$ and
    for some constant $A > 0$.
\end{definition}
If we find a \textit{valid} Lyapunov function with contraction factor $r \in (0,1)$--- that is, for all $k \ge 0$, we have
$\Psi_{k+1} \le r \Psi_{k}$, then we can deduce that
\begin{align}
    K &= \gO \bigopen{ \frac{1}{1 - r} \cdot \log \frac{\Psi_0}{A \epsilon} } \label{eq:validilav}
\end{align}
iterations are sufficient to ensure $\| \vz_K - \vz_{\star} \|^2 \le \epsilon$.
We refer to $K$ as the \textit{iteration complexity}, and the rate in the right-hand side of \eqref{eq:validilav} as the \textit{iteration complexity upper bound}.

%% file: sec3.tex
\section{Convergence Analysis of Sim-GDA}
\label{sec:3}

Given an objective function $f \in \gF (\mu_x, \mu_y, L_x, L_y, L_{xy})$, for which the Nash equilibrium is unique, we define the scaled distance to the Nash equilibrium $V (\vx, \vy)$ as
\begin{align*}
    V (\vx, \vy) = \frac{1}{\alpha} \norm{\vx - \vx_{\star}}^2 + \frac{1}{\beta} \norm{\vy - \vy_{\star}}^2.
\end{align*}
For \simgda, we focus on the convergence rate in terms of the Lyapunov function $\Psisim_k = V (\vx_k, \vy_k)$.
Note that $\Psisim_k$ is always nonnegative, and is valid since we have
$A^{\text{Sim}} \| \vz_k - \vz_{\star} \|^2 \le \Psisim_k$
for $A^{\text{Sim}} = \min \bigset{\frac{1}{\alpha}, \frac{1}{\beta}}$.
\note{This potential function is a popular choice in minimax optimization or variance reduction problems with step sizes of different scales \citep{palaniappan16}.}

\subsection{Convergence Upper Bound}

\cref{thm:simgda} yields a contraction result for \simgda{}.

\begin{restatable}{theorem}{theoremsimgda}
    \label{thm:simgda}
    Suppose that $f \in \gF (\mu_x, \mu_y, L_x, L_y, L_{xy})$.
    Then, there exists a pair of step sizes $\alpha, \beta$ with
    \begin{align*}
        \alpha \mu_x = \beta \mu_y &= \Theta \bigopen{\frac{1}{\kappa_x + \kappa_y + \kappa_{xy}^2}},
    \end{align*}
    such that \simgda{} satisfies $\Psi^{\text{\emph{Sim}}}_{k+1} \le r \Psi^{\text{\emph{Sim}}}_{k}$ with
    \begin{align}
        r &= 
        \bigopen{\frac{\bigopen{\kappa_{xy} + \sqrt{ \max \bigset{\kappa_x, \kappa_y} + \kappa_{xy}^2} \ }^2 - 1}{\bigopen{\kappa_{xy} + \sqrt{ \max \bigset{\kappa_x, \kappa_y} + \kappa_{xy}^2} \ }^2 + 1}}^2.
        \label{eq:simcontrate}
    \end{align}
\end{restatable}

While we defer the proof of \cref{thm:simgda} to \cref{sec:appsimgda}, by (\ref{eq:validilav}) we can restate the convergence rate upper bound in terms of the iteration complexity as follows.

\begin{restatable}{corollary}{corsimgda}
    \label{cor:simgda}
    For the step sizes given as in \cref{thm:simgda}, \simgda{} linearly converges with iteration complexity
    \begin{align*}
        \gO \bigopen{ \bigopen{\kappa_x + \kappa_y + \kappa_{xy}^2} \cdot \log \frac{\Psi^{\text{\emph{Sim}}}_0}{A^{\emph{Sim}} \epsilon} }, 
    \end{align*}
    where $A^{\emph{Sim}} = \min \bigset{\frac{1}{\alpha}, \frac{1}{\beta}}$.
\end{restatable}

We defer the proof of \cref{cor:simgda} to \cref{sec:corsimgda}.

\tightparagraph{Comparison with Previous Work.} 
The previously known iteration complexity upper bound of \simgda{} was $\tilde{\gO} (\kappa^2)$ \citep{mescheder2017numerics,azizian20,zhang22near}, where the condition number is defined as $\kappa = \frac{\max \{L_x, L_y, L_{xy} \}}{\min \{\mu_x, \mu_y\}}$.
However, using a single condition number might oversimplify the problem and lead to loose results; for instance, if the condition numbers follow $\kappa_{x}, \kappa_{y} = \Theta (t^2)$ and $\kappa_{xy} = \Theta(t)$ for some $t$, then previous results can only guarantee up to $\tilde{\gO} (t^4)$, while \cref{cor:simgda} suggests a better rate of $\tilde{\gO} (t^2)$.
This shows that separating the condition numbers helps capture how $\kappa_{xy}$, or the \textit{interaction between $\vx$ and $\vy$}, affects convergence speed.

Meanwhile, a recent work by \citet{zamani22convergence} proposes an iteration complexity upper bound for \simgda{} of $\tilde{\gO} (\overline{\kappa} + \kappa_{xy}^2)$ for $\overline{\kappa} = \frac{\max\{L_x, L_y\}}{\min\{\mu_x, \mu_y\}}$, but the proof heavily relies on a computer-assisted method known as the Performance Estimation Problem (PEP) \citep{drori14performance}.
Our fine-grained analysis subsumes all of these previous results, and|to the best of our knowledge|is the first to clarify the exact convergence rate of \simgda{} in terms of individual condition numbers $\kappa_x$, $\kappa_y$, and $\kappa_{xy}$. 


\subsection{Convergence Lower Bound}

\cref{thm:simgdalb} provides a convergence lower bound of the iteration complexity of \simgda{} which holds for all possible step sizes $\alpha, \beta > 0$.

\begin{restatable}{theorem}{theoremsimgdalb}
    \label{thm:simgdalb}
    There exists a $6$-dimensional function $f \in \gF (\mu_x, \mu_y, L_x, L_y, L_{xy})$ with $d_x = d_y = 3$ such that for any constant step sizes $\alpha, \beta > 0$, the convergence of \simgda{} requires an iteration complexity of rate at least
    \begin{align*}
        \Omega \bigopen{ \bigopen{\kappa_x + \kappa_y + \kappa_{xy}^2} \cdot \log {\frac{1}{\epsilon}} }
    \end{align*}
    in order to have $\| \vz_K - \vz_{\star} \|^2 \le \epsilon$.
\end{restatable}

The iteration complexity lower bound in \cref{thm:simgdalb} exactly matches the upper bound in \cref{cor:simgda}, ensuring that our analysis on \simgda{} is tight (ignoring log factors). 
We defer the proof of \cref{thm:simgdalb} to \cref{sec:thmsimgdalb}.

\begin{remark*}
    \note{Unlike typical lower bounds for which the initialization is specifically chosen along with the function, our results in \cref{thm:simgdalb} works for any initialization, while the dependency on initialization is hidden in the numerator in the $\log (1/\epsilon)$ part similarly as in the upper bound results.
    All we need is an initialization point with $\gO(1)$ distance from the optimum, and the same applies to the lower bound results we present in \cref{thm:alexgdalb}.}
\end{remark*}

%% file: sec4.tex
\section{Convergence Analysis of Alt-GDA}
\label{sec:4}


For \altgda{}, the half-step iterates alternating between $\vx$ and $\vy$ updates make theoretical analysis much harder than when dealing with simultaneous updates. 
We address this by focusing on the convergence rate in terms of the following Lyapunov function (instead of just $V (\vx_k, \vy_k)$):
\begin{align*}
        \Psi^{\text{Alt}}_k &= V^{\text{Alt}} (\vx_{k}, \vy_{k}) + V^{\text{Alt}} (\vx_{k+1}, \vy_{k}) \\
        &\phantom{=} - \alpha (1 - \alpha L_x) \norm{\nabla_{\vx} f(\vx_k, \vy_k)}^2,
\end{align*}
where $V^{\text{Alt}} (\vx, \vy)$ is defined as
\begin{align*}
    \bigopen{\frac{1}{\alpha} - \mu_x} \norm{\vx - \vx_{\star}}^2 + \bigopen{\frac{1}{\beta} - \mu_y} \norm{\vy - \vy_{\star}}^2.
\end{align*}
Note that we capture the two-step-alternating nature of the algorithm by considering two adjacent iterates at a time, which turns out to be the key idea in the proofs.

\subsection{Convergence Upper Bound}

\cref{thm:altgda} yields a contraction result for \altgda{}.

\begin{restatable}{theorem}{theoremaltgda}
    \label{thm:altgda}
    Suppose that $f \in \gF (\mu_x, \mu_y, L_x, L_y, L_{xy})$ and we run \altgda{} with step sizes $\alpha, \beta > 0$ that satisfy
    \begin{align*}
        \alpha &\le \frac{1}{2} \cdot \min \bigset{\frac{1}{L_x}, \ \frac{\sqrt{\mu_y}}{L_{xy} \sqrt{L_x}}}, \\
        \beta &\le \frac{1}{2} \cdot \min \bigset{\frac{1}{L_y}, \ \frac{\sqrt{\mu_x}}{L_{xy} \sqrt{L_y}}}.
    \end{align*}
    Then $\Psi^{\text{\emph{Alt}}}_k$ is valid, and satisfies $\Psi^{\text{\emph{Alt}}}_{k+1} \le r \Psi^{\text{\emph{Alt}}}_{k}$ with
    \begin{align*}
        r &= \max \bigset{\frac{\frac{1}{\alpha} - \mu_x}{\frac{1}{\alpha} - 2 \beta^2 L_y L_{xy}^2}, \ \frac{\frac{1}{\beta} - \mu_y}{\frac{1}{\beta} - \alpha^2 L_x L_{xy}^2}, \ \frac{\frac{1}{\alpha} - \mu_x}{\frac{1}{\alpha}}},
    \end{align*}
    where we have $0 < r < 1$.
\end{restatable}

While we defer the proof of \cref{thm:altgda} to \cref{sec:appaltgda}, by (\ref{eq:validilav}) we can restate the convergence rate upper bound in terms of the iteration complexity as follows.

\begin{restatable}{corollary}{coraltgda}
    \label{cor:altgda}
    For step sizes given by the maximum possible values in \cref{thm:altgda}, \altgda{} linearly converges with iteration complexity
    \begin{align*}
        \gO \bigopen{ \bigopen{\kappa_x + \kappa_y + \kappa_{xy} (\sqrt{\kappa_x} + \sqrt{\kappa_y})} \cdot \log \frac{\Psi^{\emph{Alt}}_{0}}{A^{\emph{Alt}} \epsilon} },
    \end{align*}
    where $A^{\emph{Alt}} = \min \bigset{\frac{1}{2 \alpha} - \mu_x, 2 \bigopen{\frac{3}{4 \beta} - \mu_y}} > 0$.
\end{restatable}

We defer the proof of \cref{cor:altgda} to \cref{sec:coraltgda}.

Recall that for \simgda{} we have an upper bound of $\tilde{\gO} \bigopen{\kappa_x + \kappa_y + \kappa_{xy}^2}$, and a lower bound which shows that this rate cannot be improved.
Comparing this with \cref{cor:altgda}, we can conclude that the convergence rate of \altgda{} is faster than \simgda{}.

\paragraph{Comparison with Sim-GDA.}
\note{
Our fine-grained analysis clarifies how the dependence of the convergence speed of \simgda{} and \altgda{} on $\kappa_{xy}$, corresponding to the \textit{interaction between $\boldsymbol{x}$ and $\boldsymbol{y}$}, are different from each other.
If $\kappa_{xy} = O(\sqrt{\kappa_x} + \sqrt{\kappa_y})$, then the diagonal blocks of the Hessian dominate, for which both \simgda{} and \altgda{} exhibit similar convergence dynamics to plain GD.
If $\kappa_{xy} = \omega (\sqrt{\kappa_x} + \sqrt{\kappa_y})$, i.e., the off-diagonal block dominates, then the relatively large interaction between $\boldsymbol{x}$ and $\boldsymbol{y}$ slows down convergence.
Our results show that \altgda{} is capable of faster convergence essentially because its dependency on $\kappa_{xy}$ is of smaller order.}

\paragraph{Comparison with Local Analysis.} \citet{zhang22near} show that the \textit{local} convergence rates of \simgda{} and \altgda{} are $\tilde{\gO} (\kappa^2)$ and $\tilde{\gO} (\kappa)$, respectively, where $\kappa = \frac{\max \{L_x, L_y, L_{xy} \}}{\min \{\mu_x, \mu_y\}}$.
Such kinds of \textit{local} convergence rates of operators, including GDA iterates, rely on (the spectral radius of) the Jacobian matrix of the operator at the optimum \citep{bertsekas99nonlinear} and require that the iterates are in a small neighborhood around the optimum, or|for gradient methods|that the objective function is \textit{quadratic}, so that the Jacobian is constant and the same spectral arguments hold everywhere in the domain.
In contrast, \cref{cor:simgda,cor:altgda} both show \textit{global} convergence rates for all initialization and SCSC objectives without such assumptions.


While we can see that \cref{cor:simgda} naturally subsumes the local convergence rate $\tilde{\gO} (\kappa^2)$, it turns out that \cref{cor:altgda} is analogous to $\tilde{\gO} (\kappa^{3/2})$, which is has a gap of $\kappa^{1/2}$ with the local convergence rate of $\tilde{\gO} (\kappa)$ by \citet{zhang22near}. 
Viewing the local convergence result as a global convergence bound for the smaller class of \textit{quadratic} SCSC functions, 
we believe that there may exist a \textit{non-quadratic} function for which \altgda{} requires an iteration complexity of $\tilde \omega(\kappa)$, \note{which we discuss in detail in \cref{conj:alt}.}


%% file: sec5.tex
\section{Alternating-Extrapolation GDA}
\label{sec:5}

A natural way of unifying the baseline algorithms \simgda{} and \altgda{} is to think of taking a \textit{linear combination} between the two.
That is, we can write:
\begin{align}
    \begin{aligned}
    \vx_{k+1} &= \vx_{k} - \alpha \nabla_{\vx} f(\vx_{k}, \vy_{k}), \\
    \blue{\tilde{\vx}_{k+1}} &= (1 - \gamma) \green{\vx_k} + \gamma \red{\vx_{k+1}}, \\
    \vy_{k+1} &= \vy_k + \beta \nabla_{\vy} f(\blue{\tilde{\vx}_{k+1}}, \vy_k).
    \end{aligned}
    \label{eq:greenred}
\end{align}
Note that this formulation provides an \emph{interpolation} between \simgda{} ($\gamma = 0$) and \altgda{} ($\gamma = 1$). In the previous sections, we demonstrated a provable gap in the iteration complexity between the two endpoints $\gamma = 0$ and $1$; this motivates us to consider an \emph{extrapolation} to $\gamma > 1$ and see if we can achieve a further speed-up. 

However, if we extrapolate the $\vx$ side alone, the update equations for $\vx$ and $\vy$ will no longer be of the same form.
By symmetrizing the $\vx$ and $\vy$ sides, we now obtain the following general framework:
\begin{align*}
    \vx_{k+1} &= \vx_{k} - \alpha \nabla_{\vx} f(\vx_{k}, \tilde{\vy}_{k}), \\
    \tilde{\vx}_{k+1} &= (1 - \gamma) \vx_{k} + \gamma \vx_{k+1}, \\
    \vy_{k+1} &= \vy_{k} + \beta \nabla_{\vy} f(\tilde{\vx}_{k+1}, \vy_{k}), \\
    \tilde{\vy}_{k+1} &= (1 - \delta) \vy_{k} + \delta \vy_{k+1},
\end{align*}
where $\tilde{\vx}_{k+1}$ and $\tilde{\vy}_{k+1}$ are the points where we compute the gradients, and $\gamma, \delta \geq 0$ are hyperparameters. 
Notice that choosing $(\gamma,\delta) = (0,1)$ recovers \simgda{} and $(\gamma,\delta) = (1,1)$ corresponds to \altgda{}.

We can rewrite our updates in terms of gradient updates (\cref{alg:alexgda}). We name our algorithm framework \textbf{Al}ternating-\textbf{Ex}trapolation \textbf{GDA} (\alexgda{}), after the fact that our analysis mainly focuses on the case $\gamma, \delta > 1$ in which we compute gradients using \emph{extrapolated} iterates, and we make alternating updates between $\vx$ and $\vy$.

\begin{algorithm}[ht]
    \caption{\textbf{Al}ternating-\textbf{Ex}trapolation \textbf{GDA} (\alexgda{})}
    \label{alg:alexgda}
    \begin{algorithmic}
        \STATE {\bfseries Input:} Number of epochs $K$, step sizes $\alpha, \beta > 0$, 
        \STATE hyperparameters $\gamma, \delta \ge 0$
        \STATE {\bfseries Initialize:} $(\vx_0, \vy_0) \in \R^{d_x} \times \R^{d_y}$ and $\tilde{\vy}_0 = \vy_0 \in \R^{d_y}$
        \FOR{$k = 0, \dots, K-1$}
            \STATE $\vx_{k+1} = \vx_k - \alpha \nabla_{\vx} f(\vx_k, \blue{\tilde{\vy_k}})$
            \STATE $\tilde{\vx}_{k+1} = \vx_k - \gamma \alpha \nabla_{\vx} f(\vx_k, \blue{\tilde{\vy}_k})$
            \STATE $\vy_{k+1} = \vy_k + \beta \nabla_{\vy} f(\blue{\tilde{\vx}_{k+1}}, \vy_k)$
            \STATE $\tilde{\vy}_{k+1} = \vy_k + \delta \beta \nabla_{\vy} f(\blue{\tilde{\vx}_{k+1}}, \vy_k)$
        \ENDFOR
        \STATE {\bfseries Output:} $(\vx_K, \vy_K) \in \R^{d_x} \times \R^{d_y}$
    \end{algorithmic}
\end{algorithm}

\paragraph{Initialization.}
Some careful readers might notice that the first step of \alexgda{} is a bit different from the rest of the iterations; for $k = 0$ we set $\tilde{\vy}_0 = \vy_0$, whereas we use $\tilde{\vy}_k = \vy_k + (\delta - 1) \beta \nabla_{\vy} f(\tilde{\vx}_{k}, \vy_{k-1})$ for all subsequent steps ($k \ge 1$).
This requires a bit more careful analysis, just as in how we define the Lyapunov function for \alexgda{}:
\begin{align*}
    &\Psi_k^{\text{Alex}} = V (\vx_k, \vy_k) + V (\vx_{k+1}, \vy_k) \\
    &- \alpha \norm{\nabla_{\vx} f(\vx_k, \tilde{\vy}_k)}^2 + (\delta - 1) \beta \norm{\nabla_{\vx} f(\tilde{\vx}_k, \vy_{k-1})}^2 \\
    &+ \frac{(\gamma - 1) (\delta - 1) \alpha \beta}{1 - \alpha \mu_x} \cdot L_{xy} \sqrt{\frac{\mu_y}{\mu_x}} \cdot \norm{\nabla_{\vx} f(\vx_{k-1}, \tilde{\vy}_{k-1})}^2
\end{align*}
for $k \ge 1$, and
\begin{align*}
    &\Psi_0^{\text{Alex}} = V (\vx_0, \vy_0) + V (\vx_{1}, \vy_0) - \alpha \norm{\nabla_{\vx} f(\vx_0, \tilde{\vy}_0)}^2 \\
    &+ \frac{(\gamma - 1) (\delta - 1) \alpha \beta}{(1 - \alpha \mu_x)(1 - \beta \mu_y)} \cdot L_{xy} \sqrt{\frac{\mu_y}{\mu_x}} \cdot \norm{\nabla_{\vx} f(\vx_{0}, \tilde{\vy}_{0})}^2
\end{align*}
for $k = 0$.

\subsection{Convergence Upper Bound}

\cref{thm:alexgda} yields a contraction result for \alexgda{}.

\begin{restatable}{theorem}{theoremalexgda}
    \label{thm:alexgda}
    Suppose that $f \in \gF (\mu_x, \mu_y, L_x, L_y, L_{xy})$ and we run \alexgda{} with $\gamma, \delta > 1$ and step sizes $\alpha, \beta > 0$ that satisfy
    \begin{align*}
        \alpha &\le C \cdot \min \bigset{\frac{1}{L_x}, \ \frac{\sqrt{\mu_y}}{L_{xy} \sqrt{\mu_x}}}, \\
        \beta &\le C \cdot \min \bigset{\frac{1}{L_y}, \ \frac{\sqrt{\mu_x}}{L_{xy} \sqrt{\mu_y}}}.
    \end{align*}
    for some constant $C > 0$ (which only depends on $\gamma$ and $\delta$). Then $\Psi^{\text{\emph{Alex}}}_k$ is valid, and satisfies $\Psi^{\text{\emph{Alex}}}_{k+1} \le r \Psi^{\text{\emph{Alex}}}_{k}$ with
    \begin{align*}
        r &= \max \bigset{1 - \alpha \mu_x, 1 - \beta \mu_y}.
    \end{align*}
\end{restatable}

While we defer the proof of \cref{thm:alexgda} to \cref{sec:appalexgda}, by (\ref{eq:validilav}) we can restate the convergence rate upper bound in terms of the iteration complexity as follows.
\begin{restatable}{corollary}{coralexgda}
    \label{cor:alexgda}
    For step sizes given by the maximum possible values in \cref{thm:alexgda}, \alexgda{} linearly converges with iteration complexity
    \begin{align*}
        \gO \bigopen{ \bigopen{\kappa_x + \kappa_y + \kappa_{xy}} \cdot \log \frac{\Psi^{\emph{Alex}}_{0}}{A^{\emph{Alex}} \epsilon} },
    \end{align*}
    where $A^{\emph{Alex}} = \min \bigset{\frac{1}{2 \alpha}, \frac{1}{\beta}} > 0$.
\end{restatable}

While we defer the proof of \cref{cor:alexgda} to \cref{sec:coralexgda},
we can observe that \cref{cor:alexgda} provides a stronger iteration complexity upper bound than \cref{cor:altgda}.

\subsection{Convergence Lower Bound}

\cref{thm:alexgdalb} provides a convergence lower bound of the iteration complexity of \alexgda{} which holds for all possible step sizes $\alpha, \beta > 0$.

\begin{restatable}{theorem}{thmalexgdalb}
    \label{thm:alexgdalb}
    There exists a $6$-dimensional function $f \in \gF (\mu_x, \mu_y, L_x, L_y, L_{xy})$ with $d_x = d_y = 3$ such that for any constant step sizes $\alpha, \beta > 0$, the convergence of \alexgda{} with $\gamma, \delta > 1$ requires an iteration complexity of
    \begin{align*}
        \Omega \bigopen{ \bigopen{\kappa_x + \kappa_y + \kappa_{xy}} \cdot \log \frac{1}{\epsilon} }
    \end{align*}
    in order to have $\| \vz_K - \vz_{\star} \|^2 \le \epsilon$.
\end{restatable}

The iteration complexity rate in \cref{thm:alexgdalb} exactly matches the upper bound in \cref{cor:alexgda}, which ensures that our analysis on \alexgda{} is tight (ignoring log factors).
We defer the proof of \cref{thm:alexgdalb} to \cref{sec:thmalexgdalb}.

\subsection{Comparison with EG}
\label{sec:comparison_eg}

Here we compare \alexgda{} to the Extra-gradient (\textbf{EG}) method \citep{korpelevich1976extragradient}, an algorithm based on \textit{simultaneous} updates of the form:
\begin{align*}
    &\left.\begin{aligned}
        \vx_{k+\frac{1}{2}} &= \vx_{k} - \alpha \nabla_{\vx} f(\vx_{k}, \vy_{k}), \\
    \vy_{k+\frac{1}{2}} &= \vy_{k} + \beta \nabla_{\vy} f(\vx_{k}, \vy_{k}),
    \end{aligned}\right\}\text{\footnotesize exploration steps} \\
    &\left.\begin{aligned}
        \vx_{k+1} &= \vx_{k} - \alpha \nabla_{\vx} f(\vx_{k+\frac{1}{2}}, \vy_{k+\frac{1}{2}}), \\
    \vy_{k+1} &= \vy_{k} + \beta \nabla_{\vy} f(\vx_{k+\frac{1}{2}}, \vy_{k+\frac{1}{2}}).
    \end{aligned}\right\}\text{\footnotesize update steps}
\end{align*}
It is known by \citet{mokhtari2019a} that \textbf{EG} converges with iteration complexity $\tilde{\gO} (\kappa)$, where $\kappa = \frac{\max \{L_x, L_y, L_{xy}\}}{\min \{\mu_x, \mu_y\}}$.
While \textbf{EG} is famous for its simplicity and fast convergence, we can show that \textbf{EG} must satisfy the same lower bound with \alexgda{} via the following proposition.

\begin{restatable}{proposition}{propeglb}
    \label{prop:eglb}
    There exists a $6$-dimensional function $f \in \gF (\mu_x, \mu_y, L_x, L_y, L_{xy})$ with $d_x = d_y = 3$ such that for any constant step sizes $\alpha, \beta > 0$, the convergence of \textbf{EG} requires an iteration complexity of rate at least
    \begin{align*}
        \Omega \bigopen{ \bigopen{\kappa_x + \kappa_y + \kappa_{xy}} \cdot \log \frac{1}{\epsilon} }
    \end{align*}
    in order to have $\| \vz_K - \vz_{\star} \|^2 \le \epsilon$.
\end{restatable}

We defer the proof of \cref{prop:eglb} to \cref{sec:propeglb}.

By comparing \cref{prop:eglb} with the upper (and lower) bound for \alexgda{}, it is clear that \textbf{EG} \textit{cannot} be strictly faster than \alexgda{} in terms of iteration complexity rates.
Moreover, \alexgda{} requires only two gradient values (one for $\vx$, $\vy$ each) per a single iteration, while \textbf{EG} needs to perform exactly twice the amount of computations (two for $\vx$, $\vy$ each).
Nevertheless, \alexgda{} is provably as fast as \textbf{EG}, and in fact, it showcases faster \textit{empirical} convergence compared to \textbf{EG} as shown in \cref{fig:scsc}.

In \cref{sec:a}, we also compare \alexgda{} with another well-known baseline algorithm, Optimistic Gradient Descent (\textbf{OGD}) \citep{popov1980modification}.

%% file: sec6.tex
\section{Alex-GDA Converges on Bilinear Problems}
\label{sec:6}


One drawback shared by \simgda{} and \altgda{} is that both algorithms fail to converge for simple unconstrained bilinear problems of the form $\min_\vx \max_\vy f(\vx, \vy) = \vx^\top \mB \vy$ \citep{gidel2019negative}, an important special case of a convex-concave but non-SCSC problem with Lipschitz gradients.

Surprisingly, we show that \alexgda{},
on the other hand,
\textit{does} converge on bilinear problems. In order to present the result, we define $\mu_{xy}$ as the smallest \textit{nonzero} singular value of $\mB$. Note that it is natural to assume the existence of nonzero singular values---if not, then $\mB = \vzero$, and the objective is constantly zero. Analogously to previous definitions, we choose $L_{xy}$ as the largest singular value of $\mB$. 

We first characterize the exact condition for convergent step sizes of \alexgda{} on bilinear problems. 
Interestingly, it allows a larger range of parameters $\gamma$ and $\delta$: we no longer require $\gamma > 1$ and $\delta > 1$ here.


\begin{restatable}{theorem}{thmalexgdabilinearconvergentstepsize}
    \label{thm:alexgdabilinearconvergentstepsize}
    With a proper choice of step sizes $\alpha$ and $\beta$, \alexgda{} linearly converges to a Nash equilibrium of a bilinear problem if and only if $\gamma+\delta>2$.
    In this case, the exact conditions for convergent step sizes $\alpha$ and $\beta$ are:
    \begin{align*}
        \begin{cases}
            \alpha\beta \!<\! \frac{4}{(2\gamma-1)(2\delta-1)L_{xy}^2}, \!\! & \! \text{if } 4\gamma\delta\!-\!3(\gamma\!+\!\delta) \!+\! 2 \ge 0,\\
            \alpha\beta \!<\! \frac{\gamma+\delta-2}{-(\gamma-1)(\delta-1)(\gamma+\delta-1)L_{xy}^2}, \!\! & \! \text{if } 4\gamma\delta\!-\!3(\gamma\!+\!\delta) \!+\! 2 < 0.
        \end{cases}
    \end{align*}
\end{restatable}
We defer the proof of \cref{thm:alexgdabilinearconvergentstepsize} to \cref{sec:alexgdabilinearconvergentstepsize}.

Furthermore, if we properly choose the step size, we can obtain the iteration complexity of \alexgda{} on bilinear problems.

\begin{restatable}{theorem}{thmalexgdabilinearcomplexity}
    \label{thm:alexgdabilinearcomplexity}
    For general $\gamma\ge1$ and $\delta\ge1$ such that $\gamma+\delta>2$, If we choose the step sizes $\alpha$ and $\beta$ so that $\alpha\beta = \frac{1}{C_{\gamma,\delta} L_{xy}^2}$ where $C_{\gamma,\delta}>0$ is a constant that only depends on $\gamma$ and $\delta$,
    an iteration complexity upper bound of \alexgda{} is
    \begin{align*}
        \gO\left(\frac{C_{\gamma,\delta}}{\gamma+\delta-2} \cdot \left(\frac{L_{xy}}{\mu_{xy}}\right)^{\!2}\cdot \log\bigopen{\frac{\norm{\vw_0}^2}{\epsilon}}\right),
    \end{align*}
    where $\norm{\vw_0}^2 \!=\! \norm{\vx_0\!-\!\vx_\star}^2 \!+\! 2\norm{\vy_0\!-\!\vy_\star}^2$ and $\vz_\star \!=\!(\vx_\star,\vy_\star)$ is a uniquely determined Nash equilibrium if $\vz_0$ is given.

    If $\delta=1$, the optimal rate exponent of \alexgda{} is 
    \begin{align*}
        \lim_{k\rightarrow\infty}\frac{\norm{\vz_k-\vz_\star}}{\norm{\vz_{k-1}-\vz_\star}} = \sqrt{\frac{L_{xy}^2 - \mu_{xy}^2}{L_{xy}^2 + \mu_{xy}^2}},
    \end{align*}
    where the optimal choice of parameters are
    \begin{align*}
        \alpha\beta = \frac{2\mu_{xy}^2/L_{xy}^2}{L_{xy}^2+ \mu_{xy}^2}, \quad \gamma = 1+\frac{L_{xy}^2}{\mu_{xy}^2}.
    \end{align*}
\end{restatable}
While we defer the proof of \cref{thm:alexgdabilinearcomplexity} to \cref{sec:alexgdabilinearcomplexity}, we remark that the convergence speed depends on a new type of condition number, namely ${L_{xy}}/{\mu_{xy}} \ge 1$, which is distinct from our definition of $\kappa_{xy}$.


\subsection{Comparison with EG}
\label{sec:comparisonwitheg_bilinear}

A work by \citet{zhang2020convergence} analyzes optimal convergence rates of \textbf{EG} and several other minimax optimization algorithms on bilinear problems. They prove that the optimal rate exponent of \textbf{EG} is $\frac{L_{xy}^2 - \mu_{xy}^2}{L_{xy}^2 + \mu_{xy}^2}$, which boils down to the iteration complexity $\tilde{\gO}\bigopen{\bigopen{{L_{xy}}/{\mu_{xy}}}^2}$; it matches the iteration complexity of \alexgda{} up to constant factor.

It seems that the optimal rate exponent of \textbf{EG} is quadratically better than that of \alexgda{} with $\delta=1$. However, since \textbf{EG} takes \emph{twice} more gradient computation per iteration than \alexgda{}, the optimal \emph{gradient computation complexity} of \textbf{EG} and \alexgda{} with $\delta=1$ are exactly identical. Still, there is room for further improvement in the convergence rate of \alexgda{} by choosing $\delta$ other than 1, but we leave it as a future work.

We also compare \alexgda{} with \textbf{OGD} in \cref{sec:a}.

%% file: sec7.tex
\section{Experiments}
\label{sec:7}

The details of the experiments are illustrated in \cref{sec:g}.

\subsection{SCSC Quadratic Games} 

\note{
\begingroup
An SCSC quadratic game is a minimax problem:
\begin{align*}
    \min_{\vx} \max_{\vy} \tfrac{1}{2} \vx^\top\mA\vx + \vx^\top\mB\vy - \tfrac{1}{2}\vy^\top\mC\vy, 
\end{align*}
where $\mA$ and $\mC$ are positive definite matrices.
\endgroup
}

\tightparagraph{\note{(1) Small-scale.}}
We conducted experiments on a $(3+3)$-dimensional SCSC quadratic game to visually compare the convergence speed of the algorithms in \cref{fig:scsc}. 
We choose appropriate step sizes for each algorithm by applying grid search, regarding the number of gradient computations to arrive at an $\eps$-distant point from the Nash equilibrium, among convergent step sizes.
As shown in the figure and already observed in \citet{zhang22near}, \altgda{} beats \simgda{} in terms of the convergence rate. 
We additionally observe that the gradient complexity of \altgda{} seems comparable to that of \textbf{EG} and \textbf{OGD}.%
\footnote{In all experiments, we allow \textbf{EG} with different step sizes used at the exploration and update steps, which is of a more general formulation than the description in \cref{sec:comparison_eg}. 
We also use a more general formulation of \textbf{OGD} than that explained in \cref{sec:a}.
See \cref{sec:experiment_scsc}.}
Furthermore, with moderately tuned parameters $\gamma$ and $\delta$, our \alexgda{} achieves a convergence rate that is even faster than \textbf{EG} and \textbf{OGD}. 

\note{
\tightparagraph{(2) Higher Dimension, Extensive Comparisons.}
We run further experiments on $(100+100)$-dimensional SCSC quadratic games to extensively compare GDA, \textbf{EG}, \textbf{OGD}, and \alexgda{}.
We test both simultaneous/alternating versions and (either positive or negative) momentum variants.
In particular, we investigate five different configurations of problem parameters ($\mu_x\!=\!\mu_y, \mu_{xy}, L_x\!=\!L_y, L_{xy}$), where $\mu_{xy}$ is the smallest singular value of the matrix $\mB$.
The results are shown in \cref{tab:quadratic_extensive}.
We observe \altgda{} is much faster than \simgda{} and even faster than \simgda{} with momentum.
Among algorithms \textit{without} momentum, \alexgda{} exhibits the best gradient complexity.
If we include algorithms with momentum, a variant of \alexgda{} (\cref{alg:alexgda_momentum} in \cref{sec:experiment_scsc_large}) achieves the best performance among all compared algorithms, while the alternating \& momentum variant of \textbf{OGD} showcases the second-best performance for most of problem parameters. 
Lastly, we verify our theoretical findings by observing an increasing trend of gradient complexity in terms of condition numbers $L/\mu$($=\kappa_x=\kappa_y$) and $L_{xy}/\sqrt{\mu_x\mu_y}$($=\kappa_{xy}$), but not in terms of $L_{xy}/\mu_{xy}$ (introduced for analysis of bilinear problems).
}

\begin{table*}[ht]
\scriptsize
\addtolength{\tabcolsep}{-2pt}
\centering
\caption{\textbf{(100+100)-dimensional SCSC quadratic games.} 
We report the number of gradient computations, averaged over 30 runs. 
Every algorithm was run until the squared distance from the optimum reached $\le\!\epsilon$. 
We set $\mu_x = \mu_y = \mu$, $L_x = L_y = L$. 
Note that \textbf{Sim} means \simgda{} and \textbf{Alt} means \altgda{}.
Also, \textbf{+M} means momentum (positive or negative), while \textbf{+A} means alternating updates. 
For each row, we mark the first, second, and third places as $\ast$, $\dagger$, and $\ddagger$, respectively.}
\label{tab:quadratic_extensive}
\begin{tabular}{c|cccc|cccc|cccc|>{\columncolor[HTML]{EFEFEF}}c>{\columncolor[HTML]{EFEFEF}}c}
    \hline
    \multirow{2}{*}{$(\mu, \mu_{xy}, L, L_{xy}, \epsilon)$} & \multirow{2}{*}{\textbf{Sim}} & \textbf{Sim} & \multirow{2}{*}{\textbf{Alt}} & \textbf{Alt} & \multirow{2}{*}{\textbf{EG}} & \textbf{EG} & \textbf{EG} & \textbf{EG} & \multirow{2}{*}{\textbf{OGD}} & \textbf{OGD} & \textbf{OGD} & \textbf{OGD} & \cellcolor[HTML]{EFEFEF} & \textbf{Alex} \\
     &  & \textbf{+M} &  & \textbf{+M} &  & \textbf{+M} & \textbf{+A} & \textbf{+AM} &  & \textbf{+M} & \textbf{+A} & \textbf{+AM} & \multirow{-2}{*}{\textbf{Alex}} & \textbf{+M} \\ \hline
    $(0.1, 0.1, 1, 1, 10^{-8})$ & 1974.2 & 421.0 & 105.9 & 78.9 & 133.8 & 115.6 & 139.9 & 102.0 & 132.8 & 105.3 & 90.0 & 67.6$^\ddagger$ & 62.7$^\dagger$ & 44.7$^\ast$ \\
    $(0.1, 0.05, 1, 2, 10^{-8})$ & 7865.0 & 839.8 & 149.1 & 105.6 & 253.2 & 210.6 & 278.9 & 186.9 & 215.1 & 177.3 & 116.1 & 93.6$^\dagger$ & 100.6$^\ddagger$ & 69.1$^\ast$ \\
    $(0.01, 0.001, 1, 0.5, 10^{-4})$ & 42762.1 & 3824.3 & 394.9 & 182.0 & 291.1 & 225.9 & 380.7 & 228.4 & 281.1 & 176.1 & 182.3 & 127.7$^\dagger$ & 133.1$^\ddagger$ & 58.3$^\ast$ \\
    $(0.01, 0.01, 1, 1, 10^{-4})$ & 104220.5 & 8539.4 & 567.6 & 157.2 & 308.8 & 223.9 & 299.9 & 175.7 & 280.5 & 184.9 & 200.5 & 117.3$^\dagger$ & 138.8$^\ddagger$ & 73.2$^\ast$ \\
    $(0.01, 0.05, 1, 2, 10^{-4})$ & 416822.5 & 16719.4 & 777.4 & 149.0 & 347.5 & 253.9 & 363.6 & 231.0 & 337.6 & 213.3 & 162.0 & 108.6$^\dagger$ & 135.4$^\ddagger$ & 83.9$^\ast$ \\ \hline
\end{tabular}
\addtolength{\tabcolsep}{3pt}
\end{table*}

\subsection{\note{Generative Adversarial Networks}} 

\begin{table}[ht]
    \footnotesize
    \centering
    \caption{\textbf{WGAN-GP.} 
    We report the mean (and standard deviation) of FID scores (the lower the better).}
    \label{tab:wgan}
    \begin{tabular}{r|ccc}
    \hline
    \phantom{$y_y^k$} & \textbf{MNIST} & \textbf{CIFAR-10} & \textbf{LSUN Bedroom} \\
    \hline
    Sim-Adam  & 3.97 (1.3) & 45.0 (1.2) & 131.2 (8.4) \\
    Alt-Adam  & 1.85 (0.3) & 24.2 (1.7) & 9.0 (1.2) \\
    \rowcolor[HTML]{EFEFEF}
    Alex-Adam & \textbf{1.53 (0.3)} & \textbf{23.8 (1.5)} & \textbf{6.3 (0.6)} \\
    \hline
    \end{tabular}
\end{table}

\note{
To examine the efficacy of \alexgda{},
we train WGAN-GP \citep{arjovsky2017wasserstein,gulrajani2017improved} for the image generation task, mostly following the implementation details in \citet{heusel2017gans}.
We examine the natural combinations of Adam \citep{kingma2015adam} and (stochastic variants of) \green{\bf Sim-}/\red{\bf Alt-}/\alexgda{}, which we call Sim-/Alt-/Alex-Adam, respectively.
We highlight that \alexgda{} can be easily implemented on top of any existing base optimizers including Adam because all we need to implement additionally is a couple of extrapolation steps; we provide a brief PyTorch \citep{paszke2019pytorch} implementation of Alex-Adam for GANs in \cref{lst:alex_adam} of \cref{sec:experiment_wgan}.
We moderately tune the step sizes and the values of $\gamma$ and $\delta$.
As a result, we use $(\gamma, \delta)=(1,4)$ for MNIST \citep{deng2012mnist}, $(\gamma, \delta)=(1, 1.2)$ for CIFAR-10 \citep{krizhevsky2009learning}, and $(\gamma, \delta)=(1, 2)$ for LSUN Bedroom $64\!\times\!64$ dataset \citep{yu15lsun}.
The result is shown in \cref{tab:wgan}, where we report Fr{\'e}chet inception distance (FID) scores \citep{heusel2017gans}.
To the best of our knowledge, we achieve state-of-the-art image generation performance in terms of FID scores for MNIST and LSUN Bedroom $64\!\times\!64$ datasets with Alex-Adam.}

\note{
The experiments in \cref{tab:quadratic_extensive,tab:wgan} can be reproduced with our code available at GitHub.\footnote{\href{https://github.com/HanseulJo/Alex-GDA}{\texttt{github.com/HanseulJo/Alex-GDA}}}
}

%% file: sec8.tex
\section{Conclusion}
\label{sec:8}

We present 
global convergence rates of \simgda{} and \altgda{} on SCSC, Lipschitz-gradient objectives in terms of the condition numbers $\kappa_x, \kappa_y$, and $\kappa_{xy}$.
For \simgda{} we prove an iteration complexity of $\tilde{\Theta} (\kappa_x + \kappa_y + \kappa_{xy}^2)$, while for \altgda{} we obtain a smaller iteration complexity of $\tilde{\gO} (\kappa_x + \kappa_y + \kappa_{xy} (\sqrt{\kappa_x} + \sqrt{\kappa_y}))$.
Comparing the results, we show that \altgda{} is provably faster than \simgda{} in terms of global convergence.

Moreover, we propose a novel algorithm called \alexgda{}, inspired by an extension of \simgda{} and \altgda{} via linear extrapolation.
\alexgda{} shows a faster iteration complexity of $\tilde{\Theta} (\kappa_x + \kappa_y + \kappa_{xy})$, matching the convergence rate of \textbf{EG} with less gradient computations per iteration. 
We also show that \alexgda{} converges linearly for bilinear problems, for which \simgda{} and \altgda{} diverge. 

We believe that our results, altogether, are valuable demonstrations of \textbf{\textit{\underline{the benefit of alternating updates}}} in GDA algorithms for minimax optimization.

\paragraph{Future Work.}


\note{As an effort to check if it is possible to obtain $\gO(\kappa)$ convergence of \altgda{}, we have tried using a computer-assisted method called the performance estimation problem (PEP) \citep{drori14performance}, a powerful tool originally designed to infer \textit{tight} worst-case complexities of \textit{convex} optimization algorithms.
Based on the work by \citet{gupta24}, we devised a PEP-based tool that automatically finds the worst-case convergence rate of an algorithm by optimizing the function, step size, and performance measure altogether.
While it is known by \citet{ryu20} that the extension of such methods to \textit{minimax} optimization can only yield a possibly loose \textit{upper bound}, the estimate we obtained for \altgda{} was approximately $\gO (\kappa^{1.4})$.
Moreover, the estimated rate for \simgda{} was $\gO (\kappa^{1.99})$, which is very close to our theoretical results.
You may refer to \cref{sec:h} for more details.}



Based on these observations and the discussions about \cref{thm:altgda} at the end of \cref{sec:4}, we leave the following conjecture on the convergence lower bound of \altgda{}.

\begin{conjecture}
    \label{conj:alt}
    There exists a \textbf{non-quadratic} function $f \in \gF (\mu_x, \mu_y, L_x, L_y, L_{xy})$ such that for any constant step sizes $\alpha, \beta > 0$, the convergence of \altgda{} requires an iteration complexity of
    \begin{align*}
        \Theta \bigopen{ \bigopen{\kappa_x + \kappa_y + \kappa_{xy}(\kappa_x + \kappa_y)^p} \cdot \log \frac{1}{\epsilon} }
    \end{align*}
    for $p \in (0, \frac{1}{2})$.
\end{conjecture} 

    
Also, on top of our findings on bilinear functions in \cref{sec:6}, we also leave the following conjecture on \alexgda{} on general \textit{convex-concave} objectives for future work.
    
\begin{conjecture}
    Suppose that the objective function $f$ is convex-concave and has $(L_x, L_y, L_{xy})$-Lipschitz gradients.
    Then, we conjecture that \alexgda{} exhibits last-iterate convergence to a Nash equilibrium of $f$.
\end{conjecture} 


%% file: seca.tex
\begin{center}
\huge \bf Supplementary Material
\end{center}

\section{Comparison with OGD} 
\label{sec:a}

Here we compare \alexgda{} to the Optimistic Gradient Descent (\textbf{OGD}) method \citep{popov1980modification}, an algorithm based on \textit{simultaneous} updates of the form:
\begin{align}
    \label{eq:ogd}
    \begin{aligned}
        \vx_{k+1} &= \vx_{k} - 2 \alpha \nabla_{\vx} f(\vx_{k}, \vy_{k}) + \alpha \nabla_{\vx} f(\vx_{k-1}, \vy_{k-1}), \\
        \vy_{k+1} &= \vy_{k} + 2 \beta \nabla_{\vy} f(\vx_{k}, \vy_{k}) - \beta \nabla_{\vy} f(\vx_{k-1}, \vy_{k-1}).
    \end{aligned}
\end{align}

We remark that \textbf{OGD} takes the same amount of gradient computation as \simgda{}, \altgda{}, and \alexgda{}.
One may observe that \alexgda{} stores the previous \textit{iterates} $\vx_k$ and $\vy_k$ to compute $\tilde{\vx}_{k+1}$ and $\tilde{\vy}_{k+1}$, whereas the implementation of \textbf{OGD} requires storing the previous \emph{gradients} $\nabla_{\vx} f(\vx_{k-1}, \vy_{k-1})$ and $\nabla_{\vy} f(\vx_{k-1}, \vy_{k-1})$ instead.
As a result, while these two algorithms exploit different types of information, the memory consumption of \alexgda{} and \textbf{OGD} are identical.

As \textbf{EG}, it is also known that \textbf{OGD} converges with iteration complexity $\tilde{\gO} (\kappa)$, where $\kappa = \frac{\max \{L_x, L_y, L_{xy}\}}{\min \{\mu_x, \mu_y\}}$ \citep{mokhtari2019a}.
We show that the iteration complexity cannot be strictly better than \alexgda{} through the following proposition.

\begin{restatable}{proposition}{propogdlb}
    \label{prop:ogdlb}
    There exists a $6$-dimensional function $f \in \gF (\mu_x, \mu_y, L_x, L_y, L_{xy})$ with $d_x = d_y = 3$ such that for any constant step sizes $\alpha, \beta > 0$, the convergence of \textbf{OGD} requires an iteration complexity of rate at least
    \begin{align*}
        \Omega \bigopen{ \bigopen{\kappa_x + \kappa_y + \kappa_{xy}} \cdot \log \frac{1}{\epsilon} }
    \end{align*}
    in order to have $\| \vz_K - \vz_{\star} \|^2 \le \epsilon$.
\end{restatable}

We prove \cref{prop:ogdlb} in \cref{sec:f}.

We also compare \alexgda{} with \textbf{OGD} in terms of bilinear problem, based upon the analysis of \citet{zhang2020convergence}.
From their results, the iteration complexity of \textbf{OGD} is translated to $\tilde{\gO}\bigopen{\bigopen{{L_{xy}}/{\mu_{xy}}}^2}$; it matches to the iteration complexity of \alexgda{} up to constant factor.
In more detail, the authors proved that optimal convergence rate exponent of \textbf{OGD} of the form in \cref{eq:ogd} is approximately $1-\frac{\mu_{xy}^2}{6L_{xy}^2}$, which is 6 times slower than our optimal rate exponent $\sqrt{\frac{L_{xy}^2 - \mu_{xy}^2}{L_{xy}^2 + \mu_{xy}^2}} \approx 1-\frac{\mu_{xy}^2}{L_{xy}^2}$ proved in \cref{thm:alexgdabilinearcomplexity}.
On the other hand, \citet{zhang2020convergence} also proved that the alternating variant of \textbf{OGD}, \textit{i.e.}, Gauss-Siedel OGD (\textbf{GS-OGD}), has an optimal convergence rate exponent $\sqrt{\frac{L_{xy}^2 - \mu_{xy}^2}{L_{xy}^2 + \mu_{xy}^2}}$. It exactly matches our optimal rate of \alexgda{} with $\delta=1$. 
These facts again buttress our claim that alternating updates are beneficial in minimax optimization.

%% file: secb1.tex
\section{\texorpdfstring{Proofs used in \cref{sec:3}}{Proofs used in Section 3}} 
\label{sec:b}

Here we prove all theorems related to \simgda{} presented in \cref{sec:3}.
\begin{itemize}
    \item  In \cref{sec:appsimgda} we prove \cref{thm:simgda} which yields a contraction inequality for \simgda{}. 
    \item In \cref{sec:corsimgda} we prove \cref{cor:simgda} which derives the corresponding iteration complexity upper bound.
    \item In \cref{sec:thmsimgdalb} we prove \cref{thm:simgdalb} which yields a matching lower bound for \simgda{}.
    \item In \cref{sec:simtech} we prove technical propositions and lemmas used throughout the proofs in \cref{sec:b}.
\end{itemize}

\subsection{\texorpdfstring{Proof of \cref{thm:simgda}}{Proof of Theorem 3.1}} \label{sec:appsimgda}

Here we prove \cref{thm:simgda} of \cref{sec:3}, restated below for the sake of readability.

\theoremsimgda*

\begin{proof}

Recall that we define the Lyapunov function as
\begin{align*}
    \Psisim_k &= \frac{1}{\alpha} \| \vx_k - \vx_{\star} \|^2 + \frac{1}{\beta} \| \vy_k - \vy_{\star} \|^2.
\end{align*}
Now we will show that $\Psisim_1 \le \gamma \Psisim_0$ for any choice of initialization points $\vx_0$ and $\vy_0$ ({\it i.e.}, set $k=0$ W.L.O.G.), which directly implies $\Psisim_{k+1} \le \gamma \Psisim_k$ for all $k$.
\cref{prop:simcont} yields a one-step contraction inequality that applies to \simgda{} with $\alpha < \frac{1}{L_x}$ and $\beta < \frac{1}{L_y}$, {\it i.e.}, when the step sizes are small enough.
\begin{restatable}{proposition}{propsimcont}
    \label{prop:simcont}
    For $f \in \gF (\mu_x, \mu_y, L_x, L_y, L_{xy})$, \simgda{} with step sizes $\alpha < \frac{1}{L_x}$ and $\beta < \frac{1}{L_y}$ satisfies
    \begin{align*}
        \frac{1}{\alpha} \norm{\vx_1 - \vx_{\star}}^2 + \frac{1}{\beta} \norm{\vy_1 - \vy_{\star}}^2 &\le r \bigopen{\frac{1}{\alpha} \norm{\vx_0 - \vx_{\star}}^2 + \frac{1}{\beta} \norm{\vy_0 - \vy_{\star}}^2},
    \end{align*}
    where the contraction factor is given by
    \begin{align*}
        r &= \max \left\{ \left\| \begin{bmatrix}
            1 - \alpha L_x & - \sqrt{\alpha \beta} L_{xy} \\
            \sqrt{\alpha \beta} L_{xy} & 1 - \beta \mu_y
        \end{bmatrix} \right\|^2, \ \ \left\| \begin{bmatrix}
            1 - \alpha \mu_x & - \sqrt{\alpha \beta} L_{xy} \\
            \sqrt{\alpha \beta} L_{xy} & 1 - \beta L_y
        \end{bmatrix} \right\|^2 \right\}.
    \end{align*}
\end{restatable}
We prove \cref{prop:simcont} in \cref{sec:propsimcont}.

To find the right step sizes, we search among $\alpha, \beta$ which satisfies ${\alpha}/{\beta} = {\mu_y}/{\mu_x}$.
This allows us to reduce the problem to optimizing the choice of $\zeta$, which can be defined as
\begin{align*}
    \zeta = \alpha \mu_x = \beta \mu_y.
\end{align*}
Then the contraction factor can be rewritten as
\begin{align*}
    r &= \max \left\{ \left\| \begin{bmatrix}
        1 - \zeta \kappa_x & - \zeta \kappa_{xy} \\
        \zeta \kappa_{xy} & 1 - \zeta
    \end{bmatrix} \right\|^2, \ \ \left\| \begin{bmatrix}
        1 - \zeta & - \zeta \kappa_{xy} \\
        \zeta \kappa_{xy} & 1 - \zeta \kappa_y
    \end{bmatrix} \right\|^2 \right\}.
\end{align*}
For $\kappa \ge 1$, let us define the function $f_{\kappa} : (0, \infty) \rightarrow (0, \infty)$ as:
\begin{align}
\begin{aligned}
    f_{\kappa} (\zeta) &= \left\| \begin{bmatrix}
        1 - \zeta \kappa & - \zeta \kappa_{xy} \\
        \zeta \kappa_{xy} & 1 - \zeta
    \end{bmatrix} \right\| = \frac{\kappa - 1}{2} \cdot \zeta + \sqrt{\bigopen{1 - \frac{\kappa + 1}{2} \cdot \zeta}^2 + \zeta^2 \kappa_{xy}^2}.
\end{aligned} \label{eq:fzeta}
\end{align}
Then we can simplify as follows:
\begin{align*}
    r &= \max \bigset{\bigopen{f_{\kappa_x} (\zeta)}^2, \bigopen{f_{\kappa_y} (\zeta)}^2}.
\end{align*}

\cref{prop:normopt} characterizes the optimal choice of $\zeta$ and the optimal function value of $f_{\kappa} (\zeta)$ defined as in (\ref{eq:fzeta}).
\begin{restatable}{proposition}{propnormopt}
    \label{prop:normopt}
    For $f_{\kappa} : (0, \infty) \rightarrow (0, \infty)$ defined as in (\ref{eq:fzeta}), the minimizer $\zeta^{\star}$ is equal to
    \begin{align*}
        \zeta^{\star} &= \frac{1}{\sqrt{\kappa + \kappa_{xy}^2}} \cdot \frac{2 \bigopen{\kappa_{xy} + \sqrt{\kappa + \kappa_{xy}^2} \ }}{1 + \bigopen{\kappa_{xy} + \sqrt{\kappa + \kappa_{xy}^2} \ }^2}
    \end{align*}
    and the minimum value of $f_{\kappa}$ attained at $\zeta^{\star}$ is equal to
    \begin{align*}
        f_{\kappa} (\zeta^{\star}) &= \frac{\bigopen{\kappa_{xy} + \sqrt{\kappa + \kappa_{xy}^2} \ }^2 - 1}{\bigopen{\kappa_{xy} + \sqrt{\kappa + \kappa_{xy}^2} \ }^2 + 1}.
    \end{align*}
    Moreover, we have $f_{\kappa_x} (\zeta) \ge f_{\kappa_y} (\zeta)$ for all $\zeta \in (0, \infty)$ if and only if $\kappa_x \ge \kappa_y$.
\end{restatable}
We prove \cref{prop:normopt} in \cref{sec:propnormopt}.

If $\kappa_x \ge \kappa_y$, we choose $\alpha, \beta$ such that
\begin{align*}
    \alpha \mu_x &= \beta \mu_y = \zeta_{x}^{\star} := \frac{1}{\sqrt{\kappa_x + \kappa_{xy}^2}} \cdot \frac{2 \bigopen{\kappa_{xy} + \sqrt{\kappa_x + \kappa_{xy}^2} \ }}{1 + \bigopen{\kappa_{xy} + \sqrt{\kappa_x + \kappa_{xy}^2} \ }^2}.
\end{align*}
Note that $\zeta_{x}^{\star} = \Theta \bigopen{\frac{1}{\kappa_x + \kappa_{xy}^2}}$.
Then, since $f_{\kappa_x} (\zeta) \ge f_{\kappa_y} (\zeta)$, we have
\begin{align*}
    r &= \max \bigset{\bigopen{f_{\kappa_x} (\zeta_{x}^{\star})}^2, \bigopen{f_{\kappa_y} (\zeta_{x}^{\star})}^2} = \bigopen{f_{\kappa_x} (\zeta_{x}^{\star})}^2 = \bigopen{\frac{\bigopen{\kappa_{xy} + \sqrt{\kappa_x + \kappa_{xy}^2} \ }^2 - 1}{\bigopen{\kappa_{xy} + \sqrt{\kappa_x + \kappa_{xy}^2} \ }^2 + 1}}^2
\end{align*}
which is identical to (\ref{eq:simcontrate}) when $\kappa_x \ge \kappa_y$.

Similarly, if $\kappa_x < \kappa_y$, we choose $\alpha, \beta$ such that
\begin{align*}
    \alpha \mu_x &= \beta \mu_y = \zeta_{y}^{\star} := \frac{1}{\sqrt{\kappa_y + \kappa_{xy}^2}} \cdot \frac{2 \bigopen{\kappa_{xy} + \sqrt{\kappa_y + \kappa_{xy}^2} \ }}{1 + \bigopen{\kappa_{xy} + \sqrt{\kappa_y + \kappa_{xy}^2} \ }^2}.
\end{align*}
Note that $\zeta_{y}^{\star} = \Theta \bigopen{\frac{1}{\kappa_y + \kappa_{xy}^2}}$. 
Then, since $f_{\kappa_y} (\zeta) \le f_{\kappa_y} (\zeta)$, we have
\begin{align*}
    r &= \max \bigset{\bigopen{f_{\kappa_x} (\zeta_{y}^{\star})}^2, \bigopen{f_{\kappa_y} (\zeta_{y}^{\star})}^2} = \bigopen{f_{\kappa_y} (\zeta_{y}^{\star})}^2 = \bigopen{\frac{\bigopen{\kappa_{xy} + \sqrt{\kappa_y + \kappa_{xy}^2} \ }^2 - 1}{\bigopen{\kappa_{xy} + \sqrt{\kappa_y + \kappa_{xy}^2} \ }^2 + 1}}^2
\end{align*}
which is identical to (\ref{eq:simcontrate}) when $\kappa_x < \kappa_y$.
Note that for either case, we have that
\begin{align*}
    \alpha \mu_x &= \beta \mu_y = \Theta \bigopen{\frac{1}{\max \{\kappa_x, \kappa_y \} + \kappa_{xy}}} = \Theta \bigopen{\frac{1}{\kappa_x + \kappa_y + \kappa_{xy}}},
\end{align*}
which concludes the proof.\footnote{Note that for $a, b \ge 0$, we have $\max \{ a, b \} = \Theta (a + b)$ since $\frac{a+b}{2} \le \max \{ a, b \} \le a + b$.}
\end{proof}

%% file: secb2.tex
\subsection{\texorpdfstring{Proof of \cref{cor:simgda}}{Proof of Corollary 3.2}} \label{sec:corsimgda}

Here we prove \cref{cor:simgda} of \cref{sec:3}, restated below for the sake of readability.

\corsimgda*

\begin{proof}
    Let us define $\xi := \kappa_{xy} + \sqrt{ \max \bigset{\kappa_x, \kappa_y} + \kappa_{xy}^2}$ so that $r = \bigopen{\frac{\xi^2 - 1}{\xi^2 + 1}}^2$ by \cref{thm:simgda}. By definition we have $\xi^2 = \Theta \bigopen{\kappa_x + \kappa_y + \kappa_{xy}^2}$ and $\xi \ge 1$, which gives us
    \begin{align*}
        \frac{1}{1 - r} &= \frac{1}{1 - \bigopen{\frac{\xi^2 - 1}{\xi^2 + 1}}^2} = \frac{\bigopen{\xi^2 + 1}^2}{\bigopen{\xi^2 + 1}^2 - \bigopen{\xi^2 - 1}^2} = \frac{1}{4}\bigopen{\xi + \frac{1}{\xi}}^2 = \Theta \bigopen{\kappa_x + \kappa_y + \kappa_{xy}^2}.
    \end{align*}
    Therefore it is sufficient to run
    \begin{align*}
        K  &= \gO \bigopen{ \bigopen{\kappa_x + \kappa_y + \kappa_{xy}^2} \cdot \log \frac{\Psisim_0}{A^{\text{Sim}} \epsilon} }
    \end{align*}
    iterations to ensure that $\| \vz_K - \vz_{\star} \|^2 \le \epsilon$, where $A^{\text{Sim}} = \min \bigset{\frac{1}{\alpha}, \frac{1}{\beta}}$.
\end{proof}

\begin{remark*}
    \note{Here we present a simpler proof of \cref{cor:simgda} we discovered afterwards.
    The proof can achieve the same iteration complexity upper bound with a similar yet slightly different choice of step sizes $\alpha, \beta$.
    Compared to the one using \cref{thm:simgda}, this proof does not require complicated matrix analyses and better extends to algorithms with alternating updates, such as \altgda{} (as in \cref{prop:altstepone}) or \alexgda{} (as in \cref{prop:alexgda}).}
\end{remark*}

\subsubsection*{Step 1. Contraction Inequality}
We first prove the following proposition.
\begin{proposition}
    \label{prop:simsimple}
    For $f \in \gF (\mu_x, \mu_y, L_x, L_y, L_{xy})$ and \simgda{} with step sizes $\alpha \le \frac{1}{2L_x}$ and $\beta \le \frac{1}{2L_y}$, we have
    \begin{align}
        \frac{1}{\alpha} \| \vx_1 - \vx_{\star} \|^2 + \frac{1}{\beta} \| \vy_1 - \vy_{\star} \|^2 &\le \bigopen{\frac{1}{\alpha} - \mu_x + 2\beta L_{xy}^2} \norm{\vx_0 - \vx_{\star}}^2 + \bigopen{\frac{1}{\beta} - \mu_y + 2\alpha L_{xy}^2} \norm{\vy_0 - \vy_{\star}}^2 \label{eq:simineq}
    \end{align}
    for all $\vx_0 \in \R^{d_x}, \vy_0 \in \R^{d_y}$.
\end{proposition}
\begin{proof}
    Recall that \simgda{} takes updates of the form:
\begin{align*}
    \vx_{1} &= \vx_{0} - \alpha \nabla_{\vx} f(\vx_{0}, \vy_{0}), \\
    \vy_{1} &= \vy_{0} + \beta \nabla_{\vy} f(\vx_{0}, \vy_{0}).
\end{align*}
From this, we can deduce that
\begin{align*}
    \frac{1}{\alpha} \| \vx_1 - \vx_{\star} \|^2 &= \frac{1}{\alpha} \norm{\vx_0 - \vx_{\star}}^2 + \frac{2}{\alpha} \inner{\vx_1 - \vx_0, \vx_0 - \vx_{\star}} + \frac{1}{\alpha} \norm{\vx_1 - \vx_0}^2 \\
    &= \frac{1}{\alpha} \norm{\vx_0 - \vx_{\star}}^2 - 2 \inner{\nabla_{\vx} f(\vx_{0}, \vy_{0}), \vx_0 - \vx_{\star}} + \alpha \norm{\nabla_{\vx} f(\vx_{0}, \vy_{0})}^2, \\
    \frac{1}{\beta} \| \vy_1 - \vy_{\star} \|^2 &= \frac{1}{\beta} \norm{\vy_0 - \vy_{\star}}^2 + \frac{2}{\beta} \inner{\vy_1 - \vy_0, \vy_0 - \vy_{\star}} + \frac{1}{\beta} \norm{\vy_1 - \vy_0}^2 \\
    &= \frac{1}{\beta} \norm{\vy_0 - \vy_{\star}}^2 + 2 \inner{\nabla_{\vy} f(\vx_{0}, \vy_{0}), \vy_0 - \vy_{\star}} + \beta \norm{\nabla_{\vy} f(\vx_{0}, \vy_{0})}^2.
\end{align*}
Note that $\mu_x$-strong convexity of $f(\cdot, \vy_0)$ yields
\begin{align}
    - 2 \inner{\nabla_{\vx} f({\vx_0}, \vy_0), {\vx_0 - \vx_{\star}}} &\le - \mu_x \norm{{\vx_0 - \vx_{\star}}}^2 - 2 (f({\vx_0}, \vy_0) - f({\vx_{\star}}, \vy_0)), \label{eq:simproof1}
\end{align}
and $\mu_y$-strong concavity of $f(\vx_0, \cdot)$ yields
\begin{align}
    2 \inner{\nabla_{\vy} f(\vx_0, {\vy_0}), {\vy_0 - \vy_{\star}}} &\le - \mu_y \norm{{\vy_0 - \vy_{\star}}}^2 - 2 (f(\vx_0, {\vy_{\star}}) - f(\vx_0, {\vy_0})). \label{eq:simproof3}
\end{align}
Moreover, since $f$ is convex-concave and has Lipschitz gradients\footnote{Note that the Lipschitz gradient conditions for $L_x$ and $L_y$ are equivalent to the widely used notion of \textit{smoothness} in convex optimization literature.}, we have
\begin{align}
    - 2 (f(\vx_0, \vy_{\star}) - f(\vx_{\star}, \vy_{\star})) &\le - \frac{1}{L_x} \norm{\nabla_{\vx} f (\vx_0, \vy_{\star})}^2, \label{eq:simproof2} \\
    - 2 (f(\vx_{\star}, \vy_{\star}) - f(\vx_{\star}, \vy_0)) &\le - \frac{1}{L_y} \norm{\nabla_{\vy} f (\vx_{\star}, \vy_0)}^2. \label{eq:simproof4}
\end{align}
Applying (\ref{eq:simproof1})--(\ref{eq:simproof4}), we have
\begin{align*}
    &\frac{1}{\alpha} \| \vx_1 - \vx_{\star} \|^2 + \frac{1}{\beta} \| \vy_1 - \vy_{\star} \|^2 \\
    &\le \bigopen{\frac{1}{\alpha} - \mu_x} \norm{\vx_0 - \vx_{\star}}^2 + \bigopen{\frac{1}{\beta} - \mu_y} \norm{\vy_0 - \vy_{\star}}^2 \\
    & + \alpha \norm{\nabla_{\vx} f(\vx_{0}, \vy_{0})}^2 + \beta \norm{\nabla_{\vy} f(\vx_{0}, \vy_{0})}^2 - \frac{1}{L_x} \norm{\nabla_{\vx} f (\vx_0, \vy_{\star})}^2 - \frac{1}{L_y} \norm{\nabla_{\vy} f (\vx_{\star}, \vy_0)}^2.
\end{align*}
If $\alpha \le \frac{1}{2L_x}$ and $\beta \le \frac{1}{2L_y}$, we can use the triangle inequality and the Lipschitz gradient condition for $L_{xy}$ to obtain
\begin{align*}
    \alpha \norm{\nabla_{\vx} f(\vx_{0}, \vy_{0})}^2 - \frac{1}{L_x} \norm{\nabla_{\vx} f (\vx_0, \vy_{\star})}^2 &\le \alpha \norm{\nabla_{\vx} f(\vx_{0}, \vy_{0})}^2 - 2\alpha \norm{\nabla_{\vx} f (\vx_0, \vy_{\star})}^2 \\
    &\le 2\alpha \norm{\nabla_{\vx} f(\vx_{0}, \vy_{0}) - \nabla_{\vx} f (\vx_0, \vy_{\star})}^2 \\
    &\le 2\alpha L_{xy}^2 \norm{\vy_{0} - \vy_{\star}}^2, \\
    \beta \norm{\nabla_{\vy} f(\vx_{0}, \vy_{0})}^2 - \frac{1}{L_y} \norm{\nabla_{\vy} f (\vx_{\star}, \vy_{0})}^2 &\le \beta \norm{\nabla_{\vx} f(\vx_{0}, \vy_{0})}^2 - 2\beta \norm{\nabla_{\vy} f (\vx_{\star}, \vy_{0})}^2 \\
    &\le 2\beta \norm{\nabla_{\vx} f(\vx_{0}, \vy_{0}) - \nabla_{\vy} f (\vx_{\star}, \vy_{0})}^2 \\
    &\le 2\beta L_{xy}^2 \norm{\vx_{0} - \vx_{\star}}^2,
\end{align*}
which boils down to \eqref{eq:simineq}.
\end{proof}

\subsubsection*{Step 2. Iteration Complexity}
Now let us show that \cref{prop:simsimple} can guarantee the same iteration complexity as in \cref{cor:simgda} when
\begin{align*}
    \alpha &= \frac{1}{2}\cdot \min \bigset{\frac{1}{L_x}, \frac{\mu_y}{2 L_{xy}^2}}, \quad \beta = \frac{1}{2}\cdot \min \bigset{\frac{1}{L_y}, \frac{\mu_x}{2 L_{xy}^2}}.
\end{align*}
\begin{proof}
    If $\alpha \le \frac{\mu_y}{4 L_{xy}^2}$ and $\beta \le \frac{\mu_x}{4 L_{xy}^2}$, \cref{prop:simsimple} implies
    \begin{align*}
        \frac{1}{\alpha} \| \vx_1 - \vx_{\star} \|^2 + \frac{1}{\beta} \| \vy_1 - \vy_{\star} \|^2 &\le \bigopen{\frac{1}{\alpha} - \mu_x + 2\beta L_{xy}^2} \norm{\vx_0 - \vx_{\star}}^2 + \bigopen{\frac{1}{\beta} - \mu_y + 2\alpha L_{xy}^2} \norm{\vy_0 - \vy_{\star}}^2 \\
        &\le \bigopen{\frac{1}{\alpha} - \frac{\mu_x}{2}} \norm{\vx_0 - \vx_{\star}}^2 + \bigopen{\frac{1}{\beta} - \frac{\mu_y}{2}} \norm{\vy_0 - \vy_{\star}}^2.
    \end{align*}
    Hence we have $\Psisim_1 \le r \Psisim_0$ for $r = \max \bigset{1 - \alpha \mu_x / 2, 1 - \beta \mu_y / 2}$, and
    \begin{align*}
        \frac{1}{1 - r} \le \max \bigset{\frac{1}{\alpha \mu_x}, \ \frac{1}{\beta \mu_y}} &= \max \bigset{\Theta \bigopen{\kappa_x + \kappa_{xy}^2}, \ \Theta \bigopen{\kappa_y + \kappa_{xy}^2}} \\
        &= \Theta \bigopen{\kappa_x + \kappa_y + \kappa_{xy}^2}.
    \end{align*}
    Therefore it is sufficient to take
    \begin{align*}
        K &= \gO \bigopen{ \bigopen{\kappa_x + \kappa_y + \kappa_{xy}^2} \cdot \log \frac{\Psi^{\text{Sim}}_{0}}{A^{\text{Sim}} \epsilon} }
    \end{align*}
    iterations to ensure that $\| \vz_K - \vz_{\star} \|^2 \le \epsilon$, where $A^{\text{Sim}} = \min \bigset{\frac{1}{\alpha}, \frac{1}{\beta}}$.
\end{proof}

%% file: secb3.tex
\subsection{\texorpdfstring{Proof of \cref{thm:simgdalb}}{Proof of Theorem 3.3}} \label{sec:thmsimgdalb}

Here we prove \cref{thm:simgdalb} of \cref{sec:3}, restated below for the sake of readability.

\theoremsimgdalb*

\begin{proof}
    We construct the worst-case function as follows:
    \begin{align*}
        f(\vx, \vy) &=
        \frac{1}{2} \begin{bmatrix}
            {x} \\ {s} \\ {t} \\ {y} \\ {u} \\ {v}
        \end{bmatrix}^{\top}
        \begin{bmatrix}
            \mu_x & 0 & 0 & L_{xy} & 0 & 0 \\
            0 & \mu_x & 0 & 0 & 0 & 0 \\
            0 & 0 & L_x & 0 & 0 & 0 \\
            L_{xy} & 0 & 0 & -\mu_y & 0 & 0 \\
            0 & 0 & 0 & 0 & -\mu_y & 0 \\
            0 & 0 & 0 & 0 & 0& -L_y 
        \end{bmatrix}
        \begin{bmatrix}
            {x} \\ {s} \\ {t} \\ {y} \\ {u} \\ {v}
        \end{bmatrix}, \\
    \end{align*}
    where ${\vx = (x, s, t)}$ and ${\vy = (y, u, v)}$. 
    We can easily check that $f$ is a quadratic function ({\it i.e.,} the Hessian is constant) such that $f \in \gF(\mu_x, \mu_y, L_x, L_y, L_{xy})$ and $\vx_\star = \vy_\star = \bm0 \in \R^3$. 

    As a first step, we will find a set of necessary conditions on step sizes for convergence, and then compute (at least) how large the number of iterations $K$ of \textbf{\green{Sim-GDA}} we need to accomplish $\bignorm{\vx_K}^2 + \bignorm{\vy_K}^2 < \epsilon.$
    To this end, we first observe that the $k$-th step of \textbf{\green{Sim-GDA}} satisfies
    \begin{align}
        \begin{bmatrix}
            x_{k+1} \\ y_{k+1}
        \end{bmatrix} &= 
        \underbrace{\begin{bmatrix}
            1 - \alpha \mu_x & - \alpha L_{xy} \\
           \beta L_{xy} & 1 - \beta \mu_y
        \end{bmatrix}}_{\triangleq \mP}
        \begin{bmatrix}
            x_k \\ y_k
        \end{bmatrix}, \label{eq:x_k_y_k}\\
        s_{k+1} &= (1 - \alpha \mu_x) s_k, \label{eq:s_k} \\
        t_{k+1} &= (1 - \alpha L_x) t_k, \label{eq:t_k} \\
        u_{k+1} &= (1 - \beta \mu_y) u_k, \label{eq:u_k} \\
        v_{k+1} &= (1 - \beta L_y) v_k. \label{eq:v_k} 
    \end{align}
    To assure the convergence of iterations \eqref{eq:t_k}~and~\eqref{eq:v_k}, the step sizes $\alpha$ and $\beta$ are required to be 
    \begin{equation}
        \label{eq:stepsize_basic}
        {\alpha < \frac{2}{L_x}} \quad \text{and} \quad {\beta < \frac{2}{L_y}}.
    \end{equation}
    Also, to guarantee $\bignorm{\vx_K}^2 + \bignorm{\vy_K}^2 < \epsilon$, we need from \eqref{eq:s_k}~and~\eqref{eq:u_k} that $s_K^2 < \gO(\epsilon)$ and $u_K^2 < \gO(\epsilon)$, respectively.
    These two necessary conditions require an iteration number of at least:
    \begin{equation}
        \label{eq:iteration_complexity_base}
        K = \Omega\bigopen{\bigopen{{\frac{1}{\alpha \mu_x}} + {\frac{1}{\beta \mu_y}}} \cdot \log\frac{1}{\epsilon}}.
    \end{equation}
    Note that \eqref{eq:stepsize_basic} automatically yields
    \begin{align}
        \frac{1}{\alpha \mu_x} + \frac{1}{\beta \mu_y} = \Omega (\kappa_x + \kappa_y).
        \label{eq:banana}
    \end{align}

    From now on, we deal with the remaining proof case by case with respect to the step sizes $\alpha$ and $\beta$.
    
    \paragraph{Case 1.} Suppose that $\alpha$ and $\beta$ satisfies ${\bigopen{\frac{\alpha\mu_x - \beta\mu_y}{2}}^2 \le \alpha\beta L_{xy}^2}$, which is equivalent to the eigenvalues of the matrix $\mP$ defined in Equation~\eqref{eq:x_k_y_k} being complex.
    We can check that, for $i=\sqrt{-1}$, the eigenvalues of $\mP$ can be expressed as
    \begin{equation*}
        \lambda = 1 - \frac{\alpha\mu_x + \beta\mu_y}{2} \pm {i} \sqrt{{\alpha\beta L_{xy}^2 - \bigopen{\frac{\alpha\mu_x - \beta\mu_y}{2}}^2}}.
    \end{equation*}
    We recall a well-known convergence theory of matrix iteration in \cref{prop:spectralradius}.
    \begin{proposition}[\citet{horn2012matrix}, Theorem~5.6.12, Corollary~5.6.13] 
        \label{prop:spectralradius}
        For a square matrix $\mA\in \R^{m\times m}$ and a sequence of $m$-dimensional vectors $(\vv_k)$, the matrix iteration $\vv_{k+1}  = \mA \vv_k$ converges as $\vv_k \rightarrow \bm0$ with arbitrarily chosen initialization $\vv_0$ if and only if the spectral radius $\rho(\mA)$ of $\mA$ is less than 1. In this case, the convergence rate is written as $O((\rho(\mA)+\epsilon)^k)$, where $\epsilon$ is an any given positive number.
    \end{proposition}
    Noting that
    \begin{align*}
        \rho(\mP)^2 &= \bigopen{1 - \frac{\alpha\mu_x + \beta\mu_y}{2}}^2 + \alpha\beta L_{xy}^2 - \bigopen{\frac{\alpha\mu_x - \beta\mu_y}{2}}^2 = 1 - (\alpha\mu_x + \beta\mu_y) + \alpha\beta \bigopen{\mu_x\mu_y + L_{xy}^2},
    \end{align*}
    in order to assure convergence of iteration~\eqref{eq:x_k_y_k}, we need 
    \begin{align*}
        \rho(\mP)^2 &< 1 \ \iff \ \beta < \frac{\mu_x + r \mu_y}{\mu_x\mu_y + L_{xy}^2} \iff  \alpha < \frac{\frac{1}{r} \mu_x + \mu_y}{\mu_x\mu_y + L_{xy}^2},
    \end{align*}
    where $r = \frac{\beta}{\alpha}$ is the ratio of step sizes.
    Combined with \eqref{eq:stepsize_basic}, we have
    \begin{align}
        {\frac{1}{\alpha\mu_x}} &> \max\bigset{{\frac{L_x}{2\mu_x}}, {\frac{rL_y}{2\mu_x}}, \frac{\mu_x\mu_y + L_{xy}^2}{\frac{1}{r} \mu_x^2 + \mu_x\mu_y}}, \label{eq:1_amu} \\
        {\frac{1}{\beta\mu_y}} &> \max\bigset{{\frac{L_x}{2r\mu_y}}, {\frac{L_y}{2\mu_y}}, \frac{\mu_x\mu_y + L_{xy}^2}{\mu_x\mu_y + r \mu_y^2}}. \label{eq:1_bmu}
    \end{align}
    If $r \ge \frac{\mu_x}{\mu_y}$, then by \eqref{eq:1_amu}, ${\frac{1}{\alpha\mu_x}} = \Omega\bigopen{\kappa_x + \kappa_y + \kappa_{xy}^2}$. On the other hand, if $r < \frac{\mu_x}{\mu_y}$, then by \eqref{eq:1_bmu}, ${\frac{1}{\beta\mu_y}} = \Omega\bigopen{\kappa_x + \kappa_y + \kappa_{xy}^2}$. Therefore, we have a desired lower bound of iteration complexity for the first case, deduced from \eqref{eq:iteration_complexity_base}.

    \paragraph{Case 2.} Suppose that $\alpha$ and $\beta$ satisfies ${\bigopen{\frac{\alpha\mu_x - \beta\mu_y}{2}}^2 > \alpha\beta L_{xy}^2 }$.
    Note that this is equivalent to 
    \begin{align}
        \bigabs{\sqrt{\frac{\alpha\mu_x}{\beta\mu_y}} - \sqrt{\frac{\beta\mu_y}{\alpha\mu_x}}} > 2\kappa_{xy}.
        \label{eq:realeigenvalues}
    \end{align}
    If $r \ge \frac{\mu_x}{\mu_y}$, {\it i.e.}, $\frac{\alpha\mu_x}{\beta\mu_y} \le \frac{\beta\mu_y}{\alpha\mu_x}$, then it implies $\frac{\beta\mu_y}{\alpha\mu_x} > 4\kappa_{xy}^2$. Thus, combined with \eqref{eq:banana}, we have 
    \begin{equation*}
        {\frac{1}{\alpha \mu_x}} + {\frac{1}{\beta \mu_y}} = \frac{1}{2}\cdot{\frac{1}{\alpha \mu_x}} + {\frac{1}{\beta\mu_y}}\bigopen{\frac{1}{2}\cdot\frac{{\beta\mu_y}}{{\alpha\mu_x}} + 1} = \Omega\bigopen{\kappa_x + \kappa_y(\kappa_{xy}^2 + 1)} = \Omega\bigopen{\kappa_x + \kappa_y + \kappa_{xy}^2}.
    \end{equation*}
    On the other hand, if $r < \frac{\mu_x}{\mu_y}$, {\it i.e.}, $\frac{\alpha\mu_x}{\beta\mu_y} > \frac{\beta\mu_y}{\alpha\mu_x}$, then it implies $\frac{\alpha\mu_x}{\beta\mu_y} > 4\kappa_{xy}^2$. Thus, combined with \eqref{eq:banana}, we have
    \begin{equation*}
        {\frac{1}{\alpha \mu_x}} + {\frac{1}{\beta \mu_y}} = {\frac{1}{\alpha \mu_x}}\bigopen{1 + \frac{1}{2}\cdot\frac{{\alpha\mu_x}}{{\beta\mu_y}}} + \frac{1}{2}\cdot{\frac{1}{\beta\mu_y}} = \Omega\bigopen{\kappa_x(1 + \kappa_{xy}^2) + \kappa_y} = \Omega\bigopen{\kappa_x + \kappa_y + \kappa_{xy}^2}.
    \end{equation*}
    Therefore, from \eqref{eq:iteration_complexity_base} we can obtain the desired lower bound for the second case as well, which concludes the proof.
\end{proof}

%% file: secb4.tex
\subsection{\texorpdfstring{Proofs used in \cref{sec:b}}{Proofs used in Appendix B}} \label{sec:simtech}

Here we prove some technical propositions and lemmas used throughout \cref{sec:b}.

\subsubsection{\texorpdfstring{Proof of \cref{prop:simcont}}{Proof of Proposition B.1}} \label{sec:propsimcont}

Here we prove \cref{prop:simcont}, restated below for the sake of readability.

\propsimcont*

\begin{proof}
    Recall that \simgda{} takes updates of the form:
    \begin{equation}
        \begin{aligned}
        \vx_{1} &= \vx_0 - \alpha \nabla_{\vx} f(\vx_0, \vy_0), \\
        \vy_{1} &= \vy_0 + \beta \nabla_{\vy} f(\vx_0, \vy_0).
        \end{aligned}
        \label{eq:simgda}
    \end{equation}
    For simplicity, let us denote $\vz = \begin{bmatrix}
        \vx^{\top} & \vy^{\top}
    \end{bmatrix}^{\top} \in \R^{d_x + d_y}$, and define
    \begin{align*}
        \nu (\vz) &:= \begin{bmatrix}
            \nabla_{\vx} f (\vz) \\
            - \nabla_{\vy} f (\vz)
        \end{bmatrix}.
     \end{align*}
    For instance, $\vz_0 = \begin{bmatrix}
        \vx_0^{\top} & \vy_0^{\top}
    \end{bmatrix}^{\top}$ and $\vz_{\star} = \begin{bmatrix}
        \vx^{\star \top} & \vy^{\star \top}
    \end{bmatrix}^{\top}$.
    
    Let us define matrices $\mA \in \R^{d_x \times d_x}$, $\mB \in \R^{d_x \times d_y}$, and $\mC \in \R^{d_y \times d_y}$ as
    \begin{align*}
        \mA &:= \int_0^1 \nabla_{\vx \vx}^2 f (t \vz_0 + (1-t)\vz_{\star}) dt, \ \
        \mB := \int_0^1 \nabla_{\vx \vy}^2 f (t \vz_0 + (1-t)\vz_{\star}) dt, \ \
        \mC := - \int_0^1 \nabla_{\vy \vy}^2 f (t \vz_0 + (1-t)\vz_{\star}) dt.
    \end{align*}
    Since $f \in \gF (\mu_x, \mu_y, L_x, L_y, L_{xy})$, we have $\mu_x \mI \preceq \mA \preceq L_x \mI$, $\mu_y \mI \preceq \mC \preceq L_y \mI$, and $\norm{\mB} \le L_{xy}$.
    
    Also, by chain rule, we have the following identities:
    \begin{align*}
        \nabla_{\vx} f(\vx_0, \vy_0) &= \mA (\vx_0 - \vx_{\star}) + \mB (\vy_0 - \vy_{\star}), \\
        \nabla_{\vy} f(\vx_0, \vy_0) &= \mB^{\top} (\vx_0 - \vx_{\star}) - \mC (\vy_0 - \vy_{\star}).
    \end{align*}
    
    For simplicity, we assume W.L.O.G. $\vx_{\star} = \bm{0}$ $(\in \R^{d_x})$ and $\vy_{\star} = \bm{0}$ $(\in \R^{d_y})$. Then we have
    \begin{align*}
        \begin{bmatrix}
            \frac{1}{\sqrt{\alpha}} \vx_1 \\ \frac{1}{\sqrt{\beta}} \vy_1
        \end{bmatrix} 
        &= \begin{bmatrix}
            \frac{1}{\sqrt{\alpha}} \vx_0 \\ \frac{1}{\sqrt{\beta}} \vy_0
        \end{bmatrix} - \begin{bmatrix}
            \sqrt{\alpha} \nabla_{\vx} f(\vx_0, \vy_0) \\ - \sqrt{\beta} \nabla_{\vy} f(\vx_0, \vy_0)
        \end{bmatrix}
        = \begin{bmatrix}
            \frac{1}{\sqrt{\alpha}} \vx_0 \\ \frac{1}{\sqrt{\beta}} \vy_0
        \end{bmatrix} - \begin{bmatrix}
            \sqrt{\alpha} (\mA \vx_0 + \mB \vy_0) \\
            - \sqrt{\beta} (\mB^{\top} \vx_0 - \mC \vy_0)
        \end{bmatrix} \\
        &= \begin{bmatrix}
            \frac{1}{\sqrt{\alpha}} \vx_0 \\ \frac{1}{\sqrt{\beta}} \vy_0
        \end{bmatrix} - \begin{bmatrix}
            \alpha \mA & \sqrt{\alpha \beta} \mB \\
            - \sqrt{\alpha \beta} \mB^{\top} & \beta \mC
        \end{bmatrix} \begin{bmatrix}
            \frac{1}{\sqrt{\alpha}} \vx_0 \\ \frac{1}{\sqrt{\beta}} \vy_0
        \end{bmatrix}
        = \begin{bmatrix}
            \mI - \alpha \mA & - \sqrt{\alpha \beta} \mB \\
            \sqrt{\alpha \beta} \mB^{\top} & \mI - \beta \mC
        \end{bmatrix} \begin{bmatrix}
            \frac{1}{\sqrt{\alpha}} \vx_0 \\ \frac{1}{\sqrt{\beta}} \vy_0
        \end{bmatrix}.
    \end{align*}
    This means that it is enough to show that
    \begin{align}
        \bignorm{\begin{bmatrix}
            \mI - \alpha \mA & - \sqrt{\alpha \beta} \mB \\
            \sqrt{\alpha \beta} \mB^{\top} & \mI - \beta \mC
        \end{bmatrix}}^2 &\le r = \max \left\{ \left\| \begin{bmatrix}
                1 - \alpha L_x & - \sqrt{\alpha \beta} L_{xy} \\
                \sqrt{\alpha \beta} L_{xy} & 1 - \beta \mu_y
            \end{bmatrix} \right\|^2, \left\| \begin{bmatrix}
                1 - \alpha \mu_x & - \sqrt{\alpha \beta} L_{xy} \\
                \sqrt{\alpha \beta} L_{xy} & 1 - \beta L_y
            \end{bmatrix} \right\|^2 \right\}, \label{eq:normmax}
    \end{align}
    since if this is true, then we automatically have
    \begin{align*}
        \frac{1}{\alpha} \norm{\vx_1}^2 + \frac{1}{\beta} \norm{\vy_1}^2 &= \bignorm{\begin{bmatrix}
            \frac{1}{\sqrt{\alpha}} \vx_1 \\ \frac{1}{\sqrt{\beta}} \vy_1
        \end{bmatrix}}^2 \le r \bignorm{\begin{bmatrix}
            \frac{1}{\sqrt{\alpha}} \vx_0 \\ \frac{1}{\sqrt{\beta}} \vy_0
        \end{bmatrix}}^2 = r \bigopen{\frac{1}{\alpha} \norm{\vx_0}^2 + \frac{1}{\beta} \norm{\vy_0}^2}.
    \end{align*}
    
    To prove \cref{eq:normmax}, the matrix norm can be bounded via \cref{lem:normineq}.
    \begin{restatable}{lemma}{normineq}
        \label{lem:normineq}
        Suppose that $\mX \in \R^{d_x \times d_x}$, $\mY \in \R^{d_y \times d_y}$, $\mW \in \R^{d_x \times d_y}$ satisfy
        \begin{align*}
            t_x \mI \preceq \mX \preceq s_x \mI, \ \ t_y \mI \preceq \mY \preceq s_y \mI, \ \ \norm{\mW} \le \ell
        \end{align*}
        for some constants $t_x, t_y, s_x, s_y > 0$ and $\ell \ge 0$. Then the block matrix $\mM \in \R^{(d_x + d_y) \times (d_x + d_y)}$ of the form
        \begin{align*}
            \mM &= \begin{bmatrix}
                \mX & -\mW \\
                \mW^{\top} & \mY
            \end{bmatrix}
        \end{align*}
        satisfies the matrix norm inequality
        \begin{align*}
            \| \mM \| &\le \max \left\{ \left\| \begin{bmatrix}
                s_x & -\ell \\
                \ell & t_y
            \end{bmatrix} \right\|, \left\| \begin{bmatrix}
                t_x & -\ell \\
                \ell & s_y
            \end{bmatrix} \right\| \right\}.
        \end{align*}
    \end{restatable}
    We prove \cref{lem:normineq} in \cref{sec:lemnormineq}.
    
    By observing that $1 - \alpha L_x > 0$, $1 - \beta L_y > 0$ and
    \begin{align*}
        (1 - \alpha L_x) \mI \preceq \mI - \alpha \mA \preceq (1 - \alpha \mu_x) \mI, \ \ (1 - \beta L_y) \mI \preceq \mI - \beta \mC \preceq (1 - \beta \mu_y) \mI, \ \ \| \sqrt{\alpha \beta} \mB \| \le \sqrt{\alpha \beta} L_{xy},
    \end{align*}
    we can use \cref{lem:normineq} with $\mX = \mI - \alpha \mA$, $\mY = \mI - \beta \mC$, $\mW = \sqrt{\alpha \beta} \mB$, and
    \begin{align*}
        t_x &= 1 - \alpha L_x, \ \ t_y =1 - \beta L_y, \ \ s_x = 1 - \alpha \mu_x, \ \ s_y = 1 - \beta \mu_y, \ \ \ell = \sqrt{\alpha \beta} L_{xy}
    \end{align*}
    which immediately proves \cref{eq:normmax}, and therefore \cref{prop:simcont}.
\end{proof}

\subsubsection{\texorpdfstring{Proof of \cref{prop:normopt}}{Proof of Proposition B.2}} \label{sec:propnormopt}

Here we prove \cref{prop:normopt}, restated below for the sake of readability.

\propnormopt*

\begin{proof}
    Recall that we define
    \begin{align*}
        f_{\kappa} (\zeta) &= \frac{\kappa - 1}{2} \cdot \zeta + \sqrt{\bigopen{1 - \frac{\kappa + 1}{2} \cdot \zeta}^2 + \zeta^2 \kappa_{xy}^2}.
    \end{align*}
    Then the first two results of the proposition are direct consequences of \cref{lem:optrate}.
    \begin{restatable}{lemma}{lemoptrate}
        \label{lem:optrate}
        Suppose that $A, B, C \ge 0$ and $A < B$. Then for the function $f : (0, \infty) \rightarrow (0, \infty)$ of the following form:
        \begin{align*}
            f(x) &= Ax + \sqrt{(1-Bx)^2 + C^2 x^2},
        \end{align*}
        the minimizer is equal to
        \begin{align*}
            x_{\star} &= \frac{1}{D} \cdot \frac{2(C+D)(B-A)}{(C+D)^2 + (B-A)^2},
        \end{align*}
        and the minimum value attained at $x_{\star}$ is equal to
        \begin{align*}
            f (x_{\star}) = \frac{(C+D)^2 - (B-A)^2}{(C+D)^2 + (B-A)^2},
        \end{align*}
        where $D = \sqrt{B^2 + C^2 - A^2}$.
    \end{restatable}
    We prove \cref{lem:optrate} in \cref{sec:lemoptrate}.

    We can use $\zeta$ as $x$ and plug in the following values into \cref{lem:optrate}:
    \begin{align*}
        A &= \frac{\kappa - 1}{2}, \ \ B = \frac{\kappa + 1}{2}, \ \ C = \kappa_{xy}, \ \ D = \sqrt{B^2 + C^2 - A^2} = \sqrt{\kappa + \kappa_{xy}^2},
    \end{align*}
    which yields
    \begin{align*}
        B - A &= 1, \ \ C + D = \kappa_{xy} + \sqrt{\kappa + \kappa_{xy}^2}.
    \end{align*}
    Then, for the choice
    \begin{align*}
        \zeta^{\star} &= \frac{1}{D} \cdot \frac{2(C+D)(B-A)}{(C+D)^2 + (B-A)^2} = \frac{1}{\sqrt{\kappa + \kappa_{xy}^2}} \cdot \frac{2 \bigopen{\kappa_{xy} + \sqrt{\kappa + \kappa_{xy}^2} \ }}{1 + \bigopen{\kappa_{xy} + \sqrt{\kappa + \kappa_{xy}^2} \ }^2},
    \end{align*}
    we can obtain the optimal value
    \begin{align*}
        f_{\kappa} (\zeta^{\star}) &= \frac{1}{D} \cdot \frac{(C+D)^2 - (B-A)^2}{(C+D)^2 + (B-A)^2} = \frac{\bigopen{\kappa_{xy} + \sqrt{\kappa + \kappa_{xy}^2} \ }^2 - 1}{\bigopen{\kappa_{xy} + \sqrt{\kappa + \kappa_{xy}^2} \ }^2 + 1}.
    \end{align*}

    The last result of the proposition is a direct consequence of the following lemma.
    \begin{restatable}{lemma}{fzetaineq}
        \label{lem:fzetaineq}
        Suppose that $A_1, A_2, B_1, B_2, C \ge 0$ satisfies $A_1 \le B_1$, $A_2 \le B_2$, and $A_2 - A_1 = B_2 - B_1 \ge 0$. Then for the functions $f_1, f_2 : (0, \infty) \rightarrow (0, \infty)$ of the following form:
        \begin{align*}
            f_1 (x) &= A_1 x + \sqrt{(1 - B_1 x)^2 + C^2 x^2}, \ \ f_2 (x) = A_2 x + \sqrt{(1 - B_2 x)^2 + C^2 x^2},
        \end{align*}
        we have $f_1 (x) \le f_2 (x)$ for all $x > 0$.
    \end{restatable}
    We prove \cref{lem:fzetaineq} in \cref{sec:lemfzetaineq}.
    
    If $\kappa_x \ge \kappa_y$, we can plug in the following values:
    \begin{align*}
        A_1 &= \frac{\kappa_y - 1}{2}, \ \ A_2 = \frac{\kappa_x - 1}{2}, \ \ B_1 = \frac{\kappa_y + 1}{2}, \ \ B_2 = \frac{\kappa_x + 1}{2}, \ \ C = \kappa_{xy},
    \end{align*}
    so that $A_2 - A_1 = B_2 - B_1 = \kappa_x - \kappa_y \ge 0$ and \cref{lem:fzetaineq} implies $f_{\kappa_x} (\zeta) \ge f_{\kappa_y} (\zeta)$.
    
    If $\kappa_x \le \kappa_y$, we can change orders as:
    \begin{align*}
        A_1 &= \frac{\kappa_x - 1}{2}, \ \ A_2 = \frac{\kappa_y - 1}{2}, \ \ B_1 = \frac{\kappa_x + 1}{2}, \ \ B_2 = \frac{\kappa_y + 1}{2}, \ \ C = \kappa_{xy},
    \end{align*}
    so that $A_2 - A_1 = B_2 - B_1 = \kappa_y - \kappa_x \ge 0$ and \cref{lem:fzetaineq} implies $f_{\kappa_x} (\zeta) \le f_{\kappa_y} (\zeta)$.
    
    Therefore we can conclude that $f_{\kappa_x} (\zeta) \ge f_{\kappa_y} (\zeta)$ for all $\zeta \in (0, \infty)$ if and only if $\kappa_x \ge \kappa_y$.
\end{proof}

\subsubsection{\texorpdfstring{Proof of \cref{lem:normineq}}{Proof of Lemma B.4}}
\label{sec:lemnormineq}

Here we prove \cref{lem:normineq}, restated below for the sake of readability.

\normineq*

\begin{proof}
    We first observe that the following matrix norms are equal:
    \begin{align*}
        \bignorm{\begin{bmatrix}
            \mX & -\mW \\
            \mW^{\top} & \mY
        \end{bmatrix}} 
        = \bigg{\|} \underbrace{\begin{bmatrix}
            \mX & \mW \\
            \mW^{\top} & -\mY
        \end{bmatrix}}_{\triangleq \mM'} \bigg{\|}.
    \end{align*}

    Let $\lambda_{\max}^{\mM'}$ and $\lambda_{\min}^{\mM'}$ be the maximum and minimum eigenvalues of $\mM'$, respectively. Since $\mM'$ is a symmetric matrix, the matrix norm of $\mM'$ is equal to
    \begin{align}
        \norm{\mM'} &= \max \bigset{| \lambda_{\max}^{\mM'} |, \ | \lambda_{\min}^{\mM'} |}. \label{eq:lmbdmaxmin}
    \end{align}
    Since $\mX \succ 0$ and $-\mY \prec 0$, we can observe that $\mM'$ is neither positive definite nor negative definite\footnote{It is easy if we think of the contrapositive--- any block partition of a PD matrix must have PD block diagonals.}, i.e., $\lambda_{\max}^{\mM'} \ge 0 \ge \lambda_{\min}^{\mM'}$. Hence we can rewrite:
    \begin{align}
        \norm{\mM'} &= \max \bigset{\lambda_{\max}^{\mM'} , \ - \lambda_{\min}^{\mM'}}.
    \end{align}
    
    Given a symmetric matrix $\mS \in \sS^{d}$, we have the following identities:
    \begin{align}
        \lambda_{\max}^{\mS} &= \sup_{\vz \in \R^{d}, \| \vz \| = 1} \vz^{\top} \mS \vz, \quad \lambda_{\min}^{\mS} = \inf_{\vz \in \R^{d}, \| \vz \| = 1} \vz^{\top} \mS \vz, \label{eq:rayleigh}
    \end{align}
    where $\lambda_{\max}^{\mS}$ and $\lambda_{\min}^{\mS}$ are the maximum and minimum eigenvalues of $\mS$, respectively. Moreover, the $\sup$ for $\lambda_{\max}^{\mS}$ and $\inf$ for $\lambda_{\min}^{\mS}$ is attained when the unit vector $\vz$ is aligned with the eigenvectors corresponding to $\lambda_{\max}^{\mS}$ and $\lambda_{\min}^{\mS}$.
    
    Now we will show that
    \begin{align}
        \lambda_{\max}^{\mM'} \le \left\| \begin{bmatrix}
            s_x & -\ell \\
            \ell & t_y
        \end{bmatrix} \right\| \quad \text{and} \quad - \lambda_{\min}^{\mM'} \le \left\| \begin{bmatrix}
            t_x & -\ell \\
            \ell & s_y
        \end{bmatrix} \right\|.
        \label{eq:gulugulu}
    \end{align}
    
    \paragraph{Maximum Eigenvalue.} The maximum eigenvalue of $\mM'$ is equal to
    \begin{align*}
        \lambda_{\max}^{\mM'} &= \sup_{\vz \in \R^{d_x + d_y}, \| \vz \| = 1} \vz^{\top} \mM' \vz \\
        &= \sup_{\substack{p, q \in [0, 1] \\ p^2 + q^2 = 1}} \sup_{\substack{\vx \in \R^{d_x}, \| \vx \| = 1 \\ \vy \in \R^{d_y}, \| \vy \| = 1}} \begin{bmatrix}
            p \vx \\ q \vy
        \end{bmatrix}^{\top}
        \begin{bmatrix}
            \mX & \mW \\
            \mW^{\top} & -\mY
        \end{bmatrix}
        \begin{bmatrix}
            p \vx \\ q \vy
        \end{bmatrix} \\
        &= \sup_{\substack{p, q \in [0, 1] \\ p^2 + q^2 = 1}} \sup_{\substack{\vx \in \R^{d_x}, \| \vx \| = 1 \\ \vy \in \R^{d_y}, \| \vy \| = 1}} \bigopen{p^2 \vx^{\top} \mX \vx + 2 pq \vx^{\top} \mW \vy - q^2 \vy^{\top} \mY \vy},
    \end{align*}
    where we reparameterize $\vz = \begin{bmatrix} p \vx^\top & q \vy^\top \end{bmatrix}^\top$ such that $\vx \in \R^{d_x}$, $\vy \in \R^{d_y}$ satisfies $\| \vx \| = \| \vy \| = 1$, and $p^2 + q^2 = 1$.

    First, suppose that $\ell > 0$, \textit{i.e.}, $\mW \ne 0$.
    Let $\mW = \mU \bm{\Sigma} \mV^{\top}$ be the singular value decomposition of $\mW$, where $\mU = [\vu_1, \dots, \vu_{r}] \in \R^{d_x \times r}$ and $\mV = [\vv_1, \dots, \vv_{r}] \in \R^{d_y \times r}$ are matrices with orthonormal columns and $\bm{\Sigma} = \text{diag} (\sigma_1, \dots, \sigma_r) \in \R^{r \times r}$ is a diagonal matrix with (strictly) positive entries.
    (Note that $1 \le r \le \min \{ d_x, d_y \}$.)
    Assume $\sigma_1 \ge \dots \ge \sigma_r$ W.L.O.G., so that $\| \mW \| \le \ell$ is equivalent to $\sigma_1 \le \ell$.
    Then we have
    \begin{align}
        p^2 \vx^{\top} \mX \vx + 2 pq \vx^{\top} \mW \vy - q^2 \vy^{\top} \mY \vy &= p^2 \vx^{\top} \mX \vx + 2 pq \sum_{k=1}^{r} \sigma_k \vx^{\top} \vu_{k} \vv_{k}^{\top} \vy - q^2 \vy^{\top} \mY \vy \notag \\
        &= p^2 \vx^{\top} \mX \vx + 2 pq \sum_{k=1}^{r} \sigma_k \vu_{k}^{\top} \vx \vy^{\top} \vv_{k}^{\top} - q^2 \vy^{\top} \mY \vy. \label{eq:maxeigenval}
    \end{align}
    Since we aim to show an upper bound of \eqref{eq:maxeigenval}, we now consider another optimization problem over a ``bigger'' search space and try to characterize its optimum value; this value will give us an upper bound of $\lambda_{\max}^{\mM'}$. Namely, we now additionally treat $\vu_1, \dots, \vu_{r}$ and $\vv_1, \dots, \vv_{r}$ in \eqref{eq:maxeigenval} as optimization variables. 
    With this addition, from now we treat the following items as optimization variables:
    \begin{enumerate}
        \item Choice of unit vectors $\vu_1, \dots, \vu_{r} \in \R^{d_x}$ of $\mU$ and $\vv_1, \dots, \vv_{r} \in \R^{d_y}$ of $\mV$
        \item Choice of unit vectors $\vx \in \R^{d_x}$, $\vy \in \R^{d_y}$
        \item Choice of values $p, q \in [0, 1]$ such that $p^2 + q^2 = 1$
    \end{enumerate}
    Our problem boils down to finding the maximum value of (\ref{eq:maxeigenval}) over all possible choices of these variables.
    (Note that the subsequent arguments and the resulting upper bound are true for all cases of $r \le \min \{ d_x, d_y \}$.)

    First, note that our choices of $\vu_1, \dots, \vu_{r}$ and $\vv_1, \dots, \vv_{r}$ only affect the middle term, which is bounded by
    \begin{align*}
        2 pq \sum_{k=1}^{r} \sigma_k \vu_{k}^{\top} \vx \vy^{\top} \vv_{k}^{\top} \le 2 pq \sigma_1,
    \end{align*}
    for which, for any given $\vx$, $\vy$ and $p, q$, the maximum is attained when we choose $\vu_1 = \vx$, $\vv_1 = \vy$. (Note that the terms for $k \ge 2$ all disappear by orthogonality.)

    Now we can observe that over possible choices of $\vx$ and $\vy$, we have
    \begin{align*}
        p^2 \vx^{\top} \mX \vx + 2 pq \sigma_1 - q^2 \vy^{\top} \mY \vy &\le p^2 \lambda_{\max}^{\mX} + 2 pq \sigma_1 - q^2 \lambda_{\min}^{\mY},
    \end{align*}
    where equality holds if the unit vector $\vx$ (or $\vy$) is aligned with the eigenvector corresponding to the maximum (or minimum) eigenvalue of $\mX$ (or $\mY$). We can use the given conditions to obtain
    \begin{align*}
        p^2 \lambda_{\max}^{\mX} + 2 pq \sigma_1 - q^2 \lambda_{\min}^{\mY} &\le p^2 s_x + 2 pq \ell - q^2 t_y = \begin{bmatrix}
            p \\ q
        \end{bmatrix}^{\top}
        \begin{bmatrix}
            s_x & \ell \\
            \ell & -t_y
        \end{bmatrix}
        \begin{bmatrix}
            p \\ q
        \end{bmatrix}.
    \end{align*}
    Finally, if we take the maximum over $p, q \in [0, 1]$ with $p^2 + q^2 = 1$, we have that
    \begin{align*}
        \sup_{\substack{p, q \in [0, 1] \\ p^2 + q^2 = 1}} \begin{bmatrix}
            p \\ q
        \end{bmatrix}^{\top}
        \begin{bmatrix}
            s_x & \ell \\
            \ell & -t_y
        \end{bmatrix}
        \begin{bmatrix}
            p \\ q
        \end{bmatrix} &= \bignorm{\begin{bmatrix}
            s_x & \ell \\
            \ell & -t_y
        \end{bmatrix}} = \bignorm{\begin{bmatrix}
            s_x & -\ell \\
            \ell & t_y
        \end{bmatrix}}
    \end{align*}
    and hence we can conclude that
    \begin{align*}
        \lambda_{\max}^{\mM'} \le \bignorm{\begin{bmatrix}
            s_x & -\ell \\
            \ell & t_y
        \end{bmatrix}}.
    \end{align*}
    For the degenerate case $\ell = 0$, we can just apply $r = 1$ and $\sigma_1 = 0$, which does not hurt the validity of the proof.
    
    \paragraph{Minimum Eigenvalue.} Similarly, the minimum eigenvalue of $\mM'$ is equal to
    \begin{align*}
        \lambda_{\min}^{\mM'} &= \inf_{\vz \in \R^{d_x + d_y}, \| \vz \| = 1} \vz^{\top} \mM' \vz \\
        &= \inf_{\substack{p, q \in [0, 1] \\ p^2 + q^2 = 1}} \inf_{\substack{\vx \in \R^{d_x}, \| \vx \| = 1 \\ \vy \in \R^{d_y}, \| \vy \| = 1}} \begin{bmatrix}
            p \vx \\ q \vy
        \end{bmatrix}^{\top}
        \begin{bmatrix}
            \mX & \mW \\
            \mW^{\top} & -\mY
        \end{bmatrix}
        \begin{bmatrix}
            p \vx \\ q \vy
        \end{bmatrix} \\
        &= \inf_{\substack{p, q \in [0, 1] \\ p^2 + q^2 = 1}} \inf_{\substack{\vx \in \R^{d_x}, \| \vx \| = 1 \\ \vy \in \R^{d_y}, \| \vy \| = 1}} \bigopen{p^2 \vx^{\top} \mX \vx + 2 pq \vx^{\top} \mW \vy - q^2 \vy^{\top} \mY \vy} \\
        &= - \sup_{\substack{p, q \in [0, 1] \\ p^2 + q^2 = 1}} \sup_{\substack{\vx \in \R^{d_x}, \| \vx \| = 1 \\ \vy \in \R^{d_y}, \| \vy \| = 1}} \bigopen{- p^2 \vx^{\top} \mX \vx - 2 pq \vx^{\top} \mW \vy + q^2 \vy^{\top} \mY \vy},
    \end{align*}
    and therefore
    \begin{align*}
        - \lambda_{\min}^{\mM'} &= \sup_{\substack{p, q \in [0, 1] \\ p^2 + q^2 = 1}} \sup_{\substack{\vx \in \R^{d_x}, \| \vx \| = 1 \\ \vy \in \R^{d_y}, \| \vy \| = 1}} \bigopen{- p^2 \vx^{\top} \mX \vx - 2 pq \vx^{\top} \mW \vy + q^2 \vy^{\top} \mY \vy},
    \end{align*}
    where we use the same reparameterization: $\vz = \begin{bmatrix} p \vx^\top & q \vy^\top \end{bmatrix}^\top$ with $\vx \in \R^{d_x}$, $\vy \in \R^{d_y}$ with $\| \vx \| = \| \vy \| = 1$, and $p^2 + q^2 = 1$.

    As in the maximum case, we first assume that $\ell > 0$ and define the singular value decomposition of $\mW$ as $\mW = \mU \bm{\Sigma} \mV^{\top}$.
    Then we can write
    \begin{align}
        - p^2 \vx^{\top} \mX \vx - 2 pq \vx^{\top} \mW \vy + q^2 \vy^{\top} \mY \vy &= - p^2 \vx^{\top} \mX \vx - 2 pq \sum_{k=1}^{r} \sigma_k \vu_{k}^{\top} \vx \vy^{\top} \vv_{k}^{\top} + q^2 \vy^{\top} \mY \vy. \label{eq:mineigenval}
    \end{align}
    to observe that
    \begin{align*}
        - 2 pq \sum_{k=1}^{r} \sigma_k \vu_{k}^{\top} \vx \vy^{\top} \vv_{k}^{\top} \le 2 pq \sigma_1,
    \end{align*}
    for which the maximum is attained when we choose $\vu_1 = \vx$ and $\vv_1 = - \vy$.
    Then we have
    \begin{align*}
        - p^2 \vx^{\top} \mX \vx + 2 pq \sigma_1 + q^2 \vy^{\top} \mY \vy &\le - p^2 \lambda_{\min}^{\mX} + 2 pq \sigma_1 + q^2 \lambda_{\max}^{\mY},
    \end{align*}
    where equality holds if the unit vector $\vx$ (or $\vy$) is aligned with the eigenvector corresponding to the minimum (or maximum) eigenvalue of $\mX$ (or $\mY$).
    We can use the given conditions to obtain
    \begin{align*}
        - p^2 \lambda_{\min}^{\mX} + 2 pq \sigma_1 + q^2 \lambda_{\max}^{\mY} &\le - p^2 t_x + 2 pq \ell - q^2 s_y = \begin{bmatrix}
            p \\ q
        \end{bmatrix}^{\top}
        \begin{bmatrix}
            t_x & \ell \\
            \ell & -s_y
        \end{bmatrix}
        \begin{bmatrix}
            p \\ q
        \end{bmatrix}.
    \end{align*}
    Finally, if we take the maximum over $p, q \in [0, 1]$ with $p^2 + q^2 = 1$, we have that
    \begin{align*}
        \sup_{\substack{p, q \in [0, 1] \\ p^2 + q^2 = 1}} \begin{bmatrix}
            p \\ q
        \end{bmatrix}^{\top}
        \begin{bmatrix}
            t_x & \ell \\
            \ell & -s_y
        \end{bmatrix}
        \begin{bmatrix}
            p \\ q
        \end{bmatrix} &= \bignorm{\begin{bmatrix}
            t_x & \ell \\
            \ell & -s_y
        \end{bmatrix}} = \bignorm{\begin{bmatrix}
            t_x & -\ell \\
            \ell & s_y
        \end{bmatrix}}
    \end{align*}
    and hence we can conclude that
    \begin{align*}
        - \lambda_{\min}^{\mM'} \le \bignorm{\begin{bmatrix}
            t_x & -\ell \\
            \ell & s_y
        \end{bmatrix}}.
    \end{align*}
    Combining the results with \eqref{eq:lmbdmaxmin}, we have
    \begin{align*}
        \norm{\mM'} &= \max \bigset{\lambda_{\max}^{\mM'} , \ - \lambda_{\min}^{\mM'}} = \max \left\{ \left\| \begin{bmatrix}
            s_x & -\ell \\
            \ell & t_y
        \end{bmatrix} \right\|, \left\| \begin{bmatrix}
            t_x & -\ell \\
            \ell & s_y
        \end{bmatrix} \right\| \right\}.
    \end{align*}
    For the degenerate case $\ell = 0$, we can just apply $r = 1$ and $\sigma_1 = 0$, which does not hurt the validity of the proof.
    
    Therefore we have shown \eqref{eq:gulugulu}, which completes the proof of \cref{lem:normineq}.
\end{proof}

\begin{remark*}
    \note{An anonymous reviewer has found a much simpler proof of \cref{lem:normineq}.
    By definition we have}
    \begin{align*} 
        \begin{bmatrix}
        \| \vx \| \\ \| \vy \|
        \end{bmatrix}^{\top}\begin{bmatrix}
            t_x & -l \\ -l & -s_y
        \end{bmatrix}
        \begin{bmatrix}
            \| \vx \| \\ \| \vy \|
        \end{bmatrix} \le \begin{bmatrix}
            \vx \\ \vy
        \end{bmatrix}^{\top} \mM' \begin{bmatrix}
            \vx \\ \vy
        \end{bmatrix} \le 
        \begin{bmatrix}
        \| \vx \| \\ \| \vy \|
        \end{bmatrix}^{\top}
        \begin{bmatrix}
            s_x & l \\ l & -t_y
        \end{bmatrix}
        \begin{bmatrix}
        \| \vx \| \\ \| \vy \|
        \end{bmatrix},
    \end{align*}
    \note{for all $\vx \in \R^{d_x}$ and $\vy \in \R^{d_y}$. 
    Then we immediately obtain the desired inequality as the matrix norm is invariant with respect to multiplication by $-1$ on rows and columns. \hfill $\square$}
\end{remark*}

\subsubsection{\texorpdfstring{Proof of \cref{lem:optrate}}{Proof of Lemma B.5}}
\label{sec:lemoptrate}

Here we prove \cref{lem:optrate}, restated below for the sake of readability.

\lemoptrate*

\begin{proof}
    Observing that $A^2 + D^2 = B^2 + C^2$ by definition, we start by substituting
    \begin{align*}
        R = \sqrt{B^2 + C^2} = \sqrt{A^2 + D^2}, \ \ \sin \phi &= \frac{A}{R}, \ \ \sin \psi = \frac{B}{R}
    \end{align*}
    for $\phi \in [0, \frac{\pi}{2})$ and $\psi \in (0, \frac{\pi}{2}]$. Note that we have $\phi < \psi$ from $A < B$, and
    \begin{align*}
        \cos \phi = \frac{D}{R}, \ \ \cos \psi = \frac{C}{R}.
    \end{align*}
    We can compute
    \begin{align*}
        Ax + \sqrt{(1-Bx)^2 + C^2 x^2} &= R x \sin \phi + \sqrt{(1 - R x \sin \psi)^2 + R^2 x^2 \cos^2 \psi} \\
        &= R x \sin \phi + \sqrt{1 - 2 R x \sin \psi + R^2 x^2}.
    \end{align*}
    By using change of variables as 
    \begin{align*}
        y = \tan \psi - R x \sec \psi \ \Leftrightarrow \ x = \frac{1}{R} \bigopen{\sin \psi - y \cos \psi},
    \end{align*}
    we have $y \in [-\infty, \tan \psi]$, and
    \begin{align*}
        1 - 2 x R \sin \psi + R^2 x^2 = (1 + y^2) \cos^2 \psi.
    \end{align*}
    Plugging in, we can obtain the following reparameterization:
    \begin{align*}
        R x \sin \phi + \sqrt{1 - 2 R x \sin \psi + R^2 x^2} &= \sin \phi \sin \psi - y \sin \phi \cos \psi + \sqrt{1 + y^2} \cos \psi.
    \end{align*}
    We can easily observe that if we again reparameterize as $y = \sinh \theta$, we can write as
    \begin{align}
        \sin \phi \sin \psi - \sin \phi \cos \psi \cdot \sinh \theta + \cos \psi \cdot \cosh \theta = \sin \phi \sin \psi + \cos \psi \bigopen{\cosh \theta - \sin \phi \cdot \sinh \theta}. \label{eq:carrot}
    \end{align}
    The derivative of \eqref{eq:carrot} with respect to $\theta$ is equal to
    \begin{align}
        \cos \psi \bigopen{\sinh \theta - \sin \phi \cdot \cosh \theta} . \label{eq:onion}
    \end{align}
    As the \textit{second} derivative of \eqref{eq:carrot} satisfies $\cos \psi \bigopen{\cosh \theta - \sin \phi \cdot \sinh \theta} \ge \cos \psi \cdot (- \sinh \theta + \cosh \theta) \ge 0$, we have that \eqref{eq:onion} is an increasing function.
    Therefore, the minimizer of \eqref{eq:carrot} must be equal to the point where \eqref{eq:onion} is zero, which is\footnote{To clarify, we are just using the fact that $a \sinh t - b \cosh t = 0$ if $\sinh t = \frac{b}{\sqrt{a^2 - b_2}}$.}
    \begin{align*}
        y_{\star} &= \sinh \theta_{\star} = \frac{\sin \phi}{\sqrt{1 - \sin^2 \phi}} = \frac{\sin \phi}{\cos \phi} = \tan \phi.
    \end{align*}
    Note that we have $\cos \phi > 0$ since $\phi \in [0, \frac{2}{\pi})$, and using the square root expression above, we can compute
    \begin{align}
        \cosh \theta_{\star} - \sin \phi \cdot \sinh \theta_{\star} &= \frac{1}{\sqrt{1 - \sin^2 \phi}} - \frac{\sin^2 \phi}{\sqrt{1 - \sin^2 \phi}} = \sqrt{1 - \sin^2 \phi} = \cos \phi.
        \label{eq:tomato}
    \end{align}
    The range of $y$ contains $y_{\star}$, since $\phi < \psi$ implies $\tan \phi \in [-\infty, \tan \psi)$. We can substitute back as
    \begin{align*}
        x_{\star} &= \frac{1}{R} \bigopen{\sin \psi - \tan \phi \cos \psi} = \frac{1}{R \cos \phi} \bigopen{\cos \phi \sin \psi - \sin \phi \cos \psi} = \frac{1}{R \cos \phi} \sin (\psi - \phi).
    \end{align*}
    By using the trigonometric identity:
    \begin{align}
        \frac{\sin \psi - \sin \phi}{\cos \psi + \cos \phi} &= \frac{2 \cos \bigopen{\frac{\psi + \phi}{2}} \sin \bigopen{\frac{\psi - \phi}{2}}}{2 \cos \bigopen{\frac{\psi + \phi}{2}} \cos \bigopen{\frac{\psi - \phi}{2}}} = \tan \bigopen{\frac{\psi - \phi}{2}},
        \label{eq:potato}
    \end{align}
    we can compute
    \begin{align*}
        \sin (\psi - \phi) &= \frac{2 \tan \bigopen{\frac{\psi - \phi}{2}}}{1 + \tan^2 \bigopen{\frac{\psi - \phi}{2}}} = \frac{2 (\cos \psi + \cos \phi) (\sin \psi - \sin \phi)}{(\cos \psi + \cos \phi)^2 + (\sin \psi - \sin \phi)^2},
    \end{align*}
    and combined with $D = R \cos \phi$ we can conclude that
    \begin{align*}
        x_{\star} &= \frac{1}{D} \cdot \frac{2(C+D)(B-A)}{(C+D)^2 + (B-A)^2}.
    \end{align*}
    
    Also, by \eqref{eq:tomato}, the minimum value can also be computed as
    \begin{align*}
        f_{\star} &= \sin \phi \sin \psi + \cos \psi \bigopen{\cosh \theta_{\star} - \sin \phi \cdot \sinh \theta_{\star}} = \sin \phi \sin \psi + \cos \phi \cos \psi = \cos (\psi - \phi).
    \end{align*}
    By using the trigonometric identity in \eqref{eq:potato}, we can compute
    \begin{align*}
        \cos (\psi - \phi) &= \frac{1 - \tan^2 \bigopen{\frac{\psi - \phi}{2}}}{1 + \tan^2 \bigopen{\frac{\psi - \phi}{2}}} = \frac{(\cos \psi + \cos \phi)^2 - (\sin \psi - \sin \phi)^2}{(\cos \psi + \cos \phi)^2 + (\sin \psi - \sin \phi)^2},
    \end{align*}
    and we can conclude that
    \begin{align*}
        f_{\star} &= \frac{(\cos \psi + \cos \phi)^2 - (\sin \psi - \sin \phi)^2}{(\cos \psi + \cos \phi)^2 + (\sin \psi - \sin \phi)^2} = \frac{(C+D)^2 - (B-A)^2}{(C+D)^2 + (B-A)^2}
    \end{align*}
    as desired.
\end{proof}

\subsubsection{\texorpdfstring{Proof of \cref{lem:fzetaineq}}{Proof of Lemma B.6}}
\label{sec:lemfzetaineq}

Here we prove \cref{lem:fzetaineq}, restated below for the sake of readability.

\fzetaineq*

\begin{proof}
    We must show that for all $x > 0$ we have $f_1 (x) \le f_2 (x)$, i.e.,
    \begin{align*}
        A_1 x + \sqrt{\bigopen{1 - B_1 x}^2 + C^2 x^2} &\le A_2 x + \sqrt{\bigopen{1 - B_2 x}^2 + C^2 x^2},
    \end{align*}
    
    which is equivalent to
    \begin{align*}
        \sqrt{\bigopen{\frac{1}{x} - B_1}^2 + C^2} - \sqrt{\bigopen{\frac{1}{x} - B_2}^2 + C^2} &\le A_2 - A_1.
    \end{align*}
    Let us substitute as follows:
    \begin{align*}
        D = A_2 - A_1 = B_2 - B_1, \ \ s = \frac{1}{x} - \frac{B_1 + B_2}{2},
    \end{align*}
    where $D \ge 0$ and $s > - \frac{B_1 + B_2}{2}$ by assumption. We are left to show that
    \begin{align}
        \sqrt{\bigopen{s + \frac{D}{2}}^2 + C^2} - \sqrt{\bigopen{s - \frac{D}{2}}^2 + C^2} &\le D. \label{eq:sibelius}
    \end{align}
    If $D = 0$, then we can observe that both sides become $0$ and hence \eqref{eq:sibelius} is indeed true.
    
    If $D > 0$, the LHS of \eqref{eq:sibelius} as a function of $s$ is a (monotonically) increasing function.
    Moreover, since
    \begin{align*}
        \lim_{s \rightarrow - \infty} \sqrt{\bigopen{s + \frac{D}{2}}^2 + C^2} - \sqrt{\bigopen{s - \frac{D}{2}}^2 + C^2} &= -D, \\
        \lim_{s \rightarrow \infty} \sqrt{\bigopen{s + \frac{D}{2}}^2 + C^2} - \sqrt{\bigopen{s - \frac{D}{2}}^2 + C^2} &= D,
    \end{align*}
    the range of the LHS is equal to $(-D, D)$, including when $C = 0$, which completes the proof.
\end{proof}

%% file: secc1.tex
\section{\texorpdfstring{Proofs used in \cref{sec:4}}{Proofs used in Section 4}} \label{sec:c}

Here we prove all theorems related to \altgda{} presented in \cref{sec:4}.
\begin{itemize}
    \item In \cref{sec:appaltgda} we prove \cref{thm:altgda} which yields a contraction inequality for \altgda{}. 
    \item In \cref{sec:coraltgda} we prove \cref{cor:altgda} which derives the corresponding iteration complexity upper bound.
    \item In \cref{sec:alttech} we prove the two main propositions introduced in \cref{sec:c}.
\end{itemize}

\paragraph{Notations.} For notational simplicity, in \cref{sec:b} we define and use the following notations for gradients:
\begin{align*}
    \vg_{ij}^{x} &:= \nabla_{\vx} f(\vx_i, \vy_j), \quad \vg_{ij}^{y} := \nabla_{\vy} f(\vx_i, \vy_j).
\end{align*}
In particular, we will use indices $i, j \in \{0, 1, \star\}$ throughout the proof.

\subsection{\texorpdfstring{Proof of \cref{thm:altgda}}{Proof of Theorem 4.1}} \label{sec:appaltgda}

Here we prove \cref{thm:altgda} of \cref{sec:4}, restated below for the sake of readability.

\theoremaltgda*

\begin{proof}
    Note that the Lyapunov function $\Psi_k^{\text{Alt}}$ for \altgda{} can be written as
    \begin{align}
        \begin{aligned}
            \Psi^{\text{Alt}}_k &= \bigopen{\frac{1}{\alpha} - \mu_x} \norm{\vx_k - \vx_{\star}}^2 + 2 \bigopen{\frac{1}{\beta} - \mu_y} \norm{\vy_k - \vy_{\star}}^2 \\
            &\phantom{=} + \bigopen{\frac{1}{\alpha} - \mu_x} \norm{\vx_{k+1} - \vx_{\star}}^2 - \alpha (1 - \alpha L_x) \norm{\nabla_{\vx} f(\vx_k, \vy_k)}^2.
        \end{aligned}
        \label{eq:lyapaltorig}
    \end{align}
    The proof consists of two steps; in \textsc{Step 1} we prove that $\Psi_k^{\text{Alt}}$ is a valid Lyapunov function, and in \textsc{Step 2} we show that $\Psi_{k+1}^{\text{Alt}} \le r \Psi_{k}^{\text{Alt}}$ holds for the contraction rate $r$ given as in \cref{thm:altgda}.
    For notational simplicity, W.L.O.G. we equivalently show that the statement holds for $k=0$ and \textit{any} choice of initialization $(\vx_0, \vy_0)$.
    (This is indeed safe because we can apply the results to each of the iterates of the whole sequence $\{ (\vx_k, \vy_k) \}_{k \ge 0}$ generated by \altgda{}.)

    \subsubsection*{Step 1. Validity of Lyapunov Function}
    
    Here we show that there exists some constant $A^{\text{Alt}}$ such that we have $\Psi_{0}^{\text{Alt}} \ge A^{\text{Alt}} \bigopen{\norm{\vx_0 - \vx_{\star}}^2 + \norm{\vy_0 - \vy_{\star}}^2}$ for any choice of initialization $(\vx_0, \vy_0)$, which is equivalent to showing that $\Psi_k^{\text{Alt}}$ is a valid Lyapunov function. \cref{prop:altvalid} yields a lower bound inequality from which we can derive such a constant $A^{\text{Alt}}$.
    
    \begin{restatable}{proposition}{propaltvalid}
        \label{prop:altvalid}
        For $f \in \gF (\mu_x, \mu_y, L_x, L_y, L_{xy})$ and \textbf{\red{Alt-GDA}} with step sizes given as in \cref{thm:altgda}, we have
        \begin{align}
            \Psi^{\emph{Alt}}_0 &\ge \bigopen{\frac{1}{2 \alpha} - \mu_x} \norm{\vx_0 - \vx_{\star}}^2 + 2 \bigopen{\frac{3}{4 \beta} - \mu_y} \norm{\vy_0 - \vy_{\star}}^2 + \bigopen{\frac{1}{\alpha} - \mu_x} \norm{\vx_{1} - \vx_{\star}}^2. \label{eq:altvalid}
        \end{align} 
        for \textit{any} choice of initialization $(\vx_0, \vy_0)$.
    \end{restatable}

    While we defer the proof of \cref{prop:altvalid} to \cref{sec:propaltvalid}, here we see that
    \begin{align*}
        \bigopen{\frac{1}{2 \alpha} - \mu_x} \norm{\vx_0 - \vx_{\star}}^2 + 2 \bigopen{\frac{3}{4 \beta} - \mu_y} \norm{\vy_0 - \vy_{\star}}^2 + \bigopen{\frac{1}{\alpha} - \mu_x} \norm{\vx_{1} - \vx_{\star}}^2 &\ge A^{\text{Alt}} \bigopen{\norm{\vx_0 - \vx_{\star}}^2 + \norm{\vy_0 - \vy_{\star}}^2}
    \end{align*}
    shows the validity of $\Psi^{\text{Alt}}_0$ for $A^{\text{Alt}} = \min \bigset{\frac{1}{2 \alpha} - \mu_x, 2 \bigopen{\frac{3}{4 \beta} - \mu_y}} > 0$.

    (Note that $\alpha \le \frac{1}{2 L_{x}} < \frac{1}{2 \mu_x}$ and $\beta \le \frac{1}{2 L_{y}} < \frac{1}{2 \mu_y} < \frac{3}{4 \mu_y}$ implies $A^{\text{Alt}} > 0$.)

    \subsubsection*{Step 2. Contraction Inequality}
    
    Here we show that $\Psialt_1 \le r \Psialt_0$ for any choice of initialization $(\vx_0, \vy_0)$, which is equivalent to showing that $\Psialt_{k+1} \le r \Psialt_{k}$ for all $k$.
    \cref{prop:altstepone} yields a one-step contraction inequality that applies to \altgda{} with $\alpha < \frac{1}{2 L_x}$ and $\beta < \frac{1}{2 L_y}$, \textit{i.e.,} when the step sizes are small enough.

    \begin{restatable}{proposition}{propaltstepone}
        \label{prop:altstepone}
        For $f \in \gF (\mu_x, \mu_y, L_x, L_y, L_{xy})$ and \textbf{\red{Alt-GDA}} with step sizes $\alpha \le \frac{1}{2 L_x}$ and $\beta \le \frac{1}{2 L_y}$, we have
        \begin{align}
            \begin{aligned}
            & {\phantom{\le}} \bigopen{\frac{1}{\alpha} - 2 \beta^2 L_y L_{xy}^2} \| \vx_1 - \vx_{\star} \|^2 + 2 \bigopen{\frac{1}{\beta} - \alpha^2 L_x L_{xy}^2} \| \vy_1 - \vy_{\star} \|^2 + \frac{1}{\alpha} \| \vx_2 - \vx_{\star} \|^2 - \alpha (1 - \alpha L_x) \norm{\vg_{11}^{x}}^2 \\
            &\le \bigopen{\frac{1}{\alpha} - \mu_x} \norm{\vx_0 - \vx_{\star}}^2 + 2 \bigopen{\frac{1}{\beta} - \mu_y} \norm{\vy_0 - \vy_{\star}}^2 + \bigopen{\frac{1}{\alpha} - \mu_x} \norm{\vx_1 - \vx_{\star}}^2 - \alpha (1 - \alpha L_x) \norm{\vg_{00}^{x}}^2
            \end{aligned} \label{eq:altineq}
        \end{align}
        for all $\vx_0 \in \R^{d_x}, \vy_0 \in \R^{d_y}$.
    \end{restatable}

    We prove \cref{prop:altstepone} in \cref{sec:propaltstepone}.
    
    Note that the choices of step sizes in \cref{thm:altgda} indeed satisfy $\alpha \le \frac{1}{2L_x}$ and $\beta \le \frac{1}{2L_y}$. Assume W.L.O.G. that $\vx_{\star} = \bm{0}$ $(\in \R^{d_x})$ and $\vy_{\star} = \bm{0}$ $(\in \R^{d_y})$. Observing that the RHS of (\ref{eq:altineq}) is exactly $\Psialt_0$, it is enough to show that
    \begin{align}
    \begin{aligned}
        \Psialt_1 &= \bigopen{\frac{1}{\alpha} - \mu_x} \norm{\vx_1 - \vx_{\star}}^2 + 2 \bigopen{\frac{1}{\beta} - \mu_y} \norm{\vy_1 - \vy_{\star}}^2 + \bigopen{\frac{1}{\alpha} - \mu_x} \norm{\vx_2 - \vx_{\star}}^2 - \alpha (1 - \alpha L_x) \norm{\vg_{11}^{x}}^2 \\
        &\le r \bigopen{\frac{1}{\alpha} - 2 \beta^2 L_y L_{xy}^2} \| \vx_1 - \vx_{\star} \|^2 + 2 r \bigopen{\frac{1}{\beta} - \alpha^2 L_x L_{xy}^2} \| \vy_1 - \vy_{\star} \|^2 + \frac{r}{\alpha} \| \vx_2 - \vx_{\star} \|^2 - r \alpha (1 - \alpha L_x) \norm{\vg_{11}^{x}}^2,
    \end{aligned} \label{eq:altcontraction}
    \end{align}
    after which we can combine the results as $r \cdot (\text{\ref{eq:altineq}}) + (\text{\ref{eq:altcontraction}})$ to obtain $\Psialt_1 \le r \Psialt_0$.
    
    Since $r \ge \frac{\frac{1}{\alpha} - \mu_x}{\frac{1}{\alpha} - 2 \beta^2 L_y L_{xy}^2}$, we have
    \begin{align*}
        \bigopen{\frac{1}{\alpha} - \mu_x} \norm{\vx_1 - \vx_{\star}}^2 \le r \bigopen{\frac{1}{\alpha} - 2 \beta^2 L_y L_{xy}^2} \| \vx_1 - \vx_{\star} \|^2.
    \end{align*}
    Since $r \ge \frac{\frac{1}{\beta} - \mu_y}{\frac{1}{\beta} - \alpha^2 L_x L_{xy}^2}$, we have
    \begin{align*}
        2 \bigopen{\frac{1}{\beta} - \mu_y} \norm{\vy_1 - \vy_{\star}}^2 \le 2 r \bigopen{\frac{1}{\beta} - \alpha^2 L_x L_{xy}^2} \| \vy_1 - \vy_{\star} \|^2.
    \end{align*}
    Since $r \ge \frac{\frac{1}{\alpha} - \mu_x}{\frac{1}{\alpha}}$, we have
    \begin{align*}
        \bigopen{\frac{1}{\alpha} - \mu_x} \norm{\vx_2 - \vx_{\star}}^2 \le \frac{r}{\alpha} \| \vx_2 - \vx_{\star} \|^2.
    \end{align*}
    Also, we can observe that $\alpha\le \frac{1}{2 L_x}$ and $\beta \le \frac{1}{2} \sqrt{\frac{\mu_x}{L_y}} \cdot \frac{1}{L_{xy}}$ implies
    \begin{align*}
        \alpha \beta^2 \le \frac{1}{2 L_x} \cdot \frac{\mu_x}{4 L_y L_{xy}^2} < \frac{1}{2 L_y L_{xy}^2} \ \ \wedge \ \  
        2 \beta^2 L_y L_{xy}^2 < 4 \beta^2 L_y L_{xy}^2 \le \mu_x \ \ \Rightarrow \ \ \frac{\frac{1}{\alpha} - \mu_x}{\frac{1}{\alpha} - 2 \beta^2 L_y L_{xy}^2} \in (0,1),
    \end{align*}
    and that $\alpha \le \frac{1}{2} \sqrt{\frac{\mu_y}{L_x}} \cdot \frac{1}{L_{xy}}$ and $\beta \le \frac{1}{2 L_y}$ implies
    \begin{align*}
         \alpha^2 \beta \le \frac{\mu_y}{4 L_x L_{xy}^2} \cdot \frac{1}{2 L_y} < \frac{1}{L_x L_{xy}^2} \ \ \wedge \ \  
        \alpha^2 L_x L_{xy}^2 < 4 \alpha^2 L_x L_{xy}^2 \le \mu_y \ \ \Rightarrow \ \ \frac{\frac{1}{\beta} - \mu_y}{\frac{1}{\beta} - \alpha^2 L_x L_{xy}^2} \in (0,1).
    \end{align*}
    Since it is obvious that $\frac{\frac{1}{\alpha} - \mu_x}{\frac{1}{\alpha}} \in (0,1)$, we can observe that
    \begin{align*}
        r &= \max \bigset{\frac{\frac{1}{\alpha} - \mu_x}{\frac{1}{\alpha} - 2 \beta^2 L_y L_{xy}^2}, \ \frac{\frac{1}{\beta} - \mu_y}{\frac{1}{\beta} - \alpha^2 L_x L_{xy}^2}, \ \frac{\frac{1}{\alpha} - \mu_x}{\frac{1}{\alpha}}} \in (0,1)
    \end{align*}
    and therefore
    \begin{align*}
        - \alpha (1 - \alpha L_x) \norm{\vg_{11}^{x}}^2 \le - r \alpha (1 - \alpha L_x) \norm{\vg_{11}^{x}}^2,
    \end{align*}
    which shows $r \in (0,1)$ and (\ref{eq:altcontraction}), and--altogether with \cref{prop:altstepone}--proves the given statement.
\end{proof}

%% file: secc2.tex
\subsection{\texorpdfstring{Proof of \cref{cor:altgda}}{Proof of Corollary 4.2}} \label{sec:coraltgda}

Here we prove \cref{cor:altgda} of \cref{sec:4}, restated below for the sake of readability.

\coraltgda*

\begin{proof}
    From \cref{thm:altgda}, we have
    \begin{align*}
        \frac{1}{1 - r} &= \max \bigset{\frac{\frac{1}{\alpha} - 2 \beta^2 L_y L_{xy}^2}{\mu_x - 2 \beta^2 L_y L_{xy}^2}, \ \frac{\frac{1}{\beta} - \alpha^2 L_x L_{xy}^2}{\mu_y - \alpha^2 L_x L_{xy}^2}, \ \frac{1}{\alpha \mu_x}}.
    \end{align*}
    From $\beta \le \frac{1}{2} \cdot \sqrt{\frac{\mu_x}{L_y}} \cdot \frac{1}{L_{xy}}$, we have
    \begin{align*}
        \frac{\frac{1}{\alpha} - 2 \beta^2 L_y L_{xy}^2}{\mu_x - 2 \beta^2 L_y L_{xy}^2} \le \frac{\frac{1}{\alpha} - \frac{1}{2} \mu_x}{\mu_x - \frac{1}{2} \mu_x} \le \frac{2}{\alpha \mu_x}.
    \end{align*}
    From $\alpha \le \frac{1}{2} \cdot \sqrt{\frac{\mu_y}{L_x}} \cdot \frac{1}{L_{xy}}$, we have
    \begin{align*}
        \frac{\frac{1}{\beta} - \alpha^2 L_x L_{xy}^2}{\mu_y - \alpha^2 L_x L_{xy}^2} \le \frac{\frac{1}{\beta} - \frac{1}{4} \mu_y}{\mu_y - \frac{1}{4} \mu_y} \le \frac{4}{3 \beta \mu_y}.
    \end{align*}
    We can deduce that
    \begin{align*}
        \frac{1}{1 - r} \le \max \bigset{\frac{2}{\alpha \mu_x}, \ \frac{4}{3 \beta \mu_y}} &= \max \bigset{\Theta \bigopen{\kappa_x + \kappa_{xy} \sqrt{\kappa_x}}, \ \Theta \bigopen{\kappa_y + \kappa_{xy} \sqrt{\kappa_y}}} \\
        &= \Theta \bigopen{\kappa_x + \kappa_y + \kappa_{xy} (\sqrt{\kappa_x} + \sqrt{\kappa_y})}.
    \end{align*}
    Therefore it is sufficient to take
    \begin{align*}
        K &= \gO \bigopen{ \bigopen{\kappa_x + \kappa_y + \kappa_{xy} (\sqrt{\kappa_x} + \sqrt{\kappa_y})} \cdot \log \frac{\Psi^{\text{Alt}}_{0}}{A^{\text{Alt}} \epsilon} }
    \end{align*}
    iterations to ensure that $\| \vz_K - \vz_{\star} \|^2 \le \epsilon$.

    Finally, we can check that $\alpha \le \frac{1}{2 L_{x}} < \frac{1}{2 \mu_x}$ and $\beta \le \frac{1}{2 L_{y}} < \frac{1}{2 \mu_y} < \frac{3}{4 \mu_y}$ implies $A^{\text{Alt}} > 0$.
\end{proof}

%% file: secc3.tex
\subsection{\texorpdfstring{Proofs used in \cref{sec:c}}{Proofs used in Appendix C}} \label{sec:alttech}

Here we prove the propositions introduced in \cref{sec:c}.

\subsubsection{\texorpdfstring{Proof of \cref{prop:altvalid}}{Proof of Proposition C.1}} \label{sec:propaltvalid}

Here we prove \cref{prop:altvalid}, restated below for the sake of readability.

\propaltvalid*

\begin{proof}
    For simplicity let us assume W.L.O.G. that $\vx_{\star} = \bm{0}$ $(\in \R^{d_x})$ and $\vy_{\star} = \bm{0}$ $(\in \R^{d_y})$.
    
    By triangle inequality and Lipschitz gradients, we have
    \begin{align*}
        \norm{{\vg_{00}^{x}}}^2 &\le 2 \norm{{\vg_{00}^{x} - \vg_{\star 0}^{x}}}^2 + 2 \norm{{\vg_{\star 0}^{x}}}^2 \le 2 L_x^2 \norm{{\vx_0}}^2 + 2 L_{xy}^2 \norm{{\vy_0}}^2.
    \end{align*}
    Therefore, we can obtain
    \begin{align*}
        & {\phantom{\ge}} \bigopen{\frac{1}{\alpha} - \mu_x} \norm{\vx_0}^2 + 2 \bigopen{\frac{1}{\beta} - \mu_y} \norm{\vy_0}^2 + \bigopen{\frac{1}{\alpha} - \mu_x} \norm{\vx_1}^2 - \alpha (1 - \alpha L_x) \norm{\vg_{00}^{x}}^2 \\
        & \ge \bigopen{\frac{1}{\alpha} - \mu_x - 2 \alpha (1 - \alpha L_x) L_x^2} \norm{\vx_0}^2 + 2 \bigopen{\frac{1}{\beta} - \mu_y - \alpha (1 - \alpha L_x) L_{xy}^2} \norm{\vy_0}^2 + \bigopen{\frac{1}{\alpha} - \mu_x} \norm{\vx_1}^2.
    \end{align*}
    Since $\alpha \le \frac{1}{2L_x}$, we have
    \begin{align*}
        \frac{1}{\alpha} - \mu_x - 2 \alpha (1 - \alpha L_x) L_x^2 &\ge \frac{1}{\alpha} - \mu_x - 2 \alpha L_x^2 \ge \frac{1}{\alpha} - \mu_x - \frac{1}{2 \alpha} = \frac{1}{2 \alpha} - \mu_x.
    \end{align*}
    Since $\alpha \le \frac{1}{2} \sqrt{\frac{\mu_y}{L_x}} \cdot \frac{1}{L_{xy}}$ and $\beta \le \frac{1}{2} \sqrt{\frac{\mu_x}{L_y}} \cdot \frac{1}{L_{xy}}$, we have
    \begin{align*}
        \frac{1}{\beta} - \mu_y - \alpha (1 - \alpha L_x) L_{xy}^2 \ge \frac{1}{\beta} - \mu_y - \alpha L_{xy}^2 \ge \frac{1}{\beta} - \mu_y - \frac{1}{4 \beta} \sqrt{\frac{\mu_y}{L_x}} \cdot \sqrt{\frac{\mu_x}{L_y}} \ge \frac{1}{\beta} - \mu_y - \frac{1}{4 \beta} = \frac{3}{4 \beta} - \mu_y.
    \end{align*}
    Therefore we have
    \begin{align*}
        & {\phantom{\ge}} \bigopen{\frac{1}{\alpha} - \mu_x - 2 \alpha (1 - \alpha L_x) L_x^2} \norm{\vx_0}^2 + 2 \bigopen{\frac{1}{\beta} - \mu_y - \alpha (1 - \alpha L_x) L_{xy}^2} \norm{\vy_0}^2 + \bigopen{\frac{1}{\alpha} - \mu_x} \norm{\vx_1}^2 \\
        &\ge \bigopen{\frac{1}{2 \alpha} - \mu_x} \norm{\vx_0}^2 + 2 \bigopen{\frac{3}{4 \beta} - \mu_y} \norm{\vy_0}^2 + \bigopen{\frac{1}{\alpha} - \mu_x} \norm{\vx_1}^2,
    \end{align*}
    which proves that \eqref{eq:altvalid} is indeed true. 
\end{proof}

\subsubsection{\texorpdfstring{Proof of \cref{prop:altstepone}}{Proof of Proposition C.2}} \label{sec:propaltstepone}

Here we prove \cref{prop:altstepone}, restated below for the sake of readability.

\propaltstepone*

\begin{proof}
Recall that \altgda{} takes updates of the form:
\begin{align*}
    \vx_{1} &= \vx_{0} - \alpha \nabla_{\vx} f(\vx_{0}, \vy_{0}) = \vx_{0} - \alpha \vg_{00}^{x}, \\
    \vy_{1} &= \vy_{0} + \beta \nabla_{\vy} f(\vx_{1}, \vy_{0}) = \vy_{0} + \beta \vg_{10}^{y}.
\end{align*}
From this, we can deduce that
\begin{align*}
    \frac{1}{\alpha} \| \vx_1 - \vx_{\star} \|^2 &= \frac{1}{\alpha} \norm{\vx_0 - \vx_{\star}}^2 + \frac{2}{\alpha} \inner{\vx_1 - \vx_0, \vx_1 - \vx_{\star}} - \frac{1}{\alpha} \norm{\vx_1 - \vx_0}^2 \\
    &= \frac{1}{\alpha} \norm{\vx_0 - \vx_{\star}}^2 - 2 \inner{\vg_{00}^{x}, \vx_1 - \vx_{\star}} - \alpha \norm{\vg_{00}^{x}}^2, \\
    \frac{2}{\beta} \| \vy_1 - \vy_{\star} \|^2 &= \frac{2}{\beta} \norm{\vy_0 - \vy_{\star}}^2 + \frac{2}{\beta} \inner{\vy_1 - \vy_0, (\vy_0 - \vy_{\star}) + (\vy_1 - \vy_{\star})} \\
    &= \frac{2}{\beta} \norm{\vy_0 - \vy_{\star}}^2 + 2 \inner{\vg_{10}^{y}, \vy_0 - \vy_{\star}} + 2 \inner{\vg_{10}^{y}, \vy_1 - \vy_{\star}}, \\
    \frac{1}{\alpha} \| \vx_2 - \vx_{\star} \|^2 &= \frac{1}{\alpha} \norm{\vx_1 - \vx_{\star}}^2 + \frac{2}{\alpha} \inner{\vx_2 - \vx_1, \vx_1 - \vx_{\star}} + \frac{1}{\alpha} \norm{\vx_2 - \vx_1}^2 \\
    &= \frac{1}{\alpha} \norm{\vx_1 - \vx_{\star}}^2 - 2 \inner{\vg_{11}^{x}, \vx_1 - \vx_{\star}} + \alpha \norm{\vg_{11}^{x}}^2,
\end{align*}
which sums up to
\begin{align}
    \begin{aligned}
        &{\phantom{=}} \frac{1}{\alpha} \| \vx_1 - \vx_{\star} \|^2 + \frac{2}{\beta} \| \vy_1 - \vy_{\star} \|^2 + \frac{1}{\alpha} \| \vx_2 - \vx_{\star} \|^2 \\
        &= \frac{1}{\alpha} \norm{\vx_0 - \vx_{\star}}^2 + \frac{2}{\beta} \norm{\vy_0 - \vy_{\star}}^2 + \frac{1}{\alpha} \| \vx_1 - \vx_{\star} \|^2 - \alpha \norm{\vg_{00}^{x}}^2 + \alpha \norm{\vg_{11}^{x}}^2 \\
        &{\phantom{=}} - 2 \inner{\vg_{00}^{x}, \vx_1 - \vx_{\star}} + 2 \inner{\vg_{10}^{y}, \vy_0 - \vy_{\star}} + 2 \inner{\vg_{10}^{y}, \vy_1 - \vy_{\star}} - 2 \inner{\vg_{11}^{x}, \vx_1 - \vx_{\star}}.
    \end{aligned} \label{eq:altcont1}
\end{align}
Then $\mu_x$-strong convexity and $L_x$-Lipschitz gradients\footnote{Note that the Lipschitz gradient conditions for $L_x$ and $L_y$ are equivalent to the widely used notion of \textit{smoothness} in convex optimization literature.} of $f(\cdot, \vy_0)$ yields:
\begin{align}
    2 \inner{\vg_{00}^{x}, {\vx_{\star} - \vx_0}} = 2 \inner{\nabla_{\vx} f({\vx_0}, \vy_0), {\vx_{\star} - \vx_0}} &\le - \mu_x \norm{{\vx_0 - \vx_{\star}}}^2 - 2 (f({\vx_0}, \vy_0) - f({\vx_{\star}}, \vy_0)), \label{eq:altproof1} \\
    2 \inner{\vg_{00}^{x}, {\vx_0 - \vx_1}} = - 2 \inner{\nabla_{\vx} f({\vx_0}, \vy_0), {\vx_1 - \vx_0}} &\le L_x \norm{{\vx_1 - \vx_0}}^2 + 2 (f({\vx_0}, \vy_0) - f({\vx_1}, \vy_0)). \label{eq:altproof2}
\end{align}
Similarly, $\mu_y$-strong concavity and $L_y$-Lipschitz gradients of $f(\vx_1, \cdot)$ yields:
\begin{align}
    2 \inner{\vg_{10}^{y}, {\vy_0 - \vy_{\star}}} = - 2 \inner{\nabla_{\vy} f(\vx_1, {\vy_0}), {\vy_{\star} - \vy_0}} &\le - \mu_y \norm{{\vy_0 - \vy_{\star}}}^2 - 2 (f(\vx_1, {\vy_{\star}}) - f(\vx_1, {\vy_0})), \label{eq:altproof3} \\
    2 \inner{\vg_{10}^{y}, {\vy_1 - \vy_0}} =  2 \inner{\nabla_{\vy} f(\vx_1, {\vy_0}), {\vy_1 - \vy_0}} &\le L_y \norm{{\vy_1 - \vy_0}}^2 + 2 (f(\vx_1, {\vy_1}) - f(\vx_1, {\vy_0})). \label{eq:altproof4}
\end{align}
Finally, $\mu_x$-strong convexity of $f(\cdot, \vy_1)$ yields:
\begin{align}
    2 \inner{\vg_{11}^{x}, {\vx_{\star} - \vx_1}} = 2 \inner{\nabla_{\vx} f({\vx_1}, \vy_1), {\vx_{\star} - \vx_1}} &\le - \mu_x \norm{{\vx_1 - \vx_{\star}}}^2 - 2 (f({\vx_1}, \vy_1) - f({\vx_{\star}}, \vy_1)). \label{eq:altproof5}
\end{align}
From now, for simplicity we assume W.L.O.G. ${\vx_{\star} = \bm{0}}$ $(\in \R^{d_x})$ and ${\vy_{\star} = \bm{0}}$ $(\in \R^{d_y})$. 

From $(\text{\ref{eq:altproof1}}) + (\text{\ref{eq:altproof2}})$ we have
\begin{align*}
    - 2 \inner{\vg_{00}^{x}, \vx_1} &= 2 \inner{\vg_{00}^{x}, {\vx_{\star} - \vx_0}} + 2 \inner{\vg_{00}^{x}, {\vx_0 - \vx_1}} \\
    &\le - \mu_x \norm{{\vx_0}}^2 + L_x \norm{{\vx_1 - \vx_0}}^2 + 2 (f(\vx_{\star}, \vy_0) - f(\vx_1, \vy_0)) \\
    &= - \mu_x \norm{{\vx_0}}^2 + \alpha^2 L_x \norm{{\vg_{00}^{x}}}^2 + 2 (f(\vx_{\star}, \vy_0) - f(\vx_1, \vy_0)).
\end{align*}
From $2 \times (\text{\ref{eq:altproof3}}) + (\text{\ref{eq:altproof4}})$ we have
\begin{align*}
    2 \inner{\vg_{10}^{y}, \vy_0} + 2 \inner{\vg_{10}^{y}, \vy_1} &= 4 \inner{\vg_{10}^{y}, {\vy_0 - \vy_{\star}}} + 2 \inner{\vg_{10}^{y}, {\vy_1 - \vy_0}} \\
    &\le - 2 \mu_y \norm{{\vy_0}}^2 + L_y \norm{{\vy_1 - \vy_0}}^2 - 2 (2 f(\vx_1, \vy_{\star}) - f(\vx_1, \vy_0) - f(\vx_1, \vy_1)) \\
    &= - 2 \mu_y \norm{{\vy_0}}^2 + \beta^2 L_y \norm{{\vg_{10}^{y}}}^2 - 2 (2 f(\vx_1, \vy_{\star}) - f(\vx_1, \vy_0) - f(\vx_1, \vy_1)).
\end{align*}
Finally, $(\text{\ref{eq:altproof5}})$ translates into
\begin{align*}
    - 2 \inner{\vg_{11}^{x}, \vx_1} &= 2 \inner{\vg_{11}^{x}, {\vx_{\star} - \vx_1}} \le - \mu_x \norm{{\vx_1}}^2 - 2 (f(\vx_1, \vy_1) - f(\vx_{\star}, \vy_1)).
\end{align*}
We can properly plug in the above equations to \cref{eq:altcont1} to obtain
\begin{align*}
    \frac{1}{\alpha} \| \vx_1 \|^2 + \frac{2}{\beta} \| \vy_1 \|^2 + \frac{1}{\alpha} \| \vx_2 \|^2 &\le \bigopen{\frac{1}{\alpha} - \mu_x} \norm{{\vx_0}}^2 + 2 \bigopen{\frac{1}{\beta} - \mu_y} \norm{{\vy_0}}^2 + \bigopen{\frac{1}{\alpha} - \mu_x} \norm{{\vx_1}}^2 \\
    & {\phantom{\le}} - \alpha (1 - \alpha L_x) \norm{\vg_{00}^{x}}^2 + \alpha \norm{\vg_{11}^{x}}^2 + \beta^2 L_y \norm{{\vg_{10}^{y}}}^2 \\
    & {\phantom{\le}} - 2 (2 f(\vx_1, \vy_{\star}) - f(\vx_{\star}, \vy_0) - f(\vx_{\star}, \vy_1)).
\end{align*}
Since $f$ is convex-concave and has Lipschitz gradients, we have
\begin{align*}
    - 2 (f(\vx_1, \vy_{\star}) - f(\vx_{\star}, \vy_{\star})) &\le - \frac{1}{L_x} \norm{\nabla_{\vx} f (\vx_1, \vy_{\star})}^2 = - \frac{1}{L_x} \norm{{\vg_{1 \star}^{x}}}^2, \\
    - 2 (f(\vx_{\star}, \vy_{\star}) - f(\vx_{\star}, \vy_0)) &\le - \frac{1}{L_y} \norm{\nabla_{\vy} f (\vx_{\star}, \vy_0)}^2 = - \frac{1}{L_y} \norm{{\vg_{\star 0}^{y}}}^2, \\
    - 2 (f(\vx_{\star}, \vy_{\star}) - f(\vx_{\star}, \vy_1)) &\le - \frac{1}{L_y} \norm{\nabla_{\vy} f (\vx_{\star}, \vy_1)}^2 = - \frac{1}{L_y} \norm{{\vg_{\star 1}^{y}}}^2.
\end{align*}
Therefore we have
\begin{align*}
    \frac{1}{\alpha} \| \vx_1 \|^2 + \frac{2}{\beta} \| \vy_1 \|^2 + \frac{1}{\alpha} \| \vx_2 \|^2 &\le \bigopen{\frac{1}{\alpha} - \mu_x} \norm{{\vx_0}}^2 + 2 \bigopen{\frac{1}{\beta} - \mu_y} \norm{{\vy_0}}^2 + \bigopen{\frac{1}{\alpha} - \mu_x} \norm{{\vx_1}}^2 \\
    & {\phantom{\le}} - \alpha (1 - \alpha L_x) \norm{\vg_{00}^{x}}^2 + \alpha \norm{\vg_{11}^{x}}^2 - \frac{2}{L_x} \norm{{\vg_{1 \star}^{x}}}^2 \\
    & {\phantom{\le}} + \beta^2 L_y \norm{{\vg_{10}^{y}}}^2 - \frac{1}{L_y} \norm{{\vg_{\star 0}^{y}}}^2 - \frac{1}{L_y} \norm{{\vg_{\star 1}^{y}}}^2 \\
    &= \bigopen{\frac{1}{\alpha} - \mu_x} \norm{{\vx_0}}^2 + 2 \bigopen{\frac{1}{\beta} - \mu_y} \norm{{\vy_0}}^2 + \bigopen{\frac{1}{\alpha} - \mu_x} \norm{{\vx_1}}^2 \\
    & {\phantom{\le}} + \alpha^2 L_x \norm{{\vg_{11}^{x}}}^2  - \frac{2}{L_x} \norm{{\vg_{1 \star}^{x}}}^2 + \beta^2 L_y \norm{{\vg_{10}^{y}}}^2 - \frac{1}{L_y} \norm{{\vg_{\star 0}^{y}}}^2 \\
    & {\phantom{\le}} - \alpha (1 - \alpha L_x) \norm{\vg_{00}^{x}}^2 + \alpha (1 - \alpha L_x) \norm{\vg_{11}^{x}}^2 - \frac{1}{L_y} \norm{{\vg_{\star 1}^{y}}}^2.
\end{align*}
By triangle inequality and the Lipschitz gradient condition for $L_{xy}$, we have the following inequalities:
\begin{align*}
    \norm{{\vg_{10}^{y}}}^2 - 2 \norm{{\vg_{\star 0}^{y}}}^2 &\le 2 \| {\vg_{10}^{y}} - {\vg_{\star 0}^{y}} \|^2 \le 2 L_{xy}^{2} \| {\vx_1} \|^2, \\
    \norm{{\vg_{11}^{x}}}^2 - 2 \norm{{\vg_{1 \star}^{x}}}^2 &\le 2 \| {\vg_{11}^{x}} - {\vg_{1 \star}^{x}} \|^2 \le 2 L_{xy}^{2} \| {\vy_1} \|^2.
\end{align*}
If $\alpha \le \frac{1}{2 L_x} \le \frac{1}{\sqrt{2} L_x}$ and $\beta \le \frac{1}{2 L_y} \le \frac{1}{\sqrt{2} L_y}$, then we have
\begin{align*}
    \alpha^2 L_x \norm{{\vg_{11}^{x}}}^2 - \frac{1}{L_x} \norm{{\vg_{1 \star}^{x}}}^2 + \beta^2 L_y \norm{{\vg_{10}^{y}}}^2 - \frac{1}{L_y} \norm{{\vg_{\star 0}^{y}}}^2 &\le \alpha^2 L_x \bigopen{\norm{{\vg_{11}^{x}}}^2 - 2 \norm{{\vg_{1 \star}^{x}}}^2} + \beta^2 L_y \bigopen{\norm{{\vg_{10}^{y}}}^2 - 2 \norm{{\vg_{\star 0}^{y}}}^2} \\
    &\le 2 \alpha^2 L_x L_{xy}^2 \norm{{\vy_1}}^2 + 2 \beta^2 L_y L_{xy}^2 \norm{{\vx_1}}^2,
\end{align*}
and hence
\begin{align*}
    \frac{1}{\alpha} \| \vx_1 \|^2 + \frac{2}{\beta} \| \vy_1 \|^2 + \frac{1}{\alpha} \| \vx_2 \|^2 &\le \bigopen{\frac{1}{\alpha} - \mu_x} \norm{{\vx_0}}^2 + 2 \bigopen{\frac{1}{\beta} - \mu_y} \norm{{\vy_0}}^2 + \bigopen{\frac{1}{\alpha} - \mu_x} \norm{{\vx_1}}^2 \\
    & {\phantom{\le}} + 2 \alpha^2 L_x L_{xy}^2 \norm{{\vy_1}}^2 + 2 \beta^2 L_y L_{xy}^2 \norm{{\vx_1}}^2 \\
    & {\phantom{\le}} - \alpha (1 - \alpha L_x) \norm{\vg_{00}^{x}}^2 + \alpha (1 - \alpha L_x) \norm{\vg_{11}^{x}}^2 - \frac{1}{L_x} \norm{{\vg_{1 \star}^{x}}}^2 - \frac{1}{L_y} \norm{{\vg_{\star 1}^{y}}}^2 \\
    &\le \bigopen{\frac{1}{\alpha} - \mu_x} \norm{{\vx_0}}^2 + 2 \bigopen{\frac{1}{\beta} - \mu_y} \norm{{\vy_0}}^2 + \bigopen{\frac{1}{\alpha} - \mu_x} \norm{{\vx_1}}^2 \\
    & {\phantom{\le}} + 2 \alpha^2 L_x L_{xy}^2 \norm{{\vy_1}}^2 + 2 \beta^2 L_y L_{xy}^2 \norm{{\vx_1}}^2 - \alpha (1 - \alpha L_x) \norm{\vg_{00}^{x}}^2 + \alpha (1 - \alpha L_x) \norm{\vg_{11}^{x}}^2.
\end{align*}
Rearranging terms, we immediately have \eqref{eq:altineq}.
\end{proof}

%% file: secd1.tex
\section{\texorpdfstring{Proofs used in \cref{sec:5}}{Proofs used in Section 5}}
\label{sec:d}

Here we prove all theorems related to \alexgda{} on SCSC Lipschitz gradient problems presented in \cref{sec:5}.
\begin{itemize}
    \item In \cref{sec:appalexgda} we prove \cref{thm:alexgda} which yields a contraction inequality for \alexgda{}{}. 
    \item In \cref{sec:coralexgda} we prove \cref{cor:alexgda} which derives the corresponding iteration complexity upper bound.
    \item In \cref{sec:thmalexgdalb} we prove \cref{thm:alexgdalb} which yields a matching lower bound for \alexgda{}.
    \item In \cref{sec:propeglb} we prove \cref{prop:eglb} which shows that the same lower bound holds for EG.
    \item In \cref{sec:alexgdatech} we prove technical propositions and lemmas used throughout the proofs in \cref{sec:d}.
\end{itemize}

\subsection{\texorpdfstring{Proof of \cref{thm:alexgda}}{Proof of Theorem 5.1}} \label{sec:appalexgda}

Here we prove \cref{thm:alexgda} of \cref{sec:5}, restated below for the sake of readability.

\theoremalexgda*

\begin{proof}
    Before starting the main proof, we characterize the step size condition as follows.
    
    \paragraph{Finer Step Size Condition.} We assume that the step sizes $\alpha, \beta > 0$ satisfy
        \begin{align}
            \alpha &\le \frac{C_1}{L_x}, \quad \beta \le \frac{C_2}{L_y}, \quad \alpha \le \frac{C_3}{L_{xy}} \sqrt{\frac{\mu_y}{\mu_x}}, \quad \beta \le \frac{C_4}{L_{xy}} \sqrt{\frac{\mu_x}{\mu_y}} \label{eq:alexgdastepsize}
        \end{align}
        for constants $C_1, C_2, C_3, C_4 > 0$ satisfying
        \begin{align}
            \begin{aligned}
                C_1 &\le \frac{\gamma - 1}{2 \gamma^2}, \quad C_2 \le \frac{\delta - 1}{2 \delta^2}, \\
                C_3 &\le \min \bigset{\frac{1}{3 \gamma - 2}, \frac{\delta - 1}{2 (\gamma - 1) \delta}, \frac{1}{2(\gamma - 1)(\delta - 1)}}, \\
                C_4 &\le \min \bigset{\frac{1}{3 \delta - 2}, \frac{\gamma - 1}{2 \gamma (\delta - 1)}, \frac{1}{2(\gamma - 1)(\delta - 1)}}.
            \end{aligned} 
            \label{eq:constineq}
        \end{align}

    (By choosing $C = \min \{ C_1, C_2, C_3, C_4 \}$, we can obtain the simpler form given in the theorem statement.)

    We show a few inequalities involving $C_1, C_2, C_3, C_4 > 0$ for future purposes.\footnote{Note that all arguments in the upper bounds of the constants given in (\ref{eq:constineq}) are all strictly positive whenever $\gamma > 1$ and $\delta > 1$.}

    First, we have
    \begin{align}
        C_1 &\le \frac{\gamma - 1}{2 \gamma^2} \le \frac{1}{2 \gamma} \le \frac{1}{2}, \quad
        C_2 \le \frac{\delta - 1}{2 \delta^2} \le \frac{1}{2 \delta} \le \frac{1}{2}. \label{eq:constsix}
    \end{align}

    Since $C_1 \le \frac{\gamma - 1}{2 \gamma^2} \le \frac{\gamma - 1}{\gamma^2}$ and $C_4 \le \frac{\gamma - 1}{2 \gamma (\delta - 1)} \le \frac{\gamma - 1}{\gamma (\delta - 1)}$, we have
    \begin{align}
         \gamma^2 C_1 + \gamma (\delta - 1) C_4 &\le 2 (\gamma - 1). \label{eq:sun}
    \end{align}
    Since $C_1 \le \frac{\gamma - 1}{2 \gamma^2} \le \frac{1}{\gamma + 1}$ and $C_4 \le \frac{1}{3 \delta - 2}$, we have
    \begin{align}
        (\gamma + 1) C_1 + (3 \delta - 2) C_4 &\le 2. \label{eq:moon}
    \end{align}
    Therefore, by $\eqref{eq:sun} + (\gamma - 1) \times \eqref{eq:moon}$ we have
    \begin{align}
        (2 \gamma^2 - 1) C_1 + \bigopen{4 \gamma \delta - 3 \gamma - 3 \delta + 2} C_4 &\le 4 (\gamma - 1). \label{eq:constone}
    \end{align}
    Since $C_2 \le {\frac{\delta - 1}{2 \delta^2}}$ and $C_3 \le {\frac{\delta - 1}{2 (\gamma - 1) \delta}}$, we have
    \begin{align}
        \delta^2 C_2 + (\gamma - 1) \delta C_3 &\le {\delta - 1}. \label{eq:consttwo}
    \end{align}
    Since $C_2 \le {\frac{\delta - 1}{2 \delta^2}} \le \frac{1}{\delta + 1}$ and $C_3 \le \frac{1}{3 \gamma - 2}$, we have
    \begin{align}
        (\delta + 1) C_2 + (3 \gamma - 2) C_3 &\le 2. \label{eq:constthree}
    \end{align}
    Since $C_1 \le \frac{1}{2}$ and $C_3 \le \frac{1}{2(\gamma - 1)(\delta - 1)}$, we have
    \begin{align}
        C_1 + (\gamma - 1) (\delta - 1) C_3 &\le 1, \label{eq:constfour}
    \end{align}
    and as $C_4 \le \frac{1}{2(\gamma - 1)(\delta - 1)}$, we similarly have
    \begin{align}
        C_1 + (\gamma - 1) (\delta - 1) C_4 &\le 1. \label{eq:constfive}
    \end{align}
    We also note that since $C_3 \le \frac{\delta - 1}{2 (\gamma - 1) \delta}$ and $C_4 \le \frac{\gamma - 1}{2 \gamma (\delta - 1)}$, we have
    \begin{align}
        C_3 C_4 \le \frac{1}{4 \delta \gamma}
    \end{align}
    which, along with $\gamma, \delta > 1$, directly implies the followings:
    \begin{align}
        4 C_3 C_4 &\le 1, \label{eq:constseven} \\
        4 (\delta - 1) C_3 C_4 &\le 1. \label{eq:consteight}
    \end{align}

    Now we proceed to the main proof of \cref{thm:alexgda}.

    For $k \ge 1$, the Lyapunov function $\Psi_k^{\text{Alex}}$ can be written as
    \begin{align*}
        \Psi_{k}^{\text{Alex}} &= \frac{1}{\alpha} \| \vx_k - \vx_{\star} \|^2 + \frac{2}{\beta} \| \vy_k - \vy_{\star} \|^2 + \frac{1}{\alpha} \| \vx_{k+1} - \vx_{\star} \|^2 - \alpha \norm{\nabla_{\vx} f(\vx_k, \tilde{\vy}_k)}^2 \\
        &\phantom{\le} + {(\delta - 1) \beta \norm{\nabla_{\vy} f(\tilde{\vx}_{k}, \vy_{k-1})}^2} + \frac{(\gamma - 1) (\delta - 1) \alpha \beta}{1 - \alpha \mu_x} \cdot L_{xy} \sqrt{\frac{\mu_y}{\mu_x}} \cdot \norm{\nabla_{\vx} f(\vx_{k-1}, \tilde{\vy}_{k-1})}^2,
    \end{align*}
    and for $k = 0$ as
    \begin{align*}
        \Psi_{0}^{\text{Alex}} &= \frac{1}{\alpha} \| \vx_0 - \vx_{\star} \|^2 + \frac{2}{\beta} \| \vy_0 - \vy_{\star} \|^2 + \frac{1}{\alpha} \| \vx_{1} - \vx_{\star} \|^2 - \alpha \norm{\nabla_{\vx} f(\vx_0, \tilde{\vy}_0)}^2 \\
        &\phantom{\le} + \frac{(\gamma - 1) (\delta - 1) \alpha \beta}{(1 - \alpha \mu_x)(1 - \beta \mu_y)} \cdot L_{xy} \sqrt{\frac{\mu_y}{\mu_x}} \cdot \norm{\nabla_{\vx} f(\vx_{0}, \tilde{\vy}_{0})}^2.
    \end{align*}
    
    Similarly as in the proof of \cref{thm:altgda}, the proof consists of two steps| in \textsc{Step 1} we prove that $\Psi_k^{\text{Alt}}$ is a valid Lyapunov function, and in \textsc{Step 2} we show that $\Psi_{k+1}^{\text{Alt}} \le r \Psi_{k}^{\text{Alt}}$ holds for the contraction rate $r$ given as in \cref{thm:alexgda}.

    \subsubsection*{Step 1. Validity of Lyapunov Function}
    
    Here we show that there exists some constant $A^{\text{Alex}}$ such that we have $\Psi_{k}^{\text{Alex}} \ge A^{\text{Alex}} \bigopen{\norm{\vx_k - \vx_{\star}}^2 + \norm{\vy_k - \vy_{\star}}^2}$, \textit{i.e.}, $\Psi_k^{\text{Alex}}$ is a valid Lyapunov function. 
    \cref{prop:digvalid} yields a lower bound inequality from which we can derive such a constant $A^{\text{Alex}}$.
    
    \begin{restatable}{proposition}{propdigvalid}
        \label{prop:digvalid}
        Suppose that we run \alexgda{} with $\gamma, \delta > 0$ and step sizes $\alpha, \beta$ satisfying \eqref{eq:alexgdastepsize},~and~\eqref{eq:constineq}. Then we have
        \begin{align}
            \Psi^{\emph{Alex}}_k &\ge \frac{1}{2 \alpha} \| \vx_k \|^2 + \frac{1}{2 \beta} \| \vy_k \|^2 + \frac{1}{\alpha} \| \vx_{k+1} \|^2
            \label{eq:digvalid}
        \end{align} 
        for all $(\vx_k, \vy_k)$, both when $k \ge 1$ and $k = 0$.
    \end{restatable}

    While we defer the proof of \cref{prop:digvalid} to \cref{sec:propdigvalid}, here we see that this implies
    \begin{align*}
        \frac{1}{2 \alpha} \| \vx_k \|^2 + \frac{1}{\beta} \| \vy_k \|^2 + \frac{1}{\alpha} \| \vx_{k+1} \|^2 &\ge A^{\text{Alex}} \bigopen{\norm{\vx_k}^2 + \norm{\vy_k}^2}
    \end{align*}
    for $A^{\text{Alex}} = \min \bigset{\frac{1}{2 \alpha}, \frac{1}{\beta}} > 0$.

    \subsubsection*{Step 2. Contraction Inequality}

    Note that this time we can't simply take $k = 0$ as in the proof of \cref{thm:altgda}, since for \alexgda{} there exists a slight difference between the first iterate and the rest, as we have briefly explained in \cref{sec:5}.

    To deal with this subtlety, here we allow ourselves to set $k = 0$ W.L.O.G. by focusing on a set of iterates given by
    
    \begin{minipage}[c]{0.45\textwidth}
    \begin{align*}
        \tilde{\vx}_{0} &= \vx_0 - \xi (\gamma - 1) \alpha \nabla_{\vx} f(\vx_{-1}, \tilde{\vy}_{-1}), \\
        \tilde{\vy}_0 &= \vy_0 + \xi (\delta - 1) \beta \nabla_{\vy} f(\tilde{\vx}_{0}, \vy_{-1}),
    \end{align*}
    \end{minipage}
    \begin{minipage}[c]{0.45\textwidth}
    \begin{align}
        \begin{aligned}
            \tilde{\vx}_{1} &= \vx_0 - \gamma \alpha \nabla_{\vx} f(\vx_0, \tilde{\vy}_0), \\
            \vx_{1} &= \vx_0 - \alpha \nabla_{\vx} f(\vx_0, \tilde{\vy_0}), \\
            \tilde{\vy}_{1} &= \vy_0 + \delta \beta \nabla_{\vy} f(\tilde{\vx}_{1}, \vy_0), \\
            \vy_{1} &= \vy_0 + \beta \nabla_{\vy} f(\tilde{\vx}_{1}, \vy_0), \\
            \tilde{\vx}_{2} &= \vx_1 - \gamma \alpha \nabla_{\vx} f(\vx_1, \tilde{\vy}_1), \\
            \vx_{2} &= \vx_1 - \alpha \nabla_{\vx} f(\vx_1, \tilde{\vy_1}),
        \end{aligned} \label{eq:alexgdaiter}
    \end{align}
    \end{minipage}
    
    \medskip
    where we can have either $\xi = 0$ or $1$.

    If $\xi = 0$, then we simply have $\vx_0 = \tilde{\vx}_0$ and $\vy_0 = \tilde{\vy}_0$, just as in the case of $k= 0$ of \alexgda{}. If $\xi = 1$, then we can bring the iterates $\tilde{\vx}_0$ and $\tilde{\vy}_0$ from the previous step, which corresponds to the case of $k \ge 1$ of \alexgda{}.
    Therefore it is safe to set $k = 0$ W.L.O.G., and it suffices to show a contraction inequality that holds for any iterates given by \eqref{eq:alexgdaiter} (including both cases of $\xi = 0$ and $1$), which we can apply to all iterates of the algorithm including both $k \ge 1$ and $k = 0$.

    \cref{prop:alexgda} gives us the main inequality which leads to the desired contraction inequality.
    
    \begin{restatable}{proposition}{propalexgda}
        \label{prop:alexgda}
        For $f \in \gF (\mu_x, \mu_y, L_x, L_y, L_{xy})$ and iterates given by (\ref{eq:alexgdaiter}) with $\gamma, \delta > 0$ and step sizes $\alpha, \beta$ satisfying \eqref{eq:alexgdastepsize},~and~\eqref{eq:constineq}, we have the contraction inequality
        \begin{align}
            \begin{aligned}
                &\phantom{\le} \frac{1}{\alpha} \| \vx_1 - \vx_{\star} \|^2 + \frac{2}{\beta} \| \vy_1 - \vy_{\star} \|^2 + \frac{1}{\alpha} \| \vx_2 - \vx_{\star} \|^2 \\
                &\phantom{\le} - \alpha \norm{\nabla_{\vx} f(\vx_1, \tilde{\vy}_1)}^2 + {(\delta - 1) \beta \norm{\nabla_{\vy} f(\tilde{\vx}_{1}, \vy_{0})}^2} + \xi (\gamma - 1) (\delta - 1) \alpha \beta L_{xy} \sqrt{\frac{\mu_y}{\mu_x}} \cdot \norm{\nabla_{\vx} f(\vx_{0}, \tilde{\vy}_{0})}^2 \\
                &\le \bigopen{\frac{1}{\alpha} - \mu_x} \| \vx_0 - \vx_{\star} \|^2 + 2 \bigopen{\frac{1}{\beta} - \mu_y} \| \vy_0 - \vy_{\star} \|^2 + \bigopen{\frac{1}{\alpha} - \mu_x} \| \vx_1 - \vx_{\star} \|^2 \\
                &\phantom{\le} - \alpha \norm{\nabla_{\vx} f(\vx_0, \tilde{\vy}_0)}^2 + \xi (\gamma - 1) (\delta - 1) \alpha \beta L_{xy} \sqrt{\frac{\mu_y}{\mu_x}} \cdot \norm{\nabla_{\vx} f(\vx_{-1}, \tilde{\vy}_{-1})}^2,
            \end{aligned}
            \label{eq:alexgdacontraction}
        \end{align}
        where $\xi = 0$ or $1$.
    \end{restatable}
    
    We prove \cref{prop:alexgda} in \cref{sec:propalexgda}.
    Note that we can simplify the step size conditions as given in the theorem statement by choosing $C$ as the minimum of the upper bounds of the constants given in \eqref{eq:constineq}.

    First, let us assume that $\xi = 1$. 
    Note that by \cref{prop:alexgda} we have
    \begin{align*}
        &\phantom{\le} \frac{1}{\alpha} \| \vx_1 - \vx_{\star} \|^2 + \frac{2}{\beta} \| \vy_1 - \vy_{\star} \|^2 + \frac{1}{\alpha} \| \vx_2 - \vx_{\star} \|^2 \\
        &\phantom{\le} - \alpha \norm{\nabla_{\vx} f(\vx_1, \tilde{\vy}_1)}^2 + {(\delta - 1) \beta \norm{\nabla_{\vy} f(\tilde{\vx}_{1}, \vy_{0})}^2} + (\gamma - 1) (\delta - 1) \alpha \beta L_{xy} \sqrt{\frac{\mu_y}{\mu_x}} \cdot \norm{\nabla_{\vx} f(\vx_{0}, \tilde{\vy}_{0})}^2 \\
        &\le \bigopen{\frac{1}{\alpha} - \mu_x} \| \vx_0 - \vx_{\star} \|^2 + 2 \bigopen{\frac{1}{\beta} - \mu_y} \| \vy_0 - \vy_{\star} \|^2 + \bigopen{\frac{1}{\alpha} - \mu_x} \| \vx_1 - \vx_{\star} \|^2 \\
        &\phantom{\le} - \alpha \norm{\nabla_{\vx} f(\vx_0, \tilde{\vy}_0)}^2 + (\gamma - 1) (\delta - 1) \alpha \beta L_{xy} \sqrt{\frac{\mu_y}{\mu_x}} \cdot \norm{\nabla_{\vx} f(\vx_{-1}, \tilde{\vy}_{-1})}^2.
    \end{align*}
    
    We can add $\frac{1}{1 - \alpha \mu_x} \cdot (\gamma - 1) (\delta - 1) \alpha^2 \beta L_{xy} \sqrt{\mu_x \mu_y} \norm{\nabla_{\vx} f(\vx_0, \tilde{\vy}_0)}^2$ to both sides so that we have
    \begin{align}
        \begin{aligned}
            &\phantom{\le} \frac{1}{\alpha} \| \vx_1 - \vx_{\star} \|^2 + \frac{2}{\beta} \| \vy_1 - \vy_{\star} \|^2 + \frac{1}{\alpha} \| \vx_2 - \vx_{\star} \|^2 - \alpha \norm{\nabla_{\vx} f(\vx_1, \tilde{\vy}_1)}^2 \\
            &\phantom{\le} + {(\delta - 1) \beta \norm{\nabla_{\vy} f(\tilde{\vx}_{1}, \vy_{0})}^2} + \frac{(\gamma - 1) (\delta - 1) \alpha \beta}{1 - \alpha \mu_x} L_{xy} \sqrt{\frac{\mu_y}{\mu_x}} \cdot \norm{\nabla_{\vx} f(\vx_{0}, \tilde{\vy}_{0})}^2 \\
            &\le \bigopen{\frac{1}{\alpha} - \mu_x} \| \vx_0 - \vx_{\star} \|^2 + 2 \bigopen{\frac{1}{\beta} - \mu_y} \| \vy_0 - \vy_{\star} \|^2 + \bigopen{\frac{1}{\alpha} - \mu_x} \| \vx_1 - \vx_{\star} \|^2 \\
            &\phantom{\le} - \alpha \bigopen{1 - \frac{(\gamma - 1) (\delta - 1) \alpha \beta}{1 - \alpha \mu_x} L_{xy} \sqrt{\mu_x \mu_y}} \norm{\nabla_{\vx} f(\vx_0, \tilde{\vy}_0)}^2 + (\gamma - 1) (\delta - 1) \alpha \beta L_{xy} \sqrt{\frac{\mu_y}{\mu_x}} \cdot \norm{\nabla_{\vx} f(\vx_{-1}, \tilde{\vy}_{-1})}^2.
        \end{aligned}
        \label{eq:brahms}
    \end{align}
    Now let us define $r$ as
    \begin{align*}
        r &= \max \bigset{1 - \alpha \mu_x, 1 - \beta \mu_y}.
    \end{align*}
    Since $r \ge 1 - \alpha \mu_x$ and $r \ge 1 - \beta \mu_y$, we have
    \begin{align}
        \begin{aligned}
            \bigopen{\frac{1}{\alpha} - \mu_x} \| \vx_0 - \vx_{\star} \|^2 &\le r \cdot \frac{1}{\alpha} \| \vx_0 - \vx_{\star} \|^2, \\
            2 \bigopen{\frac{1}{\beta} - \mu_y} \| \vy_0 - \vy_{\star} \|^2 &\le r \cdot \frac{2}{\beta} \| \vy_0 - \vy_{\star} \|^2, \\
            \bigopen{\frac{1}{\alpha} - \mu_x} \| \vx_1 - \vx_{\star} \|^2 &\le r \cdot \frac{1}{\alpha} \| \vx_1 - \vx_{\star} \|^2.
        \end{aligned}
        \label{eq:tchaikovsky}
    \end{align}
    Since $r \ge 1 - \alpha \mu_x$, we have
    \begin{align*}
        (\gamma - 1) (\delta - 1) \alpha \beta L_{xy} \sqrt{\frac{\mu_y}{\mu_x}} \cdot \norm{\nabla_{\vx} f(\vx_{-1}, \tilde{\vy}_{-1})}^2 &\le r \frac{(\gamma - 1) (\delta - 1) \alpha \beta}{1 - \alpha \mu_x} L_{xy} \sqrt{\frac{\mu_y}{\mu_x}} \cdot \norm{\nabla_{\vx} f(\vx_{-1}, \tilde{\vy}_{-1})}^2.
    \end{align*}

    Now we will show that the following holds for the negative gradient terms:
    \begin{align}
        - \alpha \bigopen{1 - \frac{(\gamma - 1) (\delta - 1) \alpha \beta}{1 - \alpha \mu_x} L_{xy} \sqrt{\mu_x \mu_y}} \norm{\nabla_{\vx} f(\vx_0, \tilde{\vy}_0)}^2 &\le - r \alpha \norm{\nabla_{\vx} f(\vx_0, \tilde{\vy}_0)}^2.
        \label{eq:schumann}
    \end{align}
    
    Observe that 
    \begin{align}
        \frac{1}{1 - \alpha \mu_x} \le \frac{1}{1 - \alpha L_x} \le \frac{1}{1 - C_1}. \label{eq:mozart}
    \end{align}
    Recalling inequality \eqref{eq:constfive}, we have
    \begin{align*}
        C_1 + (\gamma - 1) (\delta - 1) C_4 &\le 1,
    \end{align*}
    which, combined with \eqref{eq:mozart}, gives
    \begin{align*}
        \frac{1}{1 - \alpha \mu_x} (\gamma - 1) (\delta - 1) C_4 \le \frac{1}{1 - C_1} (\gamma - 1) (\delta - 1) C_4 \le 1.
    \end{align*}
    The condition $\beta \le \frac{C_4}{L_{xy}} \sqrt{\frac{\mu_x}{\mu_y}}$ then yields
    \begin{align*}
        \frac{1 - \alpha \mu_x}{(\gamma - 1) (\delta - 1) \alpha \beta L_{xy} \sqrt{\mu_x \mu_y}} &\ge \frac{1 - \alpha \mu_x}{(\gamma - 1) (\delta - 1) C_4 \alpha \mu_x} \ge \frac{1}{\alpha \mu_x}.
    \end{align*}

    Similarly, recalling inequality \eqref{eq:constfour}, we have
    \begin{align*}
        C_1 + (\gamma - 1) (\delta - 1) C_3 &\le 1,
    \end{align*}
    which, combined with \eqref{eq:mozart}, gives
    \begin{align*}
        \frac{1}{1 - \alpha \mu_x} (\gamma - 1) (\delta - 1) C_3 \le \frac{1}{1 - C_1} (\gamma - 1) (\delta - 1) C_3 \le 1.
    \end{align*}
    The conditions $\alpha \le \frac{C_3}{L_{xy}} \sqrt{\frac{\mu_y}{\mu_x}}$ then yields
    \begin{align*}
        \frac{1 - \alpha \mu_x}{(\gamma - 1) (\delta - 1) \alpha \beta L_{xy} \sqrt{\mu_x \mu_y}} &\ge \frac{1 - \alpha \mu_x}{(\gamma - 1) (\delta - 1) C_3 \beta \mu_y} \ge \frac{1}{\beta \mu_y}.
    \end{align*}
    
    Therefore we have
    \begin{align*}
        \frac{1 - \alpha \mu_x}{(\gamma - 1) (\delta - 1) \alpha \beta L_{xy} \sqrt{\mu_x \mu_y}} \ge \frac{1}{1 - r} = \max \bigset{\frac{1}{\alpha \mu_x}, \frac{1}{\beta \mu_y}},
    \end{align*}
    or equivalently
    \begin{align*}
        1 - \frac{(\gamma - 1) (\delta - 1) \alpha \beta}{1 - \alpha \mu_x} L_{xy} \sqrt{\mu_x \mu_y} \ge r,
    \end{align*}
    from which \eqref{eq:schumann} immediately follows.
    Finally, we can just add:
    \begin{align}
        0 \le r {(\delta - 1) \beta \norm{\nabla_{\vy} f(\tilde{\vx}_{0}, \vy_{-1})}^2}.
        \label{eq:rachmaninoff}
    \end{align}
    Aggregating \eqref{eq:brahms}, \eqref{eq:tchaikovsky}, \eqref{eq:schumann}, and \eqref{eq:rachmaninoff}, we can obtain
    \begin{align*}
        &\phantom{\le} \frac{1}{\alpha} \| \vx_1 - \vx_{\star} \|^2 + \frac{2}{\beta} \| \vy_1 - \vy_{\star} \|^2 + \frac{1}{\alpha} \| \vx_2 - \vx_{\star} \|^2 \\
        &\phantom{\le} - \alpha \norm{\nabla_{\vx} f(\vx_1, \tilde{\vy}_1)}^2 + {(\delta - 1) \beta \norm{\nabla_{\vy} f(\tilde{\vx}_{1}, \vy_{0})}^2} + \frac{(\gamma - 1) (\delta - 1) \alpha \beta}{1 - \alpha \mu_x} L_{xy} \sqrt{\frac{\mu_y}{\mu_x}} \cdot \norm{\nabla_{\vx} f(\vx_{0}, \tilde{\vy}_{0})}^2 \\
        &\le r \bigopen{\frac{1}{\alpha} \| \vx_0 - \vx_{\star} \|^2 + \frac{2}{\beta} \| \vy_0 - \vy_{\star} \|^2 + \frac{1}{\alpha} \| \vx_1 - \vx_{\star} \|^2} \\
        &\phantom{\le} - r \alpha \norm{\nabla_{\vx} f(\vx_0, \tilde{\vy}_0)}^2 + r {(\delta - 1) \beta \norm{\nabla_{\vy} f(\tilde{\vx}_{0}, \vy_{-1})}^2} + r \frac{(\gamma - 1) (\delta - 1) \alpha \beta}{1 - \alpha \mu_x} L_{xy} \sqrt{\frac{\mu_y}{\mu_x}} \cdot \norm{\nabla_{\vx} f(\vx_{-1}, \tilde{\vy}_{-1})}^2
    \end{align*}
    which -- as $\xi = 1$ corresponds to iterates of \alexgda{} for $k \ge 1$ -- concludes that $\Psi_{k+1}^{\text{Alex}} \le r \Psi_{k}^{\text{Alex}}$ for $k \ge 1$.

    Now suppose that $\xi = 0$. Then by \cref{prop:alexgda} we have
    \begin{align*}
        &\phantom{=} \frac{1}{\alpha} \| \vx_1 - \vx_{\star} \|^2 + \frac{2}{\beta} \| \vy_1 - \vy_{\star} \|^2 + \frac{1}{\alpha} \| \vx_2 - \vx_{\star} \|^2 - \alpha \norm{\nabla_{\vx} f(\vx_1, \tilde{\vy}_1)}^2 + {(\delta - 1) \beta \norm{\nabla_{\vy} f(\tilde{\vx}_{1}, \vy_{0})}^2} \\
        &\le \bigopen{\frac{1}{\alpha} - \mu_x} \| \vx_0 - \vx_{\star} \|^2 + 2 \bigopen{\frac{1}{\beta} - \mu_y} \| \vy_0 - \vy_{\star} \|^2 + \bigopen{\frac{1}{\alpha} - \mu_x} \| \vx_1 - \vx_{\star} \|^2 - \alpha \norm{\nabla_{\vx} f(\vx_0, \tilde{\vy}_0)}^2.
    \end{align*}
    Then, since $1 - \beta \mu_y \le r \le 1$ and \eqref{eq:tchaikovsky} holds for this case as well, we have
    \begin{align*}
        &\phantom{\le} \frac{1}{\alpha} \| \vx_1 - \vx_{\star} \|^2 + \frac{2}{\beta} \| \vy_1 - \vy_{\star} \|^2 + \frac{1}{\alpha} \| \vx_2 - \vx_{\star} \|^2 \\
        &\phantom{\le} - \alpha \norm{\nabla_{\vx} f(\vx_1, \tilde{\vy}_1)}^2 + {(\delta - 1) \beta \norm{\nabla_{\vy} f(\tilde{\vx}_{1}, \vy_{0})}^2} + \frac{(\gamma - 1) (\delta - 1) \alpha \beta}{1 - \alpha \mu_x}\cdot  L_{xy} \sqrt{\frac{\mu_y}{\mu_x}} \cdot \norm{\nabla_{\vx} f(\vx_{0}, \tilde{\vy}_{0})}^2 \\
        &\le \bigopen{\frac{1}{\alpha} - \mu_x} \| \vx_0 - \vx_{\star} \|^2 + 2 \bigopen{\frac{1}{\beta} - \mu_y} \| \vy_0 - \vy_{\star} \|^2 + \bigopen{\frac{1}{\alpha} - \mu_x} \| \vx_1 - \vx_{\star} \|^2 \\
        &\phantom{\le} - \alpha \norm{\nabla_{\vx} f(\vx_0, \tilde{\vy}_0)}^2 + \frac{(\gamma - 1) (\delta - 1) \alpha \beta}{1 - \alpha \mu_x}\cdot  L_{xy} \sqrt{\frac{\mu_y}{\mu_x}} \cdot \norm{\nabla_{\vx} f(\vx_{0}, \tilde{\vy}_{0})}^2 \\
        &\le r \bigopen{\frac{1}{\alpha} \| \vx_0 - \vx_{\star} \|^2 + \frac{2}{\beta} \| \vy_0 - \vy_{\star} \|^2 + \frac{1}{\alpha} \| \vx_1 - \vx_{\star} \|^2} \\
        &\phantom{\le} - r \alpha \norm{\nabla_{\vx} f(\vx_0, \tilde{\vy}_0)}^2 + \frac{r}{1 - \beta \mu_y} \cdot \frac{(\gamma - 1) (\delta - 1) \alpha \beta}{1 - \alpha \mu_x} \cdot L_{xy} \sqrt{\frac{\mu_y}{\mu_x}} \cdot \norm{\nabla_{\vx} f(\vx_{0}, \tilde{\vy}_{0})}^2
    \end{align*}
    which -- as $\xi = 0$ corresponds to the iterate of \alexgda{} for $k = 0$ -- concludes that $\Psi_{1}^{\text{Alex}} \le r \Psi_{0}^{\text{Alex}}$.
\end{proof}

%% file: secd2.tex
\subsection{\texorpdfstring{Proof of \cref{cor:alexgda}}{Proof of Corollary 5.2}} \label{sec:coralexgda}

Here we prove \cref{cor:alexgda} of \cref{sec:5}, restated below for the sake of readability.

\coralexgda*

\begin{proof}
    In \cref{thm:alexgda} we have shown that $\Psi_{k+1}^{\text{Alex}} \le r \Psi_{k}^{\text{Alex}}$ for all $k \ge 0$ with $r = \max \bigset{1 - \alpha \mu_x, 1 - \beta \mu_y}$.
    
    Since we choose $\alpha = \Theta (\min \bigset{\frac{1}{L_x}, \frac{\sqrt{\mu_y}}{L_{xy} \sqrt{\mu_x}}})$ and $\beta = \Theta (\min \bigset{\frac{1}{L_y}, \frac{\sqrt{\mu_x}}{L_{xy} \sqrt{\mu_y}}})$, we have
    \begin{align*}
        \frac{1}{1 - r} &= \max \bigset{\frac{1}{\alpha \mu_x}, \frac{1}{\beta \mu_y}} \\
        &= \Theta \bigopen{\max \bigset{\frac{L_x}{\mu_x}, \frac{L_y}{\mu_y}, \frac{L_{xy}}{\sqrt{\mu_x \mu_y}}}} = \Theta \bigopen{\kappa_x + \kappa_y + \kappa_{xy}}.
    \end{align*}
    
    Therefore it is sufficient to run
    \begin{align*}
        K  &= \gO \bigopen{\bigopen{\kappa_x + \kappa_y + \kappa_{xy}} \cdot \log \frac{\Psi_{0}^{\text{Alex}}}{A^{\text{Alex}} \epsilon}}
    \end{align*}
    iterations to ensure that $\| \vz_K - \vz_{\star} \|^2 \le \epsilon$, where $A^{\text{Alex}} = \min \bigset{\frac{1}{2 \alpha}, \frac{1}{\beta}}$.
\end{proof}

%% file: secd3.tex
\subsection{\texorpdfstring{Proof of \cref{thm:alexgdalb}}{Proof of Theorem 5.3}} \label{sec:thmalexgdalb}

Here we prove \cref{thm:alexgdalb} of \cref{sec:5}, restated below for the sake of readability.

\thmalexgdalb*

\begin{proof}
    We use the same worst-case function as in \cref{thm:simgdalb}:
    \begin{align*}
        f(\vx, \vy) &= 
        \frac{1}{2} \begin{bmatrix}
            {x} \\ {s} \\ {t} \\ {y} \\ {u} \\ {v}
        \end{bmatrix}^{\top}
        \begin{bmatrix}
            \mu_x & 0 & 0 & L_{xy} & 0 & 0 \\
            0 & \mu_x & 0 & 0 & 0 & 0 \\
            0 & 0 & L_x & 0 & 0 & 0 \\
            L_{xy} & 0 & 0 & -\mu_y & 0 & 0 \\
            0 & 0 & 0 & 0 & -\mu_y & 0 \\
            0 & 0 & 0 & 0 & 0& -L_y 
        \end{bmatrix}
        \begin{bmatrix}
            {x} \\ {s} \\ {t} \\ {y} \\ {u} \\ {v}
        \end{bmatrix}, \\
    \end{align*}
    where ${\vx = (x, s, t)}$ and ${\vy = (y, u, v)}$. 
    It can be easily checked that $f$ is a quadratic function ({\it i.e.,} Hessian is constant) such that $f \in \gF(\mu_x, \mu_y, L_x, L_y, L_{xy})$ and $\vx_\star = \vy_\star = \bm{0} \in \R^3$.

    We first observe that if we let
    \begin{align*}
        \mA &= 
        \begin{bmatrix}
            \mu_x & 0 & 0 \\
            0 & \mu_x & 0 \\
            0 & 0 & L_x
        \end{bmatrix}, \quad
        \mB = 
        \begin{bmatrix}
            L_{xy} & 0 & 0 \\
            0 & 0 & 0 \\
            0 & 0 & 0
        \end{bmatrix}, \quad
        \mC =
        \begin{bmatrix}
            \mu_y & 0 & 0 \\
            0 & \mu_y & 0 \\
            0 & 0 & L_y
        \end{bmatrix},
    \end{align*}
    then the $k$-th step of \alexgda{} satisfies
    \begin{align*}
        \begin{bmatrix}
            \vx_{k+1} \\ \tilde{\vx}_{k+1} \\ \vy_{k+1} \\ \tilde{\vy}_{k+1}
        \end{bmatrix}
        &=
        \begin{bmatrix}
            \mI & \bm{0} & \bm{0} & \bm{0} \\ 
            \bm{0} & \mI & \bm{0} & \bm{0} \\ 
            \bm{0} & \beta \mB^{\top} & \mI - \beta \mC & \bm{0} \\
            \bm{0} & \delta \beta \mB^{\top} & \mI - \delta \beta \mC & \bm{0}
        \end{bmatrix}
        \begin{bmatrix}
            \mI - \alpha \mA & \bm{0} & \bm{0} & - \alpha \mB \\ 
            \mI - \gamma \alpha \mA & \bm{0} & \bm{0} & - \gamma \alpha \mB \\ 
            \bm{0} & \bm{0} & \mI & \bm{0} \\ 
            \bm{0} & \bm{0} & \bm{0} & \mI
        \end{bmatrix}
        \begin{bmatrix}
            \vx_{k} \\ \tilde{\vx}_{k} \\ \vy_{k} \\ \tilde{\vy}_{k}
        \end{bmatrix} \\
        &=
        \begin{bmatrix}
            \mI - \alpha \mA & \bm{0} & \bm{0} & - \alpha \mB \\ 
            \mI - \gamma \alpha \mA & \bm{0} & \bm{0} & - \gamma \alpha \mB \\ 
            \beta \mB^{\top} (\mI - \alpha \mA) & \bm{0} & \mI - \beta \mC &  - \gamma \alpha \beta \mB^{\top} \mB \\ 
            \delta \beta \mB^{\top} (\mI - \gamma \alpha \mA) & \bm{0} & \mI - \delta \beta \mC & - \gamma \delta \alpha \beta \mB^{\top} \mB
        \end{bmatrix}
        \begin{bmatrix}
            \vx_{k} \\ \tilde{\vx}_{k} \\ \vy_{k} \\ \tilde{\vy}_{k}
        \end{bmatrix}.
    \end{align*}

    Therefore we have the following coordinate-wise updates:
    \begin{align}
        \begin{bmatrix}
            x_{k+1} \\ \tilde{x}_{k+1} \\ y_{k+1} \\ \tilde{y}_{k+1}
        \end{bmatrix} &= 
        \underbrace{\begin{bmatrix}
            1 - \alpha \mu_x & 0 & 0 & - \alpha L_{xy} \\
            1 - \gamma \alpha \mu_x & 0 & 0 & - \gamma \alpha L_{xy} \\
            \beta L_{xy} (1 - \gamma \alpha \mu_x) & 0 & 1 - \beta \mu_y & - \gamma \alpha \beta L_{xy}^2 \\
            \delta \beta L_{xy} (1 - \gamma \alpha \mu_x) & 0 & 1 - \delta \beta \mu_y & - \gamma \delta \alpha \beta L_{xy}^2
        \end{bmatrix}}_{\triangleq \mP}
        \begin{bmatrix}
            x_k \\ \tilde{x}_{k} \\ y_{k} \\ y_k
        \end{bmatrix}, \label{eq:x_k_y_k_dig}\\
        s_{k+1} &= (1 - \alpha \mu_x) s_k, \ \ \tilde{s}_{k+1} = (1 - \gamma \alpha \mu_x) s_k, \label{eq:s_k_dig} \\
        t_{k+1} &= (1 - \alpha L_x) t_k, \ \ \tilde{t}_{k+1} = (1 - \gamma \alpha L_x) t_k, \label{eq:t_k_dig} \\
        u_{k+1} &= (1 - \beta \mu_y) u_k, \ \ \tilde{u}_{k+1} = (1 - \delta \beta \mu_y) u_k, \label{eq:u_k_dig} \\
        v_{k+1} &= (1 - \beta L_y) v_k, \ \ \tilde{v}_{k+1} = (1 - \delta \beta L_y) v_k. \label{eq:v_k_dig} 
    \end{align}
    To assure the convergence of iterations \eqref{eq:t_k_dig}~and~\eqref{eq:v_k_dig}, the step sizes $\alpha$ and $\beta$ are required to be 
    \begin{equation}
        \label{eq:step_base_dig}
        {\alpha < \frac{2}{L_x}} \quad \text{and} \quad {\beta < \frac{2}{L_y}}.
    \end{equation}
    Also, to guarantee $\bignorm{\vx_K}^2 + \bignorm{\vy_K}^2 < \epsilon$, we need from \eqref{eq:s_k_dig}~and~\eqref{eq:u_k_dig} that $s_K^2 < \gO(\epsilon)$ and $u_K^2 < \gO(\epsilon)$, respectively.
    These two necessary conditions require an iteration number of at least:
    \begin{equation}
        \label{eq:cpxty_base_dig}
        K = \Omega\bigopen{\bigopen{{\frac{1}{\alpha \mu_x}} + {\frac{1}{\beta \mu_y}}} \cdot \log\frac{1}{\epsilon}},
    \end{equation}
    and $\alpha L_x, \beta L_y = \gO (1)$ from \eqref{eq:cpxty_base_dig} yields
    \begin{align}
        \frac{1}{\alpha \mu_x} + \frac{1}{\beta \mu_y} = \Omega (\kappa_x + \kappa_y).
        \label{eq:screen}
    \end{align}
    
    Now, in order to ensure convergence of iteration \eqref{eq:x_k_y_k_dig}, we need the following matrix
    \begin{align*}
        \mP &= \begin{bmatrix}
            1 - \alpha \mu_x & 0 & 0 & - \alpha L_{xy} \\
            1 - \gamma \alpha \mu_x & 0 & 0 & - \gamma \alpha L_{xy} \\
            \beta L_{xy} (1 - \gamma \alpha \mu_x) & 0 & 1 - \beta \mu_y & - \gamma \alpha \beta L_{xy}^2 \\
            \delta \beta L_{xy} (1 - \gamma \alpha \mu_x) & 0 & 1 - \delta \beta \mu_y & - \gamma \delta \alpha \beta L_{xy}^2
        \end{bmatrix}
    \end{align*}
    to have a spectral radius smaller than one.
    Hence it suffices to show that $\rho(\mP) < 1$ implies that $\frac{1}{\alpha \mu_x} + \frac{1}{\beta \mu_y} = \Omega(\kappa_{xy})$.
    
    Suppose that $\lambda$ is an eigenvalue of $\mP$. Then we must have
    \begin{align*}
        \det (\lambda \mI - \mP) &= 
        \begin{vmatrix}
            (1 - \lambda) - \alpha \mu_x & 0 & 0 & - \alpha L_{xy} \\
            1 - \gamma \alpha \mu_x & - \lambda & 0 & - \gamma \alpha L_{xy} \\
            \beta L_{xy} (1 - \gamma \alpha \mu_x) & 0 & (1 - \lambda) - \beta \mu_y & - \gamma \alpha \beta L_{xy}^2 \\
            \delta \beta L_{xy} (1 - \gamma \alpha \mu_x) & 0 & 1 - \delta \beta \mu_y & - \lambda - \gamma \delta \alpha \beta L_{xy}^2
        \end{vmatrix} \\
        &= - \lambda \cdot 
        \begin{vmatrix}
            (1 - \lambda) - \alpha \mu_x & 0 & - \alpha L_{xy} \\
            \beta L_{xy} (1 - \gamma \alpha \mu_x) & (1 - \lambda) - \beta \mu_y & - \gamma \alpha \beta L_{xy}^2 \\
            \delta \beta L_{xy} (1 - \gamma \alpha \mu_x) & 1 - \delta \beta \mu_y & - \lambda - \gamma \delta \alpha \beta L_{xy}^2
        \end{vmatrix} = 0.
    \end{align*}
    We can compute
    \begin{align*}
        &\phantom{=} \begin{vmatrix}
            (1 - \lambda) - \alpha \mu_x & 0 & - \alpha L_{xy} \\
            \beta L_{xy} (1 - \gamma \alpha \mu_x) & (1 - \lambda) - \beta \mu_y & - \gamma \alpha \beta L_{xy}^2 \\
            \delta \beta L_{xy} (1 - \gamma \alpha \mu_x) & 1 - \delta \beta \mu_y & - \lambda - \gamma \delta \alpha \beta L_{xy}^2.
        \end{vmatrix} \\
        &= \bigopen{(1 - \lambda) - \alpha \mu_x} \bigopen{(1 - \lambda) - \beta \mu_y} \bigopen{- \lambda - \gamma \delta \alpha \beta L_{xy}^2} - \alpha \beta L_{xy}^2 (1 - \gamma \alpha \mu_x) (1 - \delta \beta \mu_y) \\
        &\phantom{=} + \delta \alpha \beta L_{xy}^2 (1 - \gamma \alpha \mu_x) \bigopen{(1 - \lambda) - \beta \mu_y} + \gamma \alpha \beta L_{xy}^2 (1 - \delta \beta \mu_y) \bigopen{(1 - \lambda) - \alpha \mu_x}.
    \end{align*}
    Substituting $\lambda = 1 - t$ and $\phi = \alpha \beta L_{xy}^2$, we can obtain a simpler expression:
    \begin{align*}
        &\phantom{=} \bigopen{t - \alpha \mu_x} \bigopen{t - \beta \mu_y} \bigopen{t - 1 - \gamma \delta \phi} - \phi (1 - \gamma \alpha \mu_x) (1 - \delta \beta \mu_y) \\
        &\phantom{=} + \delta \phi (1 - \gamma \alpha \mu_x) \bigopen{t - \beta \mu_y} + \gamma \phi (1 - \delta \beta \mu_y) \bigopen{t - \alpha \mu_x} \\
        &= \bigopen{t - \alpha \mu_x} \bigopen{t - \beta \mu_y} \bigopen{t - 1} - \gamma \delta \phi \bigopen{t - \alpha \mu_x} \bigopen{t - \beta \mu_y} - \phi (1 - \gamma \alpha \mu_x) (1 - \delta \beta \mu_y) \\
        &\phantom{=} + \delta \phi (1 - \gamma \alpha \mu_x) \bigopen{t - \beta \mu_y} + \gamma \phi (1 - \delta \beta \mu_y) \bigopen{t - \alpha \mu_x} \\
        &= \bigopen{t - \alpha \mu_x} \bigopen{t - \beta \mu_y} \bigopen{t - 1} - \phi \bigopen{(1 - \gamma \alpha \mu_x) - \gamma \bigopen{t - \alpha \mu_x}} \bigopen{(1 - \delta \beta \mu_y) - \delta \bigopen{t - \beta \mu_y}} \\
        &= \bigopen{t - \alpha \mu_x} \bigopen{t - \beta \mu_y} \bigopen{t - 1} - \phi (1 - \gamma t) (1 - \delta t).
    \end{align*}
    Therefore the eigenvalue $\lambda$ must be $0$ or take the form of $1-t^*$, where $t^*$ is a root of the following cubic equation:
    \begin{align*}
        \bigopen{t - \alpha \mu_x} \bigopen{t - \beta \mu_y} \bigopen{t - 1} - \phi (1 - \gamma t) (1 - \delta t) = 0.
    \end{align*}
    We can expand as
    \begin{align*}
        t^3 - (1 + \alpha \mu_x + \beta \mu_y + \gamma \delta \phi) t^2 + (\alpha \mu_x + \beta \mu_y + \alpha \beta \mu_x \mu_y + (\gamma + \delta) \phi) t - (\alpha \beta \mu_x \mu_y + \phi) = 0.
    \end{align*}
    Hence we have a cubic equation of the form $t^3 - p t^2 + qt - r = 0$ with coefficients given by
    \begin{align}
        \begin{aligned}
        p &= 1 + \alpha \mu_x + \beta \mu_y + \gamma \delta \phi, \\
        q &= \alpha \mu_x + \beta \mu_y + \alpha \beta \mu_x \mu_y + (\gamma + \delta) \phi, \\
        r &= \alpha \beta \mu_x \mu_y + \phi.
        \end{aligned} \label{eq:stupidme}
    \end{align}
    Note that we obviously have $p, q, r > 0$.

    There exists a well-known characterization of cubic polynomials having roots with absolute values less than one.
    
    \begin{proposition}[\citet{grove2004periodicities}, Theorem 1.4]
        \label{prop:routhhurwitz}
        Consider a cubic polynomial $x^3 + a_2x^2 + a_1x + a_0$, where $a_0$, $a_1$, and $a_2$ are real numbers. Then a necessary and sufficient condition that all roots of the polynomial are contained in the open disk $|x|<1$ is 
        \begin{align}
            |a_2 + a_0| < 1 + a_1, \quad |a_2 - 3a_0| < 3 - a_1, \quad a_0(a_0-a_2) + a_1 - 1 < 0. \label{eq:degthree}
        \end{align}
    \end{proposition}
    
    Also, the following corollary suggests that the coefficients are all bounded (by constants) for such cases.
        
        
    \begin{corollary}
        \label{cor:routhhurwitzbd}
        For coefficients $a_0, a_1, a_2$ satisfying \eqref{eq:degthree}, we have $|a_2| < 3$, $|a_1| < 3$, and $|a_0| < 1$.
    \end{corollary}
    \begin{proof}
        It is easy to see that $-1 < a_1 < 3$ from the first two conditions.
        
        Also, the first and the last condition together imply that 
        \begin{align*}
            |a_2 + a_0| - 1 < a_1 < a_0(a_2-a_0) + 1.
        \end{align*}
        This is a subset of the region
        \begin{align*}
            |a_2 + a_0| < 4 \quad \wedge \quad |a_2 + a_0| < a_0(a_2-a_0) + 2.
        \end{align*}
        The range of such $(a_2, a_0)$ is equal to a parallelogram with endpoints $(-3, -1)$, $(1, -1)$, $(-1, 1)$, $(3, 1)$, which implies $|a_2| < 3$ and $|a_0| < 1$.
    \end{proof} 
    
    Plugging back in $t = 1 - \lambda$, we can write the cubic polynomial in terms of $p, q, r$, and $\lambda$ as
    \begin{align*}
        (1 - \lambda)^3 - p (1 - \lambda)^2 + q(1 - \lambda) - r &= 0 \\
        \Leftrightarrow \ \ \lambda^3 + (-3 + p) \lambda^2 + (3 - 2p + q) \lambda + (-1 + p - q + r) &= 0.
    \end{align*}
    By \cref{cor:routhhurwitzbd}, we can observe that a necessary condition for $\rho(\mP) < 1$ is that%
    \begin{align*}
        |3-p| &< 3, \ \ |3 - 2p + q| < 3, \ \ |1 - p + q - r| < 1.
    \end{align*}
    We can simply deduce that $p < 6$, which implies $q < 12$ and finally $r < 14$.
    
    Therefore we can conclude that all of the coefficients in \eqref{eq:stupidme} are of order $\gO (1)$.
    In particular, this implies $\phi = \alpha \beta L_{xy}^2 = \gO (1)$ in order to assure convergence, which concludes that
    \begin{align}
        \frac{1}{\alpha \mu_x} + \frac{1}{\beta \mu_y} &\ge \frac{2}{\sqrt{\alpha \beta \mu_x \mu_y}} = \frac{2 \kappa_{xy}}{\sqrt{\alpha \beta L_{xy}^2}} = \Omega (\kappa_{xy}). 
        \label{eq:rice}
    \end{align}
    Combining \eqref{eq:screen} and \eqref{eq:rice}, we have
    \begin{align*}
        \frac{1}{\alpha \mu_x} + \frac{1}{\beta \mu_y} &= \Omega (\kappa_x + \kappa_y + \kappa_{xy})
    \end{align*}
    and therefore from \eqref{eq:cpxty_base_dig} we can show a lower bound of
    \begin{equation*}\Omega\bigopen{\bigopen{\kappa_x + \kappa_y + \kappa_{xy}} \cdot \log\frac{1}{\epsilon}}.
    \end{equation*}
\end{proof}

%% file: secd4.tex
\subsection{\texorpdfstring{Proof of \cref{prop:eglb}}{Proof of Proposition 5.4}} \label{sec:propeglb}

Here we prove \cref{prop:eglb} of \cref{sec:5}, restated below for the sake of readability.

\propeglb*

\begin{proof}
    Recall that \textbf{EG} takes updates of the form:
    \begin{align*}
        \vx_{k+\frac{1}{2}} &= \vx_{k} - \alpha \nabla_{\vx} f(\vx_{k}, \vy_{k}), \\
        \vy_{k+\frac{1}{2}} &= \vy_{k} + \beta \nabla_{\vy} f(\vx_{k}, \vy_{k}), \\
        \vx_{k+1} &= \vx_{k} - \alpha \nabla_{\vx} f(\vx_{k+\frac{1}{2}}, \vy_{k+\frac{1}{2}}), \\
        \vy_{k+1} &= \vy_{k} + \beta \nabla_{\vy} f(\vx_{k+\frac{1}{2}}, \vy_{k+\frac{1}{2}}).
    \end{align*}

    We use the same worst-case function as in \cref{thm:simgdalb}:
    \begin{align*}
        f(\vx, \vy) &= 
        \frac{1}{2} \begin{bmatrix}
            {x} \\ {s} \\ {t} \\ {y} \\ {u} \\ {v}
        \end{bmatrix}^{\top}
        \begin{bmatrix}
            \mu_x & 0 & 0 & L_{xy} & 0 & 0 \\
            0 & \mu_x & 0 & 0 & 0 & 0 \\
            0 & 0 & L_x & 0 & 0 & 0 \\
            L_{xy} & 0 & 0 & -\mu_y & 0 & 0 \\
            0 & 0 & 0 & 0 & -\mu_y & 0 \\
            0 & 0 & 0 & 0 & 0& -L_y 
        \end{bmatrix}
        \begin{bmatrix}
            {x} \\ {s} \\ {t} \\ {y} \\ {u} \\ {v}
        \end{bmatrix}, \\
    \end{align*}
    where ${\vx = (x, s, t)}$ and ${\vy = (y, u, v)}$. 
    It can be easily checked that $f$ is a quadratic function ({\it i.e.,} Hessian is constant) such that $f \in \gF(\mu_x, \mu_y, L_x, L_y, L_{xy})$ and $\vx_\star = \vy_\star = \bm{0} \in \R^3$.

    Let us define
    \begin{align*}
        \mA &= 
        \begin{bmatrix}
            \mu_x & 0 & 0 \\
            0 & \mu_x & 0 \\
            0 & 0 & L_x
        \end{bmatrix}, \quad
        \mB = 
        \begin{bmatrix}
            L_{xy} & 0 & 0 \\
            0 & 0 & 0 \\
            0 & 0 & 0
        \end{bmatrix}, \quad
        \mC =
        \begin{bmatrix}
            \mu_y & 0 & 0 \\
            0 & \mu_y & 0 \\
            0 & 0 & L_y
        \end{bmatrix}.
    \end{align*}

    We first observe that the $k$-th step of \textbf{EG} satisfies
    \begin{align*}
        \begin{bmatrix}
            \vx_{k+\frac{1}{2}} \\ \vy_{k+\frac{1}{2}}
        \end{bmatrix}
        &=
        \underbrace{\begin{bmatrix}
            \mI - \alpha \mA & - \alpha \mB \\ 
            \beta \mB^{\top} & \mI - \beta \mC
        \end{bmatrix}}_{\triangleq \mM_{\text{Sim}}}
        \begin{bmatrix}
            \vx_{k} \\ \vy_{k}
        \end{bmatrix}, \\
        \begin{bmatrix}
            \vx_{k+1} \\ \vy_{k+1}
        \end{bmatrix}
        &=
        \begin{bmatrix}
            \vx_{k} \\ \vy_{k}
        \end{bmatrix} +
        \begin{bmatrix}
            - \alpha \mA & - \alpha \mB \\ 
            \beta \mB^{\top} & - \beta \mC
        \end{bmatrix}
        \begin{bmatrix}
            \vx_{k+\frac{1}{2}} \\ \vy_{k+\frac{1}{2}}
        \end{bmatrix} \\
        &=
        \begin{bmatrix}
            \vx_{k} \\ \vy_{k}
        \end{bmatrix} +
        \begin{bmatrix}
            - \alpha \mA & - \alpha \mB \\ 
            \beta \mB^{\top} & - \beta \mC
        \end{bmatrix}
        \begin{bmatrix}
            \mI - \alpha \mA & - \alpha \mB \\ 
            \beta \mB^{\top} & \mI - \beta \mC
        \end{bmatrix}
        \begin{bmatrix}
            \vx_{k} \\ \vy_{k}
        \end{bmatrix} \\
        &= (\mI + (\mM_{\text{Sim}} - \mI) \mM_{\text{Sim}}) 
        \begin{bmatrix}
            \vx_{k} \\ \vy_{k}
        \end{bmatrix} = (\underbrace{\mI - \mM_{\text{Sim}} + \mM_{\text{Sim}}^2}_{\triangleq \mM_{\text{EG}}}) 
        \begin{bmatrix}
            \vx_{k} \\ \vy_{k}
        \end{bmatrix}.
    \end{align*}
    Hence we have that $\lambda_{\text{Sim}}$ is an eigenvalue of $\mM_{\text{Sim}}$ if and only if $\lambda_{\text{EG}} = 1 - \lambda_{\text{Sim}} + \lambda_{\text{Sim}}^2$ is an eigenvalue of $\mM_{\text{EG}}$.
    Note that the matrix $\mM_{\text{Sim}}$ is identical to the updates made by \simgda{} on the same lower bound function $f$, which allows us to utilize some results from \cref{sec:thmsimgdalb}.

    Let us define
    \begin{align*}
        \mP &\triangleq \begin{bmatrix}
            1 - \alpha \mu_x & - \alpha L_{xy} \\
            \beta L_{xy} & 1 - \beta \mu_y
        \end{bmatrix}.
    \end{align*}
    Then the $k$-th step of \textbf{EG} satisfies
    \begin{align}
        \begin{bmatrix}
            x_{k+1} \\ y_{k+1}
        \end{bmatrix} &= (\mI - \mP + \mP^2)
        \begin{bmatrix}
            x_k \\ y_k
        \end{bmatrix}, \label{eq:x_k_y_k_eg}\\
        s_{k+1} &= (1 - \alpha \mu_x + \alpha^2 \mu_x^2) s_k, \label{eq:s_k_eg} \\
        t_{k+1} &= (1  - \alpha L_x + \alpha^2 L_x^2) t_k, \label{eq:t_k_eg} \\
        u_{k+1} &= (1 - \beta \mu_y + \beta^2 \mu_y^2) u_k, \label{eq:u_k_eg} \\
        v_{k+1} &= (1 - \beta L_y + \beta^2 L_y^2) v_k. \label{eq:v_k_eg} 
    \end{align}
    We can see that the eigenvalues of $\mM_{\text{EG}}$ must be either $\lambda_{\text{EG}} = 1 - \lambda_{P} + \lambda_{P}^2$, where $\lambda_{P}$ is an eigenvalue of $\mP$, which can be explicitly computed as
    \begin{align}
        \lambda_{P} &= 1 - \frac{\alpha \mu_x + \beta \mu_y}{2} \pm \sqrt{\left( \frac{\alpha \mu_x - \beta \mu_y}{2} \right)^2 - \alpha \beta L_{xy}^2}
        \label{eq:case1}
    \end{align}
    or among the following values: 
    \begin{align}
        1 - \alpha \mu_x, \ \ 1 - \alpha L_x, \ \ 1 - \beta \mu_y, \ \ \text{and} \ \ 1 - \beta L_y.
        \label{eq:case2}
    \end{align}
    For the (real) eigenvalues in \eqref{eq:case2}, we can deduce that the corresponding eigenvalues of $\mM_{\text{EG}}$ are
    \begin{align*}
        1 - \alpha \mu_x + \alpha^2 \mu_x^2, \ \ 1  - \alpha L_x + \alpha^2 L_x^2, \ \ 1 - \beta \mu_y + \beta^2 \mu_y^2, \ \ \text{and} \ \ 1 - \beta L_y + \beta^2 L_y^2,
    \end{align*}
    all being strictly larger than the corresponding values in \eqref{eq:case2}.
    Hence, for the convergence of iterations \eqref{eq:t_k_eg}~and~\eqref{eq:v_k_eg}, the step sizes $\alpha$ and $\beta$ are required to satisfy
    \begin{equation*}
        0 < \alpha L_x (1 - \alpha L_x) < 2 \quad \text{and} \quad 0 < \beta L_y (1 - \beta L_y) < 2,
    \end{equation*}
    which (as $\alpha, \beta > 0$) is simply equivalent to
    \begin{equation}
        \label{eq:stepsize_basic_2}
        {\alpha < \frac{1}{L_x}} \quad \text{and} \quad {\beta < \frac{1}{L_y}}.
    \end{equation}
    Also, to guarantee $\bignorm{\vx_K}^2 + \bignorm{\vy_K}^2 < \epsilon$, we need from \eqref{eq:s_k_eg}~and~\eqref{eq:u_k_eg} that $s_K^2 < \gO(\epsilon)$ and $u_K^2 < \gO(\epsilon)$, respectively.
    These two necessary conditions require an iteration number of at least:
    \begin{equation}
        \label{eq:iteration_complexity_base_2}
        K = \Omega\bigopen{\bigopen{{\frac{1}{\alpha \mu_x (1 - \alpha \mu_x)}} + {\frac{1}{\beta \mu_y (1 - \beta \mu_y)}}} \cdot \log\frac{1}{\epsilon}} = \Omega\bigopen{\bigopen{{\frac{1}{\alpha \mu_x}} + {\frac{1}{\beta \mu_y}}} \cdot \log\frac{1}{\epsilon}}.
    \end{equation}
    Note that \eqref{eq:stepsize_basic_2} automatically yields
    \begin{align}
        \frac{1}{\alpha \mu_x} + \frac{1}{\beta \mu_y} = \Omega (\kappa_x + \kappa_y).
        \label{eq:hello}
    \end{align}
    

    Now we focus on the $x, y$ coordinates to complete the proof.
    We do a similar case-by-case analysis as in our proof of \cref{thm:simgdalb} in \cref{sec:thmsimgdalb}, based on whether the eigenvalues in \eqref{eq:case1} are real or complex.
    
    \paragraph{Case 1.} If the eigenvalues $\lambda_{P}$ in \eqref{eq:case1} are real, then we have
    \begin{align*}
        \bigabs{\sqrt{\frac{\alpha\mu_x}{\beta\mu_y}} - \sqrt{\frac{\beta\mu_y}{\alpha\mu_x}}} > 2\kappa_{xy}
    \end{align*}
    as in \eqref{eq:realeigenvalues} of \cref{sec:thmsimgdalb}.
    By the same logic as in \textbf{Case 2} of \cref{sec:thmsimgdalb}, we have
    \begin{align}
        \frac{1}{\alpha \mu_x} + \frac{1}{\beta \mu_y} = \Omega (\kappa_x + \kappa_y + \kappa_{xy}^2) = \Omega (\kappa_x + \kappa_y + \kappa_{xy}).
        \label{eq:david}
    \end{align}
    
    \paragraph{Case 2.} Suppose that the eigenvalues in \eqref{eq:case1} are complex.
    If we substitute as
    \begin{align*}
        s = \frac{\alpha \mu_x + \beta \mu_y}{2}, \ \ p = \sqrt{\alpha \beta \mu_x \mu_y}, \ \ K = \kappa_{xy}^2,
    \end{align*}
    then \eqref{eq:case1} can be written as:
    \begin{align*}
        \lambda_{P} &= 1 - s \pm i \sqrt{(K+1) p^2 - s^2}.
    \end{align*}
    As we consider the case when the eigenvalues are complex, here we must have
    \begin{align*}
        s^2 &\le (K+1)p^2. 
    \end{align*}
    We can explicitly compute
    \begin{align*}
        \lambda_{\text{EG}} &= \lambda_{P} + (1 - \lambda_{P})^2 \\
        &= 1 - s \pm i \sqrt{(K+1) p^2 - s^2} + \left( s \mp i \sqrt{(K+1) p^2 - s^2} \right)^2 \\
        &= 1 - s + s^2 - ((K+1) p^2 - s^2) \pm i \left( (1 - 2s) \sqrt{(K+1) p^2 - s^2} \right)
    \end{align*}
    
    Therefore $| \lambda_{\text{EG}} |^2$ can be expressed as 
    \begin{align*}
        &\phantom{=} \left( 1 - s + s^2 - ((K+1) p^2 - s^2) \right)^2 + (1 - 2s)^2 ((K+1) p^2 - s^2) \\
        &= (1 - s + s^2)^2 - 2 (1 - s + s^2) ((K+1) p^2 - s^2) + ((K+1) p^2 - s^2)^2 + (1 - 2s)^2 ((K+1) p^2 - s^2) \\
        &= (1 - s + s^2)^2 - (2(1 - s + s^2) - (1 - 2s)^2) ((K+1) p^2 - s^2) + ((K+1) p^2 - s^2)^2 \\
        &= (1 - s + s^2)^2 - (1 + 2s - 2s^2) ((K+1) p^2 - s^2) + ((K+1) p^2 - s^2)^2 \\
        &= (1 - s + s^2)^2 + ((K+1) p^2 + s^2 - 2s - 1) ((K+1) p^2 - s^2) \\
        &= (1 - s + s^2)^2 + (K+1)^2 p^4 - s^4 - (2s + 1) ((K+1) p^2 - s^2) \\
        &= (1 - s + s^2)^2 + (2s + 1) s^2 - s^4 - (K+1) p^2 (2s + 1) + (K+1)^2 p^4 \\
        &= 1 - 2s + 4s^2 - (K+1) p^2 (2s + 1) + (K+1)^2 p^4.
    \end{align*}
    
    Note that $| \lambda_{\text{EG}} | < 1$ is equivalent to
    \begin{align*}
        - (K+1) p^2 (2s + 1) + (K+1)^2 p^4 &< 2s - 4s^2.
    \end{align*}
    
    If this is true, then substituting $t = (K+1) p^2$ we obtain the following region:
    \begin{align*}
        t^2 - t (2s + 1) + 4s^2 - 2s &< 0, \ \ s^2 \le t
    \end{align*}
    which is the upper region of the interior of an ellipse cut by a parabola.
    This region is bounded, and we can compute the range of $t$ as $0 < t < 1 + 2 / \sqrt{3}$.
    Therefore we have $t = \gO (1)$, and since we can substitute back as
    \begin{align*}
        t &= (K+1) p^2 = (\kappa_{xy}^2 + 1) \alpha \beta \mu_x \mu_y,
    \end{align*}
    we can observe that $(\kappa_{xy}^2 + 1) \alpha \beta \mu_x \mu_y = \gO (1)$, and therefore
    \begin{align*}
        \frac{1}{\alpha \mu_x} + \frac{1}{\beta \mu_y} &\ge \frac{2}{\sqrt{\alpha \beta \mu_x \mu_y}} = \frac{2 \sqrt{(\kappa_{xy}^2 + 1)}}{\sqrt{(\kappa_{xy}^2 + 1) \alpha \beta \mu_x \mu_y}} = \Omega(\kappa_{xy}).
    \end{align*}
    Aggregating with \eqref{eq:hello}, we can observe that
    \begin{align*}
        \frac{1}{\alpha \mu_x} + \frac{1}{\beta \mu_y} = \Omega(\kappa_x + \kappa_y + \kappa_{xy}),
    \end{align*}
    and hence the lower bound iteration complexity holds for all possible cases of convergence.
\end{proof}

%% file: secd5.tex
\subsection{\texorpdfstring{Proofs used in \cref{sec:d}}{Proofs used in Appendix D}} \label{sec:alexgdatech}

Here we prove some technical propositions and lemmas used throughout \cref{sec:d}.

\subsubsection{\texorpdfstring{Proof of \cref{prop:digvalid}}{Proof of Proposition D.1}} \label{sec:propdigvalid}

Here we prove \cref{prop:digvalid}, restated below for the sake of readability.

\propdigvalid*

\begin{proof}
    For simplicity let us assume W.L.O.G. that $\vx_{\star} = \bm{0}$ $(\in \R^{d_x})$ and $\vy_{\star} = \bm{0}$ $(\in \R^{d_y})$.

    For $k \ge 1$, we have
    \begin{align*}
        \Psi_{k}^{\text{Alex}} &\ge \frac{1}{\alpha} \| \vx_k \|^2 + \frac{2}{\beta} \| \vy_k \|^2 + \frac{1}{\alpha} \| \vx_{k+1} \|^2 - \alpha \norm{\nabla_{\vx} f(\vx_k, \tilde{\vy}_k)}^2 + {(\delta - 1) \beta \norm{\nabla_{\vy} f(\tilde{\vx}_{k}, \vy_{k-1})}^2}.
    \end{align*}
    
    By triangle inequality and Lipschitz gradients, we have
    \begin{align*}
        \norm{\nabla_{\vx} f(\vx_k, \tilde{\vy}_k)}^2 &\le 2 \norm{\nabla_{\vx} f(\vx_k, \tilde{\vy}_k) - \nabla_{\vx} f(\vx_{\star}, \tilde{\vy}_k)}^2 + 2 \norm{\nabla_{\vx} f(\vx_{\star}, \tilde{\vy}_k)}^2 \\
        &\le 2 L_x^2 \norm{\vx_k}^2 + 2 L_{xy}^2 \norm{\tilde{\vy}_k}^2 \\
        &\le 2 L_x^2 \norm{\vx_k}^2 + 4 L_{xy}^2 \norm{\vy_k}^2 + 4 L_{xy}^2 \norm{\tilde{\vy}_k - \vy_k}^2 \\
        &\le 2 L_x^2 \norm{\vx_k}^2 + 4 L_{xy}^2 \norm{\vy_k}^2 + 4 (\delta - 1)^2 \beta^2 L_{xy}^2 \norm{\nabla_{\vy} f(\tilde{\vx}_{k}, \vy_{k-1})}^2
    \end{align*}
    Therefore, we can obtain
    \begin{align*}
        & {\phantom{\ge}} \frac{1}{\alpha} \| \vx_k \|^2 + \frac{2}{\beta} \| \vy_k \|^2 + \frac{1}{\alpha} \| \vx_{k+1} \|^2 - \alpha \norm{\nabla_{\vx} f(\vx_k, \tilde{\vy}_k)}^2 + {(\delta - 1) \beta \norm{\nabla_{\vy} f(\tilde{\vx}_{k}, \vy_{k-1})}^2} \\
        &\ge \bigopen{\frac{1}{\alpha} - 2 \alpha L_x^2} \norm{\vx_k}^2 + 2 \bigopen{\frac{1}{\beta} - 2 \alpha L_{xy}^2} \norm{\vy_k}^2 + \frac{1}{\alpha} \norm{\vx_{k+1}}^2 \\
        &\phantom{\ge} + (\delta - 1) \beta \bigopen{1 - 4 (\delta - 1) \alpha \beta L_{xy}^2} \norm{\nabla_{\vy} f(\tilde{\vx}_{k}, \vy_{k-1})}^2.
    \end{align*}
    Since $\alpha \le \frac{C_1}{L_x}$ and $C_1 \le \frac{1}{2}$ (by \eqref{eq:constsix}), we have
    \begin{align*}
        \frac{1}{\alpha} - 2 \alpha L_x^2 &\ge \frac{1}{\alpha} - \frac{2 C_1^2}{\alpha} \ge \frac{1}{2 \alpha}.
    \end{align*}
    Since $\alpha \le \frac{C_3}{L_{xy}} \sqrt{\frac{\mu_y}{\mu_x}}$, $\beta \le \frac{C_4}{L_{xy}} \sqrt{\frac{\mu_x}{\mu_y}}$, and $4 C_3 C_4 \le 1$ (by \eqref{eq:constseven}), we have
    \begin{align*}
        \frac{1}{\beta} - 2 \alpha L_{xy}^2 \ge \frac{1}{\beta} - \frac{2 C_3 C_4}{\beta} \sqrt{\frac{\mu_y}{\mu_x}} \cdot \sqrt{\frac{\mu_x}{\mu_y}} = \frac{1}{\beta} - \frac{2 C_3 C_4}{\beta} \ge \frac{1}{2 \beta}.
    \end{align*}
    Finally, since $4 (\delta - 1) C_3 C_4 \le 1$ (by \eqref{eq:consteight}), we have
    \begin{align*}
        4 (\delta - 1) \alpha \beta L_{xy}^2 &\le 4 (\delta - 1) C_3 C_4 \le 1
    \end{align*}
     and therefore we can cancel out the last term by
     \begin{align*}
         (\delta - 1) \beta \bigopen{1 - 4 (\delta - 1) \alpha \beta L_{xy}^2} \norm{\nabla_{\vy} f(\tilde{\vx}_{k}, \vy_{k-1})}^2 \ge 0.
     \end{align*}
    Therefore we have
    \begin{align*}
        & {\phantom{\ge}} \frac{1}{\alpha} \| \vx_k \|^2 + \frac{2}{\beta} \| \vy_k \|^2 + \frac{1}{\alpha} \| \vx_{k+1} \|^2 - \alpha \norm{\nabla_{\vx} f(\vx_k, \tilde{\vy}_k)}^2 + {(\delta - 1) \beta \norm{\nabla_{\vy} f(\tilde{\vx}_{k}, \vy_{k-1})}^2} \\
        &\ge \frac{1}{2 \alpha} \| \vx_k \|^2 + \frac{1}{\beta} \| \vy_k \|^2 + \frac{1}{\alpha} \| \vx_{k+1} \|^2,
    \end{align*}
    which implies that \eqref{eq:digvalid} is indeed true for $k = 1$.
    
    For $k = 0$, we have
    \begin{align*}
        \Psi_{0}^{\text{Alex}} &\ge \frac{1}{\alpha} \| \vx_0 \|^2 + \frac{2}{\beta} \| \vy_0 \|^2 + \frac{1}{\alpha} \| \vx_{1} \|^2 - \alpha \norm{\nabla_{\vx} f(\vx_0, \vy_0)}^2,
    \end{align*}
    where we note that $\tilde{\vy}_0 = \vy_0$.
    By triangle inequality and Lipschitz gradients, we have
    \begin{align*}
        \norm{\nabla_{\vx} f(\vx_0, {\vy}_0)}^2 &\le 2 \norm{\nabla_{\vx} f(\vx_0, \vy_0) - \nabla_{\vx} f(\vx_{\star}, \vy_0)}^2 + 2 \norm{\nabla_{\vx} f(\vx_{\star}, \vy_0)}^2 \le 2 L_x^2 \norm{\vx_0}^2 + 2 L_{xy}^2 \norm{\vy_0}^2.
    \end{align*}
    Therefore, we can obtain
    \begin{align*}
        & {\phantom{\ge}} \frac{1}{\alpha} \| \vx_0 \|^2 + \frac{2}{\beta} \| \vy_0 \|^2 + \frac{1}{\alpha} \| \vx_{1} \|^2 - \alpha \norm{\nabla_{\vx} f(\vx_0, \vy_0)}^2 \ge \bigopen{\frac{1}{\alpha} - 2 \alpha L_x^2} \norm{\vx_0}^2 + 2 \bigopen{\frac{1}{\beta} - \alpha L_{xy}^2} \norm{\vy_0}^2 + \frac{1}{\alpha} \norm{\vx_{1}}^2.
    \end{align*}
    Since $\alpha \le \frac{C_1}{L_x}$ and $C_1 \le \frac{1}{2}$ (by \eqref{eq:constsix}), we have
    \begin{align*}
        \frac{1}{\alpha} - 2 \alpha L_x^2 &\ge \frac{1}{\alpha} - \frac{2 C_1^2}{\alpha} \ge \frac{1}{2 \alpha}.
    \end{align*}
    Since $\alpha \le \frac{C_3}{L_{xy}} \sqrt{\frac{\mu_y}{\mu_x}}$, $\beta \le \frac{C_4}{L_{xy}} \sqrt{\frac{\mu_x}{\mu_y}}$, and $C_3 C_4 \le \frac{1}{4}$ (by \eqref{eq:constseven}), we have
    \begin{align*}
        \frac{1}{\beta} - 2 \alpha L_{xy}^2 \ge \frac{1}{\beta} - \frac{2 C_3 C_4}{\beta} \sqrt{\frac{\mu_y}{\mu_x}} \cdot \sqrt{\frac{\mu_x}{\mu_y}} = \frac{1}{\beta} - \frac{2 C_3 C_4}{\beta} \ge \frac{1}{2 \beta}.
    \end{align*}
    Therefore we have
    \begin{align*}
        & {\phantom{\ge}} \frac{1}{\alpha} \| \vx_0 \|^2 + \frac{2}{\beta} \| \vy_0 \|^2 + \frac{1}{\alpha} \| \vx_{1} \|^2 - \alpha \norm{\nabla_{\vx} f(\vx_0, \vy_0)}^2 \ge \frac{1}{2 \alpha} \| \vx_0 \|^2 + \frac{1}{\beta} \| \vy_0 \|^2 + \frac{1}{\alpha} \| \vx_{1} \|^2,
    \end{align*}
    which implies that \eqref{eq:digvalid} is indeed true for $k = 0$.
\end{proof}

\subsubsection{\texorpdfstring{Proof of \cref{prop:alexgda}}{Proof of Proposition D.2}} \label{sec:propalexgda}

Here we prove \cref{prop:alexgda}, restated below for the sake of readability.

\propalexgda*

\begin{proof}
    While the proof of the proposition is quite technical and complicated, we can largely divide the proof into three large steps.
    In \textsc{Step 1}, we use the basic notions of strong convexity (and/or strong concavity) and the Lipschitz gradient conditions involving $L_x$ and $L_y$ (\textit{i.e., smoothness} in convex optimization literature) to obtain an inequality between terms from the previous and next iterates.
    In \textsc{Step 2}, we use the $L_{xy}$-Lipschitz gradient conditions to cope with the intermediate inner product terms.
    In \textsc{Step 3}, we use the given step size conditions to cancel out the gradient norm terms as much as possible, which leaves us with the inequality given in the proposition statement.
    
    \subsubsection*{Step 1. Basic Transformations}

    We start with
    \begin{align}
        \begin{aligned}
        \frac{1}{\alpha} \| \vx_1 - \vx_{\star} \|^2 &= \frac{1}{\alpha} \norm{\vx_0 - \vx_{\star}}^2 + \frac{2}{\alpha} \inner{\vx_1 - \vx_0, \vx_0 - \vx_{\star}} + \frac{1}{\alpha} \norm{\vx_1 - \vx_0}^2 \\
        &= \frac{1}{\alpha} \norm{\vx_0 - \vx_{\star}}^2 - 2 \inner{\nabla_{\vx} f(\vx_0, \tilde{\vy}_0), \vx_0 - \vx_{\star}} + \alpha \norm{\nabla_{\vx} f(\vx_0, \tilde{\vy}_0)}^2, \\
        \frac{2}{\beta} \| \vy_1 - \vy_{\star} \|^2 &= \frac{2}{\beta} \norm{\vy_0 - \vy_{\star}}^2 + \frac{4}{\beta} \inner{\vy_1 - \vy_0, \vy_0 - \vy_{\star}} + \frac{2}{\beta} \norm{\vy_1 - \vy_0}^2 \\
        &= \frac{2}{\beta} \norm{\vy_0 - \vy_{\star}}^2 + 4 \inner{\nabla_{\vy} f(\tilde{\vx}_1, \vy_0), \vy_0 - \vy_{\star}} + 2 \beta \norm{\nabla_{\vy} f(\tilde{\vx}_1, \vy_0)}^2, \\
        \frac{1}{\alpha} \| \vx_2 - \vx_{\star} \|^2 &= \frac{1}{\alpha} \norm{\vx_1 - \vx_{\star}}^2 + \frac{2}{\alpha} \inner{\vx_2 - \vx_1, \vx_1 - \vx_{\star}} + \frac{1}{\alpha} \norm{\vx_2 - \vx_1}^2 \\
        &= \frac{1}{\alpha} \norm{\vx_1 - \vx_{\star}}^2 - 2 \inner{\nabla_{\vx} f(\vx_1, \tilde{\vy}_1), \vx_1 - \vx_{\star}} + \alpha \norm{\nabla_{\vx} f(\vx_1, \tilde{\vy}_1)}^2.
        \end{aligned} \label{eq:mercury}
    \end{align}
    
    By strong convexity (concavity), we have
    \begin{align}
        \begin{aligned}
        - 2 \inner{\nabla_{\vx} f({\vx_0}, \tilde{\vy}_0), {\vx_{0} - \vx_{\star}}} &\le - \mu_x \norm{{\vx_0 - \vx_{\star}}}^2 - 2 (f({\vx_0}, \tilde{\vy}_0) - f({\vx_{\star}}, \tilde{\vy}_0)), \\
        4 \inner{\nabla_{\vy} f(\tilde{\vx}_1, {\vy_0}), {\vy_0 - \vy_{\star}}} &\le - 2 \mu_y \norm{{\vy_0 - \vy_{\star}}}^2 + 4 (f(\tilde{\vx}_1, {\vy_0}) - f(\tilde{\vx}_1, {\vy_{\star}})), \\
        - 2 \inner{\nabla_{\vx} f({\vx_1}, \tilde{\vy}_1), {\vx_{1} - \vx_{\star}}} &\le - \mu_x \norm{{\vx_1 - \vx_{\star}}}^2 - 2 (f({\vx_1}, \tilde{\vy}_1) - f({\vx_{\star}}, \tilde{\vy}_1)).
        \end{aligned} \label{eq:venus}
    \end{align}
    
    Since $f$ has Lipschitz gradients, we have
    \begin{align}
        \begin{aligned}
        2 \inner{\nabla_{\vx} f({\vx_0}, \tilde{\vy}_0), {\vx_{0} - \tilde{\vx}_{1}}} &\le L_x \norm{{\vx_0 - \tilde{\vx}_{1}}}^2 + 2 (f({\vx_0}, \tilde{\vy}_0) - f({\tilde{\vx}_1}, \tilde{\vy}_0)), \\
        - 2 \inner{\nabla_{\vy} f(\tilde{\vx}_1, {\vy_0}), {\vy_{0} - \tilde{\vy}_0}} &\le L_y \norm{{\vy_{0} - \tilde{\vy}_0}}^2 - 2 (f(\tilde{\vx}_1, {\vy_0}) - f(\tilde{\vx}_1, {\tilde{\vy}_0})), \\
        - 2 \inner{\nabla_{\vy} f(\tilde{\vx}_1, {\vy_0}), {\vy_{0} - \tilde{\vy}_1}} &\le L_y \norm{{\vy_{0} - \tilde{\vy}_1}}^2 - 2 (f(\tilde{\vx}_1, {\vy_0}) - f(\tilde{\vx}_1, {\tilde{\vy}_1})), \\
        2 \inner{\nabla_{\vx} f({\vx_1}, \tilde{\vy}_1), {\vx_{1} - \tilde{\vx}_{1}}} &\le L_x \norm{{\vx_1 - \tilde{\vx}_{1}}}^2 + 2 (f({\vx_1}, \tilde{\vy}_1) - f({\tilde{\vx}_1}, \tilde{\vy}_1)).
        \end{aligned} \label{eq:earth}
    \end{align}
    
    Rearranging the above conditions, we have
    \begin{align}
        \begin{aligned}
        - 2 (f(\vx_0, \tilde{\vy}_0) - f(\tilde{\vx}_1, \tilde{\vy}_0)) &\le - \gamma \alpha (2 - \gamma \alpha L_x) \norm{\nabla_{\vx} f(\vx_0, \tilde{\vy}_0)}^2, \\
        2 (f(\tilde{\vx}_1, \vy_0) - f(\tilde{\vx}_1, \tilde{\vy}_0)) &\le -2 \xi (\delta - 1) \beta \inner{\nabla_{\vy} f(\tilde{\vx}_1, \vy_0), \nabla_{\vy} f(\tilde{\vx}_0, \vy_{-1})} + \xi^2 (\delta - 1)^2 \beta^2 L_y \norm{\nabla_{\vy} f(\tilde{\vx}_0, \vy_{-1})}^2, \\
        2 (f(\tilde{\vx}_1, \vy_0) - f(\tilde{\vx}_1, \tilde{\vy}_1)) &\le - \delta \beta (2 - \delta \beta L_y) \norm{\nabla_{\vy} f(\tilde{\vx}_1, \vy_0)}^2, \\
        - 2 (f(\vx_1, \tilde{\vy}_1) - f(\tilde{\vx}_1, \tilde{\vy}_1)) &\le - 2 (\gamma - 1) \alpha \inner{\nabla_{\vx} f(\vx_1, \tilde{\vy}_1), \nabla_{\vx} f(\vx_0, \tilde{\vy}_0)} + (\gamma - 1)^2 \alpha^2 L_x \norm{\nabla_{\vx} f(\vx_0, \tilde{\vy}_0)}^2.
        \end{aligned} \label{eq:mars}
    \end{align}
    
    Since $f$ is convex, we have
    \begin{align}
        \begin{aligned}
        2 (f(\vx_{\star}, \tilde{\vy}_0) - f(\vx_{\star}, \vy_{\star})) &\le 0, \ \
        -4 (f(\tilde{\vx}_{1}, \vy_{\star}) - f(\vx_{\star}, \vy_{\star})) \le 0, \ \
        2 (f(\vx_{\star}, \tilde{\vy}_1) - f(\vx_{\star}, \vy_{\star})) \le 0.
        \end{aligned} \label{eq:jupiter}
    \end{align}
    
    Summing up \eqref{eq:mercury}, \eqref{eq:venus}, \eqref{eq:mars}, and \eqref{eq:jupiter}, we have
    \begin{align}
        \begin{aligned}
        & \frac{1}{\alpha} \| \vx_1 - \vx_{\star} \|^2 + \frac{2}{\beta} \| \vy_1 - \vy_{\star} \|^2 + \frac{1}{\alpha} \| \vx_2 - \vx_{\star} \|^2 \\
        &\le \bigopen{\frac{1}{\alpha} - \mu_x} \| \vx_0 - \vx_{\star} \|^2 + 2 \bigopen{\frac{1}{\beta} - \mu_y} \| \vy_0 - \vy_{\star} \|^2 + \bigopen{\frac{1}{\alpha} - \mu_x} \| \vx_1 - \vx_{\star} \|^2 \\
        &\phantom{\le} + \alpha \norm{\nabla_{\vx} f(\vx_0, \tilde{\vy}_0)}^2 + 2 \beta \norm{\nabla_{\vy} f(\tilde{\vx}_1, \vy_0)}^2 + \alpha \norm{\nabla_{\vx} f(\vx_1, \tilde{\vy}_1)}^2 \\
        &\phantom{\le} - \gamma \alpha (2 - \gamma \alpha L_x) \norm{\nabla_{\vx} f(\vx_0, \tilde{\vy}_0)}^2 + \xi^2 (\delta - 1)^2 \beta^2 L_y \norm{\nabla_{\vy} f(\tilde{\vx}_0, \vy_{-1})}^2 \\
        &\phantom{\le} - \delta \beta (2 - \delta \beta L_y) \norm{\nabla_{\vy} f(\tilde{\vx}_1, \vy_0)}^2 + (\gamma - 1)^2 \alpha^2 L_x \norm{\nabla_{\vx} f(\vx_0, \tilde{\vy}_0)}^2 \\
        &\phantom{\le} - 2 (\gamma - 1) \alpha \inner{\nabla_{\vx} f(\vx_1, \tilde{\vy}_1), \nabla_{\vx} f(\vx_0, \tilde{\vy}_0)} - 2 \xi (\delta - 1) \beta \inner{\nabla_{\vy} f(\tilde{\vx}_1, \vy_0), \nabla_{\vy} f(\tilde{\vx}_0, \vy_{-1})}. 
        \end{aligned}
        \label{eq:totalone}
    \end{align}

    \subsubsection*{Step 2. Using the $L_{xy}$ Conditions}

    By definition, the Lipschitz gradient condition for $L_{x}$ and $L_{y}$ yields the following inequalities:
    \begin{align*}
        \norm{\nabla_{\vx} f(\vx_0, \tilde{\vy}_0) - \nabla_{\vx} f(\vx_1, \tilde{\vy}_0)} &\le L_x \norm{\vx_0 - \vx_1}, \\
        \norm{\nabla_{\vy} f(\tilde{\vx}_0, \vy_{-1}) - \nabla_{\vy} f(\tilde{\vx}_0, \vy_0)} &\le L_y \norm{\vy_{-1} - \vy_0},
    \end{align*}
    which implies
    \begin{align*}
        \inner{\nabla_{\vx} f(\vx_0, \tilde{\vy}_0) - \nabla_{\vx} f(\vx_1, \tilde{\vy}_0), \vx_0 - \vx_1} &\le L_x \norm{\vx_0 - \vx_1}^2, \\
        - \inner{\nabla_{\vy} f(\tilde{\vx}_0, \vy_{-1}) - \nabla_{\vy} f(\tilde{\vx}_0, \vy_0), \vy_{-1} - \vy_0} &\le L_y \norm{\vy_{-1} - \vy_0}^2,
    \end{align*}
    or equivalently,
    \begin{align*}
        (1 - \alpha L_x) \norm{\nabla_{\vx} f(\vx_0, \tilde{\vy}_0)}^2 &\le \inner{\nabla_{\vx} f(\vx_0, \tilde{\vy}_0), \nabla_{\vx} f(\vx_1, \tilde{\vy}_0)}, \\
        \xi^2 (1 - \beta L_y) \norm{\nabla_{\vy} f(\tilde{\vx}_0, \vy_{-1})}^2 &\le \xi \inner{\nabla_{\vy} f(\tilde{\vx}_0, \vy_0), \nabla_{\vy} f(\tilde{\vx}_0, \vy_{-1})}.
    \end{align*}
    Note that since $\xi^2 = \xi$ for both $\xi = 0$ or $1$, the inequality for the $\vy$ side is equivalent to
    \begin{align*}
        \xi (1 - \beta L_y) \norm{\nabla_{\vy} f(\tilde{\vx}_0, \vy_{-1})}^2 &\le \xi \inner{\nabla_{\vy} f(\tilde{\vx}_0, \vy_0), \nabla_{\vy} f(\tilde{\vx}_0, \vy_{-1})}.
    \end{align*}
    
    Therefore we can obtain the below inequalities:
    \begin{align*}
        - 2 (\gamma - 1) \alpha \inner{\nabla_{\vx} f(\vx_0, \tilde{\vy}_0), \nabla_{\vx} f(\vx_1, \tilde{\vy}_0)} &\le - 2 (\gamma - 1) \alpha (1 - \alpha L_x) \norm{\nabla_{\vx} f(\vx_0, \tilde{\vy}_0)}^2, \\
        - 2 \xi (\delta - 1) \beta \inner{\nabla_{\vy} f(\tilde{\vx}_0, \vy_0), \nabla_{\vy} f(\tilde{\vx}_0, \vy_{-1})} &\le - 2 \xi (\delta - 1) \beta (1 - \beta L_y) \norm{\nabla_{\vy} f(\tilde{\vx}_0, \vy_{-1})}^2.
    \end{align*}

    Now we can use the Lipschitz gradient condition for $L_{xy}$ to obtain
    \begin{align*}
        & - 2 \inner{\nabla_{\vx} f(\vx_1, \tilde{\vy}_1) - \nabla_{\vx} f(\vx_1, \tilde{\vy}_0), \nabla_{\vx} f(\vx_0, \tilde{\vy}_0)} \\
        &\le 2 \norm{\nabla_{\vx} f(\vx_1, \tilde{\vy}_1) - \nabla_{\vx} f(\vx_1, \tilde{\vy}_0)} \cdot \norm{\nabla_{\vx} f(\vx_0, \tilde{\vy}_0)} \\
        &\le 2 L_{xy} \norm{\tilde{\vy}_1 - \tilde{\vy}_0} \cdot \norm{\nabla_{\vx} f(\vx_0, \tilde{\vy}_0)} \\
        &= 2 L_{xy} \norm{(\tilde{\vy}_1 - \vy_0) - (\tilde{\vy}_0 - \vy_0)} \cdot \norm{\nabla_{\vx} f(\vx_0, \tilde{\vy}_0)} \\
        &= 2 L_{xy} \norm{\delta \beta \nabla_{\vy} f(\tilde{\vx}_1, \vy_0) - \xi (\delta - 1) \beta \nabla_{\vy} f(\tilde{\vx}_0, \vy_{-1})} \cdot \norm{\nabla_{\vx} f(\vx_0, \tilde{\vy}_0)} \\
        &\le 2 \delta \beta L_{xy} \norm{\nabla_{\vy} f(\tilde{\vx}_1, \vy_0)} \cdot \norm{\nabla_{\vx} f(\vx_0, \tilde{\vy}_0)} + 2 \xi (\delta - 1) \beta L_{xy} \norm{\nabla_{\vy} f(\tilde{\vx}_0, \vy_{-1})} \cdot \norm{\nabla_{\vx} f(\vx_0, \tilde{\vy}_0)} \\
        &\le \delta \beta L_{xy} \bigopen{\sqrt{\frac{\mu_y}{\mu_x}} \cdot \norm{\nabla_{\vx} f(\vx_0, \tilde{\vy}_0)}^2 + \sqrt{\frac{\mu_x}{\mu_y}} \cdot \norm{\nabla_{\vy} f(\tilde{\vx}_1, \vy_0)}^2} \\
        &\phantom{\le} + \xi (\delta - 1) \beta L_{xy} \bigopen{\sqrt{\frac{\mu_y}{\mu_x}} \cdot \norm{\nabla_{\vx} f(\vx_0, \tilde{\vy}_0)}^2 + \sqrt{\frac{\mu_x}{\mu_y}} \cdot \norm{\nabla_{\vy} f(\tilde{\vx}_0, \vy_{-1})}^2},
    \end{align*}
    where we use AM-GM for the last inequality.
    
    Similarly, we can obtain
    \begin{align*}
        & - 2 \xi \inner{\nabla_{\vy} f(\tilde{\vx}_1, \vy_0) - \nabla_{\vy} f(\tilde{\vx}_0, \vy_0), \nabla_{\vy} f(\tilde{\vx}_0, \vy_{-1})} \\
        &\le 2 \xi \norm{\nabla_{\vy} f(\tilde{\vx}_1, \vy_0) - \nabla_{\vy} f(\tilde{\vx}_0, \vy_0)} \cdot \norm{\nabla_{\vy} f(\tilde{\vx}_0, \vy_{-1})} \\
        &\le 2 \xi L_{xy} \norm{\tilde{\vx}_1 - \tilde{\vx}_0} \cdot \norm{\nabla_{\vy} f(\tilde{\vx}_0, \vy_{-1})} \\
        &= 2 \xi L_{xy} \norm{(\tilde{\vx}_1 - \vx_0) - (\tilde{\vx}_0 - \vx_0)} \cdot \norm{\nabla_{\vy} f(\tilde{\vx}_0, \vy_{-1})} \\
        &= 2 \xi L_{xy} \bignorm{\gamma \alpha \nabla_{\vx} f(\vx_0, \tilde{\vy}_0) - \xi (\gamma - 1) \alpha \nabla_{\vx} f(\vx_{-1}, \tilde{\vy}_{-1})} \cdot \norm{\nabla_{\vy} f(\tilde{\vx}_0, \vy_{-1})} \\
        &= 2 \xi L_{xy} \bignorm{\gamma \alpha \nabla_{\vx} f(\vx_0, \tilde{\vy}_0) - (\gamma - 1) \alpha \nabla_{\vx} f(\vx_{-1}, \tilde{\vy}_{-1})} \cdot \norm{\nabla_{\vy} f(\tilde{\vx}_0, \vy_{-1})} \\
        &\le 2 \xi \gamma \alpha L_{xy} \norm{\nabla_{\vx} f(\vx_0, \tilde{\vy}_0)} \cdot \norm{\nabla_{\vy} f(\tilde{\vx}_0, \vy_{-1})} + 2 \xi (\gamma - 1) \alpha L_{xy} \norm{\nabla_{\vx} f(\vx_{-1}, \tilde{\vy}_{-1})} \cdot \norm{\nabla_{\vy} f(\tilde{\vx}_0, \vy_{-1})} \\
        &\le \gamma \xi \alpha L_{xy} \bigopen{\sqrt{\frac{\mu_y}{\mu_x}} \cdot \norm{\nabla_{\vx} f(\vx_0, \tilde{\vy}_0)}^2 + \sqrt{\frac{\mu_x}{\mu_y}} \cdot \norm{\nabla_{\vy} f(\tilde{\vx}_0, \vy_{-1})}^2} \\
        &\phantom{\le} + (\gamma - 1) \xi \alpha L_{xy} \bigopen{\sqrt{\frac{\mu_y}{\mu_x}} \cdot \norm{\nabla_{\vx} f(\vx_{-1}, \tilde{\vy}_{-1})}^2 + \sqrt{\frac{\mu_x}{\mu_y}} \cdot \norm{\nabla_{\vy} f(\tilde{\vx}_0, \vy_{-1})}^2},
    \end{align*}
    where the third equality is true for both $\xi = 1$ and $\xi = 0$ (where everything just becomes zero).
    
    From this, we can deduce that
    \begin{align*}
        & - 2 (\gamma - 1) \alpha \inner{\nabla_{\vx} f(\vx_1, \tilde{\vy}_1), \nabla_{\vx} f(\vx_0, \tilde{\vy}_0)} - 2 \xi (\delta - 1) \beta \inner{\nabla_{\vy} f(\tilde{\vx}_1, \vy_0), \nabla_{\vy} f(\tilde{\vx}_0, \vy_{-1})} \\
        &= - 2 (\gamma - 1) \alpha \inner{\nabla_{\vx} f(\vx_1, \tilde{\vy}_0), \nabla_{\vx} f(\vx_0, \tilde{\vy}_0)} - 2 \xi (\delta - 1) \beta \inner{\nabla_{\vy} f(\tilde{\vx}_0, \vy_0), \nabla_{\vy} f(\tilde{\vx}_0, \vy_{-1})} \\
        &\phantom{=} - 2 (\gamma - 1) \alpha \inner{\nabla_{\vx} f(\vx_1, \tilde{\vy}_1) - \nabla_{\vx} f(\vx_1, \tilde{\vy}_0), \nabla_{\vx} f(\vx_0, \tilde{\vy}_0)} \\
        &\phantom{=} - 2 \xi (\delta - 1) \beta \inner{\nabla_{\vy} f(\tilde{\vx}_1, \vy_0) - \nabla_{\vy} f(\tilde{\vx}_0, \vy_0), \nabla_{\vy} f(\tilde{\vx}_0, \vy_{-1})} \\
        &\le - 2 (\gamma - 1) \alpha (1 - \alpha L_x) \norm{\nabla_{\vx} f(\vx_0, \tilde{\vy}_0)}^2 - 2 \xi (\delta - 1) \beta (1 - \beta L_y) \norm{\nabla_{\vy} f(\tilde{\vx}_0, \vy_{-1})}^2 \\
        &\phantom{=} + (\gamma - 1) \delta \alpha \beta L_{xy} \bigopen{\sqrt{\frac{\mu_y}{\mu_x}} \cdot \norm{\nabla_{\vx} f(\vx_0, \tilde{\vy}_0)}^2 + \sqrt{\frac{\mu_x}{\mu_y}} \cdot \norm{\nabla_{\vy} f(\tilde{\vx}_1, \vy_0)}^2} \\
        &\phantom{=} + \xi (\gamma - 1) (\delta - 1) \alpha \beta L_{xy} \bigopen{\sqrt{\frac{\mu_y}{\mu_x}} \cdot \norm{\nabla_{\vx} f(\vx_0, \tilde{\vy}_0)}^2 + \sqrt{\frac{\mu_x}{\mu_y}} \cdot \norm{\nabla_{\vy} f(\tilde{\vx}_0, \vy_{-1})}^2} \\
        &\phantom{=} + \xi \gamma (\delta - 1) \alpha \beta L_{xy} \bigopen{\sqrt{\frac{\mu_y}{\mu_x}} \cdot \norm{\nabla_{\vx} f(\vx_0, \tilde{\vy}_0)}^2 + \sqrt{\frac{\mu_x}{\mu_y}} \cdot \norm{\nabla_{\vy} f(\tilde{\vx}_0, \vy_{-1})}^2} \\
        &\phantom{\le} + \xi (\gamma - 1) (\delta - 1) \alpha \beta L_{xy} \bigopen{\sqrt{\frac{\mu_y}{\mu_x}} \cdot \norm{\nabla_{\vx} f(\vx_{-1}, \tilde{\vy}_{-1})}^2 + \sqrt{\frac{\mu_x}{\mu_y}} \cdot \norm{\nabla_{\vy} f(\tilde{\vx}_0, \vy_{-1})}^2}.
    \end{align*}
    
    Applying this to \eqref{eq:totalone}, we have
    \begin{align}
        \begin{aligned}
        & \frac{1}{\alpha} \| \vx_1 - \vx_{\star} \|^2 + \frac{2}{\beta} \| \vy_1 - \vy_{\star} \|^2 + \frac{1}{\alpha} \| \vx_2 - \vx_{\star} \|^2 \\
        &\le \bigopen{\frac{1}{\alpha} - \mu_x} \| \vx_0 - \vx_{\star} \|^2 + 2 \bigopen{\frac{1}{\beta} - \mu_y} \| \vy_0 - \vy_{\star} \|^2 + \bigopen{\frac{1}{\alpha} - \mu_x} \| \vx_1 - \vx_{\star} \|^2 \\
        &\phantom{\le} + \alpha \norm{\nabla_{\vx} f(\vx_0, \tilde{\vy}_0)}^2 + 2 \beta \norm{\nabla_{\vy} f(\tilde{\vx}_1, \vy_0)}^2 + \alpha \norm{\nabla_{\vx} f(\vx_1, \tilde{\vy}_1)}^2 \\
        &\phantom{\le} - \gamma \alpha (2 - \gamma \alpha L_x) \norm{\nabla_{\vx} f(\vx_0, \tilde{\vy}_0)}^2 + \xi^2 (\delta-1)^2 \beta^2 L_y \norm{\nabla_{\vy} f(\tilde{\vx}_0, \vy_{-1})}^2 \\
        &\phantom{\le} - \delta \beta (2 - \delta \beta L_y) \norm{\nabla_{\vy} f(\tilde{\vx}_1, \vy_0)}^2 + (\gamma - 1)^2 \alpha^2 L_x \norm{\nabla_{\vx} f(\vx_0, \tilde{\vy}_0)}^2 \\
        &\phantom{\le} - 2 (\gamma - 1) \alpha (1 - \alpha L_x) \norm{\nabla_{\vx} f(\vx_0, \tilde{\vy}_0)}^2 - 2 \xi (\delta - 1) \beta (1 - \beta L_y) \norm{\nabla_{\vy} f(\tilde{\vx}_0, \vy_{-1})}^2 \\
        &\phantom{\le} + (\gamma - 1) \delta \alpha \beta L_{xy} \bigopen{\sqrt{\frac{\mu_y}{\mu_x}} \cdot \norm{\nabla_{\vx} f(\vx_0, \tilde{\vy}_0)}^2 + \sqrt{\frac{\mu_x}{\mu_y}} \cdot \norm{\nabla_{\vy} f(\tilde{\vx}_1, \vy_0)}^2} \\
        &\phantom{\le} + \xi (\gamma - 1) (\delta - 1) \alpha \beta L_{xy} \bigopen{\sqrt{\frac{\mu_y}{\mu_x}} \cdot \norm{\nabla_{\vx} f(\vx_0, \tilde{\vy}_0)}^2 + \sqrt{\frac{\mu_x}{\mu_y}} \cdot \norm{\nabla_{\vy} f(\tilde{\vx}_0, \vy_{-1})}^2} \\
        &\phantom{\le} + \xi \gamma (\delta - 1) \alpha \beta L_{xy} \bigopen{\sqrt{\frac{\mu_y}{\mu_x}} \cdot \norm{\nabla_{\vx} f(\vx_0, \tilde{\vy}_0)}^2 + \sqrt{\frac{\mu_x}{\mu_y}} \cdot \norm{\nabla_{\vy} f(\tilde{\vx}_0, \vy_{-1})}^2} \\
        &\phantom{\le} + \xi (\gamma - 1) (\delta - 1) \alpha \beta L_{xy} \bigopen{\sqrt{\frac{\mu_y}{\mu_x}} \cdot \norm{\nabla_{\vx} f(\vx_{-1}, \tilde{\vy}_{-1})}^2 + \sqrt{\frac{\mu_x}{\mu_y}} \cdot \norm{\nabla_{\vy} f(\tilde{\vx}_0, \vy_{-1})}^2} . 
        \end{aligned}
        \label{eq:totaltwo}
    \end{align}

    \subsubsection*{Step 3. Simplify using Step Size Conditions}

    Let us gather all $\nabla_{\vx}$ terms in \eqref{eq:totaltwo}, and define the sum of all such terms as
    \begin{align*}
        S_{\vx} &= \alpha \norm{\nabla_{\vx} f(\vx_0, \tilde{\vy}_0)}^2 + \alpha \norm{\nabla_{\vx} f(\vx_1, \tilde{\vy}_1)}^2 - \gamma \alpha (2 - \gamma \alpha L_x) \norm{\nabla_{\vx} f(\vx_0, \tilde{\vy}_0)}^2 \\
        &\phantom{=} + (\gamma - 1)^2 \alpha^2 L_x \norm{\nabla_{\vx} f(\vx_0, \tilde{\vy}_0)}^2 - 2 (\gamma - 1) \alpha (1 - \alpha L_x) \norm{\nabla_{\vx} f(\vx_0, \tilde{\vy}_0)}^2 \\
        &\phantom{=} + (\gamma - 1) \delta \alpha \beta L_{xy} \sqrt{\frac{\mu_y}{\mu_x}} \cdot \norm{\nabla_{\vx} f(\vx_0, \tilde{\vy}_0)}^2 + \xi (\gamma - 1) (\delta - 1) \alpha \beta L_{xy} \sqrt{\frac{\mu_y}{\mu_x}} \cdot \norm{\nabla_{\vx} f(\vx_0, \tilde{\vy}_0)}^2 \\
        &\phantom{=} + \xi \gamma (\delta - 1) \alpha \beta L_{xy} \sqrt{\frac{\mu_y}{\mu_x}} \cdot \norm{\nabla_{\vx} f(\vx_0, \tilde{\vy}_0)}^2 + \xi (\gamma - 1) (\delta - 1) \alpha \beta L_{xy} \sqrt{\frac{\mu_y}{\mu_x}} \cdot \norm{\nabla_{\vx} f(\vx_{-1}, \tilde{\vy}_{-1})}^2.
    \end{align*}
    
    
    Rearranging terms, we have
    \begin{align*}
        S_{\vx} &= \alpha \norm{\nabla_{\vx} f(\vx_1, \tilde{\vy}_1)}^2 - \alpha \norm{\nabla_{\vx} f(\vx_0, \tilde{\vy}_0)}^2 \\
        &\phantom{=} - \xi (\gamma - 1) (\delta - 1) \alpha \beta L_{xy} \sqrt{\frac{\mu_y}{\mu_x}} \cdot \norm{\nabla_{\vx} f(\vx_{0}, \tilde{\vy}_{0})}^2 + \xi (\gamma - 1) (\delta - 1) \alpha \beta L_{xy} \sqrt{\frac{\mu_y}{\mu_x}} \cdot \norm{\nabla_{\vx} f(\vx_{-1}, \tilde{\vy}_{-1})}^2 \\
        &\phantom{=} + \bigopen{2 \alpha - \gamma \alpha (2 - \gamma \alpha L_x) + (\gamma - 1)^2 \alpha^2 L_x - 2 (\gamma - 1) \alpha (1 - \alpha L_x)} \norm{\nabla_{\vx} f(\vx_0, \tilde{\vy}_0)}^2 \\
        &\phantom{=} + \bigopen{(\gamma - 1) \delta + \xi (3 \gamma - 2) (\delta - 1)} \alpha \beta L_{xy} \sqrt{\frac{\mu_y}{\mu_x}} \cdot \norm{\nabla_{\vx} f(\vx_0, \tilde{\vy}_0)}^2 \\
        &= \alpha \norm{\nabla_{\vx} f(\vx_1, \tilde{\vy}_1)}^2 - \alpha \norm{\nabla_{\vx} f(\vx_0, \tilde{\vy}_0)}^2 \\
        &\phantom{=} - \xi (\gamma - 1) (\delta - 1) \alpha \beta L_{xy} \sqrt{\frac{\mu_y}{\mu_x}} \cdot \norm{\nabla_{\vx} f(\vx_{0}, \tilde{\vy}_{0})}^2 + \xi (\gamma - 1) (\delta - 1) \alpha \beta L_{xy} \sqrt{\frac{\mu_y}{\mu_x}} \cdot \norm{\nabla_{\vx} f(\vx_{-1}, \tilde{\vy}_{-1})}^2 \\
        &\phantom{=} - \alpha \bigopen{4 (\gamma - 1) - (2 \gamma^2 - 1) \alpha L_x - \bigopen{(\gamma - 1) \delta + \xi (3 \gamma - 2) (\delta - 1)} \beta L_{xy} \sqrt{\frac{\mu_y}{\mu_x}}} \cdot \norm{\nabla_{\vx} f(\vx_0, \tilde{\vy}_0)}^2 \\
        &\le \alpha \norm{\nabla_{\vx} f(\vx_1, \tilde{\vy}_1)}^2 - \alpha \norm{\nabla_{\vx} f(\vx_0, \tilde{\vy}_0)}^2 \\
        &\phantom{=} - \xi (\gamma - 1) (\delta - 1) \alpha \beta L_{xy} \sqrt{\frac{\mu_y}{\mu_x}} \cdot \norm{\nabla_{\vx} f(\vx_{0}, \tilde{\vy}_{0})}^2 + \xi (\gamma - 1) (\delta - 1) \alpha \beta L_{xy} \sqrt{\frac{\mu_y}{\mu_x}} \cdot \norm{\nabla_{\vx} f(\vx_{-1}, \tilde{\vy}_{-1})}^2 \\
        &\phantom{=} - \alpha \bigopen{4 (\gamma - 1) - (2 \gamma^2 - 1) C_1 - \bigopen{(\gamma - 1) \delta + \xi (3 \gamma - 2) (\delta - 1)} C_4} \norm{\nabla_{\vx} f(\vx_0, \tilde{\vy}_0)}^2,
    \end{align*}
    where we use $\alpha \le \frac{C_1}{L_x}$ and $\beta \le \frac{C_4}{L_{xy}} \sqrt{\frac{\mu_x}{\mu_y}}$.
    
    Since we have from \eqref{eq:constone}:
    \begin{align*}
        4(\gamma - 1) &\ge (2 \gamma^2 - 1) C_1 + \bigopen{4 \gamma \delta - 3 \gamma - 3 \delta + 2} C_4 \\
        &= (2 \gamma^2 - 1) C_1 + \bigopen{(\gamma - 1) \delta + (3 \gamma - 2) (\delta - 1)} C_4 \\
        &\ge (2 \gamma^2 - 1) C_1 + \bigopen{(\gamma - 1) \delta + \xi (3 \gamma - 2) (\delta - 1)} C_4,
    \end{align*}
    we can deduce that
    \begin{align*}
        - \alpha \bigopen{4 (\gamma - 1) - (2 \gamma^2 - 1) C_1 - \bigopen{(\gamma - 1) \delta + \xi (3 \gamma - 2) (\delta - 1)} C_4} \norm{\nabla_{\vx} f(\vx_0, \tilde{\vy}_0)}^2 \le 0
    \end{align*}
    and therefore
    \begin{align}
    \begin{aligned}
        S_{\vx} &\le \alpha \norm{\nabla_{\vx} f(\vx_1, \tilde{\vy}_1)}^2 - \alpha \norm{\nabla_{\vx} f(\vx_0, \tilde{\vy}_0)}^2 \\
        &\phantom{=} - \xi (\gamma - 1) (\delta - 1) \alpha \beta L_{xy} \sqrt{\frac{\mu_y}{\mu_x}} \cdot \norm{\nabla_{\vx} f(\vx_{0}, \tilde{\vy}_{0})}^2 + \xi (\gamma - 1) (\delta - 1) \alpha \beta L_{xy} \sqrt{\frac{\mu_y}{\mu_x}} \cdot \norm{\nabla_{\vx} f(\vx_{-1}, \tilde{\vy}_{-1})}^2.
    \end{aligned} \label{eq:xxxxxxxx}
    \end{align}
    
    Similarly, Let us gather all $\nabla_{\vy}$ terms in \eqref{eq:totaltwo}, and define the sum all such terms as
    \begin{align*}
        S_{\vy} &= 2 \beta \norm{\nabla_{\vy} f(\tilde{\vx}_1, \vy_0)}^2 + \xi^2 (\delta - 1)^2 \beta^2 L_y \norm{\nabla_{\vy} f(\tilde{\vx}_0, \vy_{-1})}^2 \\
        &\phantom{=} - \delta \beta (2 - \delta \beta L_y) \norm{\nabla_{\vy} f(\tilde{\vx}_1, \vy_0)}^2 - 2 \xi (\delta - 1) \beta (1 - \beta L_y) \norm{\nabla_{\vy} f(\tilde{\vx}_0, \vy_{-1})}^2 \\
        &\phantom{=}+ (\gamma - 1) \delta \alpha \beta L_{xy} \sqrt{\frac{\mu_x}{\mu_y}} \cdot \norm{\nabla_{\vy} f(\tilde{\vx}_1, \vy_0)}^2 + \xi (\gamma - 1) (\delta - 1) \alpha \beta L_{xy} \sqrt{\frac{\mu_x}{\mu_y}} \cdot \norm{\nabla_{\vy} f(\tilde{\vx}_0, \vy_{-1})}^2 \\
        &\phantom{=} + \xi \gamma (\delta - 1) \alpha \beta L_{xy} \sqrt{\frac{\mu_x}{\mu_y}} \cdot \norm{\nabla_{\vy} f(\tilde{\vx}_0, \vy_{-1})}^2 + \xi (\gamma - 1) (\delta - 1) \alpha \beta L_{xy} \sqrt{\frac{\mu_x}{\mu_y}} \cdot \norm{\nabla_{\vy} f(\tilde{\vx}_0, \vy_{-1})}^2.
    \end{align*}
    
    Rearranging terms, we have
    \begin{align*}
        S_{\vy} &= \bigopen{2 \beta - \delta \beta (2 - \delta \beta L_y) + (\gamma - 1) \delta \alpha \beta L_{xy} \sqrt{\frac{\mu_x}{\mu_y}}} \cdot \norm{\nabla_{\vy} f(\tilde{\vx}_1, \vy_0)}^2 \\
        &\phantom{=} + \bigopen{\xi^2 (\delta - 1)^2 \beta^2 L_y - 2 \xi (\delta - 1) \beta (1 - \beta L_y) + \xi (3 \gamma - 2) (\delta - 1) \alpha \beta L_{xy} \sqrt{\frac{\mu_x}{\mu_y}}} \cdot \norm{\nabla_{\vy} f(\tilde{\vx}_0, \vy_{-1})}^2 \\
        &= - \beta \bigopen{2 (\delta - 1) - \delta^2 \beta L_y - (\gamma - 1) \delta \alpha L_{xy} \sqrt{\frac{\mu_x}{\mu_y}}} \norm{\nabla_{\vy} f(\tilde{\vx}_{1}, \vy_{0})}^2 \\
        &\phantom{=} - \xi (\delta - 1) \beta \bigopen{2 - (\xi (\delta - 1) + 2) \beta L_y - (3 \gamma - 2) \alpha L_{xy} \sqrt{\frac{\mu_x}{\mu_y}}} \norm{\nabla_{\vy} f(\tilde{\vx}_{0}, \vy_{-1})}^2 \\
        &\le - \beta \bigopen{2 (\delta - 1) - \delta^2 C_2 - (\gamma - 1) \delta C_3} \norm{\nabla_{\vy} f(\tilde{\vx}_1, \vy_0)}^2 \\
        &\phantom{=} - \xi (\delta - 1) \beta \bigopen{2 - (\xi (\delta - 1) + 2) C_2 - (3 \gamma - 2) C_3} \norm{\nabla_{\vy} f(\tilde{\vx}_0, \vy_{-1})}^2,
    \end{align*}
    where we use $\beta \le \frac{C_2}{L_y}$ and $\alpha \le \frac{C_3}{L_{xy}} \sqrt{\frac{\mu_y}{\mu_x}}$.
    
    Since we have from (\ref{eq:consttwo}) and (\ref{eq:constthree}):
    \begin{align*}
    {\delta - 1} &\ge \delta^2 C_2 + (\gamma - 1) \delta C_3, \\
    \quad 2 &\ge (\delta + 1) C_2 + (3 \gamma - 2) C_3 \ge (\xi (\delta - 1) + 2) C_2 + (3 \gamma - 2) C_3,
    \end{align*}
    we can deduce that
    \begin{align}
        S_{\vy} &\le - (\delta -1 ) \beta \norm{\nabla_{\vy} f(\tilde{\vx}_1, \vy_0)}^2. \label{eq:yyyyyyyy}
    \end{align}
    
    By \eqref{eq:xxxxxxxx}~and~\eqref{eq:yyyyyyyy}, we can observe that \eqref{eq:totaltwo} boils down to
    \begin{align*}
        &\phantom{\le} \frac{1}{\alpha} \| \vx_1 - \vx_{\star} \|^2 + \frac{2}{\beta} \| \vy_1 - \vy_{\star} \|^2 + \frac{1}{\alpha} \| \vx_2 - \vx_{\star} \|^2 \\
        &\le \bigopen{\frac{1}{\alpha} - \mu_x} \| \vx_0 - \vx_{\star} \|^2 + 2 \bigopen{\frac{1}{\beta} - \mu_y} \| \vy_0 - \vy_{\star} \|^2 + \bigopen{\frac{1}{\alpha} - \mu_x} \| \vx_1 - \vx_{\star} \|^2 + S_{\vx} + S_{\vy} \\
        &\le \bigopen{\frac{1}{\alpha} - \mu_x} \| \vx_0 - \vx_{\star} \|^2 + 2 \bigopen{\frac{1}{\beta} - \mu_y} \| \vy_0 - \vy_{\star} \|^2 + \bigopen{\frac{1}{\alpha} - \mu_x} \| \vx_1 - \vx_{\star} \|^2 \\
        &\phantom{\le} + \alpha \norm{\nabla_{\vx} f(\vx_1, \tilde{\vy}_1)}^2 - \alpha \norm{\nabla_{\vx} f(\vx_0, \tilde{\vy}_0)}^2 \\
        &\phantom{\le} - \xi (\gamma - 1) (\delta - 1) \alpha \beta L_{xy} \sqrt{\frac{\mu_y}{\mu_x}} \cdot \norm{\nabla_{\vx} f(\vx_{0}, \tilde{\vy}_{0})}^2 + \xi (\gamma - 1) (\delta - 1) \alpha \beta L_{xy} \sqrt{\frac{\mu_y}{\mu_x}} \cdot \norm{\nabla_{\vx} f(\vx_{-1}, \tilde{\vy}_{-1})}^2 \\
        &\phantom{\le} - (\delta -1 ) \beta \norm{\nabla_{\vy} f(\tilde{\vx}_1, \vy_0)}^2,
    \end{align*}
    or equivalently
    \begin{align*}
        &\phantom{\le} \frac{1}{\alpha} \| \vx_1 - \vx_{\star} \|^2 + \frac{2}{\beta} \| \vy_1 - \vy_{\star} \|^2 + \frac{1}{\alpha} \| \vx_2 - \vx_{\star} \|^2 \\
        &\phantom{\le} - \alpha \norm{\nabla_{\vx} f(\vx_1, \tilde{\vy}_1)}^2 {+ (\delta -1 ) \beta \norm{\nabla_{\vy} f(\tilde{\vx}_1, \vy_0)}^2} + \xi (\gamma - 1) (\delta - 1) \alpha \beta L_{xy} \sqrt{\frac{\mu_y}{\mu_x}} \cdot \norm{\nabla_{\vx} f(\vx_{0}, \tilde{\vy}_{0})}^2 \\
        &\le \bigopen{\frac{1}{\alpha} - \mu_x} \| \vx_0 - \vx_{\star} \|^2 + 2 \bigopen{\frac{1}{\beta} - \mu_y} \| \vy_0 - \vy_{\star} \|^2 + \bigopen{\frac{1}{\alpha} - \mu_x} \| \vx_1 - \vx_{\star} \|^2 \\
        &\phantom{\le} - \alpha \norm{\nabla_{\vx} f(\vx_0, \tilde{\vy}_0)}^2 + \xi (\gamma - 1) (\delta - 1) \alpha \beta L_{xy} \sqrt{\frac{\mu_y}{\mu_x}} \cdot \norm{\nabla_{\vx} f(\vx_{-1}, \tilde{\vy}_{-1})}^2,
    \end{align*}
    which is identical to \eqref{eq:alexgdacontraction} and therefore concludes the proof.
\end{proof}

%% file: sece1.tex
\section{\texorpdfstring{Proofs used in \cref{sec:6}}{Proofs used in Section 6}}
\label{sec:e}

Here we prove all theorems related to \alexgda{} on bilinear problems presented in \cref{sec:6}.
\begin{itemize}
    \item In \cref{sec:alexgdabilinearconvergentstepsize} we prove \cref{thm:alexgdabilinearconvergentstepsize} which shows the exact condition for linear convergence of \alexgda{} on bilinear problems. 
    \item In \cref{sec:alexgdabilinearcomplexity} we prove \cref{thm:alexgdabilinearcomplexity} which obtains iteration complexity of \alexgda{} for bilinear problems.
    \item In \cref{sec:alexgdabilineartech} we prove technical propositions and lemmas used throughout the proofs in \cref{sec:e}.
\end{itemize}

\subsection{\texorpdfstring{Proof of \cref{thm:alexgdabilinearconvergentstepsize}}{Proof of Theorem 6.1}} \label{sec:alexgdabilinearconvergentstepsize}

Here we prove \cref{thm:alexgdabilinearconvergentstepsize} of \cref{sec:6}, restated below for the sake of readability.

\thmalexgdabilinearconvergentstepsize*

\begin{proof}
    For a bilinear problem $f(\vx, \vy) = \vx^\top \mB \vy$, each iteration ($k\ge 0$) of \alexgda{} is written as $\tilde{\vy}_0 = \vy_0$ and
    \begin{align*}
    \vx_{k+1} &= \vx_k - \alpha \mB \tilde{\vy}_k, \\
    \tilde{\vx}_{k+1} &= \vx_k - \gamma\alpha \mB \tilde{\vy}_k, \\
    \vy_{k+1} &= \vy_k + \beta \mB^{\top} \tilde{\vx}_{k+1} = \beta \mB^{\top}\vx_k + \vy_k - \gamma\alpha\beta \mB^{\top} \mB \tilde{\vy}_k, \\
    \tilde{\vy}_{k+1} &= \vy_k + \delta\beta \mB^{\top} \tilde{\vx}_{k+1} = \delta\beta \mB^{\top}\vx_k + \vy_k - \gamma\alpha\delta\beta \mB^{\top} \mB \tilde{\vy}_k.
    \end{align*}
    This can be represented in the following matrix iteration:
    \begin{align}
    \label{eq:alexgda_update_bilinear_matrix}
    \vw_{k+1} = \begin{bmatrix}
        \vx_{k+1} \\ \vy_{k+1} \\ \tilde{\vy}_{k+1}
    \end{bmatrix}
    =
    \begin{bmatrix}
        \mI & \vzero & -\alpha \mB \\
        \beta \mB^\top & \mI & - \gamma\alpha\beta \mB^\top \mB \\
        \delta\beta \mB^\top & \mI & - \gamma\alpha\delta\beta \mB^\top \mB
    \end{bmatrix}
    \begin{bmatrix}
        \vx_{k} \\ \vy_{k} \\ \tilde{\vy}_{k}
    \end{bmatrix} = \mM \vw_k.
    \end{align} 
    Consider a reduced form of singular value decomposition (SVD) of $\mB=\mU\bm\Sigma\mV^\top$: $\mU \in \R^{d_x\times s}, \mV \in \R^{d_y\times s}, \bm\Sigma \in \R^{s\times s}$ where $s=\rank(\mB)$. 
    Note that $\mU^\top \mU = \mI$, $\mV^\top\mV = \mI$, and $\bm\Sigma = \operatorname{diag}(\sigma_1, \ldots, \sigma_s)$ is a diagonal matrix with non-zero diagonal entries ($0<\mu_{xy}\le \sigma_i \le L_{xy}$ for all $i=1, \ldots, s$).
    Then the power of the matrix $\mM$ defined in \cref{eq:alexgda_update_bilinear_matrix} can be decomposed as follows for $k\ge 1$.
    \begin{align*}
        \mM^k = \underbrace{\begin{bmatrix}
            \mU & \vzero & \vzero \\
            \vzero & \mV & \vzero \\
            \vzero & \vzero & \mV 
        \end{bmatrix}}_{=: \mW}
        \underbrace{\begin{bmatrix}
            \mI & \vzero & -\alpha \bm\Sigma \\
            \beta \bm\Sigma & \mI & -\gamma\alpha\beta \bm\Sigma^2 \\
            \delta\beta \bm\Sigma & \mI & -\gamma\alpha\delta\beta \bm\Sigma^2 
        \end{bmatrix}^k}_{=: \widetilde{\mM}^k} \begin{bmatrix}
            \mU^\top & \vzero & \vzero \\
            \vzero & \mV^\top & \vzero \\
            \vzero & \vzero & \mV^\top 
        \end{bmatrix} + \begin{bmatrix}
            \mI-\mU\mU^\top & \vzero & \vzero \\
            \vzero & \mI-\mV\mV^\top & \vzero \\
            \vzero & \mI-\mV\mV^\top & \vzero
        \end{bmatrix}
    \end{align*}
    From this matrix decomposition, a decomposition of the ambient space $\R^{d_x+d_y+d_y}$ naturally arises: a space $\gN = \nullspace(\mB) \times \nullspace(\mB^\top) \times \nullspace(\mB^\top)$ and its orthogonal complement $\gN^\perp = \row(\mB)\times \row(\mB^\top)\times\row(\mB^\top)$. The $\gN$-component of the iterate $\vw_k$ is always fixed as $$\begin{bmatrix}
        (\mI-\mU\mU^\top)\vx_0 \\ (\mI-\mV\mV^\top)\vy_0 \\ (\mI-\mV\mV^\top)\vy_0
    \end{bmatrix}$$
    and does not move at all,
    while the $\gN^\perp$-component of $\vw_k$ belong to $\gN^\perp$ even after each iteration.
    Since $\nullspace(\mB) \times \nullspace(\mB^\top)$ is the space of all Nash equilibria of the bilinear problem, now it is enough to show that the $\gN^\perp$-component converges to the origin; as a result, the iterates $(\vx_k, \vy_k)$ converge to a Nash equilibrium
    \begin{align}
        \vz_\star := ((\mI-\mU\mU^\top)\vx_0, (\mI-\mV\mV^\top)\vy_0). \label{eq:bilinearNash}
    \end{align}
    To this end, we may assume that the initial iterate $\vw_0$ belongs to $\gN^\perp$ from now on. Then, by reasoning above, every iterate $\vw_k$ belongs to $\gN^\perp$ and satisfies
    \begin{align}
    \label{eq:alexgdabilinear_matrixiteration}
        \vw_k = \mW \widetilde{\mM}^k \mW^\top \vw_0.
    \end{align}
    We first claim that it suffices to show $\rho(\widetilde{\mM})<1$ to obtain (the necessary and sufficient condition for) the convergence $\vw_k \rightarrow \bm0$. To prove the claim, let $\widetilde{\vw}_k := \mW^\top \vw_k$. Then we have $\widetilde{\vw}_k = \widetilde\mM^k \widetilde{\vw}_0$. By applying the theory of matrix iteration (\cref{prop:spectralradius}), $\rho(\widetilde\mM)<1$ if and only if $\widetilde\vw_k \rightarrow \bm0$. Moreover, since $\vw_k \in \gN^\perp$, $\mW \widetilde\vw_k = \mW\mW^\top\vw_k = \vw_k$, and thus $\widetilde\vw_k \rightarrow \bm0$ if and only if $\vw_k \rightarrow \bm0$. Therefore, the rest of the proof is dedicated to finding the condition for $\rho(\widetilde\mM)<1$.

    Note that the matrix $\widetilde{\mM}\in\R^{3s\times 3s}$ does not have 1 as an eigenvalue. If it does, there exist vectors $\va, \vb, \vc \in \R^s$ such that  $\widetilde{\mM}\begin{bmatrix}
        \va^\top & \vb^\top & \vc^\top
    \end{bmatrix}^\top = \begin{bmatrix}
        \va^\top & \vb^\top & \vc^\top
    \end{bmatrix}^\top$. It implies that
    \begin{align*}
        \va - \alpha \bm\Sigma \vc &= \va, \\
        \beta \bm\Sigma\va + \vb - \gamma\alpha\beta\bm\Sigma^2\vc &= \vb, \\
        \delta\beta \bm\Sigma\va + \vb - \gamma\alpha\delta\beta\bm\Sigma^2\vc &= \vc,
    \end{align*}
    which implies that $\va=\vb=\vc=0$ because $\bm\Sigma$ is nonsingular.
    Thus, 1 cannot have an associated
    nonzero eigenvector of $\mM$.

    To inspect the eigenvalues of $\widetilde{\mM}$, we now apply the theory of Schur complement \citep{haynsworth1968schur, zhang2006schur}: namely, $\det\bigopen{\begin{bmatrix} A&B\\C&D \end{bmatrix}} = \det(A)\det(D-CA^{-1}B)$. Writing the characteristic polynomial of $\widetilde{\mM}$, 
    \begin{align*}
        \det(\lambda\mI - \widetilde{\mM}) &= \det\bigopen{\begin{bmatrix}
            (\lambda-1)\mI & \vzero & \alpha \bm\Sigma \\
            -\beta \bm\Sigma & (\lambda-1)\mI & \gamma\alpha\beta \bm\Sigma^2 \\
            -\delta\beta \bm\Sigma & -\mI & \lambda\mI + \gamma\alpha\delta\beta \bm\Sigma^2
        \end{bmatrix}} \\
        &= \det\bigopen{\begin{bmatrix}
            (\lambda-1)\mI & \vzero \\
            -\beta \bm\Sigma & (\lambda-1)\mI \\
        \end{bmatrix}} \det\bigopen{\lambda\mI + \gamma\alpha\delta\beta \bm\Sigma^2 + \frac{\alpha}{\lambda-1} \begin{bmatrix}
            \delta\beta\bm\Sigma & \mI
        \end{bmatrix} \begin{bmatrix}
            \mI & \vzero \\
            -\frac{\beta}{\lambda-1}\bm\Sigma & \mI
        \end{bmatrix}^{-1} \begin{bmatrix}
            \bm\Sigma \\ \gamma\beta\bm\Sigma^2
        \end{bmatrix}} \\
        &= (\lambda-1)^{2s} \det\bigopen{\lambda\mI + \gamma\alpha\delta\beta \bm\Sigma^2 + \frac{\alpha}{\lambda-1} \begin{bmatrix}
            \delta\beta\bm\Sigma & \mI
        \end{bmatrix} \begin{bmatrix}
            \mI & \vzero \\
            \frac{\beta}{\lambda-1}\bm\Sigma & \mI
        \end{bmatrix} \begin{bmatrix}
            \bm\Sigma \\ \gamma\beta\bm\Sigma^2
        \end{bmatrix}} \\
        &= \det\bigopen{\lambda(\lambda-1)^2\mI + \alpha\beta\bigopen{\gamma(\lambda-1)+1}\bigopen{\delta(\lambda-1)+1}\bm\Sigma^2} = 0.
    \end{align*}
    Hence, for each eigenvalue $\sigma_i^2$ of $\bm\Sigma^2$, the roots $\lambda$ of a cubic polynomial
    \begin{align}
        P_i(\lambda) &:= \lambda(\lambda-1)^2 + \alpha\beta\sigma_i^2 \bigopen{\gamma(\lambda-1)+1}\bigopen{\delta(\lambda-1)+1} \label{eq:alexgda_polynomial}\\
        &= \lambda^3 - (2-\phi_i\gamma\delta)\lambda^2 + \bigset{1- \phi_i(2\gamma\delta-\gamma-\delta)}\lambda + \phi_i (\gamma-1)(\delta-1)\nonumber
    \end{align}
    are eigenvalues of $\widetilde{\mM}$, where  $\phi_i := \alpha\beta\sigma_i^2 > 0$.
    To obtain a necessary and sufficient condition of $|\lambda|<1$, we apply \cref{prop:routhhurwitz}:
    \begin{gather}
        \phi_i\gamma\delta - 2 + \phi_i(\gamma-1)(\delta-1) < 2 - \phi_i(2\gamma\delta - \gamma - \delta), \label{eq:secc5_5}\\
        \phi_i\gamma\delta - 2 + \phi_i(\gamma-1)(\delta-1) > -2 + \phi_i(2\gamma\delta - \gamma - \delta), \label{eq:secc5_6}\\
        \phi_i\gamma\delta - 2 - 3\phi_i(\gamma-1)(\delta-1) < 2 + \phi_i(2\gamma\delta - \gamma - \delta), \label{eq:secc5_7}\\
        \phi_i\gamma\delta - 2 - 3\phi_i(\gamma-1)(\delta-1) > -2 - \phi_i(2\gamma\delta - \gamma - \delta), \label{eq:secc5_8}\\
        \phi_i(\gamma-1)(\delta-1)(\phi_i(\gamma-1)(\delta-1) - \phi_i\gamma\delta + 2) + 1 - \phi_i(2\gamma\delta - \gamma - \delta) < 1, \label{eq:secc5_9}
    \end{gather}
    which are equivalent to
    \begin{gather}
        \phi_i \stackrel{\eqref{eq:secc5_6}}{>} 0, \quad \text{\color{gray} (which is already true,)} \\
        \gamma+\delta \stackrel{\eqref{eq:secc5_8}}{>} \frac{3}{2}, \label{eq:secc5_1}\\
        \phi_i(2\gamma-1)(2\delta-1) \stackrel{\eqref{eq:secc5_5}}{<} 4, \label{eq:secc5_2}\\
        \phi_i(1-4(\gamma-1)(\delta-1))\stackrel{\eqref{eq:secc5_7}}{<}4, \label{eq:secc5_3}\\
        -(\gamma-1)(\delta-1)(\gamma+\delta-1)\phi_i \stackrel{\eqref{eq:secc5_9}}{<} \gamma+\delta-2. \label{eq:secc5_4}
    \end{gather}

    To make these conditions more concise and interpretable, we conduct a case analysis on $\gamma$ and $\delta$ to know which condition among them is essential for having $|\lambda|<1$ (in fact, $\gamma+\delta>\frac{3}{2}$ is not enough yet!) and to identify what condition on $\phi_i$ should suffice for each case.

    \paragraph{Case 1. $(\gamma-1)(\delta-1) \ge 0$ and $\gamma+\delta>2$.} Note that \cref{eq:secc5_4} is true. Also, $(2\gamma-1)(2\delta-1) > 1-4(\gamma-1)(\delta-1)$ since
    \begin{align*}
        (2\gamma-1)(2\delta-1) - 1+4(\gamma-1)(\delta-1) &= 2(4\gamma\delta-3(\gamma+\delta) + 2)\\
        &= 8(\gamma-1)(\delta-1) + 2(\gamma+\delta-2) > 0.
    \end{align*}
    Thus, \cref{eq:secc5_2} implies \cref{eq:secc5_3}. It means that \cref{eq:secc5_2} alone is enough: $|\lambda|<1$ if
    \begin{align*}
        \phi_i<\frac{4}{(2\gamma-1)(2\delta-1)}.
    \end{align*}

    \paragraph{Case 2. $(\gamma-1)(\delta-1) \ge 0$ and $\frac{3}{2} < \gamma+\delta \le 2$.} For this case, it is impossible to satisfy all four conditions \eqref{eq:secc5_1}--\eqref{eq:secc5_4} at the same time. We prove it by contradiction. Note that $0 < (2\gamma-1)(2\delta-1) < 1-4(\gamma-1)(\delta-1)$ since 
    \begin{align*}
        (2\gamma-1)(2\delta-1) &= 4(\gamma-1)(\delta-1) + 2(\gamma+\delta) - 3 > 0, \\
        (2\gamma-1)(2\delta-1) - 1+4(\gamma-1)(\delta-1) &= 2(4\gamma\delta-3(\gamma+\delta) + 2) \\
        &\le 2((\gamma+\delta)^2-3(\gamma+\delta) + 2) \\
        &= 2(\gamma+\delta - 1)(\gamma+\delta - 2) < 0.
    \end{align*}
    From \cref{eq:secc5_3} and \cref{eq:secc5_4}, it must hold that
    \begin{align*}
         \frac{2-\gamma-\delta}{(\gamma-1)(\delta-1)(\gamma+\delta-1)} < \phi_i < \frac{4}{1-4(\gamma-1)(\delta-1)}.
    \end{align*}
    However, it implies that 
    \begin{align*}
        &4(\gamma-1)(\delta-1)(\gamma+\delta-1) - (2-\gamma-\delta)(1-4(\gamma-1)(\delta-1)) \\
        &= 4\gamma\delta-3(\gamma+\delta) + 2 > 0,
    \end{align*}
    which is a contradiction.

    \paragraph{Case 3. $(\gamma-1)(\delta-1) < 0$ and $\frac{3}{2} < \gamma+\delta \le 2$.} This case is also impossible since it contradicts \cref{eq:secc5_4}.

    \paragraph{Case 4. $(\gamma-1)(\delta-1) < 0$, $\gamma+\delta > 2$, and $4\gamma\delta-3(\gamma+\delta) + 2 \ge 0$.} In this case, it holds that 
    \begin{align*}
        0<\frac{-(\gamma-1)(\delta-1)(\gamma+\delta-1)}{\gamma+\delta-2} \le \frac{1-4(\gamma-1)(\delta-1)}{4}  \le \frac{(2\gamma-1)(2\delta-1)}{4}
    \end{align*}
    since 
    \begin{align*}
        (2\gamma-1)(2\delta-1) - 1+4(\gamma-1)(\delta-1) = 2(4\gamma\delta-3(\gamma+\delta) + 2) \ge 0
    \end{align*}
    and
    \begin{align*}
        &4(\gamma-1)(\delta-1)(\gamma+\delta-1) + (\gamma + \delta - 2)(1-4(\gamma-1)(\delta-1)) \\
        &= 4\gamma\delta-3(\gamma+\delta) + 2 \ge 0.
    \end{align*}
    
    Thus, \cref{eq:secc5_2} implies \cref{eq:secc5_3} and \cref{eq:secc5_4}. Since the rightmost term is positive, we have
    $|\lambda|<1$ if
    \begin{align*}
        \phi_i<\frac{4}{(2\gamma-1)(2\delta-1)}.
    \end{align*}

    \paragraph{Case 5. $(\gamma-1)(\delta-1) < 0$, $\gamma+\delta > 2$, and $4\gamma\delta-3(\gamma+\delta) + 2 < 0$.} In this case, it holds that
    \begin{align*}
            \frac{(2\gamma-1)(2\delta-1)}{4} < \frac{1-4(\gamma-1)(\delta-1)}{4} < \frac{-(\gamma-1)(\delta-1)(\gamma+\delta-1)}{\gamma+\delta-2}
    \end{align*}
    Thus, \cref{eq:secc5_4} implies \cref{eq:secc5_2} and \cref{eq:secc5_3}. Since the rightmost term is positive, we have $|\lambda|<1$ if 
    \begin{align*}
        \phi_i < \frac{\gamma+\delta-2}{-(\gamma-1)(\delta-1)(\gamma+\delta-1)}.
    \end{align*}

    Combining all these five cases,
    \begin{enumerate}
        \item (\textbf{Case 2} + \textbf{Case 3}) If $\gamma+\delta \le 2$, the polynomial $P_i(\lambda)$ must have a root outside of the open unit disk; 
        hence, the matrix iteration in \cref{eq:alexgdabilinear_matrixiteration} diverges.
        \item (\textbf{Case 1} + \textbf{Case 4}) If $\gamma+\delta > 2$ and $4\gamma\delta-3(\gamma+\delta) + 2 \ge 0$ (which includes the case of $\gamma+\delta > 2$, $\gamma\ge1$, and $\delta\ge1$), all the roots of the polynomial $P_i(\lambda)$ lie on the open unit disk $|\lambda|<1$ if 
        $$\phi_i < \frac{4}{(2\gamma-1)(2\delta-1)}.$$
        Hence if we choose step sizes $\alpha$ and $\beta$ such that $$\alpha\beta < \frac{4}{(2\gamma-1)(2\delta-1)L_{xy}^2},$$ then all the eigenvalues of $\widetilde{\mM}$ lie on the open unit disk; 
        the matrix iteration in \cref{eq:alexgdabilinear_matrixiteration} does converge.
        \item (\textbf{Case 5}) If $\gamma+\delta > 2$ and $4\gamma\delta-3(\gamma+\delta) + 2 < 0$,
        all the roots of the polynomial $P_i(\lambda)$ lie on the open unit disk $|\lambda|<1$ if 
        $$\phi_i < \frac{\gamma+\delta-2}{-(\gamma-1)(\delta-1)(\gamma+\delta-1)}.$$
        Hence if we choose step sizes $\alpha$ and $\beta$ such that $$\alpha\beta < \frac{\gamma+\delta-2}{-(\gamma-1)(\delta-1)(\gamma+\delta-1)L_{xy}^2},$$ then all the eigenvalues of $\widetilde{\mM}$ lie on the open unit disk; 
        the matrix iteration in \cref{eq:alexgdabilinear_matrixiteration} does converge.
    \end{enumerate}

    This proves the theorem.
    
\end{proof}

%% file: sece2.tex
\subsection{\texorpdfstring{Proof of \cref{thm:alexgdabilinearcomplexity}}{Proof of Theorem 6.2}} \label{sec:alexgdabilinearcomplexity}

Here we prove \cref{thm:alexgdabilinearcomplexity} of \cref{sec:6}, restated below for the sake of readability.

\thmalexgdabilinearcomplexity*

\begin{proof}
    Recall that the Nash equilibrium that the iterates converges to is already characterized in \cref{eq:bilinearNash}. So, as in the proof in \cref{sec:alexgdabilinearcomplexity}, we again assume that $\vw_0$ (defined in \cref{eq:alexgda_update_bilinear_matrix}) belongs to $\gN^\perp = \row(\mB)\times\row(\mB^\top)\times\row(\mB^\top)$ and we inspect the convergence (to $\vzero$) of the sequence \eqref{eq:alexgdabilinear_matrixiteration}. For this reason, we analyze the spectral radius of the matrix $\widetilde{\mM}$ (defined in \cref{eq:alexgda_update_bilinear_matrix}). This will directly give us a convergence rate as well as iteration complexity ($\tilde{\gO}\bigopen{\frac{1}{1-\rho(\widetilde{\mM})}}$).
    
    We divide the proof into two parts: the case of general parameters $\gamma\ge 1$ and $\delta\ge 1$, and the case of $\delta=1$. 
    Throughout the proof, we keep the notation consistent with the proof of \cref{thm:alexgdabilinearconvergentstepsize} in \cref{sec:alexgdabilinearconvergentstepsize}.

    \subsubsection{\texorpdfstring{The Case of General $\gamma\ge 1$ and $\delta\ge 1$}{}}

    We have to find an upper bound of $|\lambda|$ which is strictly smaller than 1, whose difference with 1 is not negligible.
    Hence, we use a slightly smaller bound $\phi_i\le \frac{2}{(2\gamma-1)(2\delta-1)}$ than that in \cref{thm:alexgdabilinearconvergentstepsize}.
    
    With some substitutions
    \begin{align*}
        \psi_i := \alpha\beta\gamma\delta\sigma_i^2 = \gamma\delta\phi_i > 0, \quad \Gamma := 1-\frac{1}{\gamma} \in [0,1), \quad \Delta := 1-\frac{1}{\delta}\in [0,1),
    \end{align*} 
    we can rewrite the polynomial $P_i(\lambda)$ as
    \begin{align}
    \label{eq:alexpolynomial}
        P_i(\lambda) &= \lambda(\lambda-1)^2 + \psi_i (\lambda - \Gamma)(\lambda - \Delta) \\
        &= \lambda^3 - (2-\psi_i)\lambda^2 + \bigset{1- \psi_i(\Gamma+\Delta)}\lambda + \psi_i \Gamma \Delta. \nonumber
    \end{align}
    
    Since $P_i(0) = \psi_i \Gamma \Delta \ge 0$ and 
    \begin{align*}
        P_i\bigopen{-\frac{1}{2}} = -\frac{9}{8} + \psi_i\bigopen{\Gamma+\frac{1}{2}}\bigopen{\Delta+\frac{1}{2}} <0
    \end{align*}
    holds because 
    \begin{align*}
        \psi_i &\le \frac{2}{(\Gamma+1)(\Delta+1)} = \frac{2\gamma\delta}{(2\gamma-1)(2\delta-1)} = \frac{9\gamma\delta}{2\bigopen{3\gamma-\frac{3}{2}}\bigopen{3\delta-\frac{3}{2}}} \\
        &< \frac{9\gamma\delta}{2\bigopen{3\gamma-2}\bigopen{3\delta-2}} = \frac{9}{8\bigopen{\Gamma+\frac{1}{2}}\bigopen{\Delta+\frac{1}{2}}}.
    \end{align*}
    Thus, there exists a non-positive real root $-r\in (-\frac{1}{2}, 0]$.

    We can show that there is no positive real root if $\psi_i$ is small enough.
    \begin{restatable}{proposition}{propnopositiveroot}
        \label{prop:nopositiveroot}
        The polynomial $P_i(\lambda)$ defined in \cref{eq:alexpolynomial} has no positive real root if 
        \begin{align*}
            \psi_i |\Gamma-\Delta| \le \min\bigset{(1-\Gamma)^2, (1-\Delta)^2}.
        \end{align*}
    \end{restatable}
    The proof of this proposition can be found in \cref{sec:nopositiveroot}. From the root coefficient relationship, we know that the sum of three roots of $P_i(\lambda)$ equals $2-\psi_i>0$, which holds because $\psi_i \le \frac{2}{(\Gamma+1)(\Delta+1)}<2$. Hence, $P_i(\lambda)$ must have a single real root $-r \le0$ and two complex conjugate roots $c$ and $\bar{c}$, where $\Re[c] > 0$.

    Note that we have another bound for the unique real root. Plugging in $\lambda=-r$ to $P_i(\lambda)=0$, we have
    \begin{gather}
        \bigset{1-\psi_i(\Gamma+\Delta)}(-r) + \psi_i \Gamma\Delta = r^3 + (2-\psi_i)r^2 \ge 0, \nonumber\\
        \therefore r\le \frac{\psi_i \Gamma\Delta}{1-\psi_i(\Gamma+\Delta)}. \label{eq:realrootbound}
    \end{gather}

    Again from the root coefficient relationship, we know that
    \begin{align*}
        -r + 2\Re[c] &= 2-\psi_i, \\
        -2r\Re[c]+|c|^2 &= 1-\psi_i(\Gamma+\Delta).
    \end{align*}
    Plugging one into another, we have an expression of the squared absolute value of a complex root in terms of $r$ as
    \begin{align}
    \label{eq:alexgda_bilinear_complexrootbound}
        \begin{aligned}
            |c|^2 &= 1-\psi_i(\Gamma+\Delta)+r(2-\psi_i+r) \\
        &\stackrel{\eqref{eq:realrootbound}}{\le}  1-\psi_i(\Gamma+\Delta) + \frac{\psi_i \Gamma\Delta}{1-\psi_i(\Gamma+\Delta)}\bigopen{2-\psi_i + \frac{\psi_i \Gamma\Delta}{1-\psi_i(\Gamma+\Delta)}} \\
        &= 1- \psi_i \bigset{\Gamma+\Delta - \frac{ \Gamma\Delta}{1-\psi_i(\Gamma+\Delta)}\bigopen{2-\psi_i + \frac{\psi_i \Gamma\Delta}{1-\psi_i(\Gamma+\Delta)}}}
        \end{aligned}
    \end{align}
    To show that $|c|^2$ is strictly smaller than 1, we want to show that 
    \begin{align*}
        \Gamma+\Delta - \frac{ \Gamma\Delta}{1-\psi_i(\Gamma+\Delta)}\bigopen{2-\psi_i + \frac{\psi_i \Gamma\Delta}{1-\psi_i(\Gamma+\Delta)}}>0.
    \end{align*}
    In fact, this is shown in the following proposition.
    \begin{restatable}{proposition}{propforbilinearcomplexity}
        \label{prop:forbilinearcomplexity}
        \begin{align*}
            \Gamma+\Delta - \frac{ \Gamma\Delta}{1-\psi_i(\Gamma+\Delta)}\bigopen{2-\psi_i + \frac{\psi_i \Gamma\Delta}{1-\psi_i(\Gamma+\Delta)}} \ge \frac{1}{4}(\Gamma+\Delta-2\Gamma\Delta) > 0
        \end{align*}
        if $\psi_i \le \frac{\Gamma+\Delta-2\Gamma\Delta}{2(\Gamma+\Delta)^2}.$
    \end{restatable}
    The proof of this proposition can be found in \cref{sec:forbilinearcomplexity}. 
    Therefore, gathering the fact that $|-r|^2 < \frac{1}{4}$, \cref{eq:alexgda_bilinear_complexrootbound}, and \cref{prop:forbilinearcomplexity}, for every root $\lambda$ of the polynomial $P_i(\lambda)$,
    \begin{align}
    \label{eq:alexgda_bilinear_eigenvaluebound}
        \begin{aligned}
            |\lambda|^2 &< \max\bigset{\frac{1}{4}, 1-\frac{\psi_i}{4}(\Gamma+\Delta-2\Gamma\Delta)} \\
            &= \max\bigset{\frac{1}{4}, 1-\frac{1}{4}\alpha\beta\sigma_i^2 (\gamma+\delta-2)},
        \end{aligned}
    \end{align}
    where
    \begin{align*}
        \psi_i \le \min\bigset{\frac{2}{(\Gamma+1)(\Delta+1)}, \frac{\min\bigset{(1-\Gamma)^2, (1-\Delta)^2}}{|\Gamma-\Delta|}, \frac{\Gamma+\Delta-2\Gamma\Delta}{2(\Gamma+\Delta)^2}},
    \end{align*}
    or equivalently,
    \begin{gather}
        \alpha\beta\sigma_i^2 \le \frac{1}{C_{\gamma,\delta}}, \nonumber\\
        \text{where } C_{\gamma,\delta} := \max\bigset{\frac{(2\gamma-1)(2\delta-1)}{2}, |\gamma-\delta|\max\bigset{\gamma, \delta}^2, \frac{2(2\gamma\delta-\gamma-\delta)^2}{\gamma+\delta-2}}. \label{eq:cgammadelta}
    \end{gather}
    Hence, if we choose step sizes $\alpha$ and $\beta$ such that $\alpha\beta = \frac{1}{C_{\gamma,\delta} L_{xy}^2}$, the bound in \cref{eq:alexgda_bilinear_eigenvaluebound} holds for all $i=1, ..., s$, thereby we obtain a strict upper bound of spectral radius of the matrix $\widetilde{\mM}$ as follows:
    \begin{align*}
        \rho(\widetilde{\mM})^2 &< \max\bigset{\frac{1}{4}, 1-\frac{1}{4}\alpha\beta\mu_{xy}^2 (\gamma+\delta-2)}, \\
        &= \max\bigset{\frac{1}{4}, 1-\frac{\gamma+\delta-2}{4C_{\gamma,\delta}} \frac{\mu_{xy}^2}{L_{xy}^2}}.
    \end{align*}
    In conclusion, the matrix iteration in \cref{eq:alexgdabilinear_matrixiteration} can satisfy $\norm{\vw_k}^2 < \epsilon$ with
    \begin{align*}
        k=\gO\bigopen{\max\bigset{1, \frac{C_{\gamma,\delta}}{\gamma+\delta-2}\cdot\frac{L_{xy}^2}{\mu_{xy}^2}} \log\bigopen{\frac{\norm{\vw_0}^2}{\eps}}}
    \end{align*}
    iterations.

    \begin{remark}
        One may notice that the constant $C_{\gamma,\delta}$ defined in \cref{eq:cgammadelta} may grow as $\gamma^3$ or $\delta^3$, which can make the range of step size with certified convergence rate shrink and degrade the iteration complexity.
        However, when $\delta=1$, our analysis gets simpler and we can choose an optimal set of parameters $\alpha$, $\beta$, and $\gamma$ to ``optimize'' the spectral radius (and thus the convergence rate).
    \end{remark} 

    \subsubsection{\texorpdfstring{The case of $\delta=1$ ($\gamma > 1$)}{}}
    
    Let us go back to the polynomial $P_i(\lambda)$ (\cref{eq:alexgda_polynomial}).
    If $\delta=1$ (and thus we choose $\gamma>1$), the polynomial becomes
    \begin{align*}
        P^{(\delta=1)}_i(\lambda) = \lambda \bigset{(\lambda-1)^2 + \alpha\beta\sigma_i^2(\gamma\lambda-(\gamma-1))}.
    \end{align*}
    So, we know one root exactly: $\lambda=0$.
    Since we want a small absolute value of eigenvalues but 0 is a trivial lower bound of $|\lambda|$, we only have to care about the other two roots: $(\lambda-1)^2 + \alpha\beta\sigma_i^2(\gamma\lambda-(\gamma-1)) = 0$, or 
    \begin{align*}
        \lambda_0 &:= 1-\frac{\gamma\alpha\beta\sigma_i^2 - \sqrt{(2-\gamma\alpha\beta\sigma_i^2)^2 -4(1-(\gamma-1)\alpha\beta\sigma_i^2)}}{2}, \\
        \lambda_1 &:= 1-\frac{\gamma\alpha\beta\sigma_i^2 + \sqrt{(2-\gamma\alpha\beta\sigma_i^2)^2 -4(1-(\gamma-1)\alpha\beta\sigma_i^2)}}{2}.
    \end{align*}

    The maximum absolute value of eigenvalues can be calculated as 
    \begin{align}
        &\max\bigset{|\lambda_0|, |\lambda_1|} \nonumber\\
        &= \begin{dcases}
            \sqrt{1- (\gamma-1)\alpha\beta\sigma_i^2} & \text{if } (2-\gamma\alpha\beta\sigma_i^2)^2 \le 4(1-(\gamma-1)\alpha\beta\sigma_i^2), \\
            \bigabs{1-\frac{\gamma\alpha\beta\sigma_i^2}{2}}+\frac{\sqrt{(2-\gamma\alpha\beta\sigma_i^2)^2 -4(1-(\gamma-1)\alpha\beta\sigma_i^2)}}{2} & \text{if } (2-\gamma\alpha\beta\sigma_i^2)^2 > 4(1-(\gamma-1)\alpha\beta\sigma_i^2).
        \end{dcases} \nonumber\\
        &= \begin{dcases}
            \sqrt{1- (\gamma-1)\alpha\beta\sigma_i^2} & \text{if } \gamma^2\alpha\beta\sigma_i^2\le 4, \\
            \bigabs{1-\frac{\gamma\alpha\beta\sigma_i^2}{2}}+\frac{\sqrt{(\gamma^2\alpha\beta\sigma_i^2- 4)\alpha\beta\sigma_i^2}}{2} & \text{if }  \gamma^2\alpha\beta\sigma_i^2> 4.
        \end{dcases} \label{eq:delta=1_spectralradius}\\
        &=: r(\alpha, \beta, \gamma, \sigma_i^2) \nonumber
    \end{align}
    
    Thus, if we want to optimize the spectral radius $\rho(\widetilde{\mM})$ (which directly gives the convergence rate exponent) by choosing parameters $\alpha$, $\beta$, and $\gamma$, we have to solve the following minimax problem:
    \begin{align*}
         \min_{\alpha, \beta, \gamma} \,\max_{i=1,\ldots,s} r(\alpha, \beta, \gamma, \sigma_i^2).
    \end{align*}
    
    Suppose $L_{xy} = \sigma_1 \ge \cdots \ge \sigma_s = \mu_{xy}$. We consider 3 cases:
    
    \paragraph{Case 1. $\gamma^2\alpha\beta L_{xy}^2\le 4$.}
    In this case, $\gamma^2\alpha\beta \sigma_i^2\le 4$ holds for all $i = 1, \dots, s$, and then $r(\alpha, \beta, \gamma, \sigma_i^2) = \sqrt{1- (\gamma-1)\alpha\beta\sigma_i^2}$ is a decreasing function of $\sigma_i^2$. Hence, it suffices to
    minimize $\sqrt{1-(\gamma-1)\alpha\beta \mu_{xy}^2} $ over $\alpha$, $\beta$, and $\gamma$.
    The optimal choice of $\alpha\beta$ is $\frac{4}{\gamma^2 L_{xy}^2}$ which comes from the condition $\gamma^2\alpha\beta L_{xy}^2\le 4$, so we minimize $\sqrt{1-\frac{4(\gamma-1)\mu_{xy}^2}{\gamma^2 L_{xy}^2}}$ over $\gamma$. 
    The optimal $\gamma$ is $2$, so the optimal spectral radius is $\sqrt{1-\frac{\mu_{xy}^2}{L_{xy}^2}}$, which can be obtained with $\alpha\beta=\frac{1}{L_{xy}^2}$ and $\gamma=2$. 
    
    \paragraph{Case 2. $\gamma^2\alpha\beta \mu_{xy}^2 \ge 4$.} 
    Note that 
    \begin{align*}
        \bigabs{1-\frac{\gamma\alpha\beta\sigma_i^2}{2}}+\frac{\sqrt{(\gamma^2\alpha\beta\sigma_i^2- 4)\alpha\beta\sigma_i^2}}{2}
    \end{align*}
    is an increasing function in terms of $\sigma_i^2 \ge \frac{4}{\gamma^2 \alpha\beta}$. This can be shown by proving that
    \begin{align*}
        1-\frac{\gamma\alpha\beta\sigma_i^2}{2}+\frac{\sqrt{(\gamma^2\alpha\beta\sigma_i^2- 4)\alpha\beta\sigma_i^2}}{2}
    \end{align*}
    and
    \begin{align*}
        -1+\frac{\gamma\alpha\beta\sigma_i^2}{2}+\frac{\sqrt{(\gamma^2\alpha\beta\sigma_i^2- 4)\alpha\beta\sigma_i^2}}{2}
    \end{align*}
    are both increasing functions in terms of $\sigma_i^2$. The latter case is easy, so we show for the former one: using the derivative in $\sigma_i^2$,
    \begin{align*}
        \frac{d}{d \sigma_i^2} \bigopen{-\gamma\alpha\beta\sigma_i^2+\sqrt{(\gamma^2\alpha\beta\sigma_i^2- 4)\alpha\beta\sigma_i^2}} =  -\gamma\alpha\beta + \frac{\gamma^2\alpha^2\beta^2\sigma_i^2 - 2\alpha\beta}{\sqrt{\gamma^2\alpha^2\beta^2\sigma_i^4- 4\alpha\beta\sigma_i^2}} >0,
    \end{align*}
    So it suffices to minimize 
    \begin{align*}
        \bigabs{1-\frac{\gamma\alpha\beta L_{xy}^2}{2}}+\frac{\sqrt{(\gamma^2\alpha\beta L_{xy}^2- 4)\alpha\beta L_{xy}^2}}{2}
    \end{align*}
    over $\alpha$, $\beta$, and $\gamma$. 
    In fact, this is also an increasing function in terms of $\alpha\beta \ge \frac{4}{\gamma^2\mu_{xy}^2}$ and the optimal choice of $\alpha\beta$ is $\frac{4}{\gamma^2\mu_{xy}^2}$ (which comes from the condition $\gamma^2\alpha\beta \mu_{xy}^2 \ge 4$). So it is left to minimize
    \begin{align*}
        \bigabs{1-\frac{2L_{xy}^2}{\gamma \mu_{xy}^2}}+\frac{2 L_{xy}}{\gamma \mu_{xy}}\sqrt{\frac{L_{xy}^2}{\mu_{xy}^2}-1}
    \end{align*}
    over $\gamma$. This has a minimum $\sqrt{1-\frac{\mu_{xy}^2}{L_{xy}^2}}$ at $\gamma=\frac{2L_{xy}^2}{\mu_{xy}^2}$. Hence, the optimal spectral radius is $\sqrt{1-\frac{\mu_{xy}^2}{L_{xy}^2}}$, which is achieved with $\gamma=\frac{2L_{xy}^2}{\mu_{xy}^2}$ and $\alpha\beta=\frac{\mu_{xy}^2}{L_{xy}^4}$.
    
    \paragraph{Case 3. $\gamma^2\alpha\beta \mu_{xy}^2 \le 4 \le \gamma^2\alpha\beta L_{xy}^2$.} Maximizing $r(\alpha,\beta,\gamma,\sigma_i^2)$ over $i=1, \ldots, s$, we only need to obtain
    \begin{align*}
        \min_{\alpha,\beta, \gamma} \max\bigset{\sqrt{1-(\gamma-1)\alpha\beta \mu_{xy}^2},\,
        \bigabs{1-\frac{\gamma\alpha\beta L_{xy}^2}{2}}+\frac{\sqrt{(\gamma^2\alpha\beta L_{xy}^2- 4)\alpha\beta L_{xy}^2}}{2}}.
    \end{align*}
  
    For a fixed $\gamma$, the optimal $X := \alpha\beta$ is uniquely attained when
    \begin{align}
        \label{eq:optimalX-condition}
        \sqrt{1-(\gamma-1) \mu_{xy}^2 X} = 
        \bigabs{1-\frac{\gamma L_{xy}^2 X}{2}}+\frac{\sqrt{\gamma^2 L_{xy}^4 X^2- 4L_{xy}^2 X}}{2},
    \end{align}
    because the left hand side decreases in $X$ but the right hand side increases in $X$, as well as 
    \begin{align}
        \label{eq:leftislarge}
        \sqrt{1-(\gamma-1) \mu_{xy}^2 \frac{4}{\gamma^2 L_{xy}^2}} \ge \bigabs{1-\frac{\gamma L_{xy}^2} {2} \frac{4}{\gamma^2 L_{xy}^2}} + 0
    \end{align}
    and
    \begin{align}
    \label{eq:rightislarge}
        \sqrt{1-(\gamma-1) \mu_{xy}^2 \frac{4}{\gamma^2 \mu_{xy}^2}} \le 
        \bigabs{1-\frac{\gamma L_{xy}^2}{2} \frac{4}{\gamma^2 \mu_{xy}^2}}+\frac{\sqrt{\gamma^2 L_{xy}^4 \bigopen{\frac{4}{\gamma^2 \mu_{xy}^2}}^2- 4L_{xy}^2 \bigopen{\frac{4}{\gamma^2 \mu_{xy}^2}}}}{2}.
    \end{align}
    \cref{eq:leftislarge} can be shown as
    \begin{align*}
        \bigopen{1-(\gamma-1) \mu_{xy}^2 \frac{4}{\gamma^2 L_{xy}^2}} - \bigopen{1-\frac{2}{\gamma}}^2 = \frac{4(\gamma-1)}{\gamma^2}\bigopen{1-\frac{\mu_{xy}^2}{L_{xy}^2}} \ge 0.
    \end{align*}
    In addition, \cref{eq:rightislarge} can be shown as the following case analysis: if $\gamma \le \frac{L_{xy}^2}{\mu_{xy}^2} + 1$ then
    \begin{align*}
        \bigabs{1-\frac{2L_{xy}^2}{\gamma\mu_{xy}^2}}^2 - \bigopen{1- \frac{4(\gamma-1)}{\gamma^2}} = \frac{4}{\gamma^2}\bigopen{\frac{L_{xy}^2}{\mu_{xy}^2}-1}\bigopen{\frac{L_{xy}^2}{\mu_{xy}^2} + 1 - \gamma}\ge 0,
    \end{align*}
    if $\frac{L_{xy}^2}{\mu_{xy}^2} + 1 < \gamma \le \frac{2L_{xy}^2}{\mu_{xy}^2}$ then
    \begin{align*}
       \frac{2L_{xy}}{\gamma \mu_{xy}}\sqrt{\frac{L_{xy}^2}{\mu_{xy}^2}-1} + \bigopen{\frac{2L_{xy}^2}{\gamma\mu_{xy}^2}-1} - \bigopen{1-\frac{2}{\gamma}} \ge \frac{2}{\gamma}\bigset{\frac{L_{xy}}{\mu_{xy}} \sqrt{\frac{L_{xy}^2}{\mu_{xy}^2}-1} - \bigopen{\frac{L_{xy}^2}{\mu_{xy}^2} - 1}} \ge 0,
    \end{align*}
    and if $\gamma > \frac{2L_{xy}^2}{\mu_{xy}^2}$,
    \begin{align*}
        \frac{2L_{xy}}{\gamma \mu_{xy}}\sqrt{\frac{L_{xy}^2}{\mu_{xy}^2}-1} + \bigopen{1-\frac{2L_{xy}^2}{\gamma\mu_{xy}^2}} - \bigopen{1-\frac{2}{\gamma}} = \frac{2}{\gamma}\bigset{\frac{L_{xy}}{\mu_{xy}} \sqrt{\frac{L_{xy}^2}{\mu_{xy}^2}-1} - \bigopen{\frac{L_{xy}^2}{\mu_{xy}^2} - 1}} \ge 0.
    \end{align*}

    Solving \cref{eq:optimalX-condition},
    \begin{gather*}
        1-(\gamma-1) \mu_{xy}^2 X = 
        {1-\gamma L_{xy}^2 X + \frac{\gamma^2 L_{xy}^4}{4} X^2}
        +\frac{\gamma^2 L_{xy}^4 X^2- 4L_{xy}^2 X}{4}
        +2\bigabs{1-\frac{\gamma L_{xy}^2 X}{2}}\frac{\sqrt{\gamma^2 L_{xy}^4 X^2- 4L_{xy}^2 X}}{2}, \\
        \bigset{(\gamma+1) L_{xy}^2 -(\gamma-1) \mu_{xy}^2} X -
        \frac{\gamma^2 L_{xy}^4}{2} X^2
        = \bigabs{1-\frac{\gamma L_{xy}^2 X}{2}}\sqrt{\gamma^2 L_{xy}^4 X^2- 4L_{xy}^2 X},\\
        \begin{aligned}
            &\bigset{(\gamma+1) L_{xy}^2 -(\gamma-1) \mu_{xy}^2}^2 X^2 - \bigset{(\gamma+1) L_{xy}^2 -(\gamma-1) \mu_{xy}^2} \gamma^2 L_{xy}^4 X^3 + \frac{\gamma^4 L_{xy}^8}{4} X^4 \\
            &= \bigopen{1 - \gamma L_{xy}^2 X + \frac{\gamma^2 L_{xy}^4 X^2}{4}}\bigopen{\gamma^2 L_{xy}^4 X^2- 4L_{xy}^2 X} \\
            &= - 4L_{xy}^2 X + \bigopen{\gamma^2 + 4\gamma } L_{xy}^4 X^2 - \bigopen{\gamma^3+\gamma^2}L_{xy}^6 X^3 + \frac{\gamma^4 L_{xy}^8}{4} X^4,
        \end{aligned}
    \end{gather*}
    \begin{align}
        \label{eq:optimalX-quadratic}
        4L_{xy}^2 - \underbrace{\bigopen{(2\gamma - 1) L_{xy}^4 + 2(\gamma^2-1) L_{xy}^2 \mu_{xy}^2 -(\gamma-1)^2 \mu_{xy}^4}}_{=:B(\gamma)} X + (\gamma^3-\gamma^2)L_{xy}^4 \mu_{xy}^2 X^2 = 0.
    \end{align}
    The discriminant equals
    \begin{align}
        \label{eq:difficultdiscriminant}
        \begin{aligned}
            &B(\gamma)^2-16(\gamma^3-\gamma^2)L_{xy}^6\mu_{xy}^2 \\
            &=\bigopen{L_{xy}^2 -(\gamma-1)\mu_{xy}^2}^2 \bigopen{(2\gamma-1)^2 L_{xy}^4 - 2 L_{xy}^2\mu_{xy}^2 (2\gamma^2 - \gamma - 1) + (\gamma-1)^2 \mu_{xy}^4} \\
            &= \bigopen{L_{xy}^2 -(\gamma-1)\mu_{xy}^2}^2 \bigset{\bigopen{(2\gamma+1)L_{xy}^2 - (\gamma-1)\mu_{xy}^2}^2 - (2\sqrt{2\gamma} L_{xy}^2)^2} \ge 0,
        \end{aligned}
    \end{align}
    where the last inequality is due to
    \begin{align*}
        (2\gamma+1)L_{xy}^2 - (\gamma-1)\mu_{xy}^2 \ge (\gamma+2)L_{xy}^2 \ge 2\sqrt{2\gamma} L_{xy}^2.
    \end{align*}
    Solving the quadratic equation in \cref{eq:optimalX-quadratic}, there are two possible optimal choices of $X$.
    \begin{align*}
        X = \frac{B(\gamma) \pm \sqrt{B(\gamma)^2-16(\gamma^3-\gamma^2)L_{xy}^6\mu_{xy}^2}}{2(\gamma^3-\gamma^2)L_{xy}^4\mu_{xy}^2}
    \end{align*}
    Nevertheless, we take only the minus sign to maximize the value of $\sqrt{1-(\gamma-1)\mu_{xy}^2 X}$ among possible $X$'s. This is because, if we took the plus sign, the $X$ would be a solution of
    \begin{align*}
        \sqrt{1-(\gamma-1) \mu_{xy}^2 X} = 
        -\bigabs{1-\frac{\gamma L_{xy}^2 X}{2}}+\frac{\sqrt{\gamma^2 L_{xy}^4 X^2- 4L_{xy}^2 X}}{2},
    \end{align*}
    but would not be a solution of \cref{eq:optimalX-condition}.
    In other words, the optimal choice of $X$ given a fixed $\gamma$ is
    \begin{align}
        \label{eq:optimalX}
        X_*(\gamma) := \frac{B(\gamma) - \sqrt{B(\gamma)^2-16(\gamma^3-\gamma^2)L_{xy}^6\mu_{xy}^2}}{2(\gamma^3-\gamma^2)L_{xy}^4\mu_{xy}^2}.
    \end{align}
    Putting this into the left hand side of \cref{eq:optimalX-condition}, now we need to minimize
    \begin{align*}
        \sqrt{1-\frac{B(\gamma) - \sqrt{B(\gamma)^2-16(\gamma^3-\gamma^2)L_{xy}^6\mu_{xy}^2}}{2\gamma^2L_{xy}^4}}.
    \end{align*}
    Here we utilize the following fact.
    \begin{restatable}{proposition}{propdifficultfunction}
    \label{prop:difficultfunction}
        Recall that $B(\gamma)=(2\gamma - 1) L_{xy}^4 + 2(\gamma^2-1) L_{xy}^2 \mu_{xy}^2 -(\gamma-1)^2 \mu_{xy}^4$. Then,
        \begin{align*}
            h(\gamma) := \frac{B(\gamma) - \sqrt{B(\gamma)^2-16(\gamma^3-\gamma^2)L_{xy}^6\mu_{xy}^2}}{\gamma^2}
        \end{align*}
        is increasing for $\gamma\in \bigclosed{1, 1+\frac{L_{xy}^2}{\mu_{xy}^2}}$ and decreasing for $\gamma\in \left[1+\frac{L_{xy}^2}{\mu_{xy}^2}, \infty\right)$.
    \end{restatable}
    The proof can be found in \cref{sec:difficultfunction}.
    Thus, the optimal value of $\gamma$ is 
    \begin{align*}
        \gamma_* = 1+\frac{L_{xy}^2}{\mu_{xy}^2}. 
    \end{align*}
    In this case,
    \begin{align*}
        B(\gamma_*) &= \bigopen{\frac{2L_{xy}^2}{\mu_{xy}^2}+1}L_{xy}^4 +2\bigopen{\frac{L_{xy}^2}{\mu_{xy}^2}+2}L_{xy}^4 - L_{xy}^4\\
        &= 4 \bigopen{\frac{L_{xy}^2}{\mu_{xy}^2}+1} L_{xy}^4 ,
    \end{align*}
    from which we can check
    \begin{align*}
        B(\gamma_*)^2 - 16(\gamma_*^3-\gamma_*^2)L_{xy}^6 \mu_{xy}^2 = 0.
    \end{align*}
    Therefore, the optimal $X$ in \cref{eq:optimalX} becomes much simpler:
    \begin{align*}
        X_*(\gamma_*) &= \frac{B(\gamma_*)}{2(\gamma_*^3-\gamma_*^2)L_{xy}^4\mu_{xy}^2} 
        \\
        &= \frac{4 \bigopen{\frac{L_{xy}^2}{\mu_{xy}^2}+1} L_{xy}^4 }{2 \bigopen{1+\frac{L_{xy}^2}{\mu_{xy}^2}}^2 L_{xy}^6}\\
        &= \frac{2\mu_{xy}^2}{L_{xy}^2 (L_{xy}^2 + \mu_{xy}^2)},
    \end{align*}
    and the corresponding spectral radius is
    \begin{align*}
        \sqrt{1-(\gamma_*-1)\cdot X_*(\gamma_*)\cdot\mu_{xy}^2} &= \sqrt{1-\frac{2 \mu_{xy}^2}{L_{xy}^2 + \mu_{xy}^2}} \\
        &= \sqrt{\frac{L_{xy}^2 - \mu_{xy}^2}{L_{xy}^2 + \mu_{xy}^2}}
    \end{align*}
    which is an even better (\textit{i.e.}, smaller) spectral radius than those in \textbf{Case 1} and \textbf{Case 2}. This concludes the proof.
\end{proof}


%% file: sece3.tex
\subsection{\texorpdfstring{Proofs used in \cref{sec:e}}{Proofs used in Appendix E}} \label{sec:alexgdabilineartech}

Here we prove some technical propositions and lemmas used throughout \cref{sec:e}.

\subsubsection{\texorpdfstring{Proof of \cref{prop:nopositiveroot}}{Proof of Proposition E.1}} \label{sec:nopositiveroot}

Here we prove \cref{prop:nopositiveroot}, restated below for the sake of readability.

\propnopositiveroot*

\begin{proof}
    Without loss of generality, suppose $0\le \Gamma \le \Delta < 1$. 
    Let $p(\lambda) = \lambda(\lambda-1)^2$ and $q(\lambda) = -\psi_i (\lambda-\Gamma) (\lambda-\Delta)$; simply $P_i(\lambda) = p(\lambda)-q(\lambda)$.
    Since $p(\lambda)\ge 0$ for $\lambda \ge 0$ and $q(\lambda)\ge 0$ only if $\lambda\in[\Gamma, \Delta] \subset [0,1]$, $P_i(\lambda)$ can have a positive root only in the interval $[\Gamma, \Delta]$. 
    So it suffices to show that $P_i(\lambda)>0$ for $\lambda\in[\Gamma, \Delta]$ for proving the proposition.

    Note that, for $\lambda\in[\Gamma, \Delta]$, $P_i(\lambda) \ge \lambda(1-\Delta)^2 + \psi_i(\lambda-\Gamma)(\lambda-\Delta)=: Q(\lambda)$. 
    Now it suffices to show  $Q(\lambda)>0$ for $\lambda\in[\Gamma, \Delta]$.

    Note that $Q(\lambda)$ is a quadratic polynomial and
    \begin{align*}
        Q(\lambda) = \psi_i\bigopen{\lambda - \frac{\psi_i (\Gamma+\Delta)-(1-\Delta)^2}{2\psi_i}}^2 -\frac{\bigset{\psi_i (\Gamma+\Delta)-(1-\Delta)^2}^2}{4\psi_i} + \psi_i\Gamma\Delta.
    \end{align*}
    Since $0 < Q(\Gamma)=\Gamma(1-\Delta)^2 \le Q(\Delta)=\Delta(1-\Delta)^2$, we can ensure $Q(\lambda)>0$ on $[\Gamma, \Delta]$ if 
    $\frac{\psi_i (\Gamma+\Delta)-(1-\Delta)^2}{2\psi_i} \le \Gamma$. 
    It is equivalent to $\psi_i(\Delta-\Gamma)\le (1-\Delta)^2$, which proves the proposition.
\end{proof}

\subsubsection{\texorpdfstring{Proof of \cref{prop:forbilinearcomplexity}}{Proof of Proposition E.2}} \label{sec:forbilinearcomplexity}

Here we prove \cref{prop:forbilinearcomplexity}, restated below for the sake of readability.

\propforbilinearcomplexity*

\begin{proof}
    Since $1-\psi_i(\Gamma+\Delta) \in (0, 1]$, the left hand side can be lower bounded as
    \begin{align}
        \label{eq:alexgda_bilinear_LHSlowerbound}
        \begin{aligned}
            &\Gamma+\Delta - \frac{ \Gamma\Delta}{1-\psi_i(\Gamma+\Delta)}\bigopen{2-\psi_i + \frac{\psi_i \Gamma\Delta}{1-\psi_i(\Gamma+\Delta)}}  \\
            &= \frac{(\Gamma+\Delta)(1-\psi_i(\Gamma+\Delta))^2 - \Gamma\Delta\bigopen{(2-\psi_i)(1-\psi_i(\Gamma+\Delta)) + \psi_i \Gamma\Delta}}{(1-\psi_i(\Gamma+\Delta))^2} \\
            &\ge (\Gamma+\Delta)(1-\psi_i(\Gamma+\Delta))^2 - \Gamma\Delta\bigopen{(2-\psi_i)(1-\psi_i(\Gamma+\Delta)) + \psi_i \Gamma\Delta},
        \end{aligned}
    \end{align}
    Which is a quadratic polynomial of $\psi_i$. Let
    \begin{align*}
        R(x) &:= (\Gamma+\Delta)(1-x(\Gamma+\Delta))^2 - \Gamma\Delta\bigopen{(2-x)(1-x(\Gamma+\Delta)) + x \Gamma\Delta} \\
        &=\underbrace{\bigset{(\Gamma+\Delta) \bigopen{\Gamma^2+\Gamma\Delta+\Delta^2}}}_{=a>0}x^2 - \underbrace{\bigset{(\Gamma+\Delta-\Gamma\Delta)^2 + \Gamma^2+\Gamma\Delta+\Delta^2}}_{=b>0}x + \underbrace{\bigset{\Gamma+\Delta-2\Gamma\Delta}}_{=c>0}.
    \end{align*}
    The discriminant of $R(x)$ is 
    \begin{align*}
        D &= b^2 - 4ac \\
        &= \bigset{(\Gamma+\Delta-\Gamma\Delta)^2 + \Gamma^2+\Gamma\Delta+\Delta^2}^2 - 4(\Gamma+\Delta) \bigopen{\Gamma^2+\Gamma\Delta+\Delta^2}\bigset{\Gamma+\Delta-2\Gamma\Delta}\\
        &= \Gamma^2\Delta^2 \bigset{8 (\Gamma+\Delta)^2 + (-1 + \Gamma\Delta)^2 - 4 (\Gamma+\Delta) (1 + \Gamma\Delta)} \\
        &= \Gamma^2\Delta^2 \bigset{\Gamma^2 (\Delta^2 - 4 \Delta + 8) - 2 \Gamma ( 2 \Delta^2 - 7\Delta + 2) + 8 \Delta^2 - 4 \Delta + 1} \\
        &= \Gamma^2\Delta^2\bigset{(\Delta^2 - 4 \Delta + 8)\bigopen{\Gamma - \frac{2 \Delta^2 - 7\Delta + 2}{\Delta^2 - 4 \Delta + 8}}^2 + \frac{(\Delta^2 - 4 \Delta + 8)(8 \Delta^2 - 4 \Delta + 1) - (2 \Delta^2 - 7\Delta + 2)^2}{\Delta^2 - 4 \Delta + 8}} \\
        &= \Gamma^2\Delta^2\bigset{((\Delta-2)^2 + 4)\bigopen{\Gamma - \frac{2 \Delta^2 - 7\Delta + 2}{\Delta^2 - 4 \Delta + 8}}^2 + \frac{4(\Delta^2 + 1)(\Delta-1)^2 + 16\Delta^2}{(\Delta-2)^2 + 4}} \ge 0,
    \end{align*}
    so $R(x)$ must have two (possibly identical) positive real roots. 
    This means that if we find a lower bound $\bar{x}>0$ for the roots, we can confirm that $R(x)\ge R(\bar{x})$ for all $x\in[0, \bar{x}]$.
    Using the fact $\sqrt{1-x} \le 1-\frac{x}{2}$ for all $x\le 1$, we have
    \begin{align*}
        \frac{b-\sqrt{b^2-4ac}}{2a} &\ge \frac{b}{2a}\bigopen{1-1+\frac{2ac}{b^2}} \\
        &= \frac{c}{b} \\
        &= \frac{\Gamma+\Delta-2\Gamma\Delta}{(\Gamma+\Delta-\Gamma\Delta)^2 + \Gamma^2+\Gamma\Delta+\Delta^2} \\
        &>\frac{\Gamma+\Delta-2\Gamma\Delta}{2(\Gamma+\Delta)^2} =: \bar{x}.
    \end{align*}
    Continuing from \cref{eq:alexgda_bilinear_LHSlowerbound}, since we assumed $\psi_i \le \bar{x}$,
    \begin{align*}
        &R(\psi_i) \ge R(\bar{x}) \\
        &= (\Gamma+\Delta)\bigopen{1-\frac{\Gamma+\Delta-2\Gamma\Delta}{2(\Gamma+\Delta)}}^2 - \Gamma\Delta\bigopen{\bigopen{2-\frac{\Gamma+\Delta-2\Gamma\Delta}{2(\Gamma+\Delta)^2}}\bigopen{1-\frac{\Gamma+\Delta-2\Gamma\Delta}{2(\Gamma+\Delta)}} + \frac{\Gamma+\Delta-2\Gamma\Delta}{2(\Gamma+\Delta)^2} \Gamma\Delta} \\
        &= \frac{(\Gamma+\Delta+2\Gamma\Delta)^2}{4(\Gamma+\Delta)} - \Gamma\Delta\bigopen{\bigopen{2-\frac{\Gamma+\Delta-2\Gamma\Delta}{2(\Gamma+\Delta)^2}}\frac{\Gamma+\Delta+2\Gamma\Delta}{2(\Gamma+\Delta)} + \frac{\Gamma+\Delta-2\Gamma\Delta}{2(\Gamma+\Delta)^2} \Gamma\Delta} \\
        &\ge \frac{(\Gamma+\Delta+2\Gamma\Delta)^2}{4(\Gamma+\Delta)} - \Gamma\Delta\bigopen{\frac{\Gamma+\Delta+2\Gamma\Delta}{\Gamma+\Delta} + \frac{\Gamma+\Delta-2\Gamma\Delta}{2(\Gamma+\Delta)^2} \Gamma\Delta} \\
        &= \frac{(\Gamma+\Delta+2\Gamma\Delta)(\Gamma+\Delta-2\Gamma\Delta)}{4(\Gamma+\Delta)} -  \frac{\Gamma+\Delta-2\Gamma\Delta}{2(\Gamma+\Delta)^2} \Gamma^2\Delta^2 \\
        &= \frac{\Gamma+\Delta -2\Gamma\Delta}{4(\Gamma+\Delta)^2}\bigset{(\Gamma+\Delta+2\Gamma\Delta)(\Gamma+\Delta) - 2\Gamma^2\Delta^2} \\
        &\ge \frac{1}{4}(\Gamma+\Delta -2\Gamma\Delta) > 0.
    \end{align*}
    which concludes the proof of the proposition.
\end{proof}

\subsubsection{\texorpdfstring{Proof of \cref{prop:difficultfunction}}{Proof of Proposition E.4}} \label{sec:difficultfunction}

Here we prove \cref{prop:difficultfunction}, restated below for the sake of readability.

\propdifficultfunction*

\begin{proof}
    From the calculation in \cref{eq:difficultdiscriminant}, 
    \begin{align*}
        h(\gamma) &= \frac{B(\gamma) -\bigabs{L_{xy}^2 -(\gamma-1)\mu_{xy}^2} \sqrt{(2\gamma-1)^2 L_{xy}^4 - 2 L_{xy}^2\mu_{xy}^2 (2\gamma^2 - \gamma - 1) + (\gamma-1)^2 \mu_{xy}^4}}{\gamma^2} \\
        &= \min\bigset{F(\gamma), G(\gamma)}
    \end{align*}
    where
    \begin{align*}
        F(\gamma) &= \frac{B(\gamma) -\bigopen{L_{xy}^2 -(\gamma-1)\mu_{xy}^2} \sqrt{(2\gamma-1)^2 L_{xy}^4 - 2 L_{xy}^2\mu_{xy}^2 (2\gamma^2 - \gamma - 1) + (\gamma-1)^2 \mu_{xy}^4}}{\gamma^2}, \\
        G(\gamma) &= \frac{B(\gamma) + \bigopen{L_{xy}^2 -(\gamma-1)\mu_{xy}^2} \sqrt{(2\gamma-1)^2 L_{xy}^4 - 2 L_{xy}^2\mu_{xy}^2 (2\gamma^2 - \gamma - 1) + (\gamma-1)^2 \mu_{xy}^4}}{\gamma^2}.
    \end{align*}
    We want to show that $F(\gamma)$ is increasing and $G(\gamma)$ is decreasing for $\gamma \in [1,\infty)$. 
    Let
    \begin{align*}
        J(\gamma) &:= \frac{B(\gamma)}{\gamma^2}= \bigopen{\frac{2\gamma-1}{\gamma^2}}L_{xy}^4  + 2\bigopen{1-\frac{1}{\gamma^2}} L_{xy}^2\mu_{xy}^2 - \bigopen{1-\frac{1}{\gamma}}^2 \mu_{xy}^4, \\
        K(\gamma) &:= \frac{L_{xy}^2 - (\gamma-1)\mu_{xy}^2}{\gamma^2}, \\
        M(\gamma) &:= (2\gamma-1)^2 L_{xy}^4 - 2 L_{xy}^2\mu_{xy}^2 (2\gamma^2 - \gamma - 1) + (\gamma-1)^2 \mu_{xy}^4, 
    \end{align*}
    so that
    \begin{align*}
        F(\gamma) &= J(\gamma) - K(\gamma)\sqrt{M(\gamma)}, \\
        G(\gamma) &= J(\gamma) + K(\gamma)\sqrt{M(\gamma)}.
    \end{align*}
    Then,
    \begin{align*}
        J'(\gamma) &= -2\bigopen{\frac{\gamma-1}{\gamma^3}} (L_{xy}^{4}+\mu_{xy}^4) + \frac{4}{\gamma^3} L_{xy}^2\mu_{xy}^2, \\
        K'(\gamma) &= \frac{-2L_{xy}^2 + (\gamma-2)\mu_{xy}^2}{\gamma^3}, \\
        M'(\gamma) &= 4(2\gamma-1)L_{xy}^4 - 2(4\gamma-1) L_{xy}^2\mu_{xy}^2 + 2(\gamma-1)\mu_{xy}^4.
    \end{align*}
    So,
    \begin{align*}
        F'(\gamma) &= \frac{2J'(\gamma)\sqrt{M(\gamma)} - 2K'(\gamma) M(\gamma) - K(\gamma) M'(\gamma)}{2\sqrt{M(\gamma)}}, \\
        G'(\gamma) &= \frac{2J'(\gamma)\sqrt{M(\gamma)} + 2K'(\gamma) M(\gamma) + K(\gamma) M'(\gamma)}{2\sqrt{M(\gamma)}}.
    \end{align*}
    We proceed the calculation with $\kappa := \frac{L_{xy}}{\mu_{xy}}\ge 1$. 
    \begin{align}
        \frac{\gamma^3}{2\mu_{xy}^6} F'(\gamma)\sqrt{M(\gamma)} &= \frac{\gamma^3}{\mu_{xy}^6}\bigopen{\frac{1}{2}J'(\gamma)\sqrt{M(\gamma)} - \frac{1}{2}K'(\gamma) M(\gamma) - \frac{1}{4}K(\gamma) M'(\gamma)} \nonumber\\
        &= \bigopen{-(\gamma-1)(\kappa^{4}+1)+ 2\kappa^2}\sqrt{(2\gamma-1)^2 \kappa^4 - 2 (2\gamma^2 - \gamma - 1) \kappa^2 + (\gamma-1)^2} \nonumber\\
        & \quad - \frac{1}{2}\bigopen{-2\kappa^2 + (\gamma-2)}\bigopen{(2\gamma-1)^2 \kappa^4 - 2 (2\gamma^2 - \gamma - 1) \kappa^2  + (\gamma-1)^2} \nonumber\\
        & \quad - \frac{\gamma}{2} \bigopen{\kappa^2 - (\gamma-1)} \bigopen{(4\gamma-2)\kappa^4 - (4\gamma-1) \kappa^2 + (\gamma-1)} \nonumber\\
        &\begin{aligned}
            &= \bigopen{-(\kappa^{4}+1)\gamma + (\kappa^2+ 1)^2}\sqrt{(2\kappa^2-1)^2 \gamma^2 -(4\kappa^4 -2\kappa^2 +2)\gamma + (\kappa^2+1)^2}\\
        & \quad + (2\kappa^6 + \kappa^4 - 2\kappa^2 + 1)\gamma^2 - (\kappa^2+1)(3\kappa^4 - \kappa^2 + 2) \gamma + (\kappa^2 + 1)^3.
        \end{aligned} \label{eq:jaewook}
    \end{align}
    We show that this is indeed nonnegative for $\gamma\ge1$ and $\kappa\ge 1$. To this end, note that,
    \begin{align*}
        &(2\kappa^6 + \kappa^4 - 2\kappa^2 + 1)\gamma^2 - (\kappa^2+1)(3\kappa^4 - \kappa^2 + 2) \gamma + (\kappa^2 + 1)^3 \\
        &= 2\gamma^2 + (\kappa^2+1)\bigset{(2\kappa^4 - \kappa^2 - 1) \gamma^2 - (3\kappa^4 - \kappa^2 + 2) \gamma + (\kappa^2 + 1)^2} \\
        &\ge 2\bigset{(2\kappa^4 - \kappa^2) \gamma^2 - (3\kappa^4 - \kappa^2 + 2) \gamma + (\kappa^2 + 1)^2} \\
        &= 2\bigset{(2\kappa^4 - \kappa^2)\bigopen{\gamma - \frac{3\kappa^4 - \kappa^2 + 2}{4\kappa^4 - 2\kappa^2}}^2 -\frac{(3\kappa^4 - \kappa^2 + 2)^2-4(\kappa^2 + 1)^2(2\kappa^4 - \kappa^2)}{4(2\kappa^4 - \kappa^2)}} \\
        &= 2\bigset{(2\kappa^4 - \kappa^2)\bigopen{\gamma - \frac{3\kappa^4 - \kappa^2 + 2}{4\kappa^4 - 2\kappa^2}}^2 -\frac{\kappa^8 - 18\kappa^6 + 13\kappa^4 + 4}{4(2\kappa^4 - \kappa^2)}}.
    \end{align*}
    Note that $2\kappa^4 - \kappa^2>0$. Also,
    (i) if $1 \le \kappa < 4$ then
    \begin{align*}
        \kappa^8 - 18\kappa^6 + 13\kappa^4 + 4 &=(\kappa-1) (\kappa+1) (\kappa^3 - 5 \kappa^2 + 4 \kappa -2) (\kappa^3 + 5 \kappa^2 + 4 \kappa +2) \\
        &=(\kappa-1) (\kappa+1) \bigopen{(\kappa-4)(\kappa-1)\kappa-2} (\kappa^3 + 5 \kappa^2 + 4 \kappa +2) < 0;
    \end{align*}
    (ii) if $\kappa\ge 4$ then $\frac{3\kappa^4 - \kappa^2 + 2}{4\kappa^4 - 2\kappa^2} < 1$ and
    \begin{align*}
         (2\kappa^4 - \kappa^2)\cdot 1^2 - (3\kappa^4 - \kappa^2 + 2)\cdot 1 + (\kappa^2 + 1)^2 = 2(\kappa^2 - 1) + 1 > 0.
    \end{align*}
    By (i) and (ii), 
    \begin{align}
        \label{eq:hellohellohello}
        \begin{aligned}
            &(2\kappa^6 + \kappa^4 - 2\kappa^2 + 1)\gamma^2 - (\kappa^2+1)(3\kappa^4 - \kappa^2 + 2) \gamma + (\kappa^2 + 1)^3 \\
            &\ge 2\bigset{(2\kappa^4 - \kappa^2)\bigopen{\gamma - \frac{3\kappa^4 - \kappa^2 + 2}{4\kappa^4 - 2\kappa^2}}^2 -\frac{\kappa^8 - 18\kappa^6 + 13\kappa^4 + 4}{4(2\kappa^4 - \kappa^2)}} > 0.
        \end{aligned}
    \end{align}
    So if $1\le \gamma < \frac{(\kappa^2+1)^2}{\kappa^4 + 1}$, $\bigopen{-(\kappa^{4}+1)\gamma + (\kappa^2+ 1)^2} \ge 0$, which proves the non-negativity of \cref{eq:jaewook}.
    In addition, observe that the following inequality holds for all $\gamma\ge 1$: 
    \begin{align}
        &\bigset{(2\kappa^6 + \kappa^4 - 2\kappa^2 + 1)\gamma^2 - (\kappa^2+1)(3\kappa^4 - \kappa^2 + 2) \gamma + (\kappa^2 + 1)^3 }^2 \nonumber\\
        &\quad -\bigopen{(\kappa^{4}+1)\gamma - (\kappa^2+ 1)^2}^2\bigset{(2\kappa^2-1)^2 \gamma^2 -(4\kappa^4 -2\kappa^2 +2)\gamma + (\kappa^2+1)^2} \nonumber\\
        &= 8\kappa^6(\kappa^2-1)^2 (\gamma^4 -\gamma^3) \ge 0. \label{eq:hellohellohellohello}
    \end{align}
    This also proves the non-negativity of \cref{eq:jaewook} in the case of $\gamma \ge \frac{(\kappa^2+1)^2}{\kappa^4 + 1}$. 
    As a result, we just showed that $F'(\gamma)\ge 0$ for $\gamma\ge1$ and $\kappa\ge 1$. We now turn to prove $G'(\gamma)\le 0$.
    \begin{align*}
        \frac{\gamma^3}{2\mu_{xy}^6} G'(\gamma)\sqrt{M(\gamma)} &= \frac{\gamma^3}{\mu_{xy}^6}\bigopen{\frac{1}{2}J'(\gamma)\sqrt{M(\gamma)} + \frac{1}{2}K'(\gamma) M(\gamma) + \frac{1}{4}K(\gamma) M'(\gamma)} \\
        &= \bigopen{-(\gamma-1)(\kappa^{4}+1)+ 2\kappa^2}\sqrt{(2\gamma-1)^2 \kappa^4 - 2 (2\gamma^2 - \gamma - 1) \kappa^2 + (\gamma-1)^2} \\
        & \quad + \frac{1}{2}\bigopen{-2\kappa^2 + (\gamma-2)}\bigopen{(2\gamma-1)^2 \kappa^4 - 2 (2\gamma^2 - \gamma - 1) \kappa^2  + (\gamma-1)^2} \\
        & \quad + \frac{\gamma}{2} \bigopen{\kappa^2 - (\gamma-1)} \bigopen{(4\gamma-2)\kappa^4 - (4\gamma-1) \kappa^2 + (\gamma-1)} \\
        &= \bigopen{-(\kappa^{4}+1)\gamma + (\kappa^2+ 1)^2}\sqrt{(2\kappa^2-1)^2 \gamma^2 -(4\kappa^4 -2\kappa^2 +2)\gamma + (\kappa^2+1)^2}\\
        & \quad - \bigset{(2\kappa^6 + \kappa^4 - 2\kappa^2 + 1)\gamma^2 - (\kappa^2+1)(3\kappa^4 - \kappa^2 + 2) \gamma + (\kappa^2 + 1)^3}.
    \end{align*}
    We show that this is nonpositive for $\gamma\ge 1$ and $\kappa\ge 1$. To this end, note that again from \cref{eq:hellohellohello},
    \begin{align*}
        (2\kappa^6 + \kappa^4 - 2\kappa^2 + 1)\gamma^2 - (\kappa^2+1)(3\kappa^4 - \kappa^2 + 2) \gamma + (\kappa^2 + 1)^3 \ge 0.
    \end{align*}
    Also, if $1\le \gamma <\frac{(\kappa^2+1)^2}{\kappa^4 + 1}$, \cref{eq:hellohellohellohello} still holds. On the other hand, if  $\gamma \ge \frac{(\kappa^2+1)^2}{\kappa^4 + 1}$, $\bigopen{-(\kappa^{4}+1)\gamma + (\kappa^2+ 1)^2} \le 0$. These indeed prove that $G'(\gamma)\le 0$ for $\gamma\ge1$ and $\kappa\ge 1$.

    Now we conclude the proof by remarking that $h(\gamma) = F(\gamma)$ if $\gamma\in \bigclosed{1,1+\frac{L_{xy}^2}{\mu_{xy}^2}}$ and $h(\gamma) = G(\gamma)$ if $\gamma\in \left[1+\frac{L_{xy}^2}{\mu_{xy}^2},\infty\right)$.

    
\end{proof}

%% file: secf.tex
\section{\texorpdfstring{Proof of \cref{prop:ogdlb}}{Proof of Proposition A.1}} \label{sec:f}

Here we prove \cref{prop:ogdlb} of \cref{sec:a}, restated below for the sake of readability.

\propogdlb*

\begin{proof}
    Recall that \textbf{OGD} takes updates of the form:
    \begin{align*}
        \vx_{k+1} &= \vx_{k} - 2 \alpha \nabla_{\vx} f(\vx_{k}, \vy_{k}) + \alpha \nabla_{\vx} f(\vx_{k-1}, \vy_{k-1}), \\
        \vy_{k+1} &= \vy_{k} + 2 \beta \nabla_{\vy} f(\vx_{k}, \vy_{k}) - \beta \nabla_{\vy} f(\vx_{k-1}, \vy_{k-1}).
    \end{align*}

    We use the same worst-case function as in \cref{thm:simgdalb}:
    \begin{align*}
        f(\vx, \vy) &= 
        \frac{1}{2} \begin{bmatrix}
            {x} \\ {s} \\ {t} \\ {y} \\ {u} \\ {v}
        \end{bmatrix}^{\top}
        \begin{bmatrix}
            \mu_x & 0 & 0 & L_{xy} & 0 & 0 \\
            0 & \mu_x & 0 & 0 & 0 & 0 \\
            0 & 0 & L_x & 0 & 0 & 0 \\
            L_{xy} & 0 & 0 & -\mu_y & 0 & 0 \\
            0 & 0 & 0 & 0 & -\mu_y & 0 \\
            0 & 0 & 0 & 0 & 0& -L_y 
        \end{bmatrix}
        \begin{bmatrix}
            {x} \\ {s} \\ {t} \\ {y} \\ {u} \\ {v}
        \end{bmatrix}, \\
    \end{align*}
    where ${\vx = (x, s, t)}$ and ${\vy = (y, u, v)}$. 
    It can be easily checked that $f$ is a quadratic function ({\it i.e.,} Hessian is constant) such that $f \in \gF(\mu_x, \mu_y, L_x, L_y, L_{xy})$ and $\vx_\star = \vy_\star = \bm{0} \in \R^3$.

    Let us define
    \begin{align*}
        \mA &= 
        \begin{bmatrix}
            \mu_x & 0 & 0 \\
            0 & \mu_x & 0 \\
            0 & 0 & L_x
        \end{bmatrix}, \quad
        \mB = 
        \begin{bmatrix}
            L_{xy} & 0 & 0 \\
            0 & 0 & 0 \\
            0 & 0 & 0
        \end{bmatrix}, \quad
        \mC =
        \begin{bmatrix}
            \mu_y & 0 & 0 \\
            0 & \mu_y & 0 \\
            0 & 0 & L_y
        \end{bmatrix}.
    \end{align*}

    We first observe that the $k$-th step of \textbf{OGD} satisfies
    \begin{align*}
        \begin{bmatrix}
            \vx_{k+1} \\ \vy_{k+1} \\ \vx_{k} \\ \vy_{k}
        \end{bmatrix}
        &= \begin{bmatrix}
            \mI - 2 \alpha \mA & - 2 \alpha \mB & \alpha \mA & \alpha \mB \\ 
            2 \beta \mB^{\top} & \mI - 2 \beta \mC & - \beta \mB^{\top} & \beta \mC \\
            \mI & \bm{0} & \bm{0} & \bm{0} \\
            \bm{0} & \mI & \bm{0} & \bm{0} \\
        \end{bmatrix}
        \begin{bmatrix}
            \vx_{k} \\ \vy_{k} \\ \vx_{k-1} \\ \vy_{k-1}
        \end{bmatrix}.
    \end{align*}

    Then the coordinate-wise updates on the $k$-th step of \textbf{OGD} must be
    \begin{align}
        \begin{bmatrix}
            x_{k+1} \\ y_{k+1} \\ x_k \\ y_k
        \end{bmatrix} &= \underbrace{\begin{bmatrix}
            1 - 2 \alpha \mu_x & - 2 \alpha L_{xy} & \alpha \mu_x & \alpha L_{xy} \\ 
            2 \beta  L_{xy} & 1 - 2 \beta \mu_y & - \beta L_{xy} & \beta \mu_y \\
            1 & 0 & 0 & 0 \\
            0 & 1 & 0 & 0 \\
        \end{bmatrix}}_{\triangleq \mP}
        \begin{bmatrix}
            x_k \\ y_k \\ x_{k-1} \\ y_{k-1}
        \end{bmatrix}, \label{eq:x_k_y_k_ogd}\\
        s_{k+1} &= (1 - 2 \alpha \mu_x) s_k + \alpha \mu_x s_{k-1}, \label{eq:s_k_ogd} \\
        t_{k+1} &= (1  - 2 \alpha L_x) t_k + \alpha L_x t_{k-1}, \label{eq:t_k_ogd} \\
        u_{k+1} &= (1 - 2 \beta \mu_y) u_k + \beta \mu_y u_{k-1}, \label{eq:u_k_ogd} \\
        v_{k+1} &= (1 - 2 \beta L_y) v_k + \beta L_y v_{k-1}. \label{eq:v_k_ogd} 
    \end{align}

    First, observing that the quadratic $w^2 - (1 - 2c)w - c = 0$ has (real) roots given by
    \begin{align*}
        w &= \frac{(1 - 2c) \pm \sqrt{(1 - 2c)^2 + 4 c}}{2},
    \end{align*}
    a recurrence relation of the form $w_{k+1} = (1 - 2c) w_k + c w_{k-1}$ converges if and only if
    \begin{align*}
        r = \frac{|1 - 2c| + \sqrt{(1 - 2c)^2 + 4 c}}{2} < 1,
    \end{align*}
    which is again equivalent to $0 < c < \frac{2}{3}$.

    Moreover, if $0 < c \le \frac{1}{2}$, then we have
    \begin{align*}
        \frac{1}{1 - r} &= \frac{1}{1 - \frac{1 - 2c + \sqrt{(1 - 2c)^2 + 4 c}}{2}} \\
        &= \frac{2}{1 + 2c - \sqrt{1 + 4 c^2}} = \frac{1 + 2c + \sqrt{1 + 4 c^2}}{2c} \ge \frac{1}{2c} = \Omega \bigopen{\frac{1}{c}},
    \end{align*}
    while if $\frac{1}{2} < c < \frac{2}{3}$, then we have
    \begin{align*}
        \frac{1}{1 - r} &= \frac{1}{1 - \frac{2c - 1 + \sqrt{(1 - 2c)^2 + 4 c}}{2}} \\
        &= \frac{2}{3 - 2c - \sqrt{1 + 4 c^2}} \ge 2 + \sqrt{2} \ge \bigopen{1 + \frac{1}{\sqrt{2}}} \cdot \frac{1}{c} = \Omega \bigopen{\frac{1}{c}}
    \end{align*}
    which is because $\frac{2}{3 - 2c - \sqrt{1 + 4 c^2}}$ is an increasing function in $[\frac{1}{2}, \frac{2}{3})$.
    
    For the convergence of iterations \eqref{eq:t_k_ogd}~and~\eqref{eq:v_k_ogd}, the step sizes $\alpha$ and $\beta$ are required to satisfy
    \begin{align}
        \label{eq:stepsize_basic_3}
        0 < \alpha L_x < \frac{2}{3} \quad \text{and} \quad 0 < \beta L_y < \frac{2}{3},
    \end{align}
    by setting $c = \alpha L_x$ and/or $c = \beta L_y$.

    Also, to guarantee $\bignorm{\vx_K}^2 + \bignorm{\vy_K}^2 < \epsilon$, we need from \eqref{eq:s_k_ogd}~and~\eqref{eq:u_k_ogd} that $s_K^2 < \gO(\epsilon)$ and $u_K^2 < \gO(\epsilon)$, respectively.

    
    The two necessary conditions $s_K^2 < \gO(\epsilon)$ and $u_K^2 < \gO(\epsilon)$ require an iteration number of at least:
    \begin{align}
        K &= \Omega\bigopen{\bigopen{{\frac{1}{\alpha \mu_x}} + {\frac{1}{\beta \mu_y}}} \cdot \log\frac{1}{\epsilon}},
        \label{eq:iteration_complexity_base_3}
    \end{align}
    by setting $c = \alpha \mu_x$ and/or $c = \beta \mu_y$.
    
    Note that \eqref{eq:stepsize_basic_3} automatically yields
    \begin{align}
        \frac{1}{\alpha \mu_x} + \frac{1}{\beta \mu_y} = \Omega (\kappa_x + \kappa_y).
        \label{eq:hola}
    \end{align}

    Now, in order to ensure convergence of iteration \eqref{eq:x_k_y_k_ogd}, we need the matrix $\mP$ to have a spectral radius smaller than one.
    Hence it suffices to show that $\rho(\mP) < 1$ implies $\frac{1}{\alpha \mu_x} + \frac{1}{\beta \mu_y} = \Omega(\kappa_{xy})$.

    Suppose that $\lambda$ is an eigenvalue of $\mP$. Then we must have    
    \begin{align*}
        \det (\lambda \mI - \mP) &= 
        \begin{vmatrix}
            (1 - \lambda) - 2 \alpha \mu_x & - 2 \alpha L_{xy} & \alpha \mu_x & \alpha L_{xy} \\
            2 \beta L_{xy} & (1 - \lambda) - 2 \beta \mu_y & - \beta L_{xy} & \beta \mu_y \\
            1 & 0 & - \lambda & 0 \\
            0 & 1 & 0 & - \lambda
        \end{vmatrix} = 0.
    \end{align*}
    First, we observe that $\lambda \ne 0$, since if we plug in $\lambda = 0$ we have
    \begin{align*}
        \det (\lambda \mI - \mP) &= \det (\mP) = \alpha \beta (\mu_x \mu_y + L_{xy}^2 ) > 0.
    \end{align*}
    Therefore we can compute
    \begin{align*}
        & \begin{vmatrix}
            (1 - \lambda) - 2 \alpha \mu_x & - 2 \alpha L_{xy} & \alpha \mu_x & \alpha L_{xy} \\
            2 \beta L_{xy} & (1 - \lambda) - 2 \beta \mu_y & - \beta L_{xy} & \beta \mu_y \\
            1 & 0 & - \lambda & 0 \\
            0 & 1 & 0 & - \lambda
        \end{vmatrix} \\
        &= \frac{1}{\lambda^2} \begin{vmatrix}
            \lambda (1 - \lambda) - 2 \lambda \alpha \mu_x & - 2 \lambda \alpha L_{xy} & \alpha \mu_x & \alpha L_{xy} \\
            2 \lambda \beta L_{xy} & \lambda (1 - \lambda) - 2 \lambda \beta \mu_y & - \beta L_{xy} & \beta \mu_y \\
            \lambda & 0 & - \lambda & 0 \\
            0 & \lambda & 0 & - \lambda
        \end{vmatrix} \\
        &= \frac{1}{\lambda^2} \begin{vmatrix}
            \lambda (1 - \lambda) - (2 \lambda - 1) \alpha \mu_x & - (2 \lambda - 1) \alpha L_{xy} & \alpha \mu_x & \alpha L_{xy} \\
            (2 \lambda - 1) \beta L_{xy} & \lambda (1 - \lambda) - (2 \lambda - 1) \beta \mu_y & - \beta L_{xy} & \beta \mu_y \\
            0 & 0 & - \lambda & 0 \\
            0 & 0 & 0 & - \lambda
        \end{vmatrix} \\
        &= \begin{vmatrix}
            \lambda (1 - \lambda) - (2 \lambda - 1) \alpha \mu_x & - (2 \lambda - 1) \alpha L_{xy} \\
            (2 \lambda - 1) \beta L_{xy} & \lambda (1 - \lambda) - (2 \lambda - 1) \beta \mu_y
        \end{vmatrix} \\
        &= \bigopen{\lambda (1 - \lambda) - (2 \lambda - 1) \alpha \mu_x} \bigopen{\lambda (1 - \lambda) - (2 \lambda - 1) \beta \mu_y} + (2 \lambda - 1)^2 \alpha \beta L_{xy}^2.
    \end{align*}
    If we substitute $a = \alpha \mu_x$ and $b = \beta \mu_y$, then $\det (\lambda \mI - \mP) = 0$ is equivalent to
    \begin{align}
        \bigopen{- \lambda^2 + (1 - 2 a) \lambda + a} \bigopen{- \lambda^2 + (1 - 2 b) \lambda + b} + (2 \lambda - 1)^2 ab \kappa_{xy}^2 &= 0,
        \label{eq:jaeyoung}
    \end{align}
    where we note that $\alpha \beta L_{xy}^2 = ab \kappa_{xy}^2$.

    Hence we have a quartic equation of the form $\lambda^4 - p \lambda^3 + q \lambda^2 - r \lambda + \ell$ with coefficients given by
    \begin{align}
    \begin{aligned}
        p &= 2 - 2 (a+b), \\
        q &= 1 - 3a - 3b + 4ab (\kappa_{xy}^2 + 1), \\
        r &= - a - b + 4ab (\kappa_{xy}^2 + 1), \\
        \ell &= ab (\kappa_{xy}^2 + 1).
    \end{aligned} \label{eq:dongkuk}
    \end{align}
    Note that we obviously have $p, q, r, \ell > 0$.
    
    There exists a well-known characterization of quartic polynomials having roots with absolute values less than one.

    \begin{proposition}[\citet{grove2004periodicities}, Theorem 1.5]
        \label{prop:routhhurwitztwo}
        Consider a quartic polynomial $x^4 + a_3x^3 + a_2x^2 + a_1x + a_0$, where $a_0$, $a_1$, $a_2$, $a_3$ are real numbers. Then a necessary and sufficient condition that all roots of the polynomial are contained in the open disk $|x|<1$ is 
        \begin{align}
        \begin{aligned}
            |a_1 + a_3| < 1 + a_0 + a_2, \quad |a_1 - a_3| &< 2(1 - a_0), \quad a_2 - 3a_0 < 3, \\
            a_0 + a_2 + a_0^2 + a_1^2 + a_0^2 a_2 + a_0 a_3^2 < 1 & + 2 a_0 a_2 + a_1 a_3 + a_0 a_1 a_3 + a_0^3.
        \end{aligned} \label{eq:degfour}
        \end{align}
    \end{proposition}
    
    Also, the following corollary suggests that the coefficients are all bounded (by constants) for such cases.
    
    \begin{corollary}
        \label{cor:routhhurwitztwobd}
        For coefficients $a_0, a_1, a_2, a_3$ satisfying \eqref{eq:degfour}, we have $|a_3| < 6$, $|a_2| < 6$, $|a_1| < 6$, $|a_0| < 1$.
    \end{corollary}
    
    \begin{proof}
        From the first three conditions, we can observe that
        \begin{align*}
            0 < 1 + a_0 + a_2, \quad 0 &< 2(1 - a_0), \quad a_2 - 3a_0 < 3.
        \end{align*}
        Hence $(a_0, a_2)$ must be inside a triangle with endpoints $(-1, 0)$, $(1, -2)$, $(1, 6)$, which implies $|a_2| < 6$, $|a_0| < 1$.
        
        Using this, we can also observe from the first two conditions that
        \begin{align*}
            |a_1 + a_3| < 1 + a_0 + a_2 < 8, \quad |a_1 - a_3| &< 2(1 - a_0) < 4.
        \end{align*}
        Hence $(a_1, a_3)$ must be inside a rectangle with endpoints $(2, 6)$, \!$(6, 2)$, \!$(-2, -6)$, \!$(-6, -2)$, implying $|a_3| < 6$, $|a_1| < 6$.
    \end{proof}

    By \cref{cor:routhhurwitztwobd}, we can observe that a necessary condition for $\rho (\mP) < 1$ is that all coefficients in \eqref{eq:dongkuk} are of order $\gO(1)$.
    In particular, this implies $ab\kappa_{xy}^2 = \alpha \beta L_{xy}^2 = \gO (1)$ in order to assure convergence, which concludes that
    \begin{align}
        \frac{1}{\alpha \mu_x} + \frac{1}{\beta \mu_y} &\ge \frac{2}{\sqrt{\alpha \beta \mu_x \mu_y}} = \frac{2 \kappa_{xy}}{\sqrt{\alpha \beta L_{xy}^2}} = \Omega (\kappa_{xy}). 
        \label{eq:rice2}
    \end{align}
    Combining \eqref{eq:hola} and \eqref{eq:rice2}, we have
    \begin{align*}
        \frac{1}{\alpha \mu_x} + \frac{1}{\beta \mu_y} &= \Omega (\kappa_x + \kappa_y + \kappa_{xy})
    \end{align*}
    and therefore from \eqref{eq:iteration_complexity_base_3} we can show a lower bound of
    \begin{equation*}\Omega\bigopen{\bigopen{\kappa_x + \kappa_y + \kappa_{xy}} \cdot \log\frac{1}{\epsilon}}.
    \end{equation*}
\end{proof}

%% file: secg1.tex
\section{Details of Experiments} 
\label{sec:g}

\subsection{SCSC Quadratic Game (1): Small-scale}
\label{sec:experiment_scsc}

We run experiments on the following SCSC quadratic problem:
\begin{align*}
    f(\vx,\vy) = \frac{1}{2}\vx^\top \mU^\top \begin{bmatrix}
        \mu_x & 0 & 0 \\ 0 & L_x & 0 \\ 0 & 0 & L_x
    \end{bmatrix} \mU \vx + \vx^\top \mU^\top \begin{bmatrix}
        L_{xy} & 0 & 0 \\ 0 & L_{xy} & 0 \\ 0 & 0 & \mu_{xy}
    \end{bmatrix}  \mV \vy + \frac{1}{2}\vy^\top \mV^\top \begin{bmatrix}
        \mu_y & 0 & 0 \\ 0 & L_y & 0 \\ 0 & 0 & L_y
    \end{bmatrix} \mV \vy,
\end{align*}
where $\mU\in \R^{3\times 3}$ and $\mV\in \R^{3\times 3}$ are random orthogonal matrices (generated with QR-decompositions of random Gaussian matrices). For a clear demonstration of optimization trajectories in \cref{fig:scsc}, we set $\mU = \mV = \mI_{3\times 3}$. For the problem parameters, we use $L_x = L_y = L_{xy} = 1$ and $\mu_x = \mu_y = \mu_{xy} = 0.2$. We run each algorithm until it reaches  $\norm{\vz_k}^2 < \eps = 10^{-50}$.

\paragraph{Implementation of EG.} 
We use a general form of \textbf{EG} as follows:
\begin{gather*}
    \vx_{k+\frac{1}{2}} = \vx_{k} - \alpha_0 \nabla_{\vx} f(\vx_{k}, \vy_{k}),  \quad
    \vy_{k+\frac{1}{2}} = \vy_{k} + \beta_0 \nabla_{\vy} f(\vx_{k}, \vy_{k}), \\
    \vx_{k+1} = \vx_{k} - \alpha_1 \nabla_{\vx} f(\vx_{k+\frac{1}{2}}, \vy_{k+\frac{1}{2}}), \quad
    \vy_{k+1} = \vy_{k} + \beta_1 \nabla_{\vy} f(\vx_{k+\frac{1}{2}}, \vy_{k+\frac{1}{2}}),
\end{gather*}
where the step sizes at explorations step ($k \rightarrow k+1/2$) and at update step ($k+1/2 \rightarrow k+1$) can differ.

\paragraph{Implementation of OGD.} 
Also, we use a general form of \textbf{OGD} as follows:
\begin{align*}
    \vx_{k+1} &= \vx_{k} - \alpha_0 \nabla_{\vx} f(\vx_{k}, \vy_{k}) + \alpha_1 \nabla_{\vx} f(\vx_{k-1}, \vy_{k-1}), \\
    \vy_{k+1} &= \vy_{k} + \beta_0 \nabla_{\vy} f(\vx_{k}, \vy_{k}) - \beta_1 \nabla_{\vy} f(\vx_{k-1}, \vy_{k-1}),
\end{align*}

\paragraph{Parameter tuning.} We tuned step sizes and other parameters (like $\gamma$ and $\delta$ of \alexgda{}) by grid search. Since this is a quadratic problem (where the local convergence analysis directly applies), following the analysis by \citet{zhang22near}, we choose $\mu/L^2$-scale step size for \simgda{} and $1/L$-scale step size for the other algorithms ($L=\max\bigset{L_x, L_y, L_{xy}}$, $\mu=\min\bigset{\mu_x, \mu_y, \mu_{xy}}$).
To be more specific,
\begin{itemize}
    \item \simgda{}: (step size) $=\frac{\mu}{CL^2}$, where $C \in \{0.5, 0.51, 0.52, \cdots, 2.99, 3\}$. (If we apply $\frac{1}{L}$-scale step size, it diverges.)
    \item \altgda{}: (step size) $=\frac{1}{CL}$, where $C \in \{1, 1.01, 1.02, \cdots, 3.99, 4\}$.
    \item \textbf{EG}, \textbf{OGD}: $\alpha_0=\beta_0=\frac{1}{C_0 L}$ and $\alpha_1=\beta_1=\frac{1}{C_1 L}$, where $C_0, C_1 \in \{0.5, 0.51, 0.52, \cdots, 3.99, 4\}$
    \item \alexgda{}: (step size) $=\frac{1}{CL}$, where $C \in \{1, 1.1, 1.2, \cdots, 1.9, 2\}$, and $\gamma, \delta \in \{1.1, 1.2, 1.3, \cdots, 3.9, 4\}$
\end{itemize}

%% file: secg2.tex
\subsection{SCSC Quadratic Game (2): Higher Dimension, Extensive Comparisons}
\label{sec:experiment_scsc_large}

We generate the SCSC quadratic problems $f: \R^{d_x} \times \R^{d_y} \rightarrow \R$ as
\begin{align*}
    f(\vx,\vy) = \frac{1}{2} \vx^\top \mU^\top \mA \mU \vx + \vx^\top \mU^\top \mB  \mV \vy + \frac{1}{2} \vy^\top \mV^\top \mC \mV \vy,
\end{align*}
where we randomly sample the matrices $\mA\in \R^{d_x \times d_x}$, $\mB\in \R^{d_x \times d_y}$, $\mC\in \R^{d_y \times d_y}$, $\mU\in \R^{d_x \times d_x}$, and $\mV\in \R^{d_y \times d_y}$:
\begin{align*}
    \mA = \diag(a_1, \dots, a_{d_x}), && a_1 &= \mu_x, && a_2 = L_x, && a_i \sim \Uniform(\mu_x, L_x), \quad (i = 3, \dots, d_x) \\
    \mB = \diag(b_1, \dots, b_{\min\{d_x, d_y\}}), && b_1 &= \mu_{xy}, && b_2 = L_{xy}, && b_i \sim \Uniform(\mu_{xy}, L_{xy}), \quad (i = 3, \dots, \min\{d_x, d_y\}) \\
    \mC = \diag(c_1, \dots, c_{d_y}), && c_1 &= \mu_y, && c_2 = L_y, && c_i \sim \Uniform(\mu_y, L_y), \quad (i = 3, \dots, d_y)
\end{align*}
while $\mU \in \R^{d_x \times d_x}$ and $\mV\in \R^{d_y \times d_y}$ are random orthogonal matrices.

For each combination of $\mu$ ($=\mu_x=\mu_y$), $\mu_{xy}$, $L$ ($=L_x=L_y$), and $L_{xy}$, we test 3 random initialization points $(\vx_0, \vy_0)$ and 10 random instances of $f(\vx, \vy)$.

\paragraph{Algorithms.}
We follow a standard implementation of heavy-ball momentum as PyTorch’s implementation. 
We adopt \citet{azizian20} for \textbf{EG} with Momentum, \citet{ramirez2023omega} for \textbf{OGD} with Momentum (so-called OmegaM), and \citet{zhang2020convergence} for the alternating counterparts of \textbf{EG} and \textbf{OGD} (Alt-EG and Alt-OG, respectively).

Our implementation of \alexgda{} with momentum (Alex+M) is as in \cref{alg:alexgda_momentum}.
\begin{algorithm}[ht]
    \caption{\alexgda{} with Momentum}
    \label{alg:alexgda_momentum}
    \begin{algorithmic}
        \STATE {\bfseries Input:} Number of epochs $K$, step sizes $\alpha, \beta > 0$, hyperparameters $\gamma, \delta \ge 0$, momentum parameters $m_x,m_y \in \R$
        \STATE {\bfseries Initialize:} $(\vx_0, \vy_0) \in \R^{d_x} \times \R^{d_y}$ and $\tilde{\vy}_0 = \vy_0 \in \R^{d_y}$
        \FOR{$k = 0, \dots, K-1$}
            \STATE $\vv^{x}_{k+1} = m_x \vv^{x}_{k} + \nabla_{x} f(\vx_{k}, \blue{\tilde{\vy}_{k}})$
            \STATE $\vx_{k+1} = \vx_{k} - \alpha \vv^{x}_{k+1}$
            \STATE $\tilde{\vx}_{k+1} = \vx_{k} - \gamma\alpha \vv^{x}_{k+1}$
            \STATE $\vv^{y}_{k+1} = m_y \vv^{y}_{k} + \nabla_{y} f(\blue{\tilde{\vx}_{k+1}}, \vy_{k})$
            \STATE $\vy_{k+1} = \vy_{k} + \beta \vv^{y}_{k+1}$
            \STATE $\tilde{\vy}_{k+1} = \vy_{k} + \delta\beta \vv^{y}_{k+1}$
        \ENDFOR
        \STATE {\bfseries Output:} $(\vx_K, \vy_K) \in \R^{d_x} \times \R^{d_y}$
    \end{algorithmic}
\end{algorithm}

\paragraph{Computing gradient complexity.}
For most algorithms, the number of gradient computations equals the number of iterations. However, \textbf{EG} and its alternating counterpart (Alt-EG) take multiple gradient computations per iteration. For \textbf{EG} (with simultaneous updates), it takes two gradient computations per iteration. For Alt-EG, according to the implementation by \citet{zhang2020convergence}, it takes three gradient computations per iteration. Hence, we computed the gradient complexity by multiplying the number of iterations and the amount of gradient computation per iteration.

\paragraph{Parameter Tuning.}
Likewise in \cref{sec:experiment_scsc}, we choose $\tfrac{\mu}{\max\{L^2, L_{xy}^2\}}$-scale step size for \simgda{} and $\tfrac{1}{\max\{L, L_{xy}\}}$-scale step size for the other algorithms. 
To be specific,
\begin{itemize}[leftmargin=15pt]
    \item \simgda{} : (step size)$=\frac{C\mu}{\max\{L^2, L_{xy}^2\}}$ where $C\in \{0.1, 0.2, \ldots, 1.5\}$,
    \item The other algorithms (including \simgda{} with momentum): (step size)$=\frac{C}{\max\{L, L_{xy}\}}$ where $C\in \{0.1, 0.2, \ldots, 1.5\}$.
\end{itemize}
We tune the momentum parameters $m_x, m_y \in \{-0.99, -0.95, -0.9, -0.8, -0.7, \ldots, 0.9, 0.95, 0.99\}$. 
Note that we allow the negative momentum as per the work by \citet{gidel2019negative}.
We tune $\gamma$ and $\delta$ for \alexgda{} as $\gamma, \delta \in \{0.5, 0.6, 0.7, \ldots, 3.0\}$.
For the momentum variant of \alexgda{} (\cref{alg:alexgda_momentum}), we slightly reduced the range of search as $\gamma, \delta \in \{1.0, 1.1, 1.2, \ldots, 3.0\}$.

%% file: secg4.tex
\subsection{Generative Adversarial Networks: WGAN-GP}
\label{sec:experiment_wgan}

We name the combination of Adam \citep{kingma2015adam} and (the stochastic version of) \textbf{\green{Sim-}}/\textbf{\red{Alt-}}/\alexgda{} as Sim-/Alt-/Alex-Adam, respectively. 
In \cref{lst:alex_adam}, we provide a brief Python code based on PyTorch \citep{paszke2019pytorch} for GAN training with Alex-Adam. 
The full code base can be found at \href{https://github.com/HanseulJo/Alex-GDA/tree/main/gan}{\texttt{github.com/HanseulJo/Alex-GDA/tree/main/gan}}. 
In the code, we use the main models \verb|netD| and \verb|netG| (for which the weights correspond to $\vx$ and $\vy$, respectively) and the auxiliary models \verb|netD_| and \verb|netG_|. 
The auxiliary models are for describing the `tilde' variables $\tilde{\vx}$ and $\tilde{\vy
}$, i.e., the results of the inter-/extrapolation steps.

\paragraph{Learning Rates.}
For MNIST \citep{deng2012mnist}, we tuned the step sizes for Alex-Adam ($\{10^{-4}, 3\times 10^{-4}\}$ for both generator and discriminator) and applied the best step size ($3\times 10^{-4}$ for generator, $10^{-4}$ for discriminator) for the other algorithms.

For CIFAR-10 \citep{krizhevsky2009learning}, we tuned the step sizes for algorithms ($\{10^{-4}, 3\times 10^{-4}\}$ for both generator and discriminator). The best step sizes were ($10^{-4}$ for both generator and discriminator) for Sim-Adam and ($3\times 10^{-4}$ for both generator and discriminator) for Alt-Adam and Alex-Adam.

For LSUN-Bedroom $64\times 64$ dataset \citep{yu15lsun}, we fixed the step size as ($10^{-4}$ for generator, $3\times 10^{-4}$ for discriminator) following \citet{heusel2017gans}.

\newpage
\begin{lstlisting}[
    label={lst:alex_adam},
    language=Python,
    caption={PyTorch-based Python code for GAN Training with \alexgda{} + Adam optimizer (i.e., Alex-Adam)},
    captionpos=t  % caption on the top
]
from copy import deepcopy
import torch
from models import Discriminator, Generator  # Custom library for modeling

# Create a Discriminator (x) and a Generator (y)
netD = Discriminator(...)   # $x_0$
netG = Generator(...)       # $y_0$
netG_ = deepcopy(netG)      # $\tilde{y}_0$

# Define the optimizers
optimizerD = torch.optim.Adam(netD.parameters(), ...)
optimizerG = torch.optim.Adam(netG.parameters(), ...)

dataloader = ...    # set of real images
num_epochs = ...    # number of epochs
criterion  = ...    # loss function 
label_real = ...    # label for real image e.g. all ones
label_fake = ...    # label for fake image e.g. all zeros
gamma = ...     # Alex-GDA parameter
delta = ...     # Alex-GDA parameter

for epoch in range(1, num_epochs+1):
  for data in dataloader:
    # Generate latent vectors
    noise = torch.randn(...)

    # Save $x_k$ to `netD_`
    netD_ = deepcopy(netD)  

    # Compute Discriminator error: $f(x_k, \tilde{y}_k)$
    errD_real = criterion(netD(data), label_real) 
    errD_fake = criterion(netD(netG_(noise)), label_fake)   # `netG_` == $\tilde{y}_k$
    errD = errD_real + errD_fake 
    
    # Update Discriminator: $x_{k+1}$ 
    optimizerD.zero_grad()
    errD.backward()
    optimizerD.step()

    # Interpolation/Extrapolation step for x:
    # Compute $\tilde{x}_{k+1} = \gamma * x_{k+1} + (1-\gamma) * x_k$ and save to `netD_`
    for online, target in zip(netD.parameters(), netD_.parameters()):
      target.data = gamma * online.data + (1 - gamma) * target.data
    
    # Save $y_k$ to `netG_`
    netG_ = deepcopy(netG)

    # Compute Generator error: $f(\tilde{x}_{k+1}, y_k)$
    errG = criterion(netD_(netG(noise)), label_real)    # `netD_` == $\tilde{x}_{k+1}$

    # Update Generator: $y_{k+1}$ 
    optimizerG.zero_grad()
    errG.backward()
    optimizerG.step()

    # Interpolation/Extrapolation step for y:
    # Compute $\tilde{y}_{k+1} = \delta * y_{k+1} + (1-\delta) * y_k$ and save to `netG_`
    for online, target in zip(netG.parameters(), netG_.parameters()):
      target.data = delta * online.data + (1 - delta) * target.data
\end{lstlisting}

%% file: sech.tex
\section{Guessing the Complexity Bound of Alt-GDA} 
\label{sec:h}

We have found numerical evidence based on the performance estimation program (PEP) \citep{drori14performance} that the upper complexity bound can be strictly smaller than $\gO(\kappa^{1.5})$ for \altgda{}, which we formally state in \cref{conj:alt}.

Reproducing the work by \citet{gupta24}, we devised a PEP-based tool that automatically optimizes the convergence rate of \textbf{\green{Sim-}}/\altgda{} under SCSC and Lipschitz gradient assumptions. 
While the original PEP is a tool for finding the worst-case convergence rate of a \emph{given} algorithm (with fixed and known parameters like step sizes) by solving a semidefinite programming problem, our tool tries to minimize this worst-case rate by finding optimal step sizes and optimal coefficients of the performance measure. 
Here, the performance measure is a linear combination of (1) the squared distance from the current iterate to the optimum, (2) the gradient norm at the current iterate, and (3) their interaction term (inner product between an iterate-optimum gap and a gradient norm), where the coefficients of the linear combination are part of optimization variables.

Using this tool, we can obtain an optimized convergence rate $r$ of \textbf{\green{Sim-}}/\altgda{} for each set of problem parameters $(\mu_x, \mu_y, L_x, L_y, L_{xy})$. 
(For convenience of exhibition, we set $\mu=\mu_x=\mu_y$ and $L = L_x = L_y = L_{xy}$ and define $\kappa = L / \mu$.)
Recall from \cref{eq:validilav} that the complexity can be expressed as $\frac{1}{1-r}$ except for the logarithmic factor. 
Hence, if we find how $\frac{1}{1-r}$ can be expressed as a function of $\kappa$, we will be able to guess the actual complexity in terms of $\kappa$.
We draw log-log plots between $\frac{1}{1-r}$ and $\kappa$ and observe its slope, which would be the exponent of $\kappa$ in the complexity.
Here we tune $\kappa\in \{10^{1}, 10^{1.2}, 10^{1.4}, \dots, 10^{3}\}$, and we compute the median slope of line segments, each of them connecting a pair of adjacent points.

As shown in \cref{fig:bnbpep}, the graphs for both algorithms appear close to a straight line.
For \simgda{}, we observe the optimal complexity is $\approx \kappa^{1.999}$: it is tight up to numerical error. On the other hand, for \altgda{}, the observed lowest possible complexity is $\approx \kappa^{1.385}$ (if we utilize a pair of consecutive iterates $\boldsymbol{z}_{k} \rightarrow \boldsymbol{z}_{k+1}$): See \cref{fig:bnbpep}.

\begin{figure}[ht]
    \centering
    \includegraphics[width=0.45\linewidth]{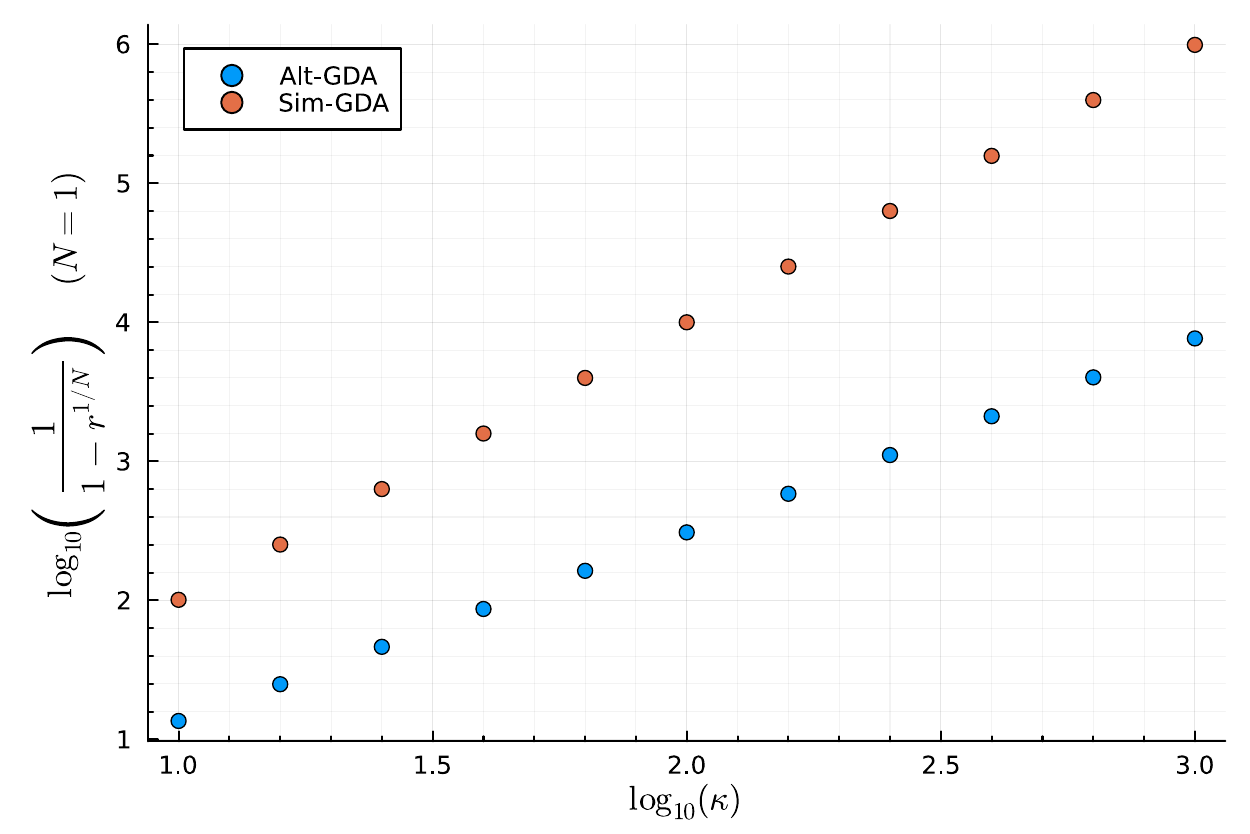}
    ~
    \includegraphics[width=0.45\linewidth]{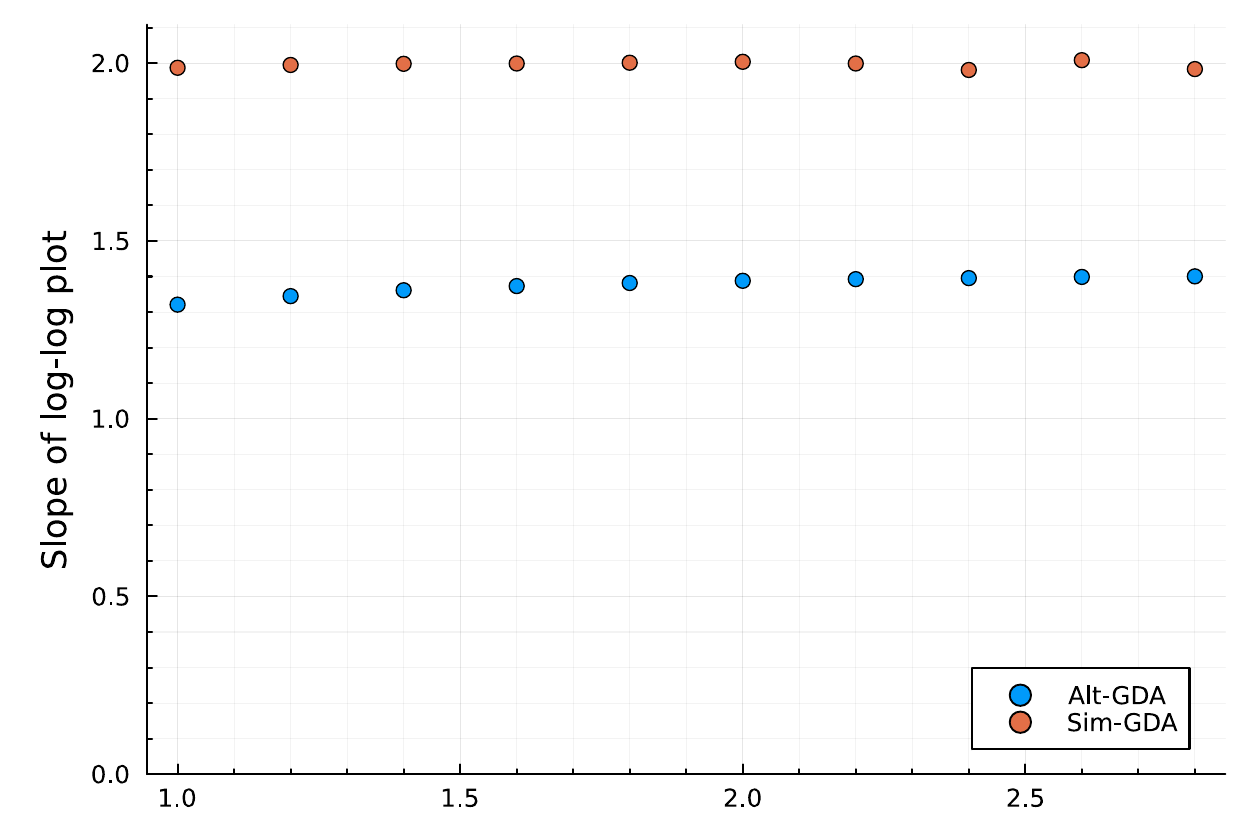}
    \caption{\textbf{Guessing the complexity bound of \textbf{\green{Sim-}}/\altgda{}.} \textbf{Left:} log-log plot between $\kappa = L/\mu$ and the near-optimal worst-case complexity. \textbf{Right:} Slope of the log-log plot. Each point corresponds to the slope of a line segment connecting a pair of adjacent points in the left plot.}
    \label{fig:bnbpep}
\end{figure}

Nevertheless, we cannot assure that this \textit{proves} that the tight complexity of \altgda{} of rate $\gO(\kappa^{1.385})$ is \emph{tight}. 
This is mainly because, in fact, our tool is not perfect in terms of the function class.
Although our tool implements every condition of SCSC and Lipschitz gradients as constraints of an optimization problem, it is not well understood (especially for minimax problems) whether such an implementation can properly simulate the class of SCSC functions with Lipschitz gradients; rather, it can only simulate a slightly \emph{larger} function class including SCSC and Lipschitz-gradient functions (this is similar to the case of monotone and Lipschitz operators \citep{ryu20}).
Thus, the numerical value $1.385$ is not tight and the true exponent can be smaller for the actual SCSC Lipschitz-gradient functions.
In other words, the complexity can be smaller than $\gO(\kappa^{1.385})$.
Nonetheless, our results altogether corroborate that the upper complexity bound of \altgda{} must be strictly smaller than $\gO(\kappa^{1.5})$.

%% file: main.bbl
\begin{thebibliography}{47}
\providecommand{\natexlab}[1]{#1}
\providecommand{\url}[1]{\texttt{#1}}
\expandafter\ifx\csname urlstyle\endcsname\relax
  \providecommand{\doi}[1]{doi: #1}\else
  \providecommand{\doi}{doi: \begingroup \urlstyle{rm}\Url}\fi

\bibitem[Arjovsky et~al.(2017)Arjovsky, Chintala, and Bottou]{arjovsky2017wasserstein}
Arjovsky, M., Chintala, S., and Bottou, L.
\newblock Wasserstein generative adversarial networks.
\newblock In \emph{International Conference on Machine Learning (ICML)}, pp.\  214--223. PMLR, 2017.

\bibitem[Azizian et~al.(2020)Azizian, Scieur, Mitliagkas, Lacoste{-}Julien, and Gidel]{azizian20}
Azizian, W., Scieur, D., Mitliagkas, I., Lacoste{-}Julien, S., and Gidel, G.
\newblock Accelerating smooth games by manipulating spectral shapes.
\newblock In Chiappa, S. and Calandra, R. (eds.), \emph{The 23rd International Conference on Artificial Intelligence and Statistics, {AISTATS} 2020, 26-28 August 2020, Online [Palermo, Sicily, Italy]}, volume 108 of \emph{Proceedings of Machine Learning Research}, pp.\  1705--1715. {PMLR}, 2020.
\newblock URL \url{http://proceedings.mlr.press/v108/azizian20a.html}.

\bibitem[Bailey et~al.(2020)Bailey, Gidel, and Piliouras]{bailey2020finite}
Bailey, J.~P., Gidel, G., and Piliouras, G.
\newblock Finite regret and cycles with fixed step-size via alternating gradient descent-ascent.
\newblock In \emph{Conference on Learning Theory (COLT)}, pp.\  391--407. PMLR, 2020.

\bibitem[Bansal \& Gupta(2019)Bansal and Gupta]{bansalgupta19}
Bansal, N. and Gupta, A.
\newblock Potential-function proofs for gradient methods.
\newblock \emph{Theory of Computing}, 15(1):\penalty0 1--32, 2019.

\bibitem[Bertsekas(1999)]{bertsekas99nonlinear}
Bertsekas, D.
\newblock \emph{Nonlinear Programming}.
\newblock Athena Scientific, 1999.

\bibitem[Das~Gupta et~al.(2023)Das~Gupta, Van~Parys, and Ryu]{gupta24}
Das~Gupta, S., Van~Parys, B. P.~G., and Ryu, E.~K.
\newblock Branch-and-bound performance estimation programming: a unified methodology for constructing optimal optimization methods.
\newblock \emph{Math. Program.}, 204\penalty0 (1–2):\penalty0 567–639, jun 2023.
\newblock ISSN 0025-5610.
\newblock \doi{10.1007/s10107-023-01973-1}.
\newblock URL \url{https://doi.org/10.1007/s10107-023-01973-1}.

\bibitem[Dem'yanov \& Pevnyi(1972)Dem'yanov and Pevnyi]{demyanov1972numerical}
Dem'yanov, V. and Pevnyi, A.
\newblock Numerical methods for finding saddle points.
\newblock \emph{USSR Computational Mathematics and Mathematical Physics}, 12\penalty0 (5):\penalty0 11--52, 1972.
\newblock ISSN 0041-5553.
\newblock \doi{https://doi.org/10.1016/0041-5553(72)90002-X}.
\newblock URL \url{https://www.sciencedirect.com/science/article/pii/004155537290002X}.

\bibitem[Deng(2012)]{deng2012mnist}
Deng, L.
\newblock The mnist database of handwritten digit images for machine learning research.
\newblock \emph{IEEE Signal Processing Magazine}, 29\penalty0 (6):\penalty0 141--142, 2012.

\bibitem[Drori \& Teboulle(2014)Drori and Teboulle]{drori14performance}
Drori, Y. and Teboulle, M.
\newblock Performance of first-order methods for smooth convex minimization: a novel approach.
\newblock \emph{Math. Program.}, 145\penalty0 (1-2):\penalty0 451--482, 2014.
\newblock \doi{10.1007/s10107-013-0653-0}.
\newblock URL \url{https://doi.org/10.1007/s10107-013-0653-0}.

\bibitem[Gidel et~al.(2019{\natexlab{a}})Gidel, Berard, Vignoud, Vincent, and Lacoste-Julien]{gidel2018a}
Gidel, G., Berard, H., Vignoud, G., Vincent, P., and Lacoste-Julien, S.
\newblock A variational inequality perspective on generative adversarial networks.
\newblock In \emph{International Conference on Learning Representations (ICLR)}, 2019{\natexlab{a}}.
\newblock URL \url{https://openreview.net/forum?id=r1laEnA5Ym}.

\bibitem[Gidel et~al.(2019{\natexlab{b}})Gidel, Hemmat, Pezeshki, Le~Priol, Huang, Lacoste-Julien, and Mitliagkas]{gidel2019negative}
Gidel, G., Hemmat, R.~A., Pezeshki, M., Le~Priol, R., Huang, G., Lacoste-Julien, S., and Mitliagkas, I.
\newblock Negative momentum for improved game dynamics.
\newblock In \emph{International Conference on Artificial Intelligence and Statistics (AISTATS)}, pp.\  1802--1811. PMLR, 2019{\natexlab{b}}.

\bibitem[Goodfellow et~al.(2020)Goodfellow, Pouget-Abadie, Mirza, Xu, Warde-Farley, Ozair, Courville, and Bengio]{goodfellow2020generative}
Goodfellow, I., Pouget-Abadie, J., Mirza, M., Xu, B., Warde-Farley, D., Ozair, S., Courville, A., and Bengio, Y.
\newblock Generative adversarial networks.
\newblock \emph{Communications of the ACM}, 63\penalty0 (11):\penalty0 139--144, 2020.

\bibitem[Grove \& Ladas(2004)Grove and Ladas]{grove2004periodicities}
Grove, E.~A. and Ladas, G.
\newblock \emph{Periodicities in nonlinear difference equations}, volume~4.
\newblock CRC Press, 2004.

\bibitem[Gulrajani et~al.(2017)Gulrajani, Ahmed, Arjovsky, Dumoulin, and Courville]{gulrajani2017improved}
Gulrajani, I., Ahmed, F., Arjovsky, M., Dumoulin, V., and Courville, A.~C.
\newblock Improved training of wasserstein gans.
\newblock \emph{Advances in Neural Information Processing Systems (NeurIPS)}, 30, 2017.

\bibitem[Haynsworth(1968)]{haynsworth1968schur}
Haynsworth, E.~V.
\newblock On the schur complement.
\newblock \emph{Basel Mathematical Notes}, 20:\penalty0 17, 1968.

\bibitem[Heusel et~al.(2017)Heusel, Ramsauer, Unterthiner, Nessler, and Hochreiter]{heusel2017gans}
Heusel, M., Ramsauer, H., Unterthiner, T., Nessler, B., and Hochreiter, S.
\newblock {GAN}s trained by a two time-scale update rule converge to a local {Nash} equilibrium.
\newblock \emph{Advances in Neural Information Processing Systems (NeurIPS)}, 30, 2017.

\bibitem[Horn \& Johnson(2012)Horn and Johnson]{horn2012matrix}
Horn, R.~A. and Johnson, C.~R.
\newblock \emph{Matrix Analysis}.
\newblock Cambridge University Press, Cambridge, England, 2 edition, October 2012.

\bibitem[Kalman \& Bertram(1960)Kalman and Bertram]{kalman1960control}
Kalman, R.~E. and Bertram, J.~E.
\newblock {Control System Analysis and Design Via the “Second Method” of Lyapunov: I—Continuous-Time Systems}.
\newblock \emph{Journal of Basic Engineering}, 82\penalty0 (2):\penalty0 371--393, 06 1960.
\newblock ISSN 0021-9223.
\newblock \doi{10.1115/1.3662604}.
\newblock URL \url{https://doi.org/10.1115/1.3662604}.

\bibitem[Kingma \& Ba(2015)Kingma and Ba]{kingma2015adam}
Kingma, D.~P. and Ba, J.
\newblock Adam: A method for stochastic optimization.
\newblock In \emph{International Conference on Learning Representations (ICLR)}, 2015.

\bibitem[Korpelevich(1976)]{korpelevich1976extragradient}
Korpelevich, G.~M.
\newblock The extragradient method for finding saddle points and other problems.
\newblock \emph{Matecon}, 12:\penalty0 747--756, 1976.

\bibitem[Krizhevsky et~al.(2009)Krizhevsky, Hinton, et~al.]{krizhevsky2009learning}
Krizhevsky, A., Hinton, G., et~al.
\newblock Learning multiple layers of features from tiny images, 2009.

\bibitem[Latorre et~al.(2023)Latorre, Krawczuk, Dadi, Pethick, and Cevher]{latorre2023finding}
Latorre, F., Krawczuk, I., Dadi, L.~T., Pethick, T., and Cevher, V.
\newblock Finding actual descent directions for adversarial training.
\newblock In \emph{International Conference on Learning Representations (ICLR)}, 2023.
\newblock URL \url{https://openreview.net/forum?id=I3HCE7Ro78H}.

\bibitem[Lee \& Kim(2021)Lee and Kim]{lee2021fast}
Lee, S. and Kim, D.
\newblock Fast extra gradient methods for smooth structured nonconvex-nonconcave minimax problems.
\newblock \emph{Advances in Neural Information Processing Systems (NeurIPS)}, 34:\penalty0 22588--22600, 2021.

\bibitem[Li et~al.(2019)Li, Wu, Cui, Dong, Fang, and Russell]{li2019robust}
Li, S., Wu, Y., Cui, X., Dong, H., Fang, F., and Russell, S.
\newblock Robust multi-agent reinforcement learning via minimax deep deterministic policy gradient.
\newblock In \emph{Conference on Artificial Intelligence (AAAI)}, volume~33, pp.\  4213--4220, 2019.

\bibitem[Liu et~al.(2020)Liu, Yuan, Ying, and Yang]{liu2020stochastic}
Liu, M., Yuan, Z., Ying, Y., and Yang, T.
\newblock Stochastic {AUC} maximization with deep neural networks.
\newblock In \emph{International Conference on Learning Representations (ICLR)}, 2020.
\newblock URL \url{https://openreview.net/forum?id=HJepXaVYDr}.

\bibitem[Madry et~al.(2018)Madry, Makelov, Schmidt, Tsipras, and Vladu]{madry2018towards}
Madry, A., Makelov, A., Schmidt, L., Tsipras, D., and Vladu, A.
\newblock Towards deep learning models resistant to adversarial attacks.
\newblock In \emph{International Conference on Learning Representations (ICLR)}. OpenReview.net, 2018.

\bibitem[Mescheder et~al.(2017)Mescheder, Nowozin, and Geiger]{mescheder2017numerics}
Mescheder, L., Nowozin, S., and Geiger, A.
\newblock The numerics of gans.
\newblock \emph{Advances in Neural Information Processing Systems (NeurIPS)}, 30, 2017.

\bibitem[Mokhtari et~al.(2019)Mokhtari, Ozdaglar, and Pattathil]{mokhtari2019a}
Mokhtari, A., Ozdaglar, A.~E., and Pattathil, S.
\newblock A unified analysis of extra-gradient and optimistic gradient methods for saddle point problems: Proximal point approach.
\newblock In \emph{International Conference on Artificial Intelligence and Statistics (AISTATS)}, 2019.
\newblock URL \url{https://api.semanticscholar.org/CorpusID:59222714}.

\bibitem[Palaniappan \& Bach(2016)Palaniappan and Bach]{palaniappan16}
Palaniappan, B. and Bach, F.
\newblock Stochastic variance reduction methods for saddle-point problems.
\newblock In Lee, D., Sugiyama, M., Luxburg, U., Guyon, I., and Garnett, R. (eds.), \emph{Advances in Neural Information Processing Systems (NeurIPS)}, volume~29. Curran Associates, Inc., 2016.
\newblock URL \url{https://proceedings.neurips.cc/paper_files/paper/2016/file/1aa48fc4880bb0c9b8a3bf979d3b917e-Paper.pdf}.

\bibitem[Park \& Ryu(2022)Park and Ryu]{park2022exact}
Park, J. and Ryu, E.~K.
\newblock Exact optimal accelerated complexity for fixed-point iterations.
\newblock In \emph{International Conference on Machine Learning (ICML)}, pp.\  17420--17457. PMLR, 2022.

\bibitem[Paszke et~al.(2019)Paszke, Gross, Massa, Lerer, Bradbury, Chanan, Killeen, Lin, Gimelshein, Antiga, et~al.]{paszke2019pytorch}
Paszke, A., Gross, S., Massa, F., Lerer, A., Bradbury, J., Chanan, G., Killeen, T., Lin, Z., Gimelshein, N., Antiga, L., et~al.
\newblock Pytorch: An imperative style, high-performance deep learning library.
\newblock \emph{Advances in neural information processing systems}, 32, 2019.

\bibitem[Popov(1980)]{popov1980modification}
Popov, L.~D.
\newblock A modification of the arrow-hurwicz method for search of saddle points.
\newblock \emph{Mathematical notes of the Academy of Sciences of the USSR}, 28:\penalty0 845--848, 1980.

\bibitem[Ramirez et~al.(2023)Ramirez, Sukumaran, Bertrand, and Gidel]{ramirez2023omega}
Ramirez, J., Sukumaran, R., Bertrand, Q., and Gidel, G.
\newblock Omega: Optimistic ema gradients.
\newblock In \emph{International Conference on Machine Learning (ICML)}, 2023.

\bibitem[Ryu et~al.(2020)Ryu, Taylor, Bergeling, and Giselsson]{ryu20}
Ryu, E.~K., Taylor, A.~B., Bergeling, C., and Giselsson, P.
\newblock Operator splitting performance estimation: Tight contraction factors and optimal parameter selection.
\newblock \emph{SIAM Journal on Optimization}, 30\penalty0 (3):\penalty0 2251--2271, 2020.
\newblock \doi{10.1137/19M1304854}.
\newblock URL \url{https://doi.org/10.1137/19M1304854}.

\bibitem[Sinha et~al.(2018)Sinha, Namkoong, and Duchi]{sinha2018certifiable}
Sinha, A., Namkoong, H., and Duchi, J.
\newblock Certifying some distributional robustness with principled adversarial training.
\newblock In \emph{International Conference on Learning Representations (ICLR)}, 2018.
\newblock URL \url{https://openreview.net/forum?id=Hk6kPgZA-}.

\bibitem[Taylor et~al.(2018)Taylor, Van~Scoy, and Lessard]{taylor18lyapunov}
Taylor, A., Van~Scoy, B., and Lessard, L.
\newblock {L}yapunov functions for first-order methods: Tight automated convergence guarantees.
\newblock In \emph{International Conference on Machine Learning (ICML)}, 2018.

\bibitem[von Neumann(1928)]{neumann1928zur}
von Neumann, J.
\newblock Zur theorie der gesellschaftsspiele.
\newblock \emph{Mathematische Annalen}, 100\penalty0 (1):\penalty0 295--320, 1928.
\newblock \doi{10.1007/BF01448847}.
\newblock URL \url{https://doi.org/10.1007/BF01448847}.

\bibitem[Ying et~al.(2016)Ying, Wen, and Lyu]{ying2016stochastic}
Ying, Y., Wen, L., and Lyu, S.
\newblock Stochastic online {AUC} maximization.
\newblock \emph{Advances in Neural Information Processing Systems (NeurIPS)}, 29, 2016.

\bibitem[Yoon \& Ryu(2021)Yoon and Ryu]{yoon2021accelerated}
Yoon, T. and Ryu, E.~K.
\newblock Accelerated algorithms for smooth convex-concave minimax problems with o (1/k\^{}2) rate on squared gradient norm.
\newblock In \emph{International Conference on Machine Learning (ICML)}, pp.\  12098--12109. PMLR, 2021.

\bibitem[Yoon \& Ryu(2022)Yoon and Ryu]{yoon2022accelerated}
Yoon, T. and Ryu, E.~K.
\newblock Accelerated minimax algorithms flock together.
\newblock \emph{arXiv preprint arXiv:2205.11093}, 2022.

\bibitem[Yu et~al.(2015)Yu, Zhang, Song, Seff, and Xiao]{yu15lsun}
Yu, F., Zhang, Y., Song, S., Seff, A., and Xiao, J.
\newblock Lsun: Construction of a large-scale image dataset using deep learning with humans in the loop.
\newblock \emph{arXiv preprint arXiv:1506.03365}, 2015.

\bibitem[Yu et~al.(2022)Yu, Lin, Mazumdar, and Jordan]{yu2022fast}
Yu, Y., Lin, T., Mazumdar, E.~V., and Jordan, M.
\newblock Fast distributionally robust learning with variance-reduced min-max optimization.
\newblock In \emph{International Conference on Artificial Intelligence and Statistics (AISTATS)}, pp.\  1219--1250. PMLR, 2022.

\bibitem[Yuan et~al.(2021)Yuan, Yan, Sonka, and Yang]{yuan2021large}
Yuan, Z., Yan, Y., Sonka, M., and Yang, T.
\newblock Large-scale robust deep {AUC} maximization: A new surrogate loss and empirical studies on medical image classification.
\newblock In \emph{IEEE International Conference on Computer Vision (ICCV(}, pp.\  3040--3049, 2021.

\bibitem[Zamani et~al.(2022)Zamani, Abbaszadehpeivasti, and de~Klerk]{zamani22convergence}
Zamani, M., Abbaszadehpeivasti, H., and de~Klerk, E.
\newblock Convergence rate analysis of the gradient descent-ascent method for convex-concave saddle-point problems, 2022.

\bibitem[Zhang(2006)]{zhang2006schur}
Zhang, F.
\newblock \emph{The Schur complement and its applications}, volume~4.
\newblock Springer Science \& Business Media, 2006.

\bibitem[Zhang \& Yu(2020)Zhang and Yu]{zhang2020convergence}
Zhang, G. and Yu, Y.
\newblock Convergence of gradient methods on bilinear zero-sum games.
\newblock In \emph{International Conference on Learning Representations}, 2020.
\newblock URL \url{https://openreview.net/forum?id=SJlVY04FwH}.

\bibitem[Zhang et~al.(2022)Zhang, Wang, Lessard, and Grosse]{zhang22near}
Zhang, G., Wang, Y., Lessard, L., and Grosse, R.~B.
\newblock Near-optimal local convergence of alternating gradient descent-ascent for minimax optimization.
\newblock In \emph{International Conference on Artificial Intelligence and Statistics (AISTATS)}, 2022.

\end{thebibliography}
